\documentclass[
    a4paper,
    english,
    oneside]{article}

\usepackage[utf8]{inputenc}
\usepackage[T1]{fontenc}
\usepackage{expl3}
\usepackage[english]{babel}
\usepackage[a4paper,top=2.5cm,bottom=2.5cm,left=3cm,right=3cm]{geometry}

\usepackage{%
    amssymb,
    amsfonts,
    amsthm,
    mathtools,
    amsmath}

\input{./macros}

\usepackage[
    backend=biber,
    style=alphabetic,
    doi=false,
    isbn=false,
    url=true,
    date=year,
    maxbibnames=99,
    giveninits=true,
      ]{biblatex}

\AtEveryBibitem{\clearfield{series}}
\AtEveryBibitem{\clearfield{labeladdress}}

\definecolor{LinkBlue}{cmyk}{1,0.50,0,0}
\usepackage[
    hidelinks,
    colorlinks=true,
    unicode,
    linkcolor=LinkBlue,
    citecolor=LinkBlue,
    urlcolor=LinkBlue]{hyperref}

\title{A Categorical Approach to Nilspace Theory}
\author{Noa Bihlmaier\thanks{
    \textsc{Mathematisches Institut, Universität Bonn, Endenicher Allee 60, 53115 Bonn, Germany} \\
    \indent\textit{Email address:} \href{mailto:nobi@math.uni-bonn.de}{\texttt{nobi@math.uni-bonn.de}}}
        \and
        Nick Ruoff\thanks{
    \textsc{Mathematisches Institut, Universität Bonn, Endenicher Allee 60, 53115 Bonn, Germany} \\
    \indent\textit{Email address:} \href{mailto:ruoff@math.uni-bonn.de}{\texttt{ruoff@math.uni-bonn.de}}}}

\begin{document}
\maketitle
\thispagestyle{empty}

\begin{abstract}
Originating in ergodic structure theory, nilspace theory is concerned with the study of cubical spaces.
A first key result in this theory, the \emph{weak structure theorem}, is the representation of a concrete fibrant compact cubespace as
a tower of its canonical truncations, exhibiting each step over the previous one as a principal bundle by a compact abelian group.
We present a categorical proof of the weak structure theorem by systematically interpreting the relevant concepts in an appropriate context.

\end{abstract}

\tableofcontents
\clearpage

\section{Introduction}

The theory of \emph{nilspaces} originally emerged in the late 2010s after Host and Kra and independently Ziegler proved the convergence of multiple ergodic averages \cite{Host2005,Ziegler2006} by means of what is now called \emph{Host-Kra structure theory} \cite{Host2018,Eisner2025}.
One of the key insights in their proofs is that every ergodic measure-preserving $\Z$-system $X$
decomposes into \emph{structured factors $Z_k$} which capture the convergence behavior of the $k+1$-fold ergodic averages (or: the factors are \emph{characteristic} for the convergence of the $k$-fold ergodic averages),
and \emph{random parts $Z_k^\perp$} on which the $k+1$-fold ergodic averages tend to zero.
The first two such decompositions are easy to describe:
$Z_0$ is the invariant factor of the system, i.e., where the action is trivial,
and corresponds to von Neumann's Mean Ergodic Theorem,
and $Z_1$ is the \emph{Kronecker factor} on which the system is a rotational system on a compact abelian group by the Halmos-von Neumann Theorem.
The $Z_k$ make up a tower of factors of $X$
\begin{center}
\begin{tikzcd}[ampersand replacement=\&]
	X \& \dots \& {Z_{k}} \& {Z_{k-1}} \& \dots \& {Z_1} \& {Z_0}
	\arrow[two heads, from=1-1, to=1-2]
	\arrow[two heads, from=1-2, to=1-3]
	\arrow[two heads, from=1-3, to=1-4]
	\arrow[two heads, from=1-4, to=1-5]
	\arrow[two heads, from=1-5, to=1-6]
	\arrow[two heads, from=1-6, to=1-7]
\end{tikzcd}
\end{center}
which exhibits $Z_k$ as a skew-product extension of $Z_{k-1}$ by a compact abelian group $A_k$,
i.e., this tower is an iterated abelian bundle.
This statement is known under the name of \emph{weak structure theorem}.
From there, the \emph{strong structure theorem} describes the system $Z_k$ as a limit of translational systems on $k$-step nilmanifolds,
i.e., translational systems on a homogeneous space $G/\Lambda$ where $G$ is a $k$-step nilpotent Lie group and $\Lambda$ is a lattice therein.
This description is sufficient to prove convergence of the $k$-fold ergodic averages,
see also \cite{Leibman2006}.

Interestingly, underlying the Host-Kra theory there are many \emph{cubical patterns} \cite{Host2004,Host2005,Host2005a,Host2008, Host2010}.
In fact, the Host-Kra factors are the universal characteristic factors for cubical ergodic averages,
and Host-Kra-Gowers semi-norms, cubegroups and iterated self-joinings
are described using indices describing the combinatorial behavior of discrete cubes ${\bbox}^n = \{0,1\}^n$.
Extracting the algebraic properties of these cubical patterns is at the core of nilspace theory,
pioneered by Antol\'in Camarena and Szegedy in \cite{Camarena2010},
who exploit the cubical patterns to obtain an analogous weak and strong structure theorem,
showing that what matters is not the ergodic system itself
but merely its underlying cubical structure.
More groundwork on nilspaces has been put forward in a series of papers by Gutman, Manners and Varj\'u \cite{Gutman2020,Gutman2020a,Gutman2020b}
and by Candela \cite{Candela2017,Candela2017a},
who both systematized and improved the weak and strong structure theorem for nilspaces.
From there, it is straightforward to deduce topological Host-Kra theory \cite{Host2010,Shao2012}
from the strong structure theory of nilspaces, see \cite{Glasner2018,Gutman2020b,Glasner2023}.
But nilspace theory also gives ways of proving the strong structure theorem for measure-preserving systems,
as in \cite{Gutman2023} and \cite{Candela2020,Candela2023}.
Apart from its applications to ergodic theory, nilspace theory has many other applications,
e.g. in Inverse Gowers Theory, Higher Order Fourier Analysis and additive combinatorics,
see \cite{Candela2021,Candela2022,Jamneshan2026d}.

\subsection*{Overview}

The goal of this text is to give a new proof of (a generalized version of) the following theorem.

\begin{theorem*}[Corollary~\ref{cor:compact-weak-structure}]
Let $X$ be a compact ergodic nilspace.
Then there is a tower of surjective fibrations
\begin{center}
\begin{tikzcd}
	X & \cdots & {\tau_2X} & {\tau_1X} & {\tau_0X}
	\arrow[from=1-1, to=1-2]
	\arrow[from=1-2, to=1-3]
	\arrow[from=1-3, to=1-4]
	\arrow[from=1-4, to=1-5]
\end{tikzcd}
\end{center}
such that $\tau_n X$ is a compact $n$-step nilspace
and the compact group $\pi_n(X,x_0)$ acts freely on $\tau_n X$
such that the orbits are the fibers of the quotient $\tau_n X\to\tau_{n-1} X$.
\end{theorem*}

To do so, we thoroughly study and introduce the concepts of cubesets, fibrations, truncations, structure groups and related notions
from a categorical point of view.

The central definition in nilspace theory is that of a \emph{cubeset} $X$.
It consists of a family of sets $X(n)$ of $n$-dimensional cubes for every $n\in\N$,
compatible with certain restriction maps.
This information can be encoded in that $X$ is a \emph{presheaf} on cubes.
Rather than saying what a single cube is, it is easier to say what the category of cubes is,
i.e., how the cubes of different dimension interact with each other.

\begin{definition*}[Definition~\ref{def:cube_cat}]
The \emph{category of cubes $\bbox$} is the category with objects ${\bbox}^n = \{0,1\}^n$ for $n\in\N_0$
and morphisms $f\colon{\bbox}^n\to{\bbox}^m$ given by affine linear maps $\Z^n\to\Z^m$ that map $\{0,1\}^n$ into $\{0,1\}^m$.
\end{definition*}

The first $n$-cubes can be pictured as follows, drawing the images of every injective map into the cube.

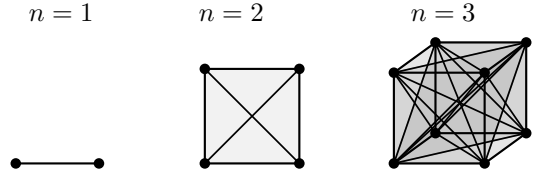
\begin{center}

\pgfdeclarelayer{symfacefill}
\pgfdeclarelayer{symfaceoutline}
\pgfsetlayers{symfacefill,symfaceoutline,main}

\newcommand{\SymFace}[1]{%
  \begin{pgfonlayer}{symfacefill}
    \path[fill=gray, fill opacity=0.10] #1 -- cycle;
  \end{pgfonlayer}
  \begin{pgfonlayer}{symfaceoutline}
    \path[draw=gray!60,line width=0.45pt] #1 -- cycle;
  \end{pgfonlayer}
}

\newcommand{\SymDiagFace}[1]{%
  \begin{pgfonlayer}{symfacefill}
    \path[fill=gray, fill opacity=0.16] #1 -- cycle;
  \end{pgfonlayer}
  \begin{pgfonlayer}{symfaceoutline}
    \path[draw=gray!70,line width=0.5pt] #1 -- cycle;
  \end{pgfonlayer}
}

\begin{tikzpicture}[
    x=1cm,y=1cm,
    font=\small,
    line cap=round,
    line join=round,
    vertex/.style={circle,fill,inner sep=1.4pt},
    mainedge/.style={draw=black, line width=0.8pt},
    affdiag/.style={draw=black, line width=0.65pt},
    simpedge/.style={
        draw=black,
        line width=0.75pt,
        -{Stealth[length=1.6mm,width=1.1mm]},
        shorten <= 8.0pt,
        shorten >= 8.0pt
    },
    rowlab/.style={anchor=east},
    collab/.style={font=\normalsize},
scale=0.5
]

\node[collab] at ( 1.2,  4.0) {$n=1$};
\node[collab] at ( 5.7,  4.0) {$n=2$};
\node[collab] at (11.3,  4.0) {$n=3$};

\begin{scope}
    \coordinate (A) at (0,0);
    \coordinate (B) at (2.2,0);

    \draw[mainedge] (A)--(B);
    \node[vertex] at (A) {};
    \node[vertex] at (B) {};
\end{scope}

\begin{scope}[shift={(5,0)}]
    \coordinate (A) at (0,0);
    \coordinate (B) at (2.5,0);
    \coordinate (C) at (0,2.5);
    \coordinate (D) at (2.5,2.5);

    \SymFace{(A)--(B)--(D)--(C)}

    \draw[mainedge] (A)--(B);
    \draw[mainedge] (B)--(D);
    \draw[mainedge] (D)--(C);
    \draw[mainedge] (C)--(A);

    \draw[affdiag] (A)--(D);
    \draw[affdiag] (B)--(C);

    \foreach \P in {A,B,C,D} {
        \node[vertex] at (\P) {};
    }
\end{scope}

\begin{scope}[shift={(10,0)}]
    \coordinate (A) at (0,0);
    \coordinate (B) at (2.4,0);
    \coordinate (C) at (0,2.4);
    \coordinate (D) at (2.4,2.4);

    \coordinate (E) at (1.1,0.8);
    \coordinate (F) at (3.5,0.8);
    \coordinate (G) at (1.1,3.2);
    \coordinate (H) at (3.5,3.2);

    \SymFace{(A)--(B)--(D)--(C)}
    \SymFace{(E)--(F)--(H)--(G)}
    \SymFace{(A)--(B)--(F)--(E)}
    \SymFace{(C)--(D)--(H)--(G)}
    \SymFace{(A)--(C)--(G)--(E)}
    \SymFace{(B)--(D)--(H)--(F)}

    \SymDiagFace{(A)--(B)--(H)--(G)}
    \SymDiagFace{(A)--(C)--(H)--(F)}
    \SymDiagFace{(A)--(E)--(H)--(D)}

    \draw[mainedge] (A)--(B);
    \draw[mainedge] (B)--(D);
    \draw[mainedge] (D)--(C);
    \draw[mainedge] (C)--(A);

    \draw[mainedge] (E)--(F);
    \draw[mainedge] (F)--(H);
    \draw[mainedge] (H)--(G);
    \draw[mainedge] (G)--(E);

    \draw[mainedge] (A)--(E);
    \draw[mainedge] (B)--(F);
    \draw[mainedge] (C)--(G);
    \draw[mainedge] (D)--(H);

    \draw[affdiag] (A)--(D);
    \draw[affdiag] (B)--(C);
    \draw[affdiag] (E)--(H);
    \draw[affdiag] (F)--(G);
    \draw[affdiag] (A)--(F);
    \draw[affdiag] (B)--(E);
    \draw[affdiag] (C)--(H);
    \draw[affdiag] (D)--(G);
    \draw[affdiag] (A)--(G);
    \draw[affdiag] (C)--(E);
    \draw[affdiag] (B)--(H);
    \draw[affdiag] (D)--(F);

    \draw[affdiag] (A)--(H);
    \draw[affdiag] (B)--(G);
    \draw[affdiag] (C)--(F);
    \draw[affdiag] (D)--(E);

    \foreach \P in {A,B,C,D,E,F,G,H} {
        \node[vertex] at (\P) {};
    }
\end{scope}
\end{tikzpicture}
\end{center}

There is an important subclass of morphisms in $\bbox$:
the morphisms $f\colon{\bbox}^n\to{\bbox}^m$ that are the restriction of a linear map.
Since they also have to map the unit cube $\{0,1\}^n$ into the unit cube,
each such linear map arises from a (not necessarily invertible) permutation of coordinates and thus
is characterized by a map of sets $\{1,\ldots,m\}\to\{1,\ldots,n\}$.
There's only one caveat: it might be possible for $f$ to be constant zero in some coordinates,
so the map of sets $\{1,\ldots,m\}\to\{1,\ldots,n\}$ only is partially defined.

\begin{definition*}[Definition~\ref{def:indexcats}]
\emph{Segal's category $\bbGamma$} is the opposite of the category of finite sets
and partially defined maps.
\end{definition*}

With this adapted category, every morphism indeed arises as a coordinate permutation
which upgrades to an equivalence of categories.

\begin{proposition*}[Proposition~\ref{lem:iso-gamma-bblox}]
There is an equivalence of categories between $\bbGamma$ and the subcategory of $\bbox$
consisting of the linear maps.
\end{proposition*}

It turns out that there is a purely categorical, coordinate-free construction to arrive from Segal's category $\bbGamma$
at the category of cubes $\bbox$.

\begin{proposition*}[Proposition~\ref{prop:cube_cat_linear}]
The cube category $\bbox$ is the unstraightening
\[
    \mathrm{Un}(\ast/\F_2^\bullet)=\int_{\bbGamma} \ast/\F_2^\bullet\to\bbGamma
\]
of the functor $\ast/\F_2^\bullet\colon\bbGamma\to\Cat$ that assigns to every $n$-element set
the delooping of the group $\F_2^n = (\Z/2\Z)^n$.
\end{proposition*}

This proposition ultimately relies on the fact that a morphism of cubes
can be uniquely decomposed into a linear map followed by a \emph{flip},
i.e. a map that only reverses the orientation of the cube along certain coordinate axes
(Lemma~\ref{lem:linearization-auto-bblox}).
The description of the cube category $\bbox$ as unstraightening over $\bbGamma$
now enables us to lift certain nice properties of $\bbGamma$ to $\bbox$:
among them a classification of epi- and monomorphisms (Lemma~\ref{lem:bbox-epi-mono}),
the existence of a Cisinski generalized Reedy structure (Lemma~\ref{lem:bbox-reedy-cat})
and the cartesian monoidal structure (Corollary~\ref{prop:bbox-symm-monoid}).

Now given the basic building blocks for the theory, a \emph{cubeset $X$} is some kind of space made out of cubes.
So $X$ consists of a collection of $n$-cubes $X(n)$ for every $n\in\N_0$,
and for any map $f\colon{\bbox}^n\to{\bbox}^m$ between cubes,
there is a restriction $X(f)\colon X(m)\to X(n)$, compatible with composition and the identity map.
This justifies calling the set $X(n)$ the collection of $n$-cubes.

\begin{definition*}[Definition~\ref{def:cubeset}]
A \emph{cubeset $X$} is a presheaf of sets on the category of cubes $\bbox$,
and the category $\cSet$ of cubesets is the presheaf category $\PSh(\bbox)=\Fun(\bbox^\op,\Set)$.
\end{definition*}

We briefly recall all relevant properties of the presheaf category of cubes in Section~\ref{ssec:cubesets-basics} and study their relation to presheaves on $\bbGamma$.
Most notably, we give an alternative description of the presheaf category
over a category that arises as the Grothendieck construction of a functor $F\colon\cC\to\Grp$
(Proposition~\ref{prop:unstrait-presehaf}).
The three main examples of cubesets are also given in this section:

\begin{example*}[Lemma~\ref{lem:point-cset}]
Given a set $S$, the \emph{trivial cubeset $i_!S$} is given by the constant presheaf
that assigns to an $n$-cube ${\bbox}^n$ the set $S$ and to a map of cubes the identity on $S$.

The \emph{discrete cubeset $i_\ast S$} is given by the presheaf
\[
    {\bbox}^n\mapsto\hom_{\set}(\{0,1\}^n,S).
\]
\end{example*}

\begin{example*}[Example~\ref{ex:host_kra_cubegroups}]
Given an abelian group $A$ and $k\in\N$ we define the \emph{$k$-th Eilenberg-Mac Lane cubeset $\cD_k(A)$} to be the subcubeset of $i_\ast A$ whose $n$-cubes are the subgroup of $A^{2^n}$ generated by elements of the form
\[
    a^F\colon\{0,1\}^n\to A,\quad\omega\mapsto\begin{cases}
        a,\quad & \omega\in F, \\
        0,\quad & \omega\notin F,
    \end{cases}
\]
where $F$ runs over all cube faces of $\{0,1\}^n$
whose dimension is at least $n-k$ and $a$ is an element of $A$.
\end{example*}

In practice, the relevant cubesets, and in particular, all examples considered above,
have the additional property of \emph{concreteness}.
A cubeset $X$ is concrete if each $n$-cube of $X$ is fully determined by its vertices.

\begin{definition*}[Definition~\ref{def:concrete-cubeset}]
A cubeset $X$ is \emph{concrete} if the natural map
\[
    X(n) \to \prod_{0\to n} X(0) = X(0)^{2^n}
\]
is injective.
We denote the full subcategory of concrete cubesets by $\ccSet$.
\end{definition*}

Since concreteness is defined via limits, concrete cubesets are stable under limits
and thus there is a \emph{concretification}, i.e. a left adjoint to the inclusion,
that assigns to a cubeset $X$ an associated concrete cubeset (Proposition~\ref{prop:concrete-left-adjoint}).
The rest of Section~\ref{ssec:concrete-cubesets} deals with stability properties of
concrete cubesets: they are stable under coproducts, filtered colimits, internal Hom and subobjects.

Another important aspect of cubesets is \emph{corner completion}.

\begin{definition*}[Definition~\ref{def:corner}]
For $n\in\N$, the \emph{$n$-corner ${\bbcor}^n$} is the union of all linear faces of ${\bbox}^n$
of dimension less than $n$ that don't contain the point $1^n=(1,\ldots,1)$.
\end{definition*}

The first few $n$-corners look as follows.
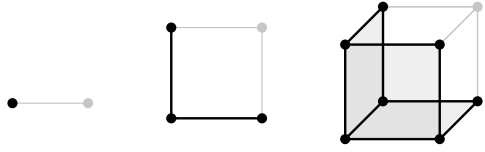
\begin{center}

\begin{tikzpicture}[
    x=1cm,y=1cm,
    font=\small,
    line cap=round,
    line join=round,
    vertex/.style={circle,fill,inner sep=1.35pt},
    greyvertex/.style={circle,fill=gray!45,inner sep=1.35pt},
    ambient/.style={draw=gray!45,line width=0.45pt},
    corneredge/.style={draw=black,line width=0.9pt},
    cornerface/.style={
        draw=gray!55,
        fill=gray!45,
        line width=0.45pt,
        fill opacity=0.28
    },
    toplabel/.style={font=\normalsize},
scale=0.5
]

\begin{scope}[shift={(0,0)}]
    \coordinate (A) at (0,0);
    \coordinate (B) at (2,0);

    \draw[ambient] (A)--(B);

    \node[vertex] at (A) {};
    \node[greyvertex] at (B) {};
\end{scope}

\begin{scope}[shift={(4.2,-0.4)}]
    \coordinate (A) at (0,0);
    \coordinate (B) at (2.4,0);
    \coordinate (C) at (0,2.4);
    \coordinate (D) at (2.4,2.4);

    \draw[ambient] (A)--(B)--(D)--(C)--cycle;

    \draw[corneredge] (A)--(B);
    \draw[corneredge] (A)--(C);

    \node[vertex] at (A) {};
    \node[vertex] at (B) {};
    \node[vertex] at (C) {};
    \node[greyvertex] at (D) {};
\end{scope}

\begin{scope}[shift={(8.8,-0.95)}]
    \coordinate (A) at (0,0);
    \coordinate (B) at (2.5,0);
    \coordinate (C) at (0,2.5);
    \coordinate (D) at (2.5,2.5);

    \coordinate (E) at (1.0,1.0);
    \coordinate (F) at (3.5,1.0);
    \coordinate (G) at (1.0,3.5);
    \coordinate (H) at (3.5,3.5);

    \draw[ambient] (A)--(B)--(D)--(C)--cycle;
    \draw[ambient] (E)--(F)--(H)--(G)--cycle;
    \draw[ambient] (A)--(E);
    \draw[ambient] (B)--(F);
    \draw[ambient] (C)--(G);
    \draw[ambient] (D)--(H);

    \filldraw[cornerface] (A)--(B)--(D)--(C)--cycle;
    \filldraw[cornerface] (A)--(B)--(F)--(E)--cycle;
    \filldraw[cornerface] (A)--(C)--(G)--(E)--cycle;

    \draw[corneredge] (A)--(B)--(D)--(C)--cycle;
    \draw[corneredge] (A)--(E)--(F)--(B);
    \draw[corneredge] (E)--(G)--(C);

    \draw[ambient] (D)--(H);
    \draw[ambient] (F)--(H);
    \draw[ambient] (G)--(H);

    \foreach \P in {A,B,C,D,E,F,G}
        \node[vertex] at (\P) {};

    \node[greyvertex] at (H) {};
\end{scope}

\end{tikzpicture}
\end{center}

This leads to the definition of a fibrant cubeset and the relative version thereof, a \emph{fibration}.

\begin{definition*}[Definition~\ref{def:nil-fibration}]
    \begin{enumerate}[(i)]
        \item
A cubeset $X$ is called \emph{fibrant} if it has \emph{corner completion},
i.e., for every corner ${\bbcor}^n\to X$ in $X$ there exists a completion to a cube ${\bbox}^n\to X$
along the inclusion ${\bbcor}^n\to{\bbox}^n$.

        \item
A map $f\colon X\to Y$ between cubesets is called \emph{fibration},
if for every commutative diagram
\begin{center}
\begin{tikzcd}[ampersand replacement=\&]
	{{\bbcor}^n} \& X \\
	{{\bbox}^n} \& Y
	\arrow[from=1-1, to=1-2]
	\arrow[hook, from=1-1, to=2-1]
	\arrow[from=1-2, to=2-2]
	\arrow[dashed, from=2-1, to=1-2]
	\arrow[from=2-1, to=2-2]
\end{tikzcd}
\end{center}
there exists a diagonal lift ${\bbox}^n\to X$.
        \end{enumerate}
\end{definition*}

This means that every $n$-corner in $X$ that can be completed in $Y$ to an $n$-cube already has a completion to an $n$-cube in $X$ which maps to the cube in $Y$.

The notion of fibration originates in model categories and is usually defined
via a right lifting property against a set of morphisms,
as is the case for fibrations of cubesets.
Applying Quillen's small object argument, we obtain a weak factorization system:
two classes of morphisms $(l(r(\Lambda)),r(\Lambda))$ such that each map of cubesets
decomposes into
\begin{center}
\begin{tikzcd}
	X & W & Y
	\arrow["h", tail, from=1-1, to=1-2]
	\arrow["f"', curve={height=12pt}, from=1-1, to=1-3]
	\arrow["g", two heads, from=1-2, to=1-3]
\end{tikzcd}
\end{center}
where $g$ is a fibration and $h$ is a \emph{trivial cofibration},
meaning that it belongs to the class $l(r(\Lambda))$, see Section~\ref{ssec:fibrations}.
There is also an explicit construction of trivial cofibrations
starting from corner inclusions ${\bbcor}^n\to{\bbox}^n$ (Construction~\ref{cons:triv-cofib}).
This gives a generalization of the extension property of simplicial cubesets from \cite[Lemma 3.1.5]{Candela2017} and \cite[Lemma 7.5]{Gutman2020}.

\begin{proposition*}[Proposition~\ref{prop:triv-nilcofib-simp}]
Let $K$ be a non-empty downwards-closed collection of subsets of $\{1,\ldots,n\}$.
Then the inclusion
\[
    \bigcup_{A\in K}{\bbox}^{\lvert A\rvert}\to{\bbox}^n
\]
is a trivial cofibration.
\end{proposition*}

The rest of the section on fibrations deals with fibrant replacements
and their interaction with concrete cubesets,
as well as with stability properties of fibrations and trivial cofibrations.

Following this theme we introduce cubesets with gluing property.
Informally, a cubeset $X$ has the gluing property if, for any two $n$-cubes $a$ and $b$ in $X$
that share a common face of dimension $n-1$,
we can glue them along this face to a new cube $c$ of $X$.
To make this precise, we define a class of morphisms $\bprism$,
and then define the gluing property to be a right lifting property against the morphisms in $\bprism$
(Definition~\ref{def:gluing-property}).
It turns out that a cubeset $X$ with the gluing property, such that the gluing of any two cubes is unique, already is concrete (Proposition~\ref{prop:unique-gluing-concrete}).
Moreover, each fibrant cubeset has the gluing property (Proposition~\ref{prop:nilfib-gluing}).

In Section~\ref{ssec:corner-gamma} we give an alternative description of corner completion.
Every cubeset $X$ can be interpreted as a $\bbGamma$-set $j^\ast X$
whose $n$-cubes are just the $n$-cubes of $X$, but we only consider maps $X(m)\to X(n)$
that arise from linear maps of cubes (Remark~\ref{rem:adjoint-triple-j}).

\begin{definition*}[Definition~\ref{def:omega_corner_gamma}]
Let $n\in\N$ and $\omega\in\{0,1\}^n\setminus\{0^n\}$.
Then define the $n$-corner ${\bbcor}^n_\omega$ w.r.t. $\omega$ as the union over all injections $f\colon{\bblox}^{n-1}\to{\bblox}^n$ in $\bbGamma$
such that $f$ misses $\omega$, i.e., $\omega\notin\im f$:
\[
    {\bbcor}^n_\omega = \bigcup_{\substack{f\colon{\bblox}^{n-1}\sub{\bblox}^n \\ \omega\notin\im f}}{\bblox}^{n-1}\hookrightarrow{\bblox}^n.
\]
\end{definition*}

In two dimensions, those corners look as follows.
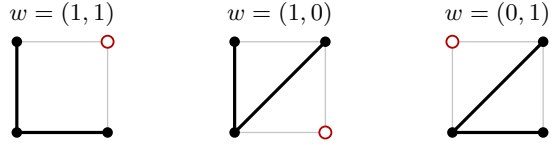
\begin{center}
\begin{tikzpicture}[
    scale=0.6,
    line cap=round,
    line join=round,
    bg/.style={draw=gray!45, line width=0.5pt},
    face/.style={draw=black, line width=1.2pt},
    vtx/.style={circle, fill=black, inner sep=1.4pt},
    missing/.style={circle, draw=red!70!black, fill=white, line width=0.7pt, inner sep=1.6pt},
    lab/.style={font=\scriptsize},
    title/.style={font=\small\bfseries}
]

\begin{scope}[shift={(0,0)}]
  \node[title] at (1,2.6) {$w=(1,1)$};

  \coordinate (O) at (0,0);
  \coordinate (A) at (2,0);
  \coordinate (B) at (0,2);
  \coordinate (C) at (2,2);

  \draw[bg] (O)--(A)--(C)--(B)--cycle;

  \draw[face] (O)--(A);
  \draw[face] (O)--(B);

  \node[vtx] at (O) {};
  \node[vtx] at (A) {};
  \node[vtx] at (B) {};
  \node[missing] at (C) {};

  \node[lab, below left]  at (O) {};
  \node[lab, below right] at (A) {};
  \node[lab, above left]  at (B) {};
  \node[lab, above right] at (C) {};
\end{scope}

\begin{scope}[shift={(4.8,0)}]
  \node[title] at (1,2.6) {$w=(1,0)$};

  \coordinate (O) at (0,0);
  \coordinate (A) at (2,0);
  \coordinate (B) at (0,2);
  \coordinate (C) at (2,2);

  \draw[bg] (O)--(A)--(C)--(B)--cycle;

  \draw[face] (O)--(B);
  \draw[face] (O)--(C);

  \node[vtx] at (O) {};
  \node[missing] at (A) {};
  \node[vtx] at (B) {};
  \node[vtx] at (C) {};

  \node[lab, below left]  at (O) {};
  \node[lab, below right] at (A) {};
  \node[lab, above left]  at (B) {};
  \node[lab, above right] at (C) {};
\end{scope}

\begin{scope}[shift={(9.6,0)}]
  \node[title] at (1,2.6) {$w=(0,1)$};

  \coordinate (O) at (0,0);
  \coordinate (A) at (2,0);
  \coordinate (B) at (0,2);
  \coordinate (C) at (2,2);

  \draw[bg] (O)--(A)--(C)--(B)--cycle;

  \draw[face] (O)--(A);
  \draw[face] (O)--(C);

  \node[vtx] at (O) {};
  \node[vtx] at (A) {};
  \node[missing] at (B) {};
  \node[vtx] at (C) {};

  \node[lab, below left]  at (O) {};
  \node[lab, below right] at (A) {};
  \node[lab, above left]  at (B) {};
  \node[lab, above right] at (C) {};
\end{scope}

\end{tikzpicture}
\end{center}

Now for a general cubeset, there is no reason as to why fibrancy against the (classical) corner in $\bbox$ implies fibrancy against the corner ${\bbcor}^n_\omega$ in $\bbGamma$.
However, if $X$ is concrete, then one cannot distinguish between those notions.

\begin{proposition*}[Proposition~\ref{prop:nilfib-restricted-nilfib-concrete}]
Let $X$ be a concrete cubeset.
Then the following assertions are equivalent.
\begin{enumerate}[(a)]
    \item The cubeset $X$ is fibrant.
    \item The underlying $\bbGamma$-set of $X$ is fibrant w.r.t. the corner inclusions
            ${\bbcor}^n_\omega\to{\bblox}^n$.
\end{enumerate}
\end{proposition*}

Another important property of a cubeset is its \emph{step}.

\begin{definition*}[Definition~\ref{def:n-step}]
A cubeset $X$ is called \emph{$k$-step} if for all $n\geq k+1$ the map
\[
    \hom({\bbox}^n, X) \to \hom({\bbcor}^n, X)
\]
is bijective, i.e., each $n$-corner in $X$ has a unique completion to an $n$-cube of $X$.
The full subcategory of $k$-step cubesets is denoted $\cSet_k$.
\end{definition*}

This means that if $X$ is $k$-step, its collection of $n$-cubes for $n>k$ is fully determined
by the $k$-cubes of $X$.
We give multiple equivalent characterizations of $k$-step cubesets in Proposition~\ref{prop:truncations-char}.
Given a general cubeset $X$, one can associate to it the maximal $k$-step quotient $\tau_k X$,
meaning that every other map $X\to Y$ where $Y$ is $k$-step factors through $X\to \tau_k X$.
Speaking categorically this states that the inclusion $\cSet_k\to\cSet$ admits a left adjoint
$\tau_k\colon\cSet\to\cSet_k$ (Proposition~\ref{prop:step-left-adjoint}).
This left adjoint preserves finite products (Proposition~\ref{prop:n-trunc-ihom-prod})
and we give stability conditions on $k$-step cubesets in Proposition~\ref{prop:n-trunc-closure}.
In general, it is not easy to compute the $k$-step quotient $\tau_k X$ of $X$.
However, if $X$ is concrete and fibrant, then a process called \emph{universal replacement}
yields $\tau_k X$ as an easy-to-compute factor of $X$.

\begin{definition*}[Definition~\ref{def:univ-replacement}]
For $x,y\in X(0)$ define $x\sim_k y$ if there exist ($k+1$)-cubes in $X$ that agree on the corner and have $x$ and $y$ on their uppermost vertex $1^{k+1}$, respectively.
\end{definition*}

This equivalence relation on $X(0)$ induces a factor $X(0)/\sim_k$
which yields a new cubeset $X/\sim_k$ whose $n$-cubes $X/\sim_k(n)$ are given by the image of $X(n)$ under the projection
\[
    X(0)^{2^n}\to (X(0)/\sim_k)^{2^n}.
\]
Explicitly this means that two $n$-cubes in $X$ are identified if all their points are equivalent in $X(0)$.

\begin{proposition*}[Proposition~\ref{prop:trunc_concfib}]
Let $X$ be a concrete fibrant cubeset.
Then the truncation $\tau_k X$ is given by $X/\sim_k$.
Moreover, the unit $\eta\colon X\to\tau_k X$ is a surjective fibration
and $\tau_k X$ is concrete and fibrant.
\end{proposition*}

This proposition generalizes the well-known fact that $\sim_k$ is the smallest equivalence relation
on $X(0)$ such that $X/\sim_k$ is $k$-step, see \cite[Proposition 6.3]{Gutman2020}.

In Section~\ref{ssec:connectivity} we further investigate the $0$-truncation of a (general) cubeset.
This includes the definition of \emph{ergodicity}.

\begin{definition*}[Definition~\ref{def:ergodic}]
A cubeset $X$ is \emph{$n$-ergodic} if $\tau_{n-1}X = \ast$.
\end{definition*}

We will then describe a decomposition into ergodic components (Proposition~\ref{prop:ergodic-decomp}).
Also, there is the analogous concept for $\bbGamma$-sets.
In Section~\ref{ssec:truncation-gamma} we compute the truncation over $\bbGamma$
for concrete cubesets which continues the theme of Section~\ref{ssec:corner-gamma}.

Next, in Section~\ref{ssec:structure-group} we study algebraic invariants of cubesets:
their structure groups.

\begin{definition*}[Definition~\ref{def:structure-group}]
If $X$ is a fibrant cubeset, then an $n$-sphere of $X$ based at $x_0\in X$
is an $n$-cube $s\colon{\bbox}^n\to X$ whose restriction to the $n$-corner is constant $x_0$.
\end{definition*}

Now one would like to use corner completion one dimension higher to define the multiplication of $n$-spheres.
In case $n = 2$, for example, given spheres $c,d$ in $X$
one can build a 3-corner of $X$ which has $c$ on one side, $d$ on another,
and the constant cube $x_0$ on the remaining inert faces.
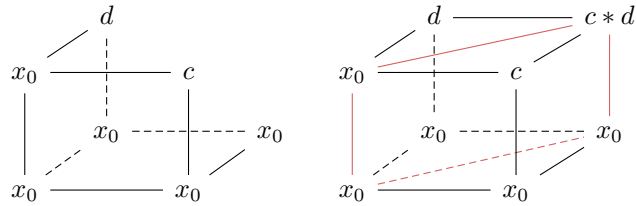
\begin{center}
\begin{tikzcd}[row sep=small, column sep=small]
	& d &&&& d && {c*d} \\
	{x_0} && c && {x_0} && c \\
	& {x_0} && {x_0} && {x_0} && {x_0} \\
	{x_0} && {x_0} && {x_0} && {x_0}
	\arrow[no head, from=1-6, to=1-8]
	\arrow[dashed, no head, from=1-6, to=3-6]
	\arrow[no head, from=2-1, to=1-2]
	\arrow[no head, from=2-1, to=2-3]
	\arrow[no head, from=2-5, to=1-6]
	\arrow[color={rgb,255:red,214;green,92;blue,92}, no head, from=2-5, to=1-8]
	\arrow[no head, from=2-5, to=2-7]
	\arrow[color={rgb,255:red,214;green,92;blue,92}, no head, from=2-5, to=4-5]
	\arrow[no head, from=2-7, to=1-8]
	\arrow[dashed, no head, from=3-2, to=1-2]
	\arrow[dashed, no head, from=3-2, to=3-4]
	\arrow[dashed, no head, from=3-6, to=3-8]
	\arrow[color={rgb,255:red,214;green,92;blue,92}, no head, from=3-8, to=1-8]
	\arrow[no head, from=4-1, to=2-1]
	\arrow[dashed, no head, from=4-1, to=3-2]
	\arrow[no head, from=4-1, to=4-3]
	\arrow[no head, from=4-3, to=2-3]
	\arrow[no head, from=4-3, to=3-4]
	\arrow[dashed, no head, from=4-5, to=3-6]
	\arrow[color={rgb,255:red,214;green,92;blue,92}, dashed, no head, from=4-5, to=3-8]
	\arrow[no head, from=4-5, to=4-7]
	\arrow[no head, from=4-7, to=2-7]
	\arrow[no head, from=4-7, to=3-8]
\end{tikzcd}
\end{center}
Then using corner completion we obtain the composition $c*d$
by taking the diagonal of the completed 3-cube (depicted in red).
But: composition might not be well-defined.
To account for this deficit, we just define two $n$-spheres to be equivalent if
they both arise as the diagonal of the completion of the same ($n+1$)-corner.
In case that $X$ is a concrete fibrant cubeset, this actually defines an equivalence relation on $n$-spheres,
yielding the set $\pi_n(X,x_0)$ on which we can install the multiplication just described,
see Constructions~\ref{cons:structure-monoid-concrete} and~\ref{cons:concrete-structure-monoid-mult}.

\begin{proposition*}[Propositions~\ref{prop:concrete-structure-group-is-monoid}~and~\ref{prop:concrete-structure-group-is-group}]
Let $X$ be a fibrant concrete cubeset.
The set $\pi_n(X,x_0)$ with the multiplication just described is an abelian group.
\end{proposition*}

The group $\pi_n(X,x_0)$ is called the \emph{$n$-th structure group of $X$}.
Actually, in Section~\ref{ssec:structure-group} we construct the structure groups for general cubesets.
In this case, $\pi_n(X,x_0)$ only is an abelian monoid.

Moreover, if $x_0$ and $y_0$ are points of $X$ that lie in the same ergodic component of $X$,
then the structure groups are the same.

\begin{proposition*}[Proposition~\ref{prop:structure-monoid-independence-base-point}]
Let $X$ be a concrete fibrant cubeset and $x_0$, $y_0$ points of $X$ in the same ergodic component.
Then there is a canonical isomorphism of abelian groups $\pi_n(X,x_0)\simeq\pi_n(X,y_0)$.
\end{proposition*}

In Section~\ref{ssec:structure-group-gamma}, we consider the structure monoids of $\bbGamma$-sets
modeled on the alternative corner ${\bbcor}^n_\omega$.
This gives another notion of structure monoid of a cubeset,
which again is isomorphic to the former notion if $X$ is a concrete fibrant cubeset.

All of the previous considerations then are combined in Section~\ref{ssec:postnikov}
to prove the following theorem, known as the \emph{weak structure theorem}.

\begin{theorem*}[Theorem~\ref{thm:weak-structure}]
Let $X$ be a concrete ergodic fibrant cubeset.
Then there exists a tower of surjective fibrations
\begin{center}
\begin{tikzcd}
	X & \cdots & {\tau_2X} & {\tau_1X} & {\tau_0X}
	\arrow[from=1-1, to=1-2]
	\arrow[from=1-2, to=1-3]
	\arrow[from=1-3, to=1-4]
	\arrow[from=1-4, to=1-5]
\end{tikzcd}
\end{center}
such that $\tau_n X$ is a concrete fibrant $n$-step cubeset
and $\tau_n X$ is a principal $\cD_n(\pi_n(X,x_0))$-bundle over $\tau_{n-1} X$.
This tower is natural in $X$.
\end{theorem*}

The existence of the tower (functorial in $X$) by itself is clear:
one just applies the $n$-truncations $\tau_n X$ to $X$ and uses the corresponding units $\tau_n X\to\tau_{n-1}X$.
We also have seen that, since $X$ is concrete and fibrant, the truncations are all themselves concrete and fibrant.
Now for the assertion on the principal $\cD_n(\pi_n(X,x_0))$-bundle,
one can pick an arbitrary basepoint $x_0$, and by ergodicity, the structure group $\pi_n(X,x_0)$
doesn't depend on this choice.
Now the fiber of $\tau_nX\to\tau_{n-1}X$ is $n$-ergodic and $n$-step:
$n$-ergodicity follows from universal replacement, used in the computation of the truncation,
and $n$-step is clear as subobject of something $n$-step.
In particular, the homotopy group $\pi_n(X,x_0)$ is in bijection to the points of the fiber of $\tau_nX\to\tau_{n-1}X$ over $x_0$.
Thus, we can install an action of $\pi_n(X,x_0)$ on each fiber that is,
since it is, up to a twist, just left multiplication in the structure group,
free and transitive.
Then one has to glue all the actions of each fiber together to obtain a global action on $\tau_n X$
making the unit of the truncation a principal bundle for the $n$-th structure group.

In the last Section~\ref{sec:condensed-cubeset} we prove a condensed variant of the weak structure theorem, generalizing the weak structure theorem for compact cubesets.
For the reader unfamiliar with condensed sets, we refer to Section~\ref{sec:condensed-cubeset}
where the most important definitions are given, with pointers to the literature.

\begin{theorem*}[Theorem~\ref{thm:cond-weak-structure}]
Let $X$ be a condensed concrete ergodic fibrant cubeset.
Then there exists a tower of surjective fibrations
\begin{center}
\begin{tikzcd}
	X & \cdots & {\tau_2X} & {\tau_1X} & {\tau_0X}
	\arrow[from=1-1, to=1-2]
	\arrow[from=1-2, to=1-3]
	\arrow[from=1-3, to=1-4]
	\arrow[from=1-4, to=1-5]
\end{tikzcd}
\end{center}
such that $\tau_n X$ is a condensed concrete fibrant $n$-step cubeset
and $\tau_n X$ is a principal $\cD_n(\pi_n(X,x_0))$-bundle over $\tau_{n-1} X$.
This tower is natural in $X$.
\end{theorem*}

We prove this theorem by taking the tensor product of presentable ($\infty$-)categories
of the category of cubesets with the category of ($\kappa$-)condensed sets.
This gives a notion of condensed cubeset, and all the definitions (concrete, ergodic, fibrant, etc.)
internalize to condensed cubesets.
It then turns out that, interpreting a condensed cubeset as a functor $\extr^\op\to\cSet$,
all the internal notions coincide with the pointwise notions.
Therefore, the condensed weak structure theorem is a direct corollary from the assertion on
functoriality in the (discrete) weak structure theorem.

When considering the full subcategory of quasicompact quasiseparated (qcqs) condensed cubesets
one obtains the classical weak structure theorem for compact Hausdorff spaces as in
\cite[Proposition 2.1.9]{Candela2017a} and \cite[Theorem 5.4]{Gutman2020}.

\subsubsection*{Acknowledgments}

We thank Pablo Candela, Asgar Jamneshan, Emilio Parini, Philipp Schmale, Elias Schuster and Pauwel van den Eeckhaut for many valuable discussions on the topics of this paper.
We thank Asgar Jamneshan and Rainer Nagel for comments on an earlier draft.

\subsubsection*{Disclosure of AI Usage}

During the process of developing this paper we used ChatGPT to generally discuss ideas and thoughts.
In particular, we extensively discussed about different approaches to model structures on categories.
It suggested the explicit counterexample in Remark~\ref{rem:fibrant-non-reflective}.
Moreover, it created some of the TikZ-figures of cubes.
We also used it to identify spelling mistakes and verify proofs.

Other than in the above disclosed parts, there has been no involvement of artificial intelligence in this paper.

\section{Cubesets}
In this section, we introduce the category of cubes and cubesets, study their basic properties and compare them to $\bbGamma$-sets.

\subsection{The Building Blocks}\label{ssec:index-cats}

In this subsection we first want to discuss several variants of the indexing category for cubesets and relations between candidates of such.

\subsubsection{Defining the Indexing Category}

Let us start with the classical category of combinatorial cubes used in nilspace theory, being introduced e.g. in \cite[p. 3]{Camarena2010}, \cite[Definition 3.1]{Gutman2020} and \cite[Definition 1.1.1]{Candela2017}.

\begin{definition}\label{def:cube_cat}
Define the \emph{cube category $\bbox$} as the skeletal category%
\footnote{If not stated otherwise, by category we mean a (static/ordinary) $(1,1)$-category.}
whose objects are the sets $\{0,1\}^n$ for $n\in\N_0$ and where a morphism $f\colon\{0,1\}^n\to\{0,1\}^m$
is a set-theoretical map that is the restriction of an affine linear map $\Z^n\to\Z^m$.
\end{definition}

To understand this category it is helpful to first study an important subcategory.

\begin{definition}\label{def:linear_cubes}
Denote by $\bblox$ the wide subcategory of $\bbox$ consisting of all maps that are restrictions of linear maps.
We denote this corresponding inclusion by $j\colon\bblox\to\bbox$.
\end{definition}

This linear cube category relates very nicely to a classically known indexing category from (higher) category theory.

\begin{definition}\label{def:indexcats}
Let us define the following small categories.
\begin{enumerate}[label=(\roman*), ref=\thedefinition(\roman*)]
    \item\label{def:finstar}
        Define \emph{Segal's category} $\bbGamma$ as the opposite of $\Fin_\ast$.
        The category $\Fin_\ast$ is (the skeleton of) the category of finite pointed sets.
        This means that its objects are indexed by the set $\N_0$, which we also denote by
        \[\langle n \rangle=\{\ast\}\sqcup\langle n\rangle^\circ=\{\ast\}\sqcup \{1,\ldots,n\}\]
        for $n\in \N_0$.
        Morphisms $\langle n\rangle \to \langle m\rangle$ in $\Fin_*$ are given by base-point preserving maps of sets,
        or equivalently, by partially defined maps $\langle n\rangle^\circ\to \langle m\rangle^\circ$.
    \item\label{def:inert_active}
        We call a morphism $\alpha$ in $\Fin_\ast$ \emph{inert} (or \emph{lower face}) if for all $i\neq\ast$ the preimage $\alpha^{-1}(i)$ consists of exactly one element.
        And we call $\alpha$ \emph{active} (or \emph{diagonal}) if it is everywhere defined, i.e., $\alpha(j)\neq\ast$ for all $j\ne \ast$.
    \item
        Define the category $\finert$ to be the wide subcategory of $\Fin_*$ consisting of inert morphisms only.
        Define $\inert$ as the opposite of $\finert$.
    \item\label{def:almost-inert}
        Define the category $\Fin_{*,\le 1}$ to be the subcategory of $\Fin_*$ where morphisms are required to have fibers of size at most one for all non-basepoints.    
\end{enumerate}
\end{definition}

\begin{proposition}\label{lem:iso-gamma-bblox}
There is a canonical equivalence of categories $\bbGamma\simeq{\bblox}$.
\end{proposition}

To prove this, let us first describe the morphisms of $\bbox$ in a more coordinate-dependent version.

\begin{lemma}\label{lem:coordinate_description}
A set-theoretical map $f\colon\{0,1\}^n\to\{0,1\}^m$ is a morphism in $\bbox$
if and only if it is of the form (its \emph{canonical representation})
\[
    f(x_1,\ldots,x_n) = (y_1,\ldots,y_m)
\]
where each $y_i$ is either $0$, $1$, or there exists a unique $j(i)$ such that $y_i = x_{j(i)}$ or $y_i = 1 - x_{j(i)}$.

Similarly, a morphism $f\colon\{0,1\}^n\to\{0,1\}^m$ in $\bblox$ can be described as
\[
    f(x_1,\ldots,x_n) = (y_1,\ldots,y_m)
\]
where each $y_i$ is either equal to zero or there exists a unique $j(i)$ such that $y_i = x_{j(i)}$.
\end{lemma}

\begin{proof}
The first part can be found in \cite[Lemma 1.1.3]{Candela2017} and \cite[Remark after Definition 3.1]{Gutman2020}.
Let us recall the argument for completeness.

First, if $f$ is of the form in the lemma, then by setting $g_i(x) = f_i(x) - f_i(0)$
we see that $g_i(x)$ is either the zero map, or $g_i(x) = \pm x_{j(i)}$ for some $j(i)$.
In particular, $f$ is the restriction of an affine map from $\Z^n$ to $\Z^m$.

On the other hand, if $f$ arises from an affine map,
then there exists a matrix $A\in\Z^{m\times n}$ and a point $p\in\Z^m$ such that
\[
    f(x) = Ax + p.
\]
From the condition $f(x)\in\{0,1\}^m$, hence $f_i(x)\in\{0,1\}$, for all $x\in\{0,1\}^n$ and all $i\in\{1,\ldots,m\}$ we know that
every row of $A$ has at most one non-zero entry which is either $1$ or $-1$.
Clearly each non-zero entry of $A$ can only be $1$ or $-1$ or else $f_i(x)$ wouldn't land in $\{0,1\}$.
Moreover, if there were two or more non-zero entries, then testing with appropriate $x$ would contradict $f_i(x)\in\{0,1\}$.
Now if in row $i$ every entry is zero, $p_i\in\{0,1\}$ and we set $y_i = p_i$.
If in row $i$ the entry in column $j(i)$ is $1$, $p_i$ has to be zero and we set $y_i = x_{j(i)}$.
If in row $i$ the entry in column $j(i)$ is $-1$, $p_i$ has to be one and we set $y_i = 1 - x_{j(i)}$.
Here, $j(i)$ is uniquely determined.
One can now easily check that the map
\[
    (x_1,\ldots,x_n) \mapsto (y_1,\ldots,y_m)
\]
is the same as the map $f$.

The second part follows from the first noting that an affine map $A$ is linear if and only if $A(0)=0$.
\end{proof}

With this at hand, we can prove that the linear part of $\bbox$ simply is $\bbGamma$.

\begin{proof}[Proof of Proposition \ref{lem:iso-gamma-bblox}]
There is a functor
\[
    F\colon{\bblox}\to\Fin_\ast^\op
\]
which acts on objects as $F(\{0,1\}^n) = \langle n\rangle$.
To a morphism $f\colon\{0,1\}^n\to\{0,1\}^m$ it assigns the morphism
\[
    F(f)\colon\langle m\rangle\to\langle n\rangle,\quad i\mapsto
        \begin{cases}
            j(i),\quad&\text{if}\; y_i = x_{j(i)}, \\
            \ast,\quad&\text{if}\; y_i = 0
        \end{cases}
\]
in $\Fin_\ast$.
In the other direction, there exists the functor
\[
    G\colon\Fin_\ast^\op\to{\bblox}
\]
which acts on objects as $G(\langle n\rangle) = \{0,1\}^n$.
To a morphism $\alpha\colon\langle m\rangle\to\langle n\rangle$ in $\Fin_\ast$ it assigns the morphism
\[
    G(\alpha)\colon\{0,1\}^n\to\{0,1\}^m,\quad (x_1,\ldots,x_n)\mapsto (y_1,\ldots,y_m)
\]
where $y_i = 0$ if $\alpha(i) = \ast$ and $y_i = x_{\alpha(i)}$ otherwise.%
\footnote{Note that this is simply pullback with $\alpha$, if one uses the convention $x_\ast=0$.}
It is a straightforward computation to show that $F$ and $G$ are inverse to each other.
\end{proof}

To reach a categorical description of the whole cube category $\bbox$,
the idea is to note that in fact $\bblox$ is not just a subcategory of $\bbox$ but a retract,
i.e., there is a functor $\bbox\to \bblox$ assigning to every affine linear map its \emph{linear part}.
This exhibits the category $\bbox$ as \emph{fibered over $\bbGamma=\bblox$},
where the fiber of $\langle n\rangle$ consists of the additional \emph{flips} induced by the non-linear base-point translation.
This leads to the description of $\bbox$ as a certain \emph{skew-product} of categories,
or, in more modern language, as \emph{unstraightening} via the classical \emph{Grothendieck construction}.

\begin{construction}\label{cons:unstraightening}
Let us recall the Grothendieck construction.
Given a (pseudo-)functor $F\colon\cC\to\Cat$ one can construct a new category
\[
    \mathrm{Un}(F)=\int_{X\in\cC} F(X)
\]
as follows:
\begin{itemize}
    \item its objects are pairs $(X,A)$ where $X\in\cC$ and $A\in F(X)$ and
    \item a morphism $(X,A)\to (Y,B)$ is a pair
    \[(f\colon X\to Y,\phi\colon F(f)(A)\to B)\]
where $f$ is a morphism in $\cC$ and $\phi$ is a morphism in $F(Y)$.
    \end{itemize}
Now composition is given in the only way possible:
given $(f,\phi)\colon (X,A)\to (Y,B)$ and $(g,\psi)\colon (Y,B)\to (Z,C)$ its composite is the pair
\[
    (g\circ f\colon X\to Y\to Z,\psi\circ F(g)(\phi)\colon F(g\circ f)(A)\to F(g)(B)\to C).
\]
This construction gives rise to a ($\infty$-)functor
\[
    \int_\cC\colon\Fun(\cC,\Cat)\to\Cat/\cC
\]
which induces an equivalence between $\Fun(\cC,\Cat)$ and the subcategory of $\Cat/\cC$
of Grothendieck opfibrations (also called cocartesian fibrations), called the \emph{straightening-unstraightening equivalence}, see \cite[Chapter 3]{Lurie2009}.
\end{construction}

Hence, let us consider an appropriate functor $\bbGamma\to \Cat$ encoding the flip-action of cubes.

\begin{definition}
The assignment $\langle n\rangle\mapsto\F_2^n$ defines a functor $\Fin_\ast^\op\to\Grp$%
\footnote{Note that we denote the group $\Z/2\Z$ by $\mathbb{F}_2$
since we want to avoid the double quotient notation $\ast/(\Z/2\Z)$.}
since a morphism $\alpha\colon\langle n\rangle\to\langle m\rangle$ in $\bbGamma$ induces
by pullback at its opposite a group homomorphism $\alpha_\ast=(\alpha^\op)^\ast\colon\F_2^n\to\F_2^m$.
At coordinate $j\in\{1,\ldots,m\}$ it is given by
\[
    (x_1,\ldots,x_n)\mapsto\begin{cases}
                                x_{\alpha(j)},\quad&\text{if}\;\alpha(j)\neq\ast, \\
                                0,\quad&\text{else}.
                            \end{cases}
\]
By composing with the functor $\Grp\to\Cat,\;G\mapsto */G=\mathrm{B}G$ we obtain a functor
\[
    \ast/\F_2^\bullet\colon\bbGamma\to\Cat,\quad\langle n\rangle\mapsto\ast/\F_2^n.
\]
\end{definition}

With this definition we arrive at the following categorical definition of the cube category.

\begin{proposition}\label{prop:cube_cat_linear}
The cube category $\bbox\to\bbGamma$ is the unstraightening
\[
    \mathrm{Un}(\ast/\F_2^\bullet)=\int_{\bbGamma} \ast/\F_2^\bullet\to\bbGamma
\]
of the functor $\ast/\F_2^\bullet\colon\bbGamma\to\Cat$.
\end{proposition}

In our case, the Grothendieck construction takes the following form.

\begin{remark}
Let us describe the category
\[
    \int_{\bbGamma} \ast/\F_2^\bullet\to\bbGamma.
\]
Since the category $\ast/\F_2^n$ consists of exactly one object $\ast_n$,
an object in ${\int_{\bbGamma}} \ast/\F_2^\bullet$ is just an object $\langle n\rangle$ of $\bbGamma$.

A morphism $(f,\phi)\colon\langle n\rangle\to\langle m\rangle$ consists of a morphism $f$ in $\bbGamma$
and a morphism $\phi\colon(\ast/\F_2^\bullet)(f)(\ast_n)=\ast_m\to\ast_m$ in $\ast/\F_2^m$ which corresponds to an element $a\in\F_2^m$.

Now composition works as follows: given $(f,a)\colon\langle n\rangle\to\langle m\rangle$
and $(g,b)\colon\langle m\rangle\to\langle k\rangle$, in the first component it is just ordinary composition $g\circ f$.
In the second component, the element of $\F_2^k$ is given by $b+g_\ast(a)$, so
\[
    (g,b)\circ(f,a)=(g\circ f, b+g(a)).
\]
\end{remark}

Note that the functor $\ast/\F_2^\bullet\colon\bbGamma\to\Cat$ factors over the isomorphism
$G\colon\bbGamma\to\bblox$ from Proposition~\ref{lem:iso-gamma-bblox} by interpreting $\{0,1\}^n$
as the group $\F_2^n$.

\begin{remark}
Since $\Aut_{\bbGamma}(\langle n\rangle) = \uS_n$ and $\Aut_{\ast/\F_2^n}(\ast) = \F_2^n$
there is a split short exact sequence
\begin{center}
\begin{tikzcd}
	0 & {\F_2^n} & {\Aut_{\int_{\bbGamma} \ast/\F_2^\bullet}(\langle n\rangle)} & {\uS_n} & 0
	\arrow[from=1-1, to=1-2]
	\arrow[from=1-2, to=1-3]
	\arrow[from=1-3, to=1-4]
	\arrow[from=1-4, to=1-5]
\end{tikzcd}
\end{center}
which exhibits $\Aut_{\int_{\bbGamma} \ast/\F_2^\bullet}(\langle n\rangle)$ as the group-theoretic semidirect product of $\F_2^n$ and $\uS_n$.
\end{remark}

The group $\F_2^n$ now acts as \emph{base-point shift} or, rather, as \emph{flip} on the $n$-cubes.
\begin{lemma}\label{lem:linearization-auto-bblox}
For every $n\in\N_0$ there exists an injective group homomorphism
\[
    \phi_{(-)}\colon\F_2^n\to\Aut_{{\bbox}}(\{0,1\}^n)
\]
with $\phi_a(0) = a$.
This group homomorphism is natural in ${\bblox}$ in the sense that for every map
$f\colon\{0,1\}^n\to\{0,1\}^m$ in ${\bblox}$ we have that $\phi_{f(a)}f = f\phi_a$ for all $a\in\{0,1\}^n$.

Moreover, for every morphism $f\colon\{0,1\}^n\to\{0,1\}^m$ there exists a unique factorization
\begin{center}
\begin{tikzcd}
	{\{0,1\}^n} & {\{0,1\}^m} & {\{0,1\}^m}
	\arrow["g", from=1-1, to=1-2]
	\arrow["\phi", from=1-2, to=1-3]
\end{tikzcd}
\end{center}
where $g$ is a morphism in ${\bblox}$ and $\phi$ is in the image of $\phi_{(-)}$.
In this case, $\phi = \phi_{f(0)}$.
\end{lemma}

\begin{proof}
Let $a\in\{0,1\}^n$.
Define $\phi_{a,i}(x_1,\ldots,x_n) = x_i$ if $a_i = 0$ and $\phi_{a,i}(x_1,\ldots,x_n) = 1-x_i$ if $a_i = 1$.
Thus, identifying $\{0,1\}^m$ with $\F_2^m$ the map $\phi_a$ simply is addition of $a$.
Hence, $\phi_a$ is self-inverse and thus an automorphism of $\{0,1\}^n$.
Moreover, $\phi_a\phi_b = \phi_{a+b}$ where the sum $a+b$ is taken in $\F_2^n$:
If $\phi_a$ is the identity, then $a = \phi_a(0) = 0$ by definition, so $\phi_{(-)}$ is injective.

To prove the second part, let $f\colon\{0,1\}^n\to\{0,1\}^m$ be a map in $\bblox$.
We need to show $\phi_{f(a)}f\phi_a = f$.
The idea is to compare coordinatewise -- either $f_i$ is zero and then $\phi_{f(a)}$ does not affect the $i$-th coordinate, resulting in $f$,
or both $\phi_a$ and $\phi_{f(a)}$ do the same thing (do or don't contribute a flip),
again resulting in $f$.

For $1\leq i\leq m$ we have that $f_i(x_1,\ldots,x_n) = y_i$ where $y_i \equiv 0$ or $y_i = x_{j(i)}$ for some unique $j(i)$ by Lemma~\ref{lem:coordinate_description}.
In the first case $f_i(x) = 0$, we have $f_i(a) = 0$, so $\phi_{f(a),i}(y_1,\ldots,y_m) = y_i$, hence
\[
    \phi_{f(a),i}(f(\phi_a(x_1,\ldots,x_n)))=f_i(\phi_a(x_1,\ldots,x_n)) = 0=f_i(a).
\]
In the second case $f_i(x) = x_{j(i)}$.
If $a_{j(i)}=0$, then $f_i(a)=0$ and so $\phi_{a,j(i)}(x)=x_{j(i)}$,
$f_i(x)= x_{j(i)}$ and $\phi_{f(a),i}(y) = y_i$, therefore
\[
    \phi_{f(a),i}(f(\phi_a(x))) = x_{j(i)} = f_i(x).
\]
If $a_{j(i)}=1$, then $f_i(a)=1$ and so $\phi_{a, j(i)}(x)=1-x_{j(i)}$,
$f_i(x)=x_{j(i)}$ and $\phi_{f(a),i}(y) = 1-y_i$, therefore
\[
    (\phi_{f(a)}(f(\phi_a))(x))_i = \phi_{f(a),i}(1-x_{j(i)}) = x_{j(i)} = f_i(x).
\]
This shows that indeed $\phi_{f(a)}f\phi_a = f$.

For the last part let $g = \phi_{f(0)}f$ which is linear since $g(0) = \phi_{f(0)}(f(0)) = 0$.
By construction, $\phi_{f(0)}$ is self-inverse and so $f$ factors as $\phi_{f(0)}\circ g$.
For uniqueness, take another factorization of $f = \phi'g'$ where $g'$ is linear and $\phi'$ is in the image of $\phi_{(-)}$.
Since $g'$ is linear, $g'(0)=0$ and we obtain that $\phi'(0) = f(0) = \phi_{f(0)}(0)$.
Therefore, $\phi'=\phi_{f(0)}$ and from self-inverseness of $\phi$ it follows that $g' = g$.
\end{proof}

\begin{proof}[Proof of Proposition~\ref{prop:cube_cat_linear}]
We obtain a functor
\[
    H\colon{\bbox}\to{\bblox}
\]
which is the identity on objects and it assigns to a morphism $f\colon\{0,1\}^n\to\{0,1\}^m$
the linear morphism $g\colon\{0,1\}^n\to\{0,1\}^m$ from Lemma~\ref{lem:linearization-auto-bblox} which is a functorial assignment since it is unique
and $\phi$ commutes as in Lemma~\ref{lem:linearization-auto-bblox}.
Clearly, $H$ exhibits the inclusion $j\colon{\bblox}\to{\bbox}$ as a retraction.

Now we can describe the functors which establish the equivalence $\int_{\bbGamma} \ast/\F_2^\bullet\simeq{\bbox}$ over $\bbGamma$ which will make for a commutative diagram
\begin{center}
\begin{tikzcd}
	\int_{\bbGamma} \ast/\F_2^\bullet & {{\bbox}} \\
	{\bbGamma} & {\bblox}
	\arrow["U", shift left, from=1-1, to=1-2]
	\arrow["{\mathrm{Un}}"', from=1-1, to=2-1]
	\arrow["V", shift left, from=1-2, to=1-1]
	\arrow["H", from=1-2, to=2-2]
	\arrow["G", shift left, from=2-1, to=2-2]
	\arrow["F", shift left, from=2-2, to=2-1]
\end{tikzcd}.
\end{center}
The functor $V$ acts as the identity on objects and it assigns to a morphism $f\colon\{0,1\}^n\to\{0,1\}^m$
\[
    V(f) = (F(H(f)),f(0)).
\]
For functoriality all that remains to be shown is that
\[
    f(g(0)) = f(0) + (F(H(f)))_\ast(g(0))
\]
in $\F_2^m$.\footnote{Which basically is affine linearity of $f$.}
But $F(H(f))_\ast = G(F(H(f))) = H(f)$ and so
\[
    (F(H(f)))_\ast(g(0)) = H(f)(g(0)) = \phi_{f(0)}(f(g(0)))
\]
for the unique $\phi_{f(0)}$ from Lemma~\ref{lem:linearization-auto-bblox}.
But in $\F_2^m$, the map $\phi_{f(0)}$ acts as addition of $f(0)$ and so
\[
    f(0)+\phi_{f(0)}(f(g(0))) = f(0) + f(0) + f(g(0)) = f(g(0)).
\]
It then automatically fulfills $\mathrm{Un}\circ V = F\circ H$.

The functor $U$ acts on objects identically and to a morphism $(f,a)$ it assigns $\phi_a\circ G(f)$.
Now functoriality follows from the statements made in Lemma~\ref{lem:linearization-auto-bblox}:
\begin{align*}
        U((f,a)\circ(g,b))
    &=  U(f\circ g,a+f_\ast(b)) \\
    &=  \phi_{a+f_\ast(b)}\circ G(f\circ g) \\
    &=  \phi_a\phi_{G(f)(b)}G(f)G(g) \\
    &=  \phi_a G(f)\phi_b G(g) \\
    &=  U(f,a)\circ U(g,b).
\end{align*}
In particular, $H\circ U = G\circ\mathrm{Un}$.
It is then easily verified that $U$ and $V$ form an equivalence of categories.
\end{proof}

\begin{remark}\label{rem:bbox-factorization}
In $\Fin_\ast$, every morphism $f\colon\langle m\rangle\to\langle n\rangle$
can be decomposed as
\begin{center}
\begin{tikzcd}
	{\langle m\rangle} & {\langle k\rangle} & {\langle n\rangle}
	\arrow["\alpha", from=1-1, to=1-2]
	\arrow["\sigma", from=1-2, to=1-3]
\end{tikzcd},
\end{center}
where $\alpha$ is the inert map that is the retraction of the order-preserving/canonical    
inclusion $\langle k\rangle^\circ\sub\langle m\rangle^\circ$ of the points $i$ where $f(i)\neq\ast$
and $\sigma$ is active.
This yields a decomposition of every map $f\colon\{0,1\}^n\to\{0,1\}^m$ in $\bbox$ as
\begin{center}
\begin{tikzcd}
	{\{0,1\}^n} & {\{0,1\}^k} & {\{0,1\}^m} & {\{0,1\}^m}
	\arrow["{\sigma^\ast}", from=1-1, to=1-2]
	\arrow["{\alpha^\ast}", from=1-2, to=1-3]
	\arrow["\phi", from=1-3, to=1-4]
\end{tikzcd}
\end{center}
where $\sigma^\ast$ is a (not necessarily bijective) permutation of coordinates arising from the active map $\sigma$,
$\alpha^\ast$ is a projection that only inserts zeros (arising from the inert $\alpha$)
and $\phi$ is the flip corresponding to a unique element $a\in\F_2^m$, in this case $a=f(0)$.
\end{remark}

\begin{example}
Let us draw the objects $n=1,2,3$ in the above categories.
Depicted are the lower-dimensional subobjects with their corresponding intersections.

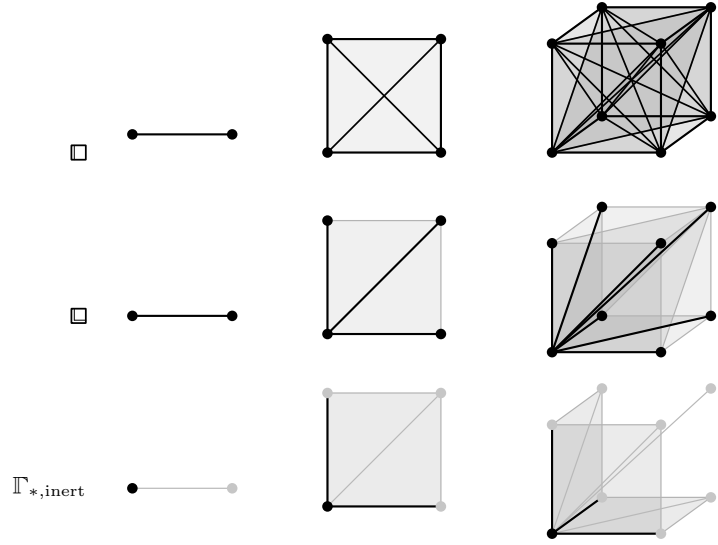
\begin{center}
\pgfdeclarelayer{facefill}
\pgfdeclarelayer{afffacefill}
\pgfdeclarelayer{faceoutline}
\pgfdeclarelayer{afffaceoutline}

\pgfsetlayers{facefill,afffacefill,faceoutline,afffaceoutline,main}

\newcommand{\Face}[1]{%
  \begin{pgfonlayer}{facefill}
    \begin{scope}[blend mode=multiply]
      \path[fill=gray, fill opacity=0.12] #1 -- cycle;
    \end{scope}
  \end{pgfonlayer}
  \begin{pgfonlayer}{faceoutline}
    \path[draw=gray!60,line width=0.45pt] #1 -- cycle;
  \end{pgfonlayer}
}

\newcommand{\CornerFace}[1]{%
  \begin{pgfonlayer}{facefill}
    \begin{scope}[blend mode=multiply]
      \path[fill=gray, fill opacity=0.16] #1 -- cycle;
    \end{scope}
  \end{pgfonlayer}
  \begin{pgfonlayer}{faceoutline}
    \path[draw=gray!60,line width=0.45pt] #1 -- cycle;
  \end{pgfonlayer}
}

\newcommand{\AffFace}[1]{%
  \begin{pgfonlayer}{afffacefill}
    \path[fill=gray, fill opacity=0.10] #1 -- cycle;
  \end{pgfonlayer}
  \begin{pgfonlayer}{afffaceoutline}
    \path[draw=gray!60,line width=0.45pt] #1 -- cycle;
  \end{pgfonlayer}
}

\newcommand{\DiagFace}[1]{%
  \begin{pgfonlayer}{afffacefill}
    \path[fill=gray, fill opacity=0.16] #1 -- cycle;
  \end{pgfonlayer}
  \begin{pgfonlayer}{afffaceoutline}
    \path[draw=gray!70,line width=0.5pt] #1 -- cycle;
  \end{pgfonlayer}
}

\begin{tikzpicture}[
    x=1cm,y=1cm,
    font=\small,
    line cap=round,
    line join=round,
    vertex/.style={circle,fill,inner sep=1.4pt},
    ambient/.style={draw=gray!55, line width=0.45pt},
    mainedge/.style={draw=black, line width=0.8pt},
    affdiag/.style={draw=black, line width=0.65pt},
    simpedge/.style={
        draw=black,
        line width=0.75pt,
        -{Stealth[length=1.6mm,width=1.1mm]},
        shorten <= 8.0pt,
        shorten >= 8.0pt
    },
    rowlab/.style={anchor=east},
    collab/.style={font=\normalsize},
scale=0.6
]

\node[rowlab] at (-0.8,  0.0) {$\bbox$};
\node[rowlab] at (-0.8, -3.6) {$\bblox$};
\node[rowlab] at (-0.8, -7.4) {$\inert$};

\begin{scope}[shift={(0,0.4)}]
    \coordinate (A) at (0,0);
    \coordinate (B) at (2.2,0);

    \draw[mainedge] (A)--(B);

    \node[vertex] at (A) {};
    \node[vertex] at (B) {};
\end{scope}

\begin{scope}[shift={(4.3,0.0)}]
    \coordinate (A) at (0,0);
    \coordinate (B) at (2.5,0);
    \coordinate (C) at (0,2.5);
    \coordinate (D) at (2.5,2.5);

    \AffFace{(A)--(B)--(D)--(C)}

    \draw[mainedge] (A)--(B);
    \draw[mainedge] (B)--(D);
    \draw[mainedge] (D)--(C);
    \draw[mainedge] (C)--(A);

    \draw[affdiag] (A)--(D);
    \draw[affdiag] (B)--(C);

    \foreach \P in {A,B,C,D} {
        \node[vertex] at (\P) {};
    }
\end{scope}

\begin{scope}[shift={(9.25,0.0)}]
    \coordinate (A) at (0,0);
    \coordinate (B) at (2.4,0);
    \coordinate (C) at (0,2.4);
    \coordinate (D) at (2.4,2.4);

    \coordinate (E) at (1.1,0.8);
    \coordinate (F) at (3.5,0.8);
    \coordinate (G) at (1.1,3.2);
    \coordinate (H) at (3.5,3.2);

    \AffFace{(A)--(B)--(D)--(C)}
    \AffFace{(E)--(F)--(H)--(G)}

    \AffFace{(A)--(B)--(F)--(E)}
    \AffFace{(C)--(D)--(H)--(G)}

    \AffFace{(A)--(C)--(G)--(E)}
    \AffFace{(B)--(D)--(H)--(F)}

    \DiagFace{(A)--(B)--(H)--(G)}
    \DiagFace{(A)--(C)--(H)--(F)}
    \DiagFace{(A)--(E)--(H)--(D)}

    \draw[mainedge] (A)--(B);
    \draw[mainedge] (B)--(D);
    \draw[mainedge] (D)--(C);
    \draw[mainedge] (C)--(A);

    \draw[mainedge] (E)--(F);
    \draw[mainedge] (F)--(H);
    \draw[mainedge] (H)--(G);
    \draw[mainedge] (G)--(E);

    \draw[mainedge] (A)--(E);
    \draw[mainedge] (B)--(F);
    \draw[mainedge] (C)--(G);
    \draw[mainedge] (D)--(H);

    \draw[affdiag] (A)--(D);
    \draw[affdiag] (B)--(C);

    \draw[affdiag] (E)--(H);
    \draw[affdiag] (F)--(G);

    \draw[affdiag] (A)--(F);
    \draw[affdiag] (B)--(E);

    \draw[affdiag] (C)--(H);
    \draw[affdiag] (D)--(G);

    \draw[affdiag] (A)--(G);
    \draw[affdiag] (C)--(E);

    \draw[affdiag] (B)--(H);
    \draw[affdiag] (D)--(F);

    \draw[affdiag] (A)--(H);
    \draw[affdiag] (B)--(G);
    \draw[affdiag] (C)--(F);
    \draw[affdiag] (D)--(E);

    \foreach \P in {A,B,C,D,E,F,G,H} {
        \node[vertex] at (\P) {};
    }
\end{scope}

\begin{scope}[shift={(0,-3.6)}]
    \coordinate (A) at (0,0);
    \coordinate (B) at (2.2,0);
    \draw[mainedge] (A)--(B);
    \node[vertex] at (A) {};
    \node[vertex] at (B) {};
\end{scope}

\begin{scope}[shift={(4.3,-4.0)}]
    \coordinate (A) at (0,0);
    \coordinate (B) at (2.5,0);
    \coordinate (C) at (0,2.5);
    \coordinate (D) at (2.5,2.5);

    \Face{(A)--(B)--(D)--(C)}

    \draw[mainedge] (A)--(B);
    \draw[mainedge] (A)--(C);
    \draw[mainedge] (A)--(D);

    \foreach \P in {A,B,C,D} {
        \node[vertex] at (\P) {};
    }
\end{scope}

\begin{scope}[shift={(9.25,-4.4)}]
    \coordinate (A) at (0,0);
    \coordinate (B) at (2.4,0);
    \coordinate (C) at (0,2.4);
    \coordinate (D) at (2.4,2.4);

    \coordinate (E) at (1.1,0.8);
    \coordinate (F) at (3.5,0.8);
    \coordinate (G) at (1.1,3.2);
    \coordinate (H) at (3.5,3.2);

    \Face{(A)--(B)--(D)--(C)}
    \Face{(A)--(E)--(G)--(C)}
    \Face{(A)--(B)--(F)--(E)}
    \Face{(A)--(D)--(H)--(E)}
    \Face{(A)--(B)--(H)--(G)}
    \Face{(A)--(C)--(H)--(F)}

    \draw[mainedge] (A)--(B);
    \draw[mainedge] (A)--(C);
    \draw[mainedge] (A)--(D);
    \draw[mainedge] (A)--(E);
    \draw[mainedge] (A)--(F);
    \draw[mainedge] (A)--(G);
    \draw[mainedge] (A)--(H);

    \foreach \P in {A,B,C,D,E,F,G,H} {
        \node[vertex] at (\P) {};
    }
\end{scope}

\begin{scope}[shift={(0,-7.4)}]
    \coordinate (A) at (0,0);
    \coordinate (B) at (2.2,0);

    \draw[ambient] (A)--(B);
    \node[vertex] at (A) {};
    \node[vertex,fill=gray!45] at (B) {};
\end{scope}

\begin{scope}[shift={(4.3,-7.8)}]
    \coordinate (A) at (0,0);
    \coordinate (B) at (2.5,0);
    \coordinate (C) at (0,2.5);
    \coordinate (D) at (2.5,2.5);

    \draw[ambient] (A)--(B);
    \draw[ambient] (A)--(C);
    \draw[ambient] (A)--(D);

    \draw[mainedge] (A)--(B);
    \draw[mainedge] (A)--(C);
\CornerFace{(A)--(B)--(D)--(C)};

    \foreach \P in {B,C,D} {
        \node[vertex,fill=gray!45] at (\P) {};
    }
            \node[vertex] at (A) {};

\end{scope}

\begin{scope}[shift={(9.25,-8.4)}]
    \coordinate (A) at (0,0);
    \coordinate (B) at (2.4,0);
    \coordinate (C) at (0,2.4);
    \coordinate (D) at (2.4,2.4);

    \coordinate (E) at (1.1,0.8);
    \coordinate (F) at (3.5,0.8);
    \coordinate (G) at (1.1,3.2);
    \coordinate (H) at (3.5,3.2);

    \draw[ambient] (A)--(B);
    \draw[ambient] (A)--(C);
    \draw[ambient] (A)--(D);
    \draw[ambient] (A)--(E);
    \draw[ambient] (A)--(F);
    \draw[ambient] (A)--(G);
    \draw[ambient] (A)--(H);

    \CornerFace{(A)--(B)--(D)--(C)}
    \CornerFace{(A)--(B)--(F)--(E)}
    \CornerFace{(A)--(C)--(G)--(E)}

    \draw[mainedge] (A)--(B);
    \draw[mainedge] (A)--(E);
    \draw[mainedge] (A)--(C);

    \foreach \P in {A} {
        \node[vertex] at (\P) {};
    }
    \foreach \P in {B,C,D,E,F,G,H} {
    \node[vertex,fill=gray!45] at (\P) {};
    }
\end{scope}
\end{tikzpicture}
    \end{center}
\end{example}

From now on, we will freely switch between $\bbox$ and $\int_{\bbGamma} \ast/\F_2^\bullet$ depending on which category is better suited for certain applications.
Moreover, we will denote objects of $\bbox$ by $\langle n\rangle$, $\{0,1\}^n$ or ${\bbox}^n$,
depending on the mood of the authors when writing the corresponding parts.%
\footnote{As rough guidance: $\langle n\rangle$ is mostly used for emphasizing the abstract object,
and for the relation to $\bbGamma$,
$\{0,1\}^n$ is written for either emphasizing the relation to the geometric intuition of a cube
or for the combinatorial description of the category,
and $\bbox^n$ is mostly used as the Yoneda embedding of $\langle n\rangle$.}

\subsubsection{Properties of the Indexing Category}

Having seen their importance for nilspace theory,
let us recall some important properties of the categories $\Fin_\ast$, $\bbGamma$ and $\Fin_{\ast,\leq 1}$.

\begin{lemma}\label{lem:basic-prop-indexcat}
\begin{enumerate}[(i)]
    \item The category $\Fin_\ast$ is a Cisinski generalized Reedy category%
    \footnote{see \cite[Def 8.1.1]{Cisinski2006} or \cite{Berger2010}, and \cite{Riehl2014,Riehl2017} for the theory of Reedy categories}
        and has a symmetric monoidal structure given on objects by $\langle n\rangle\otimes\langle m\rangle = \langle n+m\rangle$.
    \item The category $\bbGamma$ is a Cisinski generalized Reedy category and has a symmetric monoidal structure given on objects by $\langle n\rangle\otimes\langle m\rangle = \langle n+m\rangle$.
    \item The category $\Fin_{\ast,\leq 1}$ is a Cisinski generalized Reedy category and has a symmetric monoidal structure given on objects by $\langle n\rangle\otimes\langle m\rangle = \langle n+m\rangle$.
\end{enumerate}
\noindent
Furthermore, in $\Fin_\ast$,
\begin{enumerate}[(1)]
    \item the epimorphisms agree with surjective maps and the split epimorphisms, and
    \item the monomorphisms agree with (everywhere defined) injective maps and the split monomorphisms.
\end{enumerate}
\end{lemma}

\begin{proof}
Let us begin with the latter assertions.
\begin{enumerate}[(1)]
\item If $\alpha\colon\langle n\rangle\to\langle m\rangle$ is an epimorphism in $\Fin_\ast$
one can test against $\langle m\rangle\to\langle 1\rangle$ to see that $\alpha$ is surjective,
and the existence of the section is evident.
\item If $\alpha\colon\langle n\rangle\to\langle m\rangle$ is a monomorphism in $\Fin_\ast$
one can test against $\langle 1\rangle\to\langle n\rangle$ to see that $\alpha$ is injective and everywhere defined.
Every injection has a retraction given by the identity on the image of $\alpha$ and undefined everywhere else.
\end{enumerate}
Now to the structures on the categories.
\begin{enumerate}[(i)]
\item The degree function is given by $\mathrm{d}(\langle n\rangle) = n$.
The degree increasing maps $\Fin_{\ast,+}$ are the monomorphisms
and the degree decreasing maps $\Fin_{\ast,-}$ are the surjections.
\item One simply takes $\bbGamma_+ = ((\Fin_\ast)_-)^\op$ and $\bbGamma_- = ((\Fin_\ast)_+)^\op$.
\item Here one uses the everywhere defined injections as $(\Fin_{\ast,\leq 1})_+$ and $(\Fin_{\ast,\leq 1})_- = \finert$ which correspond to the surjections in $\Fin_{\ast,\leq 1}$.\qedhere
\end{enumerate}
\end{proof}

\begin{remark}
Note that there does not exist a Cisinski generalized Reedy structure on $\finert$ with $\mathrm{d}(\langle n\rangle) = n$
since the only inert maps admitting a section are the isomorphisms, so $(\finert)_-$ is the core of $\finert$.
But simultaneously this implies that $(\finert)_+$ needs to consist of all maps.
This contradicts the fact that every inert map is surjective and thus lowers the degree.
The same argument works for $\inert$.
\end{remark}

\begin{lemma}\label{lem:fin-star-co-comp}
The category $\Fin_\ast$ is both finitely complete and finitely cocomplete
and the projection $\Fin_\ast\to\Fin$ creates finite limits and connected colimits.
\end{lemma}

\begin{proof}
This is well known.
Recall that the coproduct of finite pointed sets $A$ and $B$ is the pushout $A\sqcup_\ast B$.
\end{proof}

With this, we directly see that the monoidal structure on $\bbGamma$ is cartesian.

\begin{corollary}\label{rem:bblox-symm-monoid}
The symmetric monoidal structure on $\bblox$ defined by
\[
    \langle n\rangle \otimes\langle m\rangle =\langle n+m\rangle
\]
agrees with the cartesian one, i.e.,
\[
    \langle n+m\rangle=\langle n\rangle \times\langle m\rangle.
\]
\end{corollary}

Now, these properties translate through the unstraightening to the whole cube category $\bbox$ as follows.

\begin{lemma}\label{lem:bbox-epi-mono}
In the category $\bbox$, we have that
\begin{enumerate}[(i)]
    \item the epimorphisms agree with the surjective maps and with the split epimorphisms, and
    \item the monomorphisms agree with the injective maps and with the split monomorphisms.
\end{enumerate}
Furthermore, the monomorphisms are those maps in whose canonical representation
each $x_j$ appears (as $x_j$ or as $1 - x_j$) at least once.
The epimorphisms are precisely those maps such that no $y_i$ is constant and every $x_j$ appears at most once.
\end{lemma}

\begin{proof}
Let $f\colon{\bbox}^n\to{\bbox}^m$ be an epimorphism in ${\bbox}$.
We factor it as $\phi\circ g$ with $\phi\in\F_2^m$, and so we can assume that $g$ is in $\bbGamma$
and an epimorphism.
By Lemma~\ref{lem:basic-prop-indexcat}, its opposite $\alpha$ in $\Fin_\ast$ is an everywhere defined injection and there exists a splitting $\beta$ with $\beta\circ\alpha = \id$.
The splitting transfers to $f=\phi g$ since $\phi$ is invertible, and so $f$ is a split epimorphism.
Now if $f$ is a split epimorphism, it is in particular split in $\set$ and so is surjective.
On the other hand, every surjection clearly is an epimorphism (by applying the forgetful functor $\bbox\to\set$).
Moreover, from the canonical representation of $g = \alpha^\ast$ where $\alpha$ is an everywhere defined injection,
we infer that there appears no constant term $0$ and each $x_i$ appears at most once, since $\alpha$ has fibers of size at most one.
Therefore, the same holds for $f$.
On the other hand, every map $f$ with the same properties on its canonical representation is clearly surjective.

If $f\colon{\bbox}^n\to{\bbox}^m$ is a monomorphism in $\bbox$,
then in its factorization $f=\phi\circ g$, so is $g$.
Again, by Lemma~\ref{lem:basic-prop-indexcat}, we have that for $g=\alpha^\ast$ the map $\alpha$ is a split epimorphism, and hence $g$ is split monic.
The same line of reasoning as above also shows that split monomorphisms are injective,
and injections are monomorphisms.
Furthermore, $g=\alpha^\ast$ for $\alpha$ surjective shows that in the canonical representation of $g$,
each $x_i$ appears at least once.
And conversely, every such map is injective.
\end{proof}

\begin{lemma}\label{lem:bbox-reedy-cat}
The category $\bbox$ inherits a Cisinski generalized Reedy structure from $\bbGamma$
where the degree of $\{0,1\}^n$ is given by $n$,
the monomorphisms are the degree increasing maps and the epimorphisms are the degree decreasing maps.
\end{lemma}

\begin{proof}
Clearly, the non-invertible monomorphisms increase, and the non-invertible epimorphisms decrease the degree (since in $\Fin$, an injection (resp. surjection) between sets of the same cardinality is invertible).
Moreover, the isomorphisms preserve the degree and are both mono- and epimorphisms.
The factorization $f = m\circ e$ with $e$ epimorphic and $m$ monomorphic
follows from the corresponding factorization in $\bbGamma$ (Lemma~\ref{lem:basic-prop-indexcat})
which also implies uniqueness up to isomorphism.
Lastly, Lemma~\ref{lem:bbox-epi-mono} shows that every epimorphism has a section
and since the factorization $f = \phi\circ g$ is unique,
the set of sections determines the epimorphism uniquely by Lemma~\ref{lem:basic-prop-indexcat}.
\end{proof}

\begin{lemma}\label{lem:bblox-bbox-pres-limits}
The inclusion $j\colon\bbGamma\to\bbox$ preserves finite limits.
\end{lemma}

\begin{proof}
For products, we have seen that $\langle n\rangle\times\langle m\rangle = \langle n+m\rangle$.
Applying $j$ this corresponds to the set-theoretical product $\{0,1\}^n\times\{0,1\}^m = \{0,1\}^{n+m}$
with linear projections.
Given another $k$-cube with morphisms ${\bbox}^k\to{\bbox}^n$ and ${\bbox}^k\to{\bbox}^m$,
there exists a unique set-map $\{0,1\}^k\to\{0,1\}^{n+m}$ making the diagram
\begin{center}
\begin{tikzcd}
	& {\{0,1\}^{n+m}} & \\
	{\{0,1\}^n} && {\{0,1\}^m} \\
	& {\{0,1\}^k}
	\arrow[from=1-2, to=2-1]
	\arrow[from=1-2, to=2-3]
	\arrow[dashed, from=3-2, to=1-2]
	\arrow[from=3-2, to=2-1]
	\arrow[from=3-2, to=2-3]
\end{tikzcd}
\end{center}
commutative.
Looking at its $i$-th coordinate we see that, by commutativity of the diagram
and the fact that all solid arrows are in $\bbox$, it is of the form $0$, $1$, $x_{j(i)}$ or $1-x_{j(i)}$ for a unique $1\leq j(i)\leq k$.
This means that the set-map is in $\bbox$ and so $j$ preserves products.

For the preservation of equalizers consider the diagram
\begin{center}
\begin{tikzcd}
	{j\langle k\rangle} & {j\langle m\rangle} & {j\langle m'\rangle} \\
	{{\bbox}^n}
	\arrow["e", from=1-1, to=1-2]
	\arrow["f", shift left, from=1-2, to=1-3]
	\arrow["g"', shift right, from=1-2, to=1-3]
	\arrow["L"', from=2-1, to=1-2]
\end{tikzcd}
\end{center}
where $e$ is the equalizer of $f$ and $g$ in $\bbGamma$
and $fL = gL$, but $L$ is a map in $\bbox$.
Since $e$ is monic in $\bbGamma$, it is injective in $\set$.
Moreover, for all $a\in\{0,1\}^m$ with $f(a) = g(a)$ there exists $c\in\{0,1\}^k$ such that $e(c) = a$:
If this weren't the case, we could define a map $\langle 1\rangle\to\langle m\rangle$ in $\bbGamma$ mapping $1$ to $a$
which then would get equalized by $f$ and $g$, but couldn't factor through $e$, a contradiction.
Hence, upon applying the forgetful functor to $\set$,
this diagram is an equalizer diagram.
Hence there exists a unique map $T\colon\{0,1\}^n\to\{0,1\}^k$ of sets(!) making the diagram
\begin{center}
\begin{tikzcd}
	{j\langle k\rangle} & {j\langle m\rangle} & {j\langle m'\rangle} \\
	{{\bbox}^n}
	\arrow["e", from=1-1, to=1-2]
	\arrow["f", shift left, from=1-2, to=1-3]
	\arrow["g"', shift right, from=1-2, to=1-3]
	\arrow["T", dashed, from=2-1, to=1-1]
	\arrow["L"', from=2-1, to=1-2]
\end{tikzcd}
\end{center}
commutative (in $\set$).
Now since $e$ is injective, from Lemma~\ref{lem:bbox-epi-mono} we infer that the $i$-th coordinate of $T$ is of the form $0$, $1$, $x_{j(i)}$ or $1-x_{j(i)}$ for a unique $1\leq j(i)\leq n$.
This means that the set-map $T$ is in $\bbox$ and so $j$ preserves equalizers.
\end{proof}

\begin{corollary}\label{prop:bbox-symm-monoid}
There is a symmetric monoidal structure on $\bbox$ defined by the formula
\[
    {\bbox}^n\otimes{\bbox}^m={\bbox}^{n+m}
\]
which agrees with the cartesian monoidal structure.
\end{corollary}

\begin{proof}
All that is to be shown is that we have ${\bbox}^n\times{\bbox}^m\simeq{\bbox}^{n+m}$.
But this follows from Lemma~\ref{lem:bblox-bbox-pres-limits} and Lemma~\ref{rem:bblox-symm-monoid}.
\end{proof}

\begin{remark}
Note that the property that the symmetric monoidal structure is the cartesian one
is false in other classical indexing categories, e.g., in the simplex category $\bbDelta$
or other variants of cube categories in cubical homotopy theory, see below.
\end{remark}

\begin{remark}
Note that Segal's category $\bbGamma$ is used to model $\bbGamma$-spaces,
see \cite{Bousfield1978,Brown1981,Schwede1999,Schwede2000},
as well as being the (opposite of the) $\E_\infty$-operad modeling commutative monoids in $\infty$-categories,
see \cite{Lurie2017,Hebestreit2021}.
The trivial operad $\finert$ in turn is used to define operads ($\infty$-multicategories).

Often, $\bbGamma$ is understood as a commutative variant of the classical simplex category $\bbDelta$
(which even is an Eilenberg-Zilber category).
One can relate them using the \emph{cut functor} of Dedekind cuts $\bbDelta\to\bbGamma$.
This functor exhibits every symmetric operad as a non-commutative operad
and shows that $\mathbb{E}_\infty$-monoids are $\mathbb{E}_1$-monoids, see \cite{Lurie2017,Hebestreit2021}.
There are also relations to other important indexing categories,
e.g. the trees of dendroidal sets, see \cite{Cisinski2011,Heuts2022}.
\end{remark}

\begin{remark}\label{def:cubical_homotopy_cubes}
There already is a well-established cubical homotopy theory, see e.g. \cite{Brown1981,Cisinski2006,Brown2011,Doherty2024}.
The cube category used in the theory of nilspaces has many more morphisms (diagonals)
than these cube categories.
However, it misses the usual \enquote{connections}.

In particular, there are Cisinski model structures to model spaces, and Joyal model structures to model $(\infty,n)$-categories using cubical sets.
However, those structures do not fit our purposes, and in particular their definitions of fibration and homotopy group differ from ours.
\end{remark}

\subsection{Cubesets: the Basics}\label{ssec:cubesets-basics}

Now that we have established what the basic building blocks of our geometric objects are,
namely the $n$-dimensional cubes ${\bbox}^n$, we can start creating geometric spaces out of them.
As usual, one wants to glue multiple cubes together to obtain \emph{cubesets}.
This is achieved by taking the free cocompletion of the category $\bbox$.

\subsubsection{Cubesets are Presheaves}

\begin{definition}\label{def:cubeset}
The category of cubesets $\cSet$ is the presheaf category on $\bbox$,
\[
    \cSet = \PSh(\bbox) = \Fun({\bbox}^{\op},\Set).
\]
\end{definition}

As a presheaf category, the category of cubesets $\cSet$ enjoys many pleasant properties.
One of the most important certainly is the fact that it is a Grothendieck topos
where all limits and colimits can be computed pointwise.

Recall that the Yoneda embedding $\bbox\to\PSh(\bbox)$ exhibits $\cSet$ as the free cocompletion
of the cube category $\bbox$.
This is described by the following universal property of the category $\cSet$ (by which it is uniquely determined):
every functor $\bbox\to\mcD$ where $\mcD$ has all small colimits
has a unique extension to a colimit preserving functor $\cSet\to\mcD$.
Concretely, this means that every presheaf $F$ can be represented as the canonical colimit over its category of elements,
or equivalently over the slice category $\bbox_{/F}$, that is over all cubes ${\bbox}^n\to F$ in $F$:
\[
    F = \varinjlim_{{\bbox}^n\to F}{\bbox}^n
\]
Moreover, this colimit is \emph{free} in the sense that for
$F = \varinjlim_{{\bbox}^n\to F} {\bbox}^n$ and $G = \varinjlim_{{\bbox}^m\to G} {\bbox}^m$,
the double limit formula
\[
    \hom(F,G) = \varprojlim_{{\bbox}^n\to F}\varinjlim_{{\bbox}^m\to G}\hom_{\bbox}({\bbox}^n,{\bbox}^m)
\]
holds.
This means that every cube ${\bbox}^n$ in $\cSet$ is a tiny object and in particular compact projective.
Thus, the category $\PSh(\bbox)$ is $\aleph_0$-presentable, so every colimit preserving functor
$G\colon\PSh(\bbox)\to\mcD$ admits a right adjoint, given on an object $D\in\mcD$ by the presheaf
\[
    \langle n\rangle\mapsto\hom(G({\bbox}^n),D).
\]
The necessity of defining the right adjoint as it is given follows from the adjunction and the Yoneda lemma.
Moreover, if $G$ preserves limits and is accessible
(i.e. preserves $\lambda$-filtered colimits for $\lambda$ big enough), and $\mcD$ is presentable,
then it admits a left adjoint (and conversely).

Another important construction is the lift of the symmetric monoidal structure
${\bbox}^n\otimes{\bbox}^m={\bbox}^{n+m}$ in $\bbox$ onto the full $\cSet$.
This is achieved by \emph{Day convolution}.

\begin{remark}[Day Convolution]
If $\mcC$ is a monoidal category, then one can define a tensor product on its presheaf category $\PSh(\mcC)$.
Given $A,B\in\PSh(\mcC)$, their tensor product $A\otimes B$ is given by left Kan extension
of the external product $A\overline{\times}B\colon\mcC^{\op}\times\mcC^{\op}\to\set$ along the tensor product
$\otimes^\op\colon\mcC^\op\times\mcC^\op\to\mcC$,
\begin{center}
\begin{tikzcd}
	{\mcC^\op\times\mcC^\op} && \mcC^\op \\
	{\set\times\set} \\
	\set
	\arrow["\otimes^\op", from=1-1, to=1-3]
	\arrow["{A\times B}"', from=1-1, to=2-1]
	\arrow[dashed, from=1-3, to=3-1]
	\arrow["\times"', from=2-1, to=3-1]
\end{tikzcd}
\end{center}
which is given by the coend
\[
    (A\otimes B)(e) = \int^{(c,d)\in\mcC\times\mcC}\hom(e,c\otimes_{\mcC}d)\times A(c)\times B(d)
\]
or equivalently, the colimit
\[
    (A\otimes B)(e) = \varinjlim_{e\to c\otimes d} (A(c)\times B(d)).
\]
This defines a monoidal structure on $\PSh(\mcC)$.
If the tensor product on $\mcC$ is symmetric, then so is the Day convolution product.
\end{remark}

Specialized to cubesets it is surprising to see that Day convolution agrees with the cartesian monoidal structure.

\begin{proposition}\label{prop:Day-conv-cubeset}
Day convolution on $\cSet = \PSh(\bbox)$ agrees with the cartesian monoidal structure.
\end{proposition}

\begin{proof}
Recall that on $\bbox$, the symmetric monoidal structure defined by ${\bbox}^n\otimes{\bbox}^m = {\bbox}^{n+m}$ is the cartesian one, see Proposition~\ref{prop:bbox-symm-monoid}.
Then the general case follows from the Kan formula and the fact that colimits are stable under base change in $\PSh(\bbox)$.
\end{proof}

\begin{remark}
The same line of reasoning also works when replacing $\bbox$ by $\bbGamma$:
the proof of the preceding Proposition~\ref{prop:Day-conv-cubeset} relies only on the fact
that on the site $\mcC$, the monoidal structure is the cartesian one.
But this is also the case for $\bbGamma = \bblox$ by Remark~\ref{rem:bblox-symm-monoid}.
\end{remark}

Next, let us discuss the internal Hom in $\cSet$.
That Day convolution agrees with the cartesian product also implies that the tensor product is closed.
Recall that in any category of presheaves, the internal Hom of two cubesets $X$ and $Y$ is given on objects by
\[
    \ihom(X,Y)(n) = \hom(X\times{\bbox}^n,Y).
\]
This description makes $\cSet$ into a closed symmetric monoidal category where the tensor product
is the cartesian product (or equivalently, Day convolution).

\begin{example}
The \emph{path space} $\uP X$ of a cubeset $X$ is given by $\uP X = \ihom([1],X)$,
which has simply $\uP X(n) = X(n+1)$, i.e., the cubeset of edges in $X$.

\end{example}

\begin{example}\label{ex:host_kra_cubegroups}
An important class of examples of cubesets is given by \emph{Host-Kra cubegroups}.
They are defined as follows:
For a filtered group $G_\bullet=(G_0=G_1\ge G_2\ge \cdots)$ with $[G_i,G_j]\subset G_{i+j}$
define $\HK(G_\bullet)$ to be the sub-cubeset of $i_*(G_0)$
whose $n$-cubes are the subgroup of $G^{2^n}$ spanned by configurations
\[
    g^F(\omega)=\begin{cases}
        g,\quad & \omega\in F\\
        1,\quad & \omega \notin F,
    \end{cases},\qquad \omega\in \{0,1\}^n
\]
for a face $F$ of $\{0,1\}^n$ and $g\in G_{\mathrm{codim} F}$.

In this paper, the most relevant case of these cubegroups is given by an abelian group $A$
with filtration $A=G_0=\dots=G_k\ge 0=G_{k+1}$.
The corresponding cubegroup is called the \emph{$k$-th Eilenberg-Mac Lane cubeset $\cD_k(A)$}.
For more on cubegroups we refer to \cite{Bihlmaier2026b}.
\end{example}

\begin{remark}\label{rem:EM_preserves_products}
Note that the Eilenberg-Mac Lane cubeset induces a functor $\cD_k\colon \Ab\to \cSet$
for every $k\geq 1$ that factors through \emph{abelian} cubegroups.
This functor preserves products.
\end{remark}

\subsubsection{Cube- and \texorpdfstring{$\bbGamma$-}{Gamma}sets}

Since every morphism of sites (the indexing categories) induces a triple of adjunctions
between the corresponding presheaf categories, and $\bbGamma$ is a retract of the unstraightening $\bbox$,
we obtain several comparison functors between cubesets and $\bbGamma$-sets.

Let us first consider the triple of adjunctions induced by the unstraightening $\mathrm{Un}\colon \bbox\to \bbGamma$.

\begin{remark}
The cocartesian fibration/Grothendieck opfibration $\mathrm{Un}\colon\bbox\to\bbGamma$ induces a triple of adjunctions
\begin{center}
\begin{tikzcd}
	{\PSh(\bbox)} && {\PSh(\bbGamma)}
	\arrow[""{name=0, anchor=center, inner sep=0}, "{{\mathrm{Un}_\ast}}"', curve={height=24pt}, from=1-1, to=1-3]
	\arrow[""{name=0p, anchor=center, inner sep=0}, phantom, from=1-1, to=1-3, start anchor=center, end anchor=center, curve={height=24pt}]
	\arrow[""{name=1, anchor=center, inner sep=0}, "{{\mathrm{Un}_!}}", curve={height=-24pt}, from=1-1, to=1-3]
	\arrow[""{name=1p, anchor=center, inner sep=0}, phantom, from=1-1, to=1-3, start anchor=center, end anchor=center, curve={height=-24pt}]
	\arrow[""{name=2, anchor=center, inner sep=0}, "{{\mathrm{Un}^\ast}}"', from=1-3, to=1-1]
	\arrow[""{name=2p, anchor=center, inner sep=0}, phantom, from=1-3, to=1-1, start anchor=center, end anchor=center]
	\arrow[""{name=2p, anchor=center, inner sep=0}, phantom, from=1-3, to=1-1, start anchor=center, end anchor=center]
	\arrow["\dashv"{anchor=center, rotate=-81}, draw=none, from=1p, to=2p]
	\arrow["\dashv"{anchor=center, rotate=-99}, draw=none, from=2p, to=0p]
\end{tikzcd}
\end{center}
where $\mathrm{Un}_!$ and $\mathrm{Un}_\ast$ are the pointwise left, resp. right, Kan extension
of $\mathrm{Un}$ along the Yoneda embedding $\bbox\to\PSh(\bbox)$.
The right adjoint $\mathrm{Un}_\ast$ is given by
\[
    \mathrm{Un}_\ast X(n) = \hom_{\PSh(\bbGamma)}(\langle n\rangle,\mathrm{Un}_\ast X) = \hom_{\cSet}(\mathrm{Un}^\ast\langle n\rangle,X)
\]
where $\mathrm{Un}^\ast\langle m\rangle = y(\langle m\rangle)\circ\mathrm{Un}$.
The left adjoint can be computed via the formula
\begin{align*}
    \mathrm{Un}_!X(m) &= \varinjlim_{\bbox^n\to\langle m\rangle}X(n),
\end{align*}
where the index category runs over all maps ${\bbox}^n\to\langle m\rangle$ for fixed $m$.
It equivalently can be described as the unique cocontinuous extension of the horizontal functor
\begin{center}
\begin{tikzcd}
	\bbox & \bbGamma & {\PSh(\bbGamma)} \\
	\\
	\cSet
	\arrow["{\mathrm{Un}}", from=1-1, to=1-2]
	\arrow["y"', from=1-1, to=3-1]
	\arrow["y", from=1-2, to=1-3]
	\arrow["{\mathrm{Un}_!}", dashed, from=1-3, to=3-1]
\end{tikzcd}
\end{center}
making the diagram commutative, as in the universal property of the free cocompletion.
Indeed, for any $X\in\PSh(\bbGamma)$, we compute that
\begin{align*}
    \hom_{\PSh(\bbGamma)}(\mathrm{Un}_! y_{\cSet}(n),X) 
    &= \hom_{\cSet}(y_{\cSet}(n),\mathrm{Un}^\ast X) \\
    &= \mathrm{Un}^\ast X(n) \\
    &= X(\mathrm{Un}(n)) \\
    &= X(n)  \\
    &= \hom_{\PSh(\bbGamma)}(\langle n\rangle,X)
\end{align*}
and so the claim follows from the Yoneda lemma.

Such a triple of adjoint functors always induces a geometric morphism $\mathrm{Un}\colon\cSet\to\PSh(\bbGamma)$
whose right adjoint part is $\mathrm{Un}_\ast$ and the left exact left adjoint is $\mathrm{Un}^\ast$.
\end{remark}

Another important example is that of the wide subcategory inclusion $j\colon\bbGamma=\bblox\hookrightarrow\bbox$, which as before induces a triple of adjoint functors.

\begin{remark}\label{rem:adjoint-triple-j}
Consider the triple of adjunctions
\begin{center}
\begin{tikzcd}
	\cSet && {\PSh(\bbGamma)}
	\arrow[""{name=0, anchor=center, inner sep=0}, "{{j^\ast}}", from=1-1, to=1-3]
	\arrow[draw=none, from=1-1, to=1-3]
	\arrow[""{name=1, anchor=center, inner sep=0}, "{{j_\ast}}", curve={height=-24pt}, from=1-3, to=1-1]
	\arrow[""{name=2, anchor=center, inner sep=0}, "{{j_!}}"', curve={height=24pt}, from=1-3, to=1-1]
	\arrow["\dashv"{anchor=center, rotate=-65}, draw=none, from=0, to=1]
	\arrow["\dashv"{anchor=center, rotate=-115}, draw=none, from=2, to=0]
\end{tikzcd}
\end{center}
Here, $j^*$ is the pullback of $j$, simply forgetting the affine structure of a cubeset and only remembering the linear maps.
The left adjoint $j_!$ is the unique cocontinuous extension of the functor $\bbGamma\to\cSet$ assigning $\langle n\rangle$ to ${\bbox}^n$.
We have that $j_!(y_{\bbGamma}(n)) = y_{\bbox}(n)$.
Indeed, for any $X\in\cSet$, we compute that
\[
    \hom_{\cSet}(j_! y_{\bbGamma}(n),X) = \hom_{\PSh(\bbGamma)}(y_{\bbGamma}(n),j^\ast X) = j^\ast X(n) = X(j(n)) = X(n) = \hom_{\cSet}({\bbox}^n,X)
\]
and so the claim follows from the Yoneda lemma.
\end{remark}

More interestingly, in this case not only is $j_\ast\colon\PSh(\bbGamma)\to\cSet$
(the right adjoint part of) a geometric morphism but also $j^\ast\colon\cSet\to\PSh(\bbGamma)$.

\begin{lemma}\label{lem:j-geometric-morphism}
The functor $j_!$ preserves finite limits and finite unions of subobjects.
\end{lemma}

\begin{proof}
For $X$ a presheaf on $\bbGamma$, by the formula for the pointwise Kan extension we have that
\[
    j_! X({\bbox}^n) = \varinjlim_{{\bbox}^n\to j(\langle m\rangle)}X(\langle m\rangle)
\]
where we denote by ${\bbox}^n$ an object of $\bbox$ and by $\langle m\rangle$ an object in $\bbGamma$,
and the index diagram is the projection
\[
    Q\colon ({\bbox}^n\downarrow j)\to\bbGamma
\]
where $({\bbox}^n\downarrow j)$ is the category of objects in $\bbGamma$ $j$-over ${\bbox}^n$.
By Lemma~\ref{lem:fin-star-co-comp}, $\bbGamma$ has all finite limits
and by Lemma~\ref{lem:bblox-bbox-pres-limits}, $j$ preserves them.
Now it is a general fact that if $\bbGamma$ is finitely complete and $j$ preserves finite limits,
then also $({\bbox}^n\downarrow j)$ has finite limits and the projection $Q$ creates them.
This means that the index category in the colimit formula for $j_! X$ is cofiltered,
which in turn means that, since finite limits commute with filtered colimits in $\set$,
that $j_!$ preserves finite limits.

Being a left adjoint, $j_!$ preserves colimits.

The union of two subobjects is given by the pushout of their pullback,
and $j_!$ preserves both those operations, so it preserves binary unions of subobjects.
By induction, it preserves all finite unions of subobjects.
\end{proof}

\begin{remark}
For an example of a $\bbGamma$-set which is not a cubeset (doesn't arise from $j^\ast X$ for some $X\in\cSet$)
consider the \emph{upper Host-Kra cubegroups}.
For a filtered group $G_\bullet$ the upper Host-Kra $\bbGamma$-group $\HK_*(G_\bullet)$ can be constructed
by defining the value on $\langle n\rangle $ to be the subgroup of $G^{\{0,1\}^n}$ generated by all elements of the form
\[
    g^F\colon\{0,1\}^n\to G,\quad \omega\mapsto \begin{cases}
        g,\quad & \text{if}\;\omega\notin F, \\
        e,\quad & \text{else},
    \end{cases}
\]
for $F\subsetneqq \{0,1\}^n$ being an inert face of codimension $i>0$ and $g\in G_i$.
Compare also with the full Host-Kra cubegroup~\ref{ex:host_kra_cubegroups},
see also \cite{Bihlmaier2026b}.
\end{remark}

A further nice application of adjoint functors induced by morphisms of sites is the comparison to plain sets, and its adjoints.

\begin{lemma}\label{lem:point-cset}
Consider the embedding of the terminal object $i\colon\ast\to\bbox$.
This induces a triple of adjoints
\begin{center}
\begin{tikzcd}
	\cSet && {\Set.}
	\arrow[""{name=0, anchor=center, inner sep=0}, "{{i^\ast}}", from=1-1, to=1-3]
	\arrow[draw=none, from=1-1, to=1-3]
	\arrow[""{name=1, anchor=center, inner sep=0}, "{{i_\ast}}", curve={height=-24pt}, from=1-3, to=1-1]
	\arrow[""{name=2, anchor=center, inner sep=0}, "{{i_!}}"', curve={height=24pt}, from=1-3, to=1-1]
	\arrow["\dashv"{anchor=center, rotate=-65}, draw=none, from=0, to=1]
	\arrow["\dashv"{anchor=center, rotate=-115}, draw=none, from=2, to=0]
\end{tikzcd}
\end{center}
Here, the pullback $i^\ast\colon\PSh(\bbox)\to \Set=\PSh(\ast)$ of $i\colon\ast\to\bbox$ is the evaluation at $0$.
Its left adjoint $i_!$ sends a set $S$ to the constant presheaf ${\bbox}^n\mapsto S$.
The right adjoint $i_\ast$ assigns to a set $S$ the cubeset given by
\[
    {\bbox}^n\mapsto\hom_{\set}(\{0,1\}^n,S) =  S^{2^n}.
\]
Both adjoints $i_!$ and $i_\ast$ are fully faithful.
\end{lemma}

\begin{proof}
Existence of the adjoints is clear from general nonsense.
The explicit description can be seen by writing down the universal properties:
if $X$ is a cubeset and $S$ a set, then a map $S\to X(0)$ gives rise to a unique morphism
$i_!S\to X$ which factors pointwise as $i_!S(n) = S \to X(0) \to X(n)$ via the unique map $n\to 0$.
In particular, we have that evaluation at $0$ induces a natural bijection
\[
    \hom_{\cSet}(i_!S,X) \to \hom_{\set}(S,X(0)).
\]
On the other hand, a map $X(0)\to S$ induces a unique map $X(n)\to X(0)^{2^n}\to S^{2^n}$, so
\[
    \hom_{\cSet}(X,S^{2^{(-)}}) = \hom_{\set}(X(0),S).
\]
Alternatively, by the (co)limit formula for left and right Kan extension it is clear that
\begin{align*}
    i_!S(n) &= \varinjlim_{n\to 0} S = S, \\
    i_\ast S(n) &= \varprojlim_{0\to n} S = \prod_{0\to n} S = S^{2^n}.
\end{align*}
The fully faithfulness of $i_\ast$ follows directly from the fact that the counit $i^\ast i_\ast S\to S$ is an isomorphism:
$i^\ast i_\ast S$ simply is $i_\ast S(\ast) = S$.
But now the fully faithfulness of $i_*$ by general nonsense also implies fully faithfulness of $i_!$.
\end{proof}

\begin{definition}
If $S$ is a set, we call $i_! S$ the \emph{trivial cubeset} and $i_* S$ the \emph{discrete cubeset}.
\end{definition}

\begin{remark}\label{rem:left-adjoint-i_!}
The functor $i_!$ has a further left adjoint:
since $i_!$ assigns to a set $S$ the constant diagram,
its left adjoint assigns to a presheaf $X\in\cSet$ the colimit of $X$ over ${\bbox}^\op$.
In particular, $i_!$ also preserves limits.
\end{remark}

\begin{remark}
We could also consider many more triples of adjunctions between other presheaf categories, e.g. simplicial sets, by extending the diagram of index-categories
\begin{center}
\begin{tikzcd}
 \inert & {\mathrm{Fin}_{*,\le 1}^{\op}} & \bbGamma & \bbox & {\int_{\bbGamma}*/\mathbb{F}_2^{\bullet}} \\
	&& \bbDelta
	\arrow[from=1-1, to=1-2]
	\arrow[from=1-2, to=1-3]
	\arrow["j", from=1-3, to=1-4]
	\arrow["{=}"{marking, allow upside down}, draw=none, from=1-4, to=1-5]
	\arrow["{{\mathrm{Cut}}}", from=2-3, to=1-3]
\end{tikzcd}
\end{center}
where all functors are bijective on objects, the horizontal arrows are faithful and none is full,
and $\mathrm{Un}\circ j = \id$.
\end{remark}

For the rest of this subsection we will give another description of $\cSet = \PSh(\bbox)$ in terms of its unstraightening over $\bbGamma$.
Since $\bbox$ is an unstraightening of $\bbGamma$, we expect to be able to translate $\cSet$ as $\bbGamma$-sets with a certain action.
Let us explain this in more detail.

\begin{construction}
Let $\mcC$ be a small category, $G\colon\mcC\to\Grp$ a functor and let $\mathrm{Un}\colon\mcD\to\mcC$ be its unstraightening.
Given a presheaf $F\in\PSh(\mcC)$, one obtains a bifunctor $\Hom(F,F)\colon\mcC\times\mcC^\op\to\set$ 
given by precomposing the Hom-functor $\hom\colon\set^\op\times\set\to\set$ with the functor
\[
    F^\op\times F\colon\mcC\times\mcC^\op\to\set^\op\times\set.
\]
Explicitly,
\[
    \hom(F,F)(X,Y) = \hom_\set(F(X),F(Y)),\qquad X,Y\in\mcC.
\]
Moreover, we can extend $G$ to be a functor $\mcC\times\mcC^\op\to\set$ which only depends on the first
variable and uses that every group has an underlying set.

Now construct the following category $\mcE_{G,\cC}(2)$:
\begin{itemize}
    \item its objects are tuples $(F,\alpha)$ where $F$ is a presheaf on $\mcC$
        and $\alpha$ is a dinatural transformation $\alpha\colon G\to\hom(F,F)$
        whose components $\alpha_X$ are monoid homomorphisms.
	\item A morphism $\eta\colon (F,\alpha)\to(F',\alpha')$ is a natural transformation $\eta\colon F\to F'$
        such that for all $X\in\mcC$, the diagram
\begin{center}
\begin{tikzcd}
	{G(X,X)} & {\hom(F,F)(X,X)} \\
	{\hom(F',F')(X,X)} & {\hom(F(X),F'(X))}
	\arrow["{\alpha_X}", from=1-1, to=1-2]
	\arrow["{\alpha'_X}"', from=1-1, to=2-1]
	\arrow["{\eta_{X\ast}}", from=1-2, to=2-2]
	\arrow["{\eta_X^\ast}"', from=2-1, to=2-2]
\end{tikzcd}
\end{center}
is commutative.
\end{itemize}
Unwinding the definitions, the datum of the objects of $\mcE_{G,\cC}(2)$ is the same
as giving the datum of a presheaf $F\colon\mcC^\op\to\set$ together with a
$G(X)$-action $\alpha_X$ on $F(X)$ for every $X\in\mcC$ such that for each map $f\colon X\to Y$ in $\mcC$ and every $g\in G(X)$,
the diagram
\begin{center}
\begin{tikzcd}
	{F(Y)} & {F(X)} \\
	{F(Y)} & {F(X)}
	\arrow["{F(f)}", from=1-1, to=1-2]
	\arrow["{\alpha_Y(G(f)g)}"', from=1-1, to=2-1]
	\arrow["{\alpha_X(g)}", from=1-2, to=2-2]
	\arrow["{F(f)}"', from=2-1, to=2-2]
\end{tikzcd}
\end{center}
is commutative.
\end{construction}

But this information suffices to reconstruct a presheaf on $\mcD$, so we obtain the following Proposition.

\begin{proposition}\label{prop:unstrait-presehaf}
The categories $\PSh(\mcD)$ and $\mcE_{G,\cC}(2)$ are equivalent.
\end{proposition}

\begin{proof}
Restricting a presheaf $F$ on $\mcD$ to $\mcC$
yields a presheaf $F$ on $\mcC$ together with a $G(X)$-action $\alpha_X$ on $F(X)$
(which is just a monoid homomorphism $G(X)=G(X,X)\to \hom(F,F)(X,X) = \End F(X)$) by defining $\alpha_X(g)=F((\id_X, g^{-1}))$.
This action has the following property:
for every map $f\colon X\to Y$ in $\mcC$ and every $g\in G(X)$, we have that
$F(f)\circ\alpha_Y(G(f)g) = \alpha_X(g)\circ F(f)$.
This is due to the fact that in $\mcD$ we have that a morphism $(f,G(f)g^{-1})\colon X\to Y$ in $\mcD$
can be factored as $(f,e)\circ (\id_X,g^{-1})$ and as $(\id_Y,G(f)g^{-1})\circ (f,e)$,
and applying $F$ to both factorizations yields the desired equality.

In the other direction one can easily extend an object $(F,\alpha)$ in $\mcE_{G,\cC}(2)$ to a presheaf on $\mcD$
which is given on objects by $F(X)$ and a map $(f,g) = (\id_Y,g)\circ (f,e)$ in $\mcD$
is assigned to $F(f)\circ \alpha_Y(g^{-1})\colon F(Y)\to F(X)$.

Then it is routine to check that these constructions induce an equivalence of categories.
\end{proof}

This allows one to see cubesets as $\bbGamma$-sets with a compatible family of $\mathbb{F}_2^n$-actions on the set of $n$-cubes.
We can also encode the functoriality differently by not requiring a condition on the action to be compatible with some given presheaf transition maps
but rather by first giving the actions and then requiring the transition maps to be action preserving homomorphisms.

\begin{remark}
Another description of the presheaf category on $\mcD$ is the following:
define a category whose objects are tuples
$(F_X)_{X\in\mcC}$ where $F_X$ is a presheaf on $*/G(X)$ (i.e., a $G$-set/$G$-module/a set with $G$-action) together with a functorial assignment of morphisms
$\phi\colon G(f)^\ast F_Y\to F_X$ for every map $f\colon X\to Y$ in $\mcC$.
And a morphism between two such objects is a family of natural transformations,
compatible with the maps $\phi$.
Then this category also is equivalent to $\PSh(\mcD)$.
\end{remark}

\subsection{Concrete Cubesets}\label{ssec:concrete-cubesets}

The notion of concrete sheaf \cite{Dubuc1979} specializes to the topos of cubesets,
giving us the notion of concrete cubeset.
This recovers the classical notion of nilspace, since all nilspaces considered in the literature
are concrete cubesets, cf. \cite[Definition 1.2.1]{Candela2017} and \cite[Definition 3.2]{Gutman2020}.

\begin{definition}\label{def:concrete-cubeset}
A cubeset $X$ is called \emph{concrete} if the canonical map
\[
    \hom({\bbox}^n,X) = X(n) \to \prod_{0\to n} X(0) = (X(0))^{2^n}
\]
is injective.
We denote the full subcategory spanned by concrete cubesets by $\ccSet$.
\end{definition}

\begin{example}
If $S$ is a set, the trivial cubeset $i_! S$ and the discrete cubeset $i_* S$ are concrete.
Moreover, the Host-Kra cubegroups from Example~\ref{ex:host_kra_cubegroups} are concrete.
\end{example}

\begin{remark}
In view of the triple of adjoint functors from Lemma~\ref{lem:point-cset},
$X$ is concrete precisely if the unit $X\to i_\ast i^\ast X$ is a monomorphism
(recall that $i_{*}$ assigns to a set the discrete cubeset $n\mapsto X(0)^{2^n}$, and $i^*$ is simply evaluating at $0$).
Since being a monomorphism can be checked pointwise on sections, this means that the map
\[
    \hom({\bbox}^n,X)\to\hom({\bbox}^n,i_\ast i^\ast X) = \hom(i_! i^\ast {\bbox}^n,X),
\]
given by pullback along the counit $i_!i^\ast{\bbox}^n\to{\bbox}^n$ is a monomorphism.
Taking
\[
    M = \{i_!i^\ast{\bbox}^n\to{\bbox}^n : n\in\N_0\},
\]
we see that the concrete cubesets are precisely the $M$-separated presheaves on ${\bbox}^n$.
The $M$-local presheaves (=$M$-sheaves) are precisely the ones where $X \simeq i_\ast i^\ast X$,
And by fully faithfulness of $i_\ast$ this implies that the category of $M$-local presheaves is equivalent to $\Set$ under the inclusion $i_\ast\colon\set\to\cSet$.

Put differently, the geometric morphism $i\colon\set\to\cSet$ induces a Lawvere-Tierney topology on $\cSet$,
which corresponds uniquely to a Grothendieck topology on $\bbox$,
whose covers are given by the maps $i_!i^\ast{\bbox}^n\to{\bbox}^n$,
such that the sheaves for this topology are precisely $\set$, and the separated presheaves are the concrete cubesets.
See also \cite[Chapter V]{MacLane1994} and \cite[Chapter A4.3]{Johnstone2002}.
\end{remark}

\begin{lemma}\label{lem:concrete-morphism}
A map $f\colon X\to Y$ of concrete cubesets is fully determined by its value $f_0$ on points.
\end{lemma}

\begin{proof}
This follows from the commutativity of the diagram
\begin{center}
\begin{tikzcd}
	{X(n)} & {Y(n)} \\
	{\prod\limits_{0\to n}X(0)} & {\prod\limits_{0\to n}Y(0)}
	\arrow["{f_n}", from=1-1, to=1-2]
	\arrow[hook, from=1-1, to=2-1]
	\arrow[hook, from=1-2, to=2-2]
	\arrow["{\prod f_0}"', from=2-1, to=2-2]
\end{tikzcd}
\end{center}
and since the right vertical map is monic.
\end{proof}

\begin{proposition}\label{prop:concrete-left-adjoint}
The inclusion $\ccSet\to\cSet$ preserves limits and has a left adjoint $(-)^\mathrm{conc}\colon\cSet\to\ccSet$.
This left adjoint, the \emph{concretification},%
\footnote{also called concretization}
takes a cubeset $X$ and assigns to it, at stage $n$, the image of the map $X(n)\to\prod_{0\to n}X(0)$.
\end{proposition}

\begin{proof}
Being concrete is stable under limits and so $\ccSet\sub\cSet$ preserves them.
Since the image is functorial (in any elementary topos, at least),
the described construction indeed makes a concrete cubeset out of a (general) cubeset.
Given a map $f\colon X\to Y$ where $Y$ is concrete, by functoriality of the image
there exists a unique map $g_n\colon\im_n\to Y_n$ making the diagram
\begin{center}
\begin{tikzcd}
	{X(n)} & {\im_n} & {\prod\limits_{0\to n}X(0)} \\
	{Y(n)} & {Y(n)} & {\prod\limits_{0\to n}Y(0)}
	\arrow[two heads, from=1-1, to=1-2]
	\arrow["{{f_n}}", from=1-1, to=2-1]
	\arrow[hook, from=1-2, to=1-3]
	\arrow["{{g_n}}", dashed, from=1-2, to=2-2]
	\arrow["{{\prod f_0}}", from=1-3, to=2-3]
	\arrow["{=}"', from=2-1, to=2-2]
	\arrow[hook, from=2-2, to=2-3]
\end{tikzcd}
\end{center}
commutative.
This shows that taking the image is left adjoint to the inclusion $\ccSet\sub\cSet$.
\end{proof}

Differently than with limits, not every colimit of concrete cubesets is concrete.

\begin{example}\label{ex:concrete-colimit-not}
Consider the 2-corner ${\bbcor}^2\sub{\bbox}^2$ which is defined as the pushout of the inert map $\ast\to{\bbox}^1$ along itself.
Then ${\bbcor}^2$, ${\bbox}^2$ and $\ast$ are concrete cubesets, but the pushout
\begin{center}
\begin{tikzcd}
	{{\bbcor}^2} & \ast \\
	{{\bbox}^2} & {\bbox}^2/{\bbcor}^2
	\arrow[from=1-1, to=1-2]
	\arrow[from=1-1, to=2-1]
	\arrow[from=1-2, to=2-2]
	\arrow[from=2-1, to=2-2]
	\arrow["\lrcorner"{anchor=center, pos=0.125, rotate=180}, draw=none, from=2-2, to=1-1]
\end{tikzcd}
\end{center}
is not.
First of all, ${\bbox}^2/{\bbcor}^2(0) = \{0,1\}$ where $0$ corresponds to the corner ${\bbcor}^2$ and $1$ to the other point in ${\bbox}^2$.
Second, note that ${\bbox}^2/{\bbcor}^2(1)$ is the set of all maps ${\bbox}^1\to{\bbox}^2$ where two maps get identified if and only if they both factor over the corner ${\bbcor}^2\sub{\bbox}^2$.
This follows from the fact that pushouts are computed pointwise in $\cSet$ and then one checks the universal property directly.
This means that ${\bbox}^2/{\bbcor}^2(1)$ has more than four elements and so cannot be concrete since ${\bbox}^2/{\bbcor}^2(0)^2 = \{0,1\}^2$.
More explicitly, in ${\bbox}^2/{\bbcor}^2$ there exists a nontrivial 1-cube with vertices $0$ and $0$, induced by the off-diagonal.
\end{example}

However, as usual, concrete cubesets are stable under filtered colimits and coproducts.

\begin{proposition}\label{prop:concrete-stability}
Coproducts and filtered colimits of concrete cubesets are concrete.
Moreover, subobjects of concrete cubesets are concrete.
\end{proposition}

\begin{proof}
The claim about filtered colimits follows from the fact that filtered colimits commute with finite limits in $\set$.
For the claim about coproducts note that the map $\coprod_i X_i(n) \to \prod_{0\to n}\coprod_i X_i(0)$ factors as
\[
    \coprod_i X_i(n) \to \coprod_i\prod_{0\to n} X_i(0) \to \prod_{0\to n}\coprod_i X_i(0)
\]
where the second map is the natural map arising from swapping the colimit with the limit,
which in our case is injective.
The first map is just a coproduct of injections and hence also injective.

For subobjects consider an inclusion $X\to Y$ with $Y$ concrete.
Now, since monomorphisms are injections sectionwise,
by naturality of the map we obtain the commutative square
\begin{center}
\begin{tikzcd}[ampersand replacement=\&]
	{Y(n)} \& {\prod_{0\to n}Y(0)} \\
	{X(n)} \& {\prod_{0\to n}X(0)}
	\arrow[hook, from=1-1, to=1-2]
	\arrow[hook, from=2-1, to=1-1]
	\arrow[from=2-1, to=2-2]
	\arrow[from=2-2, to=1-2]
\end{tikzcd}
\end{center}
where the left vertical and the upper horizontal maps are monomorphisms.
Thus, $X(n)\to \prod_{0\to n}X(0)$ is monic as well.
\end{proof}

Moreover, concrete cubesets are stable under taking the internal Hom,
essentially due to the fact that ${\bbox}^n$ is concrete and a morphism between concrete cubesets is determined on points.

\begin{proposition}\label{prop:concrete-ihom-prod}
Let $X$ and $Y$ be cubesets, where $Y$ is concrete.
Then $\ihom(X,Y)$ is concrete.
Furthermore, concretification preserves finite products.
\end{proposition}

\begin{proof}
First we treat the case where $X$ and $Y$ are concrete, and hence so is $X\times{\bbox}^n$.
Therefore, $\hom(X\times{\bbox}^n,Y)$ consists of all maps $f_0\colon X(0)\times\{0,1\}^n\to Y(0)$ that satisfy some extra condition (see Lemma~\ref{lem:concrete-morphism}).
But $\prod_{0\to n} \hom(X,Y)$ consists of $2^n$ maps $f_0\colon X(0)\to Y(0)$ that satisfy the same kind of condition.
Therefore, the canonical map
\[
    \ihom(X,Y)(n) = \hom(X\times{\bbox}^n,Y) \to \prod_{0\to n}\hom(X,Y) = \prod_{0\to n}\ihom(X,Y)(0)
\]
is injective.
Now let $X$ be an arbitrary cubeset, and take $Z$ to be concrete.
Then we have that
\[
    \hom((X\times{\bbox}^n)^\mathrm{conc},Z) = \hom(X\times{\bbox}^n,Z) = \hom(X,\ihom({\bbox}^n,Z)).
\]
By the previous case, $\ihom({\bbox}^n,Z)$ is concrete.
Therefore, the right hand side is equal to
\[
    \hom(X,\ihom({\bbox}^n,Z)) = \hom(X^\mathrm{conc},\ihom({\bbox}^n,Z)) = \hom(X^\mathrm{conc}\times{\bbox}^n,Z) = \hom(X^\mathrm{conc}\times({\bbox}^n)^\mathrm{conc},Z)
\]
where the last equality follows from the fact that ${\bbox}^n$ is concrete.
By the Yoneda lemma we obtain that $(X\times{\bbox}^n)^\mathrm{conc} = X^\mathrm{conc}\times({\bbox}^n)^\mathrm{conc}$.
This assertion now implies that the concretification-adjunction also holds internally, i.e.,
for general $X$ and concrete $Y$ we have that
\begin{align*}
    \ihom(X,Y)(n) &= \hom(X\times{\bbox}^n,Y) \\
    &= \hom((X\times{\bbox}^n)^\mathrm{conc},Y) \\
    &= \hom(X^\mathrm{conc}\times({\bbox}^n)^\mathrm{conc},Y) \\
    &= \hom(X^\mathrm{conc}\times{\bbox}^n,Y) \\
    &= \ihom(X^\mathrm{conc},Y)(n),
\end{align*}
natural in $n$.
Thus the claim on $\ihom(X,Y)$ follows for general $X$ and concrete $Y$.
Lastly, we can deduce that $(X\times Y)^\mathrm{conc} = X^\mathrm{conc}\times Y^\mathrm{conc}$ for all cubesets $X,Y$:
If $Z$ is concrete, we have that
\begin{align*}
    \hom((X\times Y)^\mathrm{conc},Z) &= \hom(X\times Y,Z) \\
    &= \hom(X,\ihom(Y,Z)) \\
    &= \hom(X^\mathrm{conc},\ihom(Y^\mathrm{conc},Z)) \\
    &= \hom(X^\mathrm{conc}\times Y^\mathrm{conc},Z).
\end{align*}
Again, the Yoneda lemma implies that $(X\times Y)^\mathrm{conc} = X^\mathrm{conc}\times Y^\mathrm{conc}$.
Lastly, the terminal object $\ast$ is concrete, so $\ast^\mathrm{conc}=\ast$.
\end{proof}

\begin{remark}
Note that the inclusion $\ccSet\sub\cSet$ is not a geometric morphism
since concretification doesn't preserve all finite limits.

For example, consider a non-surjective map $X\to Y$  between cubesets $X$ and $Y$ inducing a bijection on underlying sets $X(0)\simeq Y(0)$, where $Y$ is concrete
(e.g., consider the counit $i_! i^* X\to X$ for some non-constant concrete $X$).
Now build the quotient $Y/X$ (the cofiber) so that the diagram
\begin{center}
\begin{tikzcd}[ampersand replacement=\&]
	X \& Y \\
	{*} \& {Y/X}
	\arrow[from=1-1, to=1-2]
	\arrow[from=1-1, to=2-1]
	\arrow[from=1-2, to=2-2]
	\arrow[from=2-1, to=2-2]
\end{tikzcd}
\end{center}
is a pushout square.
Consider the kernel pair $K$ of the map $Y\to Y/X$.
Then $K$ is a subobject of the concrete cubeset $Y\times Y$ and hence concrete.
Since $X\ne Y$, we have $Y/X\ne \ast$ and $Y\to Y/X$ is a surjection.
Thus, $K$ is a proper subobject of $Y\times Y$.
On the other hand, the concretification of $Y/X$ is the point $\ast$ since $Y/X(\ast)=\ast$.
So we conclude
\[
    (Y\times_{Y/X}Y)^\mathrm{conc} = K^\mathrm{conc} = K \neq Y\times Y = Y^\mathrm{conc}\times_{(Y/X)^\mathrm{conc}}Y^\mathrm{conc}.
\]
\end{remark}

\section{Fibrant Cubesets and Their Weak Structure Theory}
This section deals with the weak structure theory of cubesets.
First, we discuss the notion of a fibration. Then, we consider truncated cubesets and their homotopy groups.
Using this, we obtain the structure tower of a concrete fibrant cubeset, which is studied in the last subsection.

\subsection{Fibrations}\label{ssec:fibrations}

In this section we work in the unrestricted (classical) category of cubes $\bbox$
with all morphisms (projections, permutations and flips).

Recall that given $n\in\N$ there are $n$ many canonical inert maps $\epsilon^{n-1}_i\colon\langle n\rangle\to\langle n-1\rangle$ in $\fin_\ast$ which are
the unique monotonic surjections being undefined in $i$, i.e. given by $j\mapsto j$ if $j<i$, $i\mapsto\ast$ and $j\mapsto j-1$ if $j>i$.
We denote the corresponding maps in $\bbox$ by $\epsilon_i^{n-1}$
(or simply $\epsilon_i$ if the dimension $n$ is clear from the context) as well,
which now correspond to lower face maps and are inclusions ${\bbox}^{n-1}\to{\bbox}^n$.
The image of $\epsilon_i^{n-1}$ is the subset $[x_i = 0] = \{x\in\{0,1\}^n : x_i = 0\}$ for $1\leq i\leq n$.

\begin{definition}\label{def:corner}
Let $n\geq 1$.
The $n$-corner ${\bbcor}^n$ of the $n$-cube ${\bbox}^n$ is the union of all $(n-1)$-cubes ${\bbox}^{n-1}\subset{\bbox}^n$ such that the inclusion arises from an inert map in $\fin_\ast$.

Equivalently, it is the colimit over the diagram of all representable subobjects
${\bbox}^m\sub{\bbox}^n$ for $m<n$ such that the inclusion is inert,
meaning that it is a composition of maps of the form $\epsilon_i$.
\end{definition}

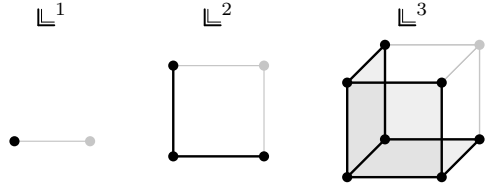
\begin{center}

\begin{tikzpicture}[
    x=1cm,y=1cm,
    font=\small,
    line cap=round,
    line join=round,
    vertex/.style={circle,fill,inner sep=1.35pt},
    greyvertex/.style={circle,fill=gray!45,inner sep=1.35pt},
    ambient/.style={draw=gray!45,line width=0.45pt},
    corneredge/.style={draw=black,line width=0.9pt},
    cornerface/.style={
        draw=gray!55,
        fill=gray!45,
        line width=0.45pt,
        fill opacity=0.28
    },
    toplabel/.style={font=\normalsize},
scale=0.5
]

\node[toplabel] at (1.0,3.35) {$\bbcor^1$};
\node[toplabel] at (5.4,3.35) {$\bbcor^2$};
\node[toplabel] at (10.55,3.35) {$\bbcor^3$};

\begin{scope}[shift={(0,0)}]
    \coordinate (A) at (0,0);
    \coordinate (B) at (2,0);

    \draw[ambient] (A)--(B);

    \node[vertex] at (A) {};
    \node[greyvertex] at (B) {};
\end{scope}

\begin{scope}[shift={(4.2,-0.4)}]
    \coordinate (A) at (0,0);
    \coordinate (B) at (2.4,0);
    \coordinate (C) at (0,2.4);
    \coordinate (D) at (2.4,2.4);

    \draw[ambient] (A)--(B)--(D)--(C)--cycle;

    \draw[corneredge] (A)--(B);
    \draw[corneredge] (A)--(C);

    \node[vertex] at (A) {};
    \node[vertex] at (B) {};
    \node[vertex] at (C) {};
    \node[greyvertex] at (D) {};
\end{scope}

\begin{scope}[shift={(8.8,-0.95)}]
    \coordinate (A) at (0,0);
    \coordinate (B) at (2.5,0);
    \coordinate (C) at (0,2.5);
    \coordinate (D) at (2.5,2.5);

    \coordinate (E) at (1.0,1.0);
    \coordinate (F) at (3.5,1.0);
    \coordinate (G) at (1.0,3.5);
    \coordinate (H) at (3.5,3.5);

    \draw[ambient] (A)--(B)--(D)--(C)--cycle;
    \draw[ambient] (E)--(F)--(H)--(G)--cycle;
    \draw[ambient] (A)--(E);
    \draw[ambient] (B)--(F);
    \draw[ambient] (C)--(G);
    \draw[ambient] (D)--(H);

    \filldraw[cornerface] (A)--(B)--(D)--(C)--cycle;
    \filldraw[cornerface] (A)--(B)--(F)--(E)--cycle;
    \filldraw[cornerface] (A)--(C)--(G)--(E)--cycle;

    \draw[corneredge] (A)--(B)--(D)--(C)--cycle;
    \draw[corneredge] (A)--(E)--(F)--(B);
    \draw[corneredge] (E)--(G)--(C);

    \draw[ambient] (D)--(H);
    \draw[ambient] (F)--(H);
    \draw[ambient] (G)--(H);

    \foreach \P in {A,B,C,D,E,F,G}
        \node[vertex] at (\P) {};

    \node[greyvertex] at (H) {};
\end{scope}
\end{tikzpicture}
\end{center}

Note that the point ${\bbox}^0$ doesn't have a corner.

\begin{remark}\label{rem:corner-combinatorial}
An $n$-corner in $X$ is a map $\lambda\colon{\bbcor}^n\to X$, i.e. an element of $\hom({\bbcor}^n,X)$.
Explicitly, it consists of a tuple of maps $\lambda_i\colon{\bbox}^{n-1}\to X$ for $1\leq i\leq n$
such that whenever $\epsilon_j\colon{\bbox}^{n-2}\to{\bbox}^{n-1}$ is an inert map for $1\leq j\leq n-1$
we have that
\[
    \begin{cases}
        \lambda_i\circ\epsilon_j = \lambda_{j+1}\circ\epsilon_i\quad&\text{for}\; i\leq j, \\
        \lambda_i\circ\epsilon_j = \lambda_j\circ\epsilon_{i-1}\quad&\text{for}\; i\geq j+1.
    \end{cases}
\]
\end{remark}

\begin{remark}
The $n$-corner ${\bbcor}^n$ is a compact connected object%
\footnote{In a category $\cC$, an object $X$ is compact if $\hom(X,-)$ preserves filtered colimits,
and it is connected if $\hom(X,-)$ preserves coproducts.}
in $\cSet$,
as it is a finite colimit of tiny (so in particular compact and connected) objects, the $k$-cubes,
over a connected diagram%
\footnote{A connected object in the category of small categories,
equivalently every two objects in the diagram are connected by a zig-zag of morphisms.}.
Explicitly this means that the functor $\hom({\bbcor}^n,-)\colon\cSet\to\set$ preserves filtered colimits and arbitrary coproducts.
\end{remark}

\begin{definition}\label{def:nil-fibration}
We set
\[
    \Lambda = \{{\bbcor}^n\to{\bbox}^n : n\in\N\}
\]
to be the set of all corner inclusions $\lambda^n\colon{\bbcor}^n\to{\bbox}^n$.
Now a map $f\colon X\to Y$ of cubesets is a \emph{fibration}
if it has the right lifting property against $\Lambda$, i.e.,
for every commutative diagram
\begin{center}
\begin{tikzcd}
	{{\bbcor}^n} & X \\
	{{\bbox}^n} & Y
	\arrow[from=1-1, to=1-2]
	\arrow[hook', from=1-1, to=2-1]
	\arrow[from=1-2, to=2-2]
	\arrow[dashed, from=2-1, to=1-2]
	\arrow[from=2-1, to=2-2]
\end{tikzcd}
\end{center}
there exists a diagonal filler.
Moreover, a cubeset $X$ is \emph{fibrant} if the projection $p\colon X\to\ast$ is a fibration.
We denote the full subcategory of fibrant cubesets by $\cSetnil$.
\end{definition}

Geometrically, a fibration is a map $f$ such that every $n$-corner in $X$,
whose image under $f$ comes from an $n$-cube in $Y$, already arose from an $n$-cube in $X$.
We also say that a fibrant cubeset has \emph{corner completion}.

\begin{example}
The $n$-cube ${\bbox}^n$ is not fibrant for $n\geq 1$, but the point ${\bbox}^0=\ast$ is.
For this, consider, e.g. the $2$-corner spanned by twice the same nondegenerate $1$-cube in ${\bbox}^n$.
This does not extend since any affine linear extension to $\bbox^2$ would necessarily map the new vertex to something outside the $n$-cube.

Injective objects such as the subobject classifier $\Omega$, or more generally $\Omega^X=\ihom(X,\Omega)$ are fibrant (see \cite[Chapter IV.10, Corollary 2]{MacLane1994}),
as well as discrete and trivial objects $i_\ast S$ and $i_!S$ for nonempty sets $S$.
The Host-Kra cubegroups from Example \ref{ex:host_kra_cubegroups} are also fibrant.
\end{example}

\begin{remark}
Note that in this definition we do not include right lifting against the map $\emptyset\to\ast$,
which would correspond to the $n=0$ case in \cite[Definition 7.1]{Gutman2020}.
Our convention seems to agree with \cite[Definition 3.3.7, Lemma 3.3.9]{Candela2017} although there is no explicit definition of the index $n$.
This means that a fibration need not be surjective in general, take for example the map $\emptyset\to\ast$ (or any map with empty domain).
However, every fibration between non-empty concrete ergodic (see Section \ref{ssec:connectivity}) cubesets is surjective.
\end{remark}

\begin{remark}
Equivalently, a cubeset $X$ is fibrant if restricting an $n$-cube to its $n$-corner
\[
    \hom({\bbox}^n,X)\to\hom({\bbcor}^n,X)
\]
is a surjection for all $n\in\N$.
If $X$ is concrete, then an $n$-corner $\lambda$ in $X$ is given by a map
\[
    \lambda\colon\{0,1\}^n\setminus\{1^n\}\to X_0
\]
such that the restriction of $\lambda$ along every lower face
$\epsilon_i\colon \{0,1\}^{n-1}\to\{0,1\}^n$ yields an ($n-1$)-cube of $X$,
i.e., an element of $X_{n-1}$.
Thus, being fibrant amounts to finding a value $\lambda(1^n)\in X_0$
such that the extension $\lambda\colon\{0,1\}^n\to X_0$ is in $X_n$.
\end{remark}

Given the notion of fibration it is possible to construct a \emph{weak factorization system}
$(l(r(\Lambda)),r(\Lambda))$ where $r(\Lambda)$ is the class of fibrations
(those maps with the {\textbf{r}ight lifting property} against $\Lambda$)
and $l(r(\Lambda))$ is the class of \emph{trivial cofibrations},
defined to be the class of all maps that have the {\textbf{l}eft lifting property} against all fibrations.
Being a weak factorization system means that every morphism $f\colon X\to Y$ in $\cSet$ factors as
\begin{center}
\begin{tikzcd}
	X & W & Y
	\arrow["h", tail, from=1-1, to=1-2]
	\arrow["f"', curve={height=12pt}, from=1-1, to=1-3]
	\arrow["g", two heads, from=1-2, to=1-3]
\end{tikzcd}
\end{center}
where $g$ is a fibration and $h$ is a trivial cofibration.

Since $\Lambda$ is a set, the domain of every map $\lambda^n$ in $\Lambda$ is compact and $\cSet$ has all (small) colimits,
Quillen's \emph{small object argument} allows for a functorial factorization:
for every morphism $f\colon X\to Y$ there exists a factorization $f=gh$ with $g\in r(\Lambda)$
and $h\in l(r(\Lambda))$, functorial in $f$:
\begin{center}
\begin{tikzcd}
	X & W & Y \\
	{X'} & {W'} & {Y'}
	\arrow["h"', tail, from=1-1, to=1-2]
	\arrow["f", curve={height=-12pt}, from=1-1, to=1-3]
	\arrow[from=1-1, to=2-1]
	\arrow["g"', two heads, from=1-2, to=1-3]
	\arrow[from=1-2, to=2-2]
	\arrow[from=1-3, to=2-3]
	\arrow["{h'}", tail, from=2-1, to=2-2]
	\arrow["{f'}"', curve={height=12pt}, from=2-1, to=2-3]
	\arrow["{g'}", two heads, from=2-2, to=2-3]
\end{tikzcd},
\end{center}
see \cite[Theorem 12.2.2, Theorem 12.5.1]{Riehl}.
Moreover, one can describe the trivial cofibrations as follows:
since we can apply the small object argument, $l(r(\Lambda))$ is the smallest class of morphisms in $\cSet$
that contains $\Lambda$ and is closed under retracts, transfinite composition and pushouts,
see \cite[Proposition 2.1.9]{Cisinski2025}.
It then automatically is also closed under coproducts.
We also say that $l(r(\Lambda))$ is the \emph{saturated class} $\mathrm{sat}(\Lambda)$ of $\Lambda$.

\begin{construction}\label{cons:triv-cofib}
There is an inductive construction of every trivial cofibration:
a \emph{relative $\Lambda$-cell complex} is a transfinite composition of pushouts of morphisms in $\Lambda$.
Then every trivial cofibration is a retract of a relative $\Lambda$-cell complex.
This means that given a trivial cofibration $f\colon X\to Y$ there exists
a relative $\Lambda$-cell complex $g\colon X\to Z$ and a retraction $r\colon Z\to Y$ such that the diagram
\begin{center}
\begin{tikzcd}
	X & X & X \\
	Y & Z & Y
	\arrow["{=}"{description}, draw=none, from=1-1, to=1-2]
	\arrow["f"', from=1-1, to=2-1]
	\arrow["{=}"{description}, draw=none, from=1-2, to=1-3]
	\arrow["g"', from=1-2, to=2-2]
	\arrow["f"', from=1-3, to=2-3]
	\arrow["s", from=2-1, to=2-2]
	\arrow["{\mathrm{id}}"', curve={height=12pt}, from=2-1, to=2-3]
	\arrow["r", from=2-2, to=2-3]
\end{tikzcd}
\end{center}
is commutative (being a retraction is a priori weaker but one can make it work with always $X$ and the identity in the upper row).
That $g$ is a relative $\Lambda$-cell complex means that there exists an ordinal indexed diagram
$A\colon\alpha\to\cSet$ such that $g = \varinjlim g_\alpha$ is the transfinite composition of $A$
and in every step $g_\alpha\colon A_\alpha\to A_{\alpha+1}$ there exists a map
$h_\alpha\colon B_\alpha\to C_\alpha$ in $\Lambda$ such that the diagram
\begin{center}
\begin{tikzcd}
	{B_\alpha} & {A_\alpha} \\
	{C_\alpha} & {A_{\alpha+1}}
	\arrow[from=1-1, to=1-2]
	\arrow["{h_\alpha}"', from=1-1, to=2-1]
	\arrow["{g_\alpha}", from=1-2, to=2-2]
	\arrow[from=2-1, to=2-2]
\end{tikzcd}
\end{center}
is a pushout square,
see \cite[Definition 2.1.9, Corollary 2.1.15]{Hovey2007}.
From this construction we also see that every trivial cofibration is a monomorphism.
\end{construction}

There is an important subclass of trivial cofibrations,
given by relative $\Lambda$-cell complexes, which we want to highlight briefly.
It is given by what is classically known to be \emph{partial cubes} \cite[Lemma 7.5]{Gutman2020} or \emph{simplicial cubesets} \cite[Definition 3.1.4]{Candela2017}.
Recall that every subset $A\sub\langle n\rangle^\circ=\{1,\dots, n\}$ with $k$ elements induces an inert map
$\langle n\rangle\to\langle k\rangle$ in $\fin_\ast$ which singles out the elements in $A$.
By pullback along this map we obtain a subobject
\[
    {\bbox}^k\to{\bbox}^n
\]
in $\bbox$ whose image consists of those $x\in\{0,1\}^n$ such that $x_i = 0$ whenever $i\not\in A$.
Note that the empty set corresponds to the point ${\bbox}^0\to{\bbox}^n$.
The following Proposition shows when multiple such subobjects glue to a trivial cofibration
and is a generalization of \cite[Lemma 3.1.5]{Candela2017} and \cite[Lemma 7.5]{Gutman2020}.

\begin{proposition}\label{prop:triv-nilcofib-simp}
Let $K$ be a non-empty collection of subsets of $\langle n\rangle^\circ$ which is downwards-closed,
i.e., if $A\in K$ and $B\sub A$ then $B\in K$.
Then the inclusion
\[
    \bigcup_{A\in K}{\bbox}^{\lvert A\rvert}\to{\bbox}^n
\]
is a trivial cofibration.
\end{proposition}

\begin{proof}
If $K$ consists of all subsets of $\langle n\rangle^\circ$,
then $\bigcup_{A\in K}{\bbox}^{\lvert A\rvert}\to{\bbox}^n$ is (isomorphic to) the identity map
so there is nothing to show.
If $K$ consists of all subsets of $\langle n\rangle^\circ$ except the whole set
then $\bigcup_{A\in K}{\bbox}^{\lvert A\rvert}\to{\bbox}^n$ is the inclusion of the $n$-corner
since for $\lvert A\rvert = n-1$ the maps ${\bbox}^{n-1}={\bbox}^{\lvert A\rvert}\to{\bbox}$
are given by $\epsilon_i$, where $i$ depends on which element $A$ misses in $\langle n\rangle^\circ$.
So again, there is nothing to show.
Now suppose that $K$ misses at least two subsets of $\langle n\rangle^\circ$
and choose $A_0\sub\langle n\rangle^\circ$ such that $A_0$ is not contained in $K$
but every proper subset of $A_0$ is (e.g., choose a set of minimal cardinality of those not contained in $K$).
Set $K' = K\cup\{A_0\}$ and consider the commutative diagram
\begin{center}
\begin{tikzcd}
	{\bigcup\limits_{A\in K}{\bbox}^{\lvert A\rvert}\times_{{\bbox}^n}{\bbox}^{\lvert A_0\rvert}} & {\bigcup\limits_{A\in K}{\bbox}^{\lvert A\rvert}} & \\
	{{\bbox}^{\lvert A_0\rvert}} & {\bigcup\limits_{A'\in K'}{\bbox}^{\lvert A'\rvert}} \\
	&& {{\bbox}^n}
	\arrow[from=1-1, to=1-2]
	\arrow[from=1-1, to=2-1]
	\arrow[from=1-2, to=2-2]
	\arrow[from=1-2, to=3-3]
	\arrow[from=2-1, to=2-2]
	\arrow[from=2-1, to=3-3]
	\arrow[from=2-2, to=3-3]
\end{tikzcd}
\end{center}
of subobjects of ${\bbox}^n$ where the outer square is a pullback square
and the inner square is bicartesian.
Now the left vertical map is a corner inclusion since for $A\in K$ the map
\[
    {\bbox}^{\lvert A\cap A_0\rvert}
    ={\bbox}^{\lvert A\rvert}\times_{{\bbox}^n}{\bbox}^{\lvert A_0\rvert}
    \to{\bbox}^{\lvert A_0\rvert}
\]
is the inclusion corresponding to the subobject $A\cap A_0\sub A_0$.
But from the assertion on $A_0$ we infer that it contains every $B\sub A_0$ that is not the whole $A_0$
and thus is a corner inclusion by the second proof step.
Iterating this argument verifies the claim.
\end{proof}

A special case is given by the following Corollary.

\begin{corollary}\label{cor:internal-nilfib}
Given $n\in\N$ and $k\in\N_0$, the \emph{roof} $\lambda^n\times\mathrm{id}_k\colon{\bbcor}^n\times{\bbox}^k\to{\bbox}^{n+k}$ is a trivial cofibration.
\end{corollary}

\begin{proof}
Take $K$ to be the downwards-closed collection of subsets of $\langle n+k\rangle^\circ$
generated by the sets $A_i = \langle n+k\rangle^\circ\setminus\{i\}$ for $1\leq i\leq n$.
\end{proof}

\begin{remark}\label{prop:internal-nilfib}
By inspecting the proof of Proposition~\ref{prop:triv-nilcofib-simp}
one sees that in the case of the previous Corollary,
the only corner completions used are of dimension at least $n$.
\end{remark}

\begin{remark}\label{rem:internal-nilfib}
One could also define a map $f\colon X\to Y$ to be a fibration if the push-pull map
\[
    \ihom({\bbox}^n,X)\to\ihom({\bbcor}^n,X)\times_{\ihom({\bbcor}^n,Y)}\ihom({\bbox}^n,Y)
\]
is an epimorphism for every $n\in\N$ \emph{internally} in $\cSet$, i.e., on the internal Hom.
This notion a priori is stronger than being a fibration since evaluation at the point
returns the condition for being a fibration since $\ihom(W,Z)(\ast) = \hom(W,Z)$.
But the preceding Corollary shows that the converse also holds, since $\ihom({\bbcor}^n,X)(k)=\hom({\bbcor}^{n}\times {\bbox}^k,X)$.
\end{remark}

Another consequence of the small object argument is the \emph{functorial fibrant replacement}
which asserts that there exists a functor $(-)^\mathrm{fib}\colon\cSet\to\cSetnil$
together with a natural transformation $\mathrm{id}\to(-)^\mathrm{fib}$
whose components are all trivial cofibrations.
This can be seen by applying the functorial factorization to the map $p\colon X\to\ast$.

\begin{construction}
One can even give an explicit construction of a functorial fibrant replacement, see \cite[Remark 12.5.6]{Riehl}.%
\footnote{We prefer Garner's algebraic small object argument over the classical one from Quillen, since it adds way less redundant cells}
Beware that this is not a left/right adjoint to the inclusion -- we will see in \ref{rem:fibrant-non-reflective} that these do not exist.
Let $X\in\cSet$ and set $X_0\coloneqq X$.
For $\alpha\in\N_0$ let $M_\alpha$ be the set of all maps $f\colon{\bbcor}^n\to X_\alpha$
such that there does not exist a lift
\begin{center}
\begin{tikzcd}
	& {X_\beta} \\
	{{\bbcor}^n} & {X_\alpha}
	\arrow[from=1-2, to=2-2]
	\arrow[dashed, from=2-1, to=1-2]
	\arrow["f"', from=2-1, to=2-2]
\end{tikzcd}
\end{center}
for some $\beta<\alpha$ (the strictness here is important for functoriality).
Then define $X_{\alpha+1}$ so that the diagram
\begin{center}
\begin{tikzcd}
	{\coprod\limits_{M_\alpha}{\bbcor}^n} & {X_\alpha} \\
	{\coprod\limits_{M_\alpha}{\bbox}^n} & {X_{\alpha+1}}
	\arrow[from=1-1, to=1-2]
	\arrow[hook', from=1-1, to=2-1]
	\arrow[from=1-2, to=2-2]
	\arrow[hook', from=2-1, to=2-2]
\end{tikzcd}
\end{center}
is a pushout square.
Now one defines $X^\mathrm{fib}$ to be the sequential colimit of the $X_\alpha$.
It is clear that with this definition the map $X\to X^\mathrm{fib}$ is a trivial cofibration
(it is in the saturated class of the set of corner inclusions $\Lambda$ which agrees with the trivial cofibrations),
and that $X^\mathrm{fib}\to\ast$ is a fibration
(by compactness, every $n$-corner in $X^\mathrm{fib}$ lies in a finite stage $X_\alpha$
with $\alpha$ minimal and so admits a filler in $X_{\alpha+1}$).
\end{construction}

\begin{remark}\label{rem:fibrant-non-reflective}
Given a notion of fibration defined via a right lifting property against a class of morphisms,
e.g. in a model category, the full subcategory of fibrant objects is in general not reflective.
The first heuristic is that almost all non-trivial fibrant replacements satisfy
$X^\mathrm{fib}\neq (X^\mathrm{fib})^\mathrm{fib}$ so that fibrant replacement can't be the reflector,
see the construction above.

As an explicit counterexample, take the model structure on $\set$ where the weak equivalences are all maps, the cofibrations the injections and the fibrations the surjections.
In this case, the fibrant objects are precisely the non-empty sets which do not form a reflective subcategory.

Usually, fibrant objects are rarely ever stable under subobjects.
In the Kan-Quillen model structure on simplicial sets one can just take an inclusion $\bbDelta^n\sub K$ where $K$ is a Kan complex,
and in the cubeset setting the same works with a cube ${\bbox}^n\sub X$ with $X$ fibrant.

More generally, if the trivial cofibrations are effective monics and there exists a fibrant replacement,
then the full subcategory of fibrant objects is reflective precisely if it is everything.
For this, assume there exists a non-fibrant object $E$.
Use fibrant replacement to find an effective monic $E\hookrightarrow X$ for fibrant $X$.
But now, by effectiveness, $E$ is the equalizer of its cokernel pair $X\rightrightarrows Y$.
Embed $Y$ into a fibrant object to obtain that $E$ is limit of a diagram of fibrant objects, hence
these are not stable under limits and thereby do not form a reflective subcategory.
\end{remark}

\begin{construction}
There is another version of fibrant replacement which \emph{does} satisfy $(X^\mathrm{fib})^\mathrm{fib} = X^\mathrm{fib}$
which will be called \emph{naive fibrant replacement}.
Let $X\in\cSet$ and set $X_0\coloneqq X$.
For $\alpha\in\N_0$ let $N_\alpha$ be the set of all maps $f\colon{\bbcor}^n\to X_\alpha$
such that there does not exist a filler
\begin{center}
\begin{tikzcd}
	{{\bbcor}^n} & {X_\alpha} \\
	{{\bbox}^n} & \ast
	\arrow["f", from=1-1, to=1-2]
	\arrow[from=1-1, to=2-1]
	\arrow[from=1-2, to=2-2]
	\arrow[dashed, from=2-1, to=1-2]
	\arrow[from=2-1, to=2-2]
\end{tikzcd}.
\end{center}
Then define $X_{\alpha+1}$ so that the diagram
\begin{center}
\begin{tikzcd}
	{\coprod\limits_{N_\alpha}{\bbcor}^n} & {X_\alpha} \\
	{\coprod\limits_{N_\alpha}{\bbox}^n} & {X_{\alpha+1}}
	\arrow[from=1-1, to=1-2]
	\arrow[hook', from=1-1, to=2-1]
	\arrow[from=1-2, to=2-2]
	\arrow[hook', from=2-1, to=2-2]
\end{tikzcd}
\end{center}
is a pushout square.
Then define $X^\mathrm{fib}$ as the sequential colimit of the $X_\alpha$.

Note that this construction is not functorial since $N_0$ doesn't depend functorially on $X$!
See \cite{Nikolaus2011} on how to resolve this issue.
\end{construction}

However, fibrant objects are stable under some operations.

\begin{lemma}
Nilfibrant cubesets are stable under coproducts and products, i.e.,
if $(X_i)_{i\in I}$ is a family of fibrant cubesets, then $\coprod_{i\in I} X_i$ and $\prod_{i\in I} X_i$ are fibrant as well.
\end{lemma}

\begin{proof}
A cubeset $X$ is fibrant if and only if the restriction $\hom({\bbox}^k,X)\to\hom({\bbcor}^k,X)$ is surjective for all $k\in\N$.
This implies the property for products, since $\hom(Y,-)$ preserves products and epimorphisms are stable under taking products in $\set$.
For coproducts note that a map ${\bbcor}^k\to\coprod_{i\in I} X_i$ factors through an inclusion $X_{i_0}\to\coprod_{i\in I} X_i$
since ${\bbcor}^k$ is connected.
\end{proof}

Furthermore, fibrancy does not behave well with concreteness.

\begin{example}
Note that a fibrant replacement of a concrete cubeset is not necessarily concrete again.
For example, consider the concrete cubeset given by
\begin{center}
\begin{tikzcd}[ampersand replacement=\&]
	b \& b \\
	a \& b
	\arrow[no head, from=1-1, to=1-2]
	\arrow[no head, from=2-1, to=1-1]
	\arrow[no head, from=2-1, to=1-2]
	\arrow[no head, from=2-1, to=2-2]
	\arrow[no head, from=2-2, to=1-2]
\end{tikzcd}
\end{center}
(i.e. three times the same edge between $a$ and $b$, connected by degenerate edges).
In a fibrant replacement we add a completion to the corner
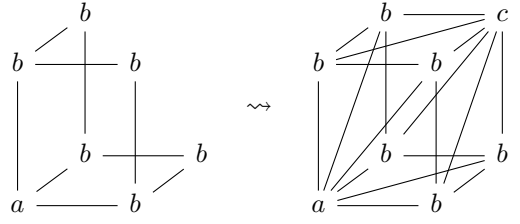
\begin{center}
\begin{tikzcd}[row sep=tiny, column sep=tiny]
	&& b &&&&&&& b &&& c \\
	b &&& b &&&& b &&& b \\
	&&&&&& \leadsto \\
	&& b &&& b &&&& b &&& b \\
	a &&& b &&&& a &&& b
	\arrow[no head, from=1-10, to=1-13]
	\arrow[no head, from=2-1, to=1-3]
	\arrow[no head, from=2-1, to=2-4]
	\arrow[no head, from=2-8, to=1-10]
	\arrow[no head, from=2-8, to=1-13]
	\arrow[no head, from=2-8, to=2-11]
	\arrow[no head, from=2-11, to=1-13]
	\arrow[no head, from=4-3, to=1-3]
	\arrow[no head, from=4-3, to=4-6]
	\arrow[no head, from=4-10, to=1-10]
	\arrow[no head, from=4-10, to=1-13]
	\arrow[no head, from=4-10, to=4-13]
	\arrow[no head, from=4-13, to=1-13]
	\arrow[no head, from=5-1, to=2-1]
	\arrow[no head, from=5-1, to=4-3]
	\arrow[no head, from=5-1, to=5-4]
	\arrow[no head, from=5-4, to=2-4]
	\arrow[no head, from=5-4, to=4-6]
	\arrow[no head, from=5-8, to=1-10]
	\arrow[no head, from=5-8, to=2-8]
	\arrow[no head, from=5-8, to=2-11]
	\arrow[no head, from=5-8, to=4-10]
	\arrow[no head, from=5-8, to=4-13]
	\arrow[no head, from=5-8, to=5-11]
	\arrow[no head, from=5-11, to=1-13]
	\arrow[no head, from=5-11, to=2-11]
	\arrow[no head, from=5-11, to=4-13]
\end{tikzcd}.
\end{center}
However, this completion adds three different(!) squares between $a$, $b$, $b$ and the new vertex, induced by the active diagonal faces resulting in a non-concrete cubespace.
\end{example}

\begin{remark}
We expect that if a cubeset is fibrant, then its concretisation is not necessarily fibrant.
A candidate for an explicit counterexample could be the naive fibrant replacement of the non-concrete cubeset given by
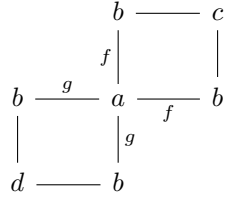
\begin{center}
\begin{tikzcd}[ampersand replacement=\&]
	\& b \& c \\
	b \& a \& b \\
	d \& b
	\arrow[no head, from=1-2, to=1-3]
	\arrow[no head, from=2-1, to=3-1]
	\arrow["f", no head, from=2-2, to=1-2]
	\arrow["g"', no head, from=2-2, to=2-1]
	\arrow["f"', no head, from=2-2, to=2-3]
	\arrow["g", no head, from=2-2, to=3-2]
	\arrow[no head, from=2-3, to=1-3]
	\arrow[no head, from=3-2, to=3-1]
\end{tikzcd}
\end{center}
	There, the corner
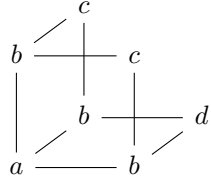
\begin{center}
\begin{tikzcd}[ampersand replacement=\&, row sep=tiny,
    column sep=tiny]
	\&\& c \&\&\& \\
	b \&\&\& c \\
	\\
	\&\& b \&\&\& d \\
	a \&\&\& b
	\arrow[no head, from=2-1, to=1-3]
	\arrow[no head, from=2-1, to=2-4]
	\arrow[no head, from=4-3, to=1-3]
	\arrow[no head, from=4-3, to=4-6]
	\arrow[no head, from=5-1, to=2-1]
	\arrow[no head, from=5-1, to=4-3]
	\arrow[no head, from=5-1, to=5-4]
	\arrow[no head, from=5-4, to=2-4]
	\arrow[no head, from=5-4, to=4-6]
\end{tikzcd}
\end{center}
only becomes a corner in the concretisation, and there is no reason how this should be filled before taking concretisation.
\end{remark}

One might wonder whether we have something like the typical pushout product stability axiom (SM7), see Chapter 2.2 of \cite{Quillen1967} for the origin of this axiom.
Unfortunately, even the usual prism lemma fails.

\begin{example}
	Consider the pushout product of $\partial \bbox^1=\ast\sqcup\ast\to\bbox^1$ with $\bbcor^1\to \bbox^1$.
	This corresponds to the prism
	\[(\partial {\bbox}^1\times {\bbox}^1)\cup ({\bbox}^1\times {\bbcor}^1)\hookrightarrow {\bbox}^2,\]
	pictorially
	
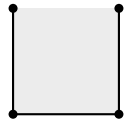
\begin{center}
\begin{tikzpicture}[
    x=0.7cm,y=0.7cm,
    font=\small,
    line cap=round,
    line join=round,
    vertex/.style={circle,fill,inner sep=1.2pt},
    ambient/.style={draw=gray!55, line width=0.45pt},
    mainedge/.style={draw=black, line width=0.8pt}
]

\coordinate (A) at (0,0);
\coordinate (B) at (2.0,0);
\coordinate (C) at (0,2.0);
\coordinate (D) at (2.0,2.0);

\fill[gray!15] (A)--(B)--(D)--(C)--cycle;

\draw[mainedge] (A)--(B);
\draw[mainedge] (A)--(C);
\draw[mainedge] (B)--(D);

\foreach \P in {A,B,C,D} {
    \node[vertex] at (\P) {};
}

\end{tikzpicture}
	\end{center}
	and is clearly not a trivial cofibration.
\end{example}

\begin{remark}
	Note that this U-shape is precisely the horn in classical cubical homotopy theory,
	see~\cite[Definition 10.3.2]{Brown2011}.
\end{remark}

This shows that the pushout product of a trivial cofibration with a general (normal) monomorphism is not necessarily a trivial cofibration.

\subsubsection{Cubesets with Gluing Property}

Given two $n$-cubes, we can glue them along an ($n-1$)-dimensional face (an inert $\epsilon_i\colon{\bbox}^{n-1}\to{\bbox}^n$) to obtain ${\bbox}^n\sqcup_{{\bbox}^{n-1}}{\bbox}^n$.
Opposite to where we have glued, we obtain a map ${\bbox}^{n-1}\sqcup{\bbox}^{n-1}\to{\bbox}^n\sqcup_{{\bbox}^{n-1}}{\bbox}^n$ induced by the \emph{upper face map} $\overline{\epsilon}_i\colon{\bbox}^{n-1}\to{\bbox}^n$
which is $\epsilon_i$ but with the $0$ swapped for a $1$.
This map can be glued along the map $\epsilon_i\sqcup\overline{\epsilon}_i\colon{\bbox}^{n-1}\sqcup{\bbox}^{n-1}\to{\bbox}^n$,
so that we obtain a pushout square
\begin{center}
\begin{tikzcd}
	{{\bbox}^{n-1}\sqcup{\bbox}^{n-1}} & {{\bbox}^n} \\
	{{\bbox}^{n}\sqcup_{{\bbox}^{n-1}}{\bbox}^{n}} & {({\bbox}^{n}\sqcup_{{\bbox}^{n-1}}{\bbox}^{n})\sqcup_{{\bbox}^{n-1}\sqcup{\bbox}^{n-1}}{\bbox}^n}
	\arrow[from=1-1, to=1-2]
	\arrow[from=1-1, to=2-1]
	\arrow[from=1-2, to=2-2]
	\arrow[from=2-1, to=2-2]
\end{tikzcd}
\end{center}
where the lower horizontal map is a monomorphism since the upper one is.
This allows us to define the set of morphisms
\[
    \bprism\coloneq \{ {{\bbox}^{n}\sqcup_{{\bbox}^{n-1}}{\bbox}^{n}}\to{({\bbox}^{n}\sqcup_{{\bbox}^{n-1}}{\bbox}^{n})\sqcup_{{\bbox}^{n-1}\sqcup{\bbox}^{n-1}}{\bbox}^n} \mid n\in\N \}.
\]

\begin{definition}\label{def:gluing-property}
A cubeset $X$ has the \emph{gluing property} if the map $p\colon X\to\ast$ has the right lifting property w.r.t. $\bprism$.
Moreover we say that $X$ has \emph{unique gluing} if the lift is unique (and exists, i.e. the object is local with respect to $\bprism$).
	We say that $X$ has \emph{at most unique gluing} if the lift is unique whenever it exists.
\end{definition}

\begin{proposition}\label{prop:unique-gluing-concrete}
A cubeset $X$ has at most unique gluing if and only if it is concrete.
\end{proposition}

\begin{proof}
The if direction is clear.
For the only if direction we argue by induction on the dimension $n$.
For $n=1$, we need to show that for two points $a,b\in X(0)$
there is at most one edge with these two boundary points.
Consider some edge $f$ between $a$ and $b$, and glue it to the degenerate $a$-cube.
Therefore, $f$ is a gluing of $f$ and the degenerate cube.
Now, any other $g$ between $a$ and $b$ also is a gluing of $f$ and the degenerate $a$-cube,
so we have $f=g$
\begin{center}
\begin{tikzcd}[row sep=tiny,column sep=tiny]
	a && \\
	\\
	a && b
	\arrow["e", no head, from=3-1, to=1-1]
	\arrow["f"', no head, from=3-1, to=3-3]
	\arrow["f"', shift right, no head, from=3-3, to=1-1]
	\arrow["g", shift left, no head, from=3-3, to=1-1]
\end{tikzcd}.
\end{center}
Thus, the map $X(1)\to X(0)^2$ is injective.
For a general dimension we argue similarly: Consider two $n$-cubes $f$ and $g$ with the same points.
Then by induction, all their faces of dimension $n-1$ agree.
Now, take any face of $f$ and consider the degenerate $n$-cube of this face.
Then, $f$ and $g$ both yield gluings of this degenerate cube and of $f$, and thus $g=f$.
\begin{center}
\begin{tikzcd}[ampersand replacement=\&,sep=small]
	\& \bullet \&\&\& \\
	\bullet \\
	\\
	\& \bullet \&\&\& \bullet \\
	\bullet \&\&\& \bullet
	\arrow[""{name=0, anchor=center, inner sep=0}, no head, from=1-2, to=4-5]
	\arrow[""{name=1, anchor=center, inner sep=0}, curve={height=-12pt}, no head, from=1-2, to=4-5]
	\arrow["a"{description}, no head, from=2-1, to=1-2]
	\arrow[""{name=2, anchor=center, inner sep=0}, no head, from=2-1, to=5-4]
	\arrow[""{name=3, anchor=center, inner sep=0}, curve={height=-12pt}, no head, from=2-1, to=5-4]
	\arrow[no head, from=4-2, to=1-2]
	\arrow[""{name=4, anchor=center, inner sep=0}, no head, from=4-2, to=4-5]
	\arrow[no head, from=5-1, to=2-1]
	\arrow["a"{description}, no head, from=5-1, to=4-2]
	\arrow[""{name=5, anchor=center, inner sep=0}, no head, from=5-1, to=5-4]
	\arrow[no head, from=5-4, to=4-5]
	\arrow["f"{description}, shift right=2, draw=none, from=2, to=0]
	\arrow["g"{description}, shift left=2, draw=none, from=3, to=1]
	\arrow["f"', shift left=4, draw=none, from=5, to=4]
\end{tikzcd}
\end{center}
\end{proof}

\begin{proposition}\label{prop:nilfib-gluing}
Let $X$ be a fibrant cubeset. Then $X$ has the gluing property.
\end{proposition}

\begin{proof}
We simply embed the 2-roof $\bbcor^2\times \bbox^{n-1}\to \bbox^2\times \bbox^{n-1}$ and then restrict to the triangle which includes the off-diagonal of the $2$-square.
This is the same as the inclusion used in the gluing property.

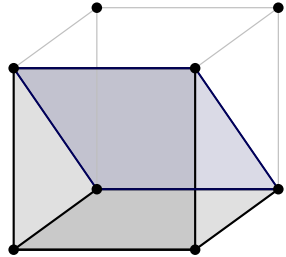
\pgfdeclarelayer{facefill}
\pgfdeclarelayer{faceoutline}
\pgfsetlayers{facefill,faceoutline,main}
\begin{center}
\newcommand{\DarkFace}[1]{%
  \begin{pgfonlayer}{facefill}
    \path[fill=gray, fill opacity=0.24] #1 -- cycle;
  \end{pgfonlayer}
  \begin{pgfonlayer}{faceoutline}
    \path[draw=gray!70,line width=0.5pt] #1 -- cycle;
  \end{pgfonlayer}
}

\newcommand{\PrismFace}[1]{%
  \begin{pgfonlayer}{facefill}
    \path[fill=blue!20!gray, fill opacity=0.24] #1 -- cycle;
  \end{pgfonlayer}
  \begin{pgfonlayer}{faceoutline}
    \path[draw=blue!35!black,line width=0.5pt] #1 -- cycle;
  \end{pgfonlayer}
}

\begin{tikzpicture}[
    x=1cm,y=1cm,
    font=\small,
    line cap=round,
    line join=round,
    vertex/.style={circle,fill,inner sep=1.4pt},
    ambient/.style={draw=gray!50, line width=0.45pt},
    mainedge/.style={draw=black, line width=0.8pt},
    prismedge/.style={draw=blue!35!black, line width=0.75pt}
]

\coordinate (A) at (0,0);
\coordinate (B) at (2.4,0);
\coordinate (C) at (0,2.4);
\coordinate (D) at (2.4,2.4);

\coordinate (E) at (1.1,0.8);
\coordinate (F) at (3.5,0.8);
\coordinate (G) at (1.1,3.2);
\coordinate (H) at (3.5,3.2);

\draw[ambient] (A)--(B)--(D)--(C)--cycle;
\draw[ambient] (E)--(F)--(H)--(G)--cycle;
\draw[ambient] (A)--(E);
\draw[ambient] (B)--(F);
\draw[ambient] (C)--(G);
\draw[ambient] (D)--(H);

\DarkFace{(A)--(B)--(D)--(C)}
\DarkFace{(A)--(B)--(F)--(E)}

\PrismFace{(C)--(D)--(F)--(E)}

\draw[mainedge] (A)--(B);
\draw[mainedge] (B)--(D);
\draw[mainedge] (D)--(C);
\draw[mainedge] (C)--(A);

\draw[mainedge] (A)--(B);
\draw[mainedge] (B)--(F);
\draw[mainedge] (F)--(E);
\draw[mainedge] (E)--(A);

\draw[prismedge] (C)--(D);
\draw[prismedge] (D)--(F);
\draw[prismedge] (F)--(E);
\draw[prismedge] (E)--(C);

\foreach \P in {A,B,C,D,E,F,G,H} {
    \node[vertex] at (\P) {};
}

\end{tikzpicture}
\end{center}
Hence, a gluing may be computed from fibrancy by taking the off-diagonals of a lift of the roof
$\bbox^{n-1}\times \bbcor^2\hookrightarrow \bbox^{n+1}$.
\end{proof}

\begin{corollary}
Concrete fibrant cubesets are precisely the fibrant cubesets with unique gluing property.
\end{corollary}

\begin{lemma}
The subcategory of cubespaces with unique gluing property forms a reflective subcategory of all cubesets, i.e. the inclusion admits a left adjoint.
\end{lemma}

\begin{proof}
This follows from the fact that cubesets with unique gluing property are the $\bprism$-local objects,
and on local objects there is always a reflection, see \cite[Theorem 1.39]{Adamek1994}.
\end{proof}

\subsubsection{Fibrations over \texorpdfstring{$\bbGamma$}{Gamma}}\label{ssec:corner-gamma}

In this section we work in the opposite category $\bbGamma$ of finite pointed sets, unless otherwise noted.
Recall that it can be interpreted as a wide subcategory $\bblox$ of the whole cube category $\bbox$
where the morphisms fix the base point.
In this category, there is another definition of $n$-corner than in the whole $\bbox$.
Recall that we denote by ${\bblox}^n$ the image of $\langle n\rangle$ under the Yoneda embedding
$\bbGamma\to\PSh(\bbGamma)$ and by ${\bbox}^n$ the image of $\{0,1\}^n$ under the Yoneda embedding
$\bbox\to\cSet$.

\begin{definition}\label{def:omega_corner_gamma}
Let $n\in\N$ and $\omega\in\{0,1\}^n\setminus\{0^n\}$.
Then define the $n$-corner ${\bbcor}^n_\omega$ w.r.t. $\omega$ as the union over all injections $f\colon{\bblox}^{n-1}\to{\bblox}^n$ in $\bbGamma$
such that $f$ misses $\omega$, i.e., $\omega\notin\im f$:
\[
    {\bbcor}^n_\omega = \bigcup_{\substack{f\colon{\bblox}^{n-1}\sub{\bblox}^n \\ \omega\notin\im f}}{\bblox}^{n-1}\hookrightarrow{\bblox}^n.
\]
\end{definition}

\begin{example}
In two dimensions, the corners look like
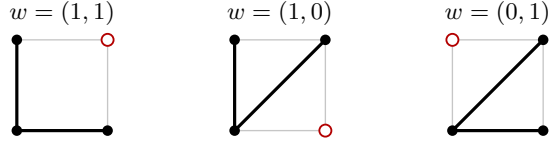
\begin{center}
\begin{tikzpicture}[
    scale=0.6,
    line cap=round,
    line join=round,
    bg/.style={draw=gray!45, line width=0.5pt},
    face/.style={draw=black, line width=1.2pt},
    vtx/.style={circle, fill=black, inner sep=1.4pt},
    missing/.style={circle, draw=red!70!black, fill=white, line width=0.7pt, inner sep=1.6pt},
    lab/.style={font=\scriptsize},
    title/.style={font=\small\bfseries}
]
\begin{scope}[shift={(0,0)}]
  \node[title] at (1,2.6) {$w=(1,1)$};

  \coordinate (O) at (0,0);
  \coordinate (A) at (2,0);
  \coordinate (B) at (0,2);
  \coordinate (C) at (2,2);

  \draw[bg] (O)--(A)--(C)--(B)--cycle;

  \draw[face] (O)--(A);
  \draw[face] (O)--(B);

  \node[vtx] at (O) {};
  \node[vtx] at (A) {};
  \node[vtx] at (B) {};
  \node[missing] at (C) {};

  \node[lab, below left]  at (O) {};
  \node[lab, below right] at (A) {};
  \node[lab, above left]  at (B) {};
  \node[lab, above right] at (C) {};
\end{scope}

\begin{scope}[shift={(4.8,0)}]
  \node[title] at (1,2.6) {$w=(1,0)$};

  \coordinate (O) at (0,0);
  \coordinate (A) at (2,0);
  \coordinate (B) at (0,2);
  \coordinate (C) at (2,2);

  \draw[bg] (O)--(A)--(C)--(B)--cycle;

  \draw[face] (O)--(B);
  \draw[face] (O)--(C);

  \node[vtx] at (O) {};
  \node[missing] at (A) {};
  \node[vtx] at (B) {};
  \node[vtx] at (C) {};

  \node[lab, below left]  at (O) {};
  \node[lab, below right] at (A) {};
  \node[lab, above left]  at (B) {};
  \node[lab, above right] at (C) {};
\end{scope}

\begin{scope}[shift={(9.6,0)}]
  \node[title] at (1,2.6) {$w=(0,1)$};

  \coordinate (O) at (0,0);
  \coordinate (A) at (2,0);
  \coordinate (B) at (0,2);
  \coordinate (C) at (2,2);

  \draw[bg] (O)--(A)--(C)--(B)--cycle;

  \draw[face] (O)--(A);
  \draw[face] (O)--(C);

  \node[vtx] at (O) {};
  \node[vtx] at (A) {};
  \node[missing] at (B) {};
  \node[vtx] at (C) {};

  \node[lab, below left]  at (O) {};
  \node[lab, below right] at (A) {};
  \node[lab, above left]  at (B) {};
  \node[lab, above right] at (C) {};
\end{scope}

\end{tikzpicture}
\end{center}
And the three-dimensional case looks as follows.
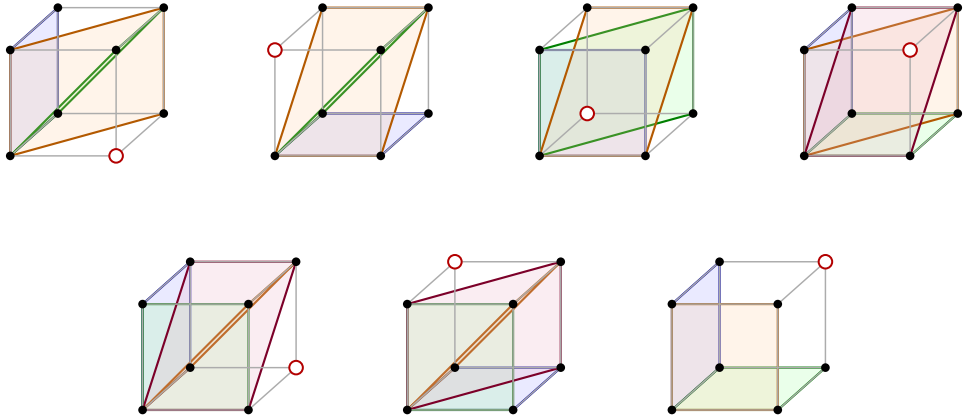
\begin{center}
\begin{tikzpicture}[
    scale=0.7,
    line cap=round,
    line join=round,
    edge/.style={draw=gray!65, line width=0.55pt},
    frontedge/.style={draw=black, line width=0.65pt},
    vertex/.style={circle, fill=black, inner sep=1.15pt},
    missing/.style={circle, draw=red!70!black, fill=white, line width=0.75pt, inner sep=1.8pt},
    faceA/.style={fill=blue!35,   draw=blue!60!black,   fill opacity=0.23, line width=0.8pt},
    faceB/.style={fill=green!35,  draw=green!50!black,  fill opacity=0.23, line width=0.8pt},
    faceC/.style={fill=orange!35, draw=orange!70!black, fill opacity=0.23, line width=0.8pt},
    faceD/.style={fill=purple!30, draw=purple!65!black, fill opacity=0.23, line width=0.8pt}
]

\begin{scope}[shift={(0,4.8)}]
  \coordinate (A) at (0,0);
  \coordinate (B) at (2,0);
  \coordinate (C) at (0,2);
  \coordinate (D) at (2,2);
  \coordinate (E) at (0.9,0.8);
  \coordinate (F) at (2.9,0.8);
  \coordinate (G) at (0.9,2.8);
  \coordinate (H) at (2.9,2.8);

  \filldraw[faceA] (A)--(C)--(G)--(E)--cycle;
  \filldraw[faceB] (A)--(D)--(H)--(E)--cycle;
  \filldraw[faceC] (A)--(C)--(H)--(F)--cycle;

  \draw[edge] (A)--(B)--(D)--(C)--cycle;
  \draw[edge] (E)--(F)--(H)--(G)--cycle;
  \draw[edge] (A)--(E);
  \draw[edge] (B)--(F);
  \draw[edge] (C)--(G);
  \draw[edge] (D)--(H);

  \node[vertex]  at (A) {};
  \node[missing] at (B) {};
  \node[vertex]  at (C) {};
  \node[vertex]  at (D) {};
  \node[vertex]  at (E) {};
  \node[vertex]  at (F) {};
  \node[vertex]  at (G) {};
  \node[vertex]  at (H) {};
\end{scope}

\begin{scope}[shift={(5.0,4.8)}]
  \coordinate (A) at (0,0);
  \coordinate (B) at (2,0);
  \coordinate (C) at (0,2);
  \coordinate (D) at (2,2);
  \coordinate (E) at (0.9,0.8);
  \coordinate (F) at (2.9,0.8);
  \coordinate (G) at (0.9,2.8);
  \coordinate (H) at (2.9,2.8);

  \filldraw[faceA] (A)--(B)--(F)--(E)--cycle;
  \filldraw[faceB] (A)--(D)--(H)--(E)--cycle;
  \filldraw[faceC] (A)--(B)--(H)--(G)--cycle;

  \draw[edge] (A)--(B)--(D)--(C)--cycle;
  \draw[edge] (E)--(F)--(H)--(G)--cycle;
  \draw[edge] (A)--(E);
  \draw[edge] (B)--(F);
  \draw[edge] (C)--(G);
  \draw[edge] (D)--(H);

  \node[vertex]  at (A) {};
  \node[vertex]  at (B) {};
  \node[missing] at (C) {};
  \node[vertex]  at (D) {};
  \node[vertex]  at (E) {};
  \node[vertex]  at (F) {};
  \node[vertex]  at (G) {};
  \node[vertex]  at (H) {};
\end{scope}

\begin{scope}[shift={(10.0,4.8)}]
  \coordinate (A) at (0,0);
  \coordinate (B) at (2,0);
  \coordinate (C) at (0,2);
  \coordinate (D) at (2,2);
  \coordinate (E) at (0.9,0.8);
  \coordinate (F) at (2.9,0.8);
  \coordinate (G) at (0.9,2.8);
  \coordinate (H) at (2.9,2.8);

  \filldraw[faceA] (A)--(B)--(D)--(C)--cycle;
  \filldraw[faceB] (A)--(C)--(H)--(F)--cycle;
  \filldraw[faceC] (A)--(B)--(H)--(G)--cycle;

  \draw[edge] (A)--(B)--(D)--(C)--cycle;
  \draw[edge] (E)--(F)--(H)--(G)--cycle;
  \draw[edge] (A)--(E);
  \draw[edge] (B)--(F);
  \draw[edge] (C)--(G);
  \draw[edge] (D)--(H);

  \node[vertex]  at (A) {};
  \node[vertex]  at (B) {};
  \node[vertex]  at (C) {};
  \node[vertex]  at (D) {};
  \node[missing] at (E) {};
  \node[vertex]  at (F) {};
  \node[vertex]  at (G) {};
  \node[vertex]  at (H) {};
\end{scope}

\begin{scope}[shift={(15.0,4.8)}]
  \coordinate (A) at (0,0);
  \coordinate (B) at (2,0);
  \coordinate (C) at (0,2);
  \coordinate (D) at (2,2);
  \coordinate (E) at (0.9,0.8);
  \coordinate (F) at (2.9,0.8);
  \coordinate (G) at (0.9,2.8);
  \coordinate (H) at (2.9,2.8);

  \filldraw[faceA] (A)--(C)--(G)--(E)--cycle;
  \filldraw[faceB] (A)--(B)--(F)--(E)--cycle;
  \filldraw[faceC] (A)--(C)--(H)--(F)--cycle;
  \filldraw[faceD] (A)--(B)--(H)--(G)--cycle;

  \draw[edge] (A)--(B)--(D)--(C)--cycle;
  \draw[edge] (E)--(F)--(H)--(G)--cycle;
  \draw[edge] (A)--(E);
  \draw[edge] (B)--(F);
  \draw[edge] (C)--(G);
  \draw[edge] (D)--(H);

  \node[vertex]  at (A) {};
  \node[vertex]  at (B) {};
  \node[vertex]  at (C) {};
  \node[missing] at (D) {};
  \node[vertex]  at (E) {};
  \node[vertex]  at (F) {};
  \node[vertex]  at (G) {};
  \node[vertex]  at (H) {};
\end{scope}

\begin{scope}[shift={(2.5,0)}]
  \coordinate (A) at (0,0);
  \coordinate (B) at (2,0);
  \coordinate (C) at (0,2);
  \coordinate (D) at (2,2);
  \coordinate (E) at (0.9,0.8);
  \coordinate (F) at (2.9,0.8);
  \coordinate (G) at (0.9,2.8);
  \coordinate (H) at (2.9,2.8);

  \filldraw[faceA] (A)--(C)--(G)--(E)--cycle;
  \filldraw[faceB] (A)--(B)--(D)--(C)--cycle;
  \filldraw[faceC] (A)--(D)--(H)--(E)--cycle;
  \filldraw[faceD] (A)--(B)--(H)--(G)--cycle;

  \draw[edge] (A)--(B)--(D)--(C)--cycle;
  \draw[edge] (E)--(F)--(H)--(G)--cycle;
  \draw[edge] (A)--(E);
  \draw[edge] (B)--(F);
  \draw[edge] (C)--(G);
  \draw[edge] (D)--(H);

  \node[vertex]  at (A) {};
  \node[vertex]  at (B) {};
  \node[vertex]  at (C) {};
  \node[vertex]  at (D) {};
  \node[vertex]  at (E) {};
  \node[missing] at (F) {};
  \node[vertex]  at (G) {};
  \node[vertex]  at (H) {};
\end{scope}

\begin{scope}[shift={(7.5,0)}]
  \coordinate (A) at (0,0);
  \coordinate (B) at (2,0);
  \coordinate (C) at (0,2);
  \coordinate (D) at (2,2);
  \coordinate (E) at (0.9,0.8);
  \coordinate (F) at (2.9,0.8);
  \coordinate (G) at (0.9,2.8);
  \coordinate (H) at (2.9,2.8);

  \filldraw[faceA] (A)--(B)--(F)--(E)--cycle;
  \filldraw[faceB] (A)--(B)--(D)--(C)--cycle;
  \filldraw[faceC] (A)--(D)--(H)--(E)--cycle;
  \filldraw[faceD] (A)--(C)--(H)--(F)--cycle;

  \draw[edge] (A)--(B)--(D)--(C)--cycle;
  \draw[edge] (E)--(F)--(H)--(G)--cycle;
  \draw[edge] (A)--(E);
  \draw[edge] (B)--(F);
  \draw[edge] (C)--(G);
  \draw[edge] (D)--(H);

  \node[vertex]  at (A) {};
  \node[vertex]  at (B) {};
  \node[vertex]  at (C) {};
  \node[vertex]  at (D) {};
  \node[vertex]  at (E) {};
  \node[vertex]  at (F) {};
  \node[missing] at (G) {};
  \node[vertex]  at (H) {};
\end{scope}

\begin{scope}[shift={(12.5,0)}]
  \coordinate (A) at (0,0);
  \coordinate (B) at (2,0);
  \coordinate (C) at (0,2);
  \coordinate (D) at (2,2);
  \coordinate (E) at (0.9,0.8);
  \coordinate (F) at (2.9,0.8);
  \coordinate (G) at (0.9,2.8);
  \coordinate (H) at (2.9,2.8);

  \filldraw[faceA] (A)--(C)--(G)--(E)--cycle;
  \filldraw[faceB] (A)--(B)--(F)--(E)--cycle;
  \filldraw[faceC] (A)--(B)--(D)--(C)--cycle;

  \draw[edge] (A)--(B)--(D)--(C)--cycle;
  \draw[edge] (E)--(F)--(H)--(G)--cycle;
  \draw[edge] (A)--(E);
  \draw[edge] (B)--(F);
  \draw[edge] (C)--(G);
  \draw[edge] (D)--(H);

  \node[vertex]  at (A) {};
  \node[vertex]  at (B) {};
  \node[vertex]  at (C) {};
  \node[vertex]  at (D) {};
  \node[vertex]  at (E) {};
  \node[vertex]  at (F) {};
  \node[vertex]  at (G) {};
  \node[missing] at (H) {};
\end{scope}
\end{tikzpicture}
\end{center}
\end{example}

For every dimension $n$ there are thus $2^n-1$ corners.
However, as the next proposition shows, there essentially are only $n$ relevant corners.
Let $d_\omega = \lvert\omega\rvert$ be the distance of $\omega$ to $0^n$.

\begin{proposition}
Let $n\in\N$ and $\omega,\epsilon\in\{0,1\}^n\setminus\{0^n\}$.
Then there is an automorphism $\phi$ on ${\bblox}^n$ in $\bbGamma$ such that the induced function $\mathrm{Sub}({\bblox}^n)\to\mathrm{Sub}({\bblox}^n)$
maps ${\bbcor}^n_\omega$ to ${\bbcor}^n_\epsilon$ precisely if $d_\omega = d_\epsilon$.
\end{proposition}

\begin{proof}
If $d_\omega = d_\epsilon$, there exists an automorphism $\phi$ of ${\bblox}^n$ in $\bbGamma$ with $\phi(\omega) = \epsilon$,
given by permuting the corresponding coordinates (since there are equally many 1's in $\omega$ as in $\epsilon$).
Now we need to see that the composition ${\bbcor}^n_\omega\to{\bblox}^n\to{\bblox}^n$ corresponds to ${\bbcor}^n_\epsilon\sub{\bblox}^n$.
Since $\phi$ is an isomorphism, it induces a bijection of subobjects.
Thus, for $f\colon{\bblox}^{n-1}\hookrightarrow{\bblox}^n$, $\omega\notin\im f$ precisely if $\epsilon\notin\im g$ for $g = \phi\circ f$.
Therefore, the unions are the same and so $\phi$ maps ${\bbcor}^n_\omega$ to ${\bbcor}^n_\epsilon$.
\end{proof}

\begin{definition}\label{def:corner_in_gamma}
We define
\[
    \Lambda_{\bbGamma} \coloneq \{{\bbcor}^n_\omega\hookrightarrow{\bblox}^n \mid n\in\N,\,\omega\in\{0,1\}^n\setminus\{0^n\}\}.
\]
\end{definition}

A priori, given a cubeset $X$, its underlying $\bbGamma$-set $j^\ast X$ (see Remark~\ref{rem:adjoint-triple-j})
supports a different notion of fibrancy w.r.t. $\Lambda_{\bbGamma}$.
However, in the concrete setting, there is no difference as the next proposition shows.

\begin{proposition}\label{prop:nilfib-restricted-nilfib-concrete}
Let $X$ be a concrete cubeset.
Then the following assertions are equivalent.
\begin{enumerate}[(a)]
    \item The cubeset $X$ is fibrant.
    \item The projection $p\colon j^\ast X\to\ast$ has the right lifting property against
        $\Lambda_{\bbGamma}$, i.e. the underlying $\bbGamma$-set is fibrant w.r.t. $\Lambda_{\bbGamma}$.
\end{enumerate}
\end{proposition}

\begin{proof}
By adjunction, a diagonal lift in the left diagram (in $\PSh(\bbGamma)$) corresponds to a lift in the right diagram (in $\cSet$):
\begin{center}
\begin{tikzcd}
	{{\bbcor}^n_\omega} & {j^\ast X} & {j_!{\bbcor}^n_\omega} & X \\
	{{\bblox}^n} & \ast & {j_!{\bblox}^n} & \ast
	\arrow[from=1-1, to=1-2]
	\arrow[from=1-1, to=2-1]
	\arrow[""{name=0, anchor=center, inner sep=0}, from=1-2, to=2-2]
	\arrow[from=1-3, to=1-4]
	\arrow[""{name=1, anchor=center, inner sep=0}, from=1-3, to=2-3]
	\arrow[from=1-4, to=2-4]
	\arrow[dashed, from=2-1, to=1-2]
	\arrow[from=2-1, to=2-2]
	\arrow[dashed, from=2-3, to=1-4]
	\arrow[from=2-3, to=2-4]
	\arrow["\leftrightsquigarrow"{description}, draw=none, from=0, to=1]
\end{tikzcd}
\end{center}
But $j_!{\bblox}^n = {\bbox}^n$ and by Lemma~\ref{lem:j-geometric-morphism}, $j_!$ preserves finite unions of subobjects
(the union of subobjects is the pushout of those subobjects along their intersection), so we have that
\[
    j_!{\bbcor}^n_\omega = \bigcup_{\substack{f\colon{\bbox}^{n-1}\sub{\bbox}^n \\ f\in\PSh(\bbGamma) \\ \omega\notin\im f}}{\bbox}^{n-1}
\]
where the union now is taken in $\cSet$ instead of $\PSh(\bbGamma)$.
Note that for $\omega=1^n$ this just corresponds to the classical corner ${\bbcor}^n$ in $\cSet$:
If $f\colon{\bbox}^{n-1}\hookrightarrow{\bbox}^n$ is a linear (= base point preserving) injective map in $\bbox$ with $1^n\notin\im f$,
then by its canonical representation we know that each $x_i$ appears at least once, and there must be at least one constant $0$.
In particular, by dimension, every $x_i$ appears exactly once, so $f=\epsilon_k$ for some $k$.
Here, $\epsilon_k$ is the inert map corresponding to the $k$-th coordinate.
Thus, $j_!{\bbcor}^n_{1^n} = {\bbcor}^n$.
This directly yields the implication $(b)\implies (a)$.

The other direction requires a bit more work for which we first develop the setup.
\end{proof}

We now want to mimic the definition of $\Lambda_{\bbGamma}$ in $\PSh(\bbGamma)$ also in $\cSet$.

\begin{definition}
Let $n\in\N$ and $\omega\in\{0,1\}^n\setminus\{0^n\}$.
Then define the $n$-corner ${\bbcor}^n_\omega$ w.r.t. $\omega$ as the union over all injections
$f\colon{\bbox}^{n-1}\to{\bbox}^n$ in $\bbox$ such that $f$ misses $\omega$, i.e., $\omega\notin\im f$:
\[
    {\bbcor}^n_\omega = \bigcup_{\substack{f\colon{\bbox}^{n-1}\sub{\bbox}^n \\ \omega\notin\im f}}{\bbox}^{n-1}\hookrightarrow{\bbox}^n.
\]
We define
\[
    \Lambda' \coloneq \{{\bbcor}^n_\omega\hookrightarrow{\bbox}^n \mid n\in\N,\,\omega\in\{0,1\}^n\setminus\{0^n\}\}
\]
as a set of morphisms in $\cSet$.
\end{definition}

\begin{lemma}\label{lem:alternative-nilfibrant-concrete}
Let $X$ be a concrete cubeset with the gluing property and consider the projection $p\colon X\to\ast$.
Then the following assertions are equivalent.
\begin{enumerate}[(a)]
    \item $p$ has the right lifting property w.r.t. $\Lambda'$.
    \item $p$ has the right lifting property w.r.t. $\Lambda$, i.e., $X$ is fibrant.
\end{enumerate}
\end{lemma}

\begin{proof}
$(a)\implies(b)$:
Let $\lambda\colon{\bbcor}^n\to X$ be an $n$-corner in $X$.
For every subobject $f\colon{\bbox}^{n-1}\to{\bbox}^n$ that misses $1^n$ and is not an inert,
there exists an $i$ such that $x_i$ appears twice in the canonical representation of $f$,
once as $x_i$ (at coordinate $k$) and once as $1-x_i$ (at coordinate $j$).
Then $f\lvert_{[x_i=0]}$ factors through the inert $\epsilon_k$ as an $n-2$-dimensional cube,
and $f\lvert_{[x_i=1]}$ factors through the inert $\epsilon_j$ as an $n-2$-dimensional cube.
Restricting the corner to $\epsilon_k\cup\epsilon_j$ yields a map ${\bbox}^{n-1}\sqcup_{{\bbox}^{n-2}}{\bbox}^{n-1}\to X$.
Since $X$ has the (unique) gluing property, we can extend it (uniquely) to a map
\[
    ({\bbox}^{n-1}\sqcup_{{\bbox}^{n-2}}{\bbox}^{n-1})\sqcup_{{\bbox}^{n-2}\sqcup{\bbox}^{n-2}}{\bbox}^{n-1}\to X
\]
whose restriction to ${\bbox}^{n-1}$ can be viewed as an extension of $\lambda$ to (the image of) $f$.
Iterating this process exhausts, after finitely many steps, all injective non-inert maps ${\bbox}^{n-1}\to{\bbox}^n$ that miss $1^n$,
compatible on intersections (by the uniqueness of the gluings),
and thus gives a map $\lambda'\colon{\bbcor}^n_{1^n}\to X$ that restricts to $X$.
This map $\lambda'$ can, by assumption, be lifted to a cube in $X$ which then also is a completion of $\lambda$.

$(b)\implies(a)$:
By applying an automorphism, we may assume that $\omega = 1^n$.
So we need to see that $p$ lifts against ${\bbcor}^n_{1^n}\to{\bbox}^n$.
Restricting to ${\bbcor}^n$ yields a lift
\begin{center}
\begin{tikzcd}
	{{\bbcor}^n} & X \\
	{{\bbcor}^n_{1^n}} & X \\
	{{\bbox}^n} & \ast
	\arrow[from=1-1, to=1-2]
	\arrow[from=1-1, to=2-1]
	\arrow[from=2-1, to=2-2]
	\arrow[from=2-1, to=3-1]
	\arrow["\id"', from=2-2, to=1-2]
	\arrow[from=2-2, to=3-2]
	\arrow[dashed, from=3-1, to=1-2]
	\arrow[from=3-1, to=3-2]
\end{tikzcd}
\end{center}
making the outer rectangle and triangles commutative.
Since $X$ is concrete, also the lower small triangle commutes:
the values of the corner completion on an antidiagonal in ${\bbcor}^n_{1^n}$ not in ${\bbcor}^n$
are uniquely determined by its values on points, which already are present in ${\bbcor}^n$.
\end{proof}

\begin{lemma}\label{lem:different-liftings-comparison}
We have the following inclusion
\[
    j_!(\Lambda_{\bbGamma})\subset\mathrm{sat}(\Lambda'\cup\bprism)
\]
of morphisms in $\cSet$.
\end{lemma}

\begin{proof}
Let $f\in\Lambda_{\bbGamma}$ be a map $f\colon{\bbcor}^n_\omega\to{\bblox}^n$.
We first extend $j_!f$ along outer faces that don't contain $\omega$,
using pushouts along maps in $\bprism$,
such that, after applying an automorphism of the cube mapping $\omega$ to $1^n$,
the domain of the extended $j_!f$ contains the classical corner ${\bbcor}^n$.
Lastly, we extend $j_!f$ to all anti-diagonals, using pushouts along maps in $\bprism$.
This procedure thus has extended $j_!f$ to a map in $\Lambda'$, yielding the claim.

Assume that there is an inert $\epsilon_i\colon{\bbox}^{n-1}\to{\bbox}^n$ such that $\omega_i = 0$,
i.e., $\omega\in\epsilon_i$.
Now take some $k$ with $\omega_k = 1$.
Then $\epsilon_k$ factors through ${\bbcor}^n_\omega$.
Define $s_k\colon{\bbox}^{n-1}\to{\bbox}^n$ by
\[
    s_k(x_1,\ldots,x_{n-1})\mapsto (x_1,\ldots,x_k,x_{k+1},\dots,x_{i-1}, x_k, x_{i}\ldots,x_{n-1})
\]
where the second $x_k$ appears at index $i$, having image $[x_k=x_i]$ (for $i<k$ swap them).
Since $s_k$ is linear and doesn't contain $\omega$, it factors through ${\bbcor}^n_\omega$.
Moreover, $\epsilon_k\cap s_k$ is isomorphic to ${\bbox}^{n-2}$,
so we obtain a map $(\epsilon_k,s_k)\colon{\bbox}^{n-1}\sqcup_{{\bbox}^{n-2}}{\bbox}^{n-1}\to{\bbcor}^n_\omega$.
Now $\epsilon_k\cap\overline{\epsilon}_i$ and $s_k\cap\overline{\epsilon}_i$ both are isomorphic to ${\bbox}^{n-2}$,
and their intersection as subobjects of ${\bbox}^n$ is empty.
This induces a map $({\bbox}^{n-1}\sqcup_{{\bbox}^{n-2}}{\bbox}^{n-1})\sqcup_{{\bbox}^{n-2}\sqcup{\bbox}^{n-2}}{\bbox}^{n-1}\to{\bbox}^n$
which on the left acts as $(\epsilon_k,s_k)$ and on the right as $\overline{\epsilon}_i$.
This map makes the outer square commutative
\begin{center}
\begin{tikzcd}
	{{\bbox}^{n-1}\sqcup_{{\bbox}^{n-2}}{\bbox}^{n-1}} && {{\bbcor}^n_\omega} & \\
	\\
	{({\bbox}^{n-1}\sqcup_{{\bbox}^{n-2}}{\bbox}^{n-1})\sqcup_{{\bbox}^{n-2}\sqcup{\bbox}^{n-2}}{\bbox}^{n-1}} && {X_i} \\
	&&& {{\bbox}^n}
	\arrow[""{name=0, anchor=center, inner sep=0}, from=1-1, to=1-3]
	\arrow[from=1-1, to=3-1]
	\arrow[from=1-3, to=3-3]
	\arrow["j_!f", from=1-3, to=4-4]
	\arrow[from=3-1, to=3-3]
	\arrow[from=3-1, to=4-4]
	\arrow[dashed, from=3-3, to=4-4]
	\arrow["\lrcorner"{anchor=center, pos=0.125, rotate=180}, draw=none, from=3-3, to=0]
\end{tikzcd}.
\end{center}
Taking the pushout $X_i$ gives a factorization of $j_!f$ where the first map
is in $\mathrm{sat}(\Lambda'\cup\bprism)$ (since the left vertical map is in $\bprism$),
and all maps are injective.
Now repeat this process for all $i$ with $\omega_i=0$.
This yields a factorization
\begin{center}
\begin{tikzcd}
	{{\bbcor}^n_\omega} & X & {{\bbox}^n}
	\arrow["h"', from=1-1, to=1-2]
	\arrow["j_!f", curve={height=-12pt}, from=1-1, to=1-3]
	\arrow["g"', from=1-2, to=1-3]
\end{tikzcd}
\end{center}
of $j_!f$ where $h$ is in $\mathrm{sat}(\Lambda'\cup\bprism)$.
If there is no $i$ with $\omega_i=0$, then $h$ is the identity.
Now we rotate the cube: there exists an automorphism $\sigma$ of ${\bbox}^n$ in $\cSet$ that interchanges $\omega$ and $1^n$
(take the corresponding flips at those coordinates $i$ where $\omega_i = 0$).
Under this rotation, $g\colon X\to{\bbox}^n$ is the subobject corresponding to
\[
    \sigma({\bbcor}^n_\omega)\cup{\bbcor}^n\sub{\bbox}^n.
\]
Lastly, we have to glue the antidiagonals to $X$ via pushouts along $\bprism$:
for two inerts $\epsilon_i$ and $\epsilon_j$ in ${\bbox}^n$ with $i<j$, such that the map $a_{ij}\colon{\bbox}^{n-1}\to{\bbox}^n$
connecting $\epsilon_i$ and $\epsilon_j$ as an antidiagonal
(explicitly given by
\[
    a_{ij}(x_1,\ldots,x_{n-1}) = (y_1,\ldots,y_n)
\]
where $y_k = x_k$ for $k<j$, $y_j = 1-x_i$, $y_k = x_{k-1}$ for $k>j$)
does not yet factor through $X$, we can glue it as in the pushout
\begin{center}
\begin{tikzcd}
	{{\bbox}^{n-1}\sqcup_{{\bbox}^{n-2}}{\bbox}^{n-1}} && X & \\
	\\
	{({\bbox}^{n-1}\sqcup_{{\bbox}^{n-2}}{\bbox}^{n-1})\sqcup_{{\bbox}^{n-2}\sqcup{\bbox}^{n-2}}{\bbox}^{n-1}} && {Y_{ij}} \\
	&&& {{\bbox}^n}
	\arrow[""{name=0, anchor=center, inner sep=0}, from=1-1, to=1-3]
	\arrow[from=1-1, to=3-1]
	\arrow[from=1-3, to=3-3]
	\arrow["g", from=1-3, to=4-4]
	\arrow[from=3-1, to=3-3]
	\arrow[from=3-1, to=4-4]
	\arrow[dashed, from=3-3, to=4-4]
	\arrow["\lrcorner"{anchor=center, pos=0.125, rotate=180}, draw=none, from=3-3, to=0]
\end{tikzcd}.
\end{center}
Iterating this process for all such pairs $(i,j)$ yields an injection $Y\to{\bbox}^n$.
So we have shown that $g$ factors as
\begin{center}
\begin{tikzcd}
	X & Y & {{\bbox}^n}
	\arrow["{h'}"', from=1-1, to=1-2]
	\arrow["g", curve={height=-12pt}, from=1-1, to=1-3]
	\arrow["{g'}"', from=1-2, to=1-3]
\end{tikzcd}
\end{center}
where $h'\in\mathrm{sat}(\Lambda'\cup\bprism)$.
To show $g'\in\mathrm{sat}(\Lambda'\cup\bprism)$, we show that
\[
    Y = \bigcup_{\substack{f\colon{\bbox}^{n-1}\sub{\bbox}^n \\ 1^n\notin\im f}}{\bbox}^{n-1}
\]
where $f$ runs over maps in $\cSet$.
By construction, $Y$ is contained in the right hand side.
On the other hand, if $f\colon{\bbox}^{n-1}\to{\bbox}^n$ is a subobject with $1^n\notin\im f$,
then in the canonical representation of $f$,
each $x_i$ appears at least once, and either $y_j \equiv 0$ for some $j$ or there exists precisely one $i$ such that $x_i$ and $1-x_i$ appear.
In the first case, $f$ is an inert and thus factors through $X$, and hence through $Y$.
In the second case, $f$ is an antidiagonal and therefore factors through $Y$.
\end{proof}

This allows us to conclude the proposition.

\begin{proof}[Proof of Proposition~\ref{prop:nilfib-restricted-nilfib-concrete}]
$(a)\implies (b)$:
By Lemma~\ref{lem:alternative-nilfibrant-concrete}, the projection $p$ has the right lifting property w.r.t. $\Lambda'$.
Since $X$ is fibrant, it has the gluing property by Proposition~\ref{prop:nilfib-gluing},
i.e., $p$ has the right lifting property w.r.t. $\bprism$.
This means that $p$ has the right lifting property w.r.t. $\mathrm{sat}(\Lambda'\cup\bprism)$.
By Lemma~\ref{lem:different-liftings-comparison}, $p$ has the right lifting property w.r.t. $j_!(\Lambda_{\bbGamma})$.
This shows that $j^\ast X\to\ast$ has the right lifting property w.r.t. $\Lambda_{\bbGamma}$.
\end{proof}

\subsection{Truncations}\label{sec:truncation}

\begin{definition}\label{def:n-step}
A cubeset $X$ is of \emph{degree $n$} (or \emph{$n$-step}) for $n\geq 0$ if corner completion strictly above dimension $n$ is unique.
This means that for every $k > n$ restriction along the $k$-corner
\[
    \hom({\bbox}^k,X)\to\hom({\bbcor}^k,X)
\]
is an isomorphism (bijective).
A cubeset $X$ is of \emph{degree $-1$} if it is either empty or the point,
and it is of \emph{degree $-2$} if it is the point.
We denote the full subcategory of cubesets of degree $n$ by $\cSet_n$.
\end{definition}

\begin{remark}
If we let $\Lambda_n = \{{\bbcor}^k\to{\bbox}^k : k > n\}$, then the cubesets of degree $n$ are precisely the $\Lambda_n$-local objects.
In particular, a $\Lambda_n$-local object $X$ cannot distinguish between a $k$-corner and a $k$-cube.
\end{remark}

\begin{remark}\label{rem:corner-no-topology}
Note that, however, $\Lambda_n$ does not induce a Lawvere-Tierney topology on $\cSet$, or equivalently a Grothendieck topology on $\bbox$.
If it were, then to any morphism $f\colon{\bbox}^n\to{\bbox}^m$, a cover of ${\bbox}^m$ would allow a pullback along $f$ to a cover of ${\bbox}^n$.
But if $1^m\colon\ast\to{\bbox}^m$ is the point $1^m$ in $\{0,1\}^m$, the corner in ${\bbox}^m$ does not factor through this point (which it would have to, since the point only is covered by the point).
\end{remark}

\begin{proposition}\label{prop:truncations-char}
Let $X$ be a cubeset. Then the following assertions are equivalent.
\begin{enumerate}[(a)]
	\item The cubeset $X$ is $n$-step, i.e. the pullback map
	\[
        \hom({\bbox}^k,X)\to\hom({\bbcor}^k,X)
    \]
    is bijective for all $k>n$.
	\item For every $f\colon Y\to Z\in l(r(\Lambda_n))$
    (i.e., a trivial cofibration/retract of a relative cell complex w.r.t. corner inclusions of dimension $\geq n+1$),
	one has that
    \[
        f^\ast\colon\hom(Z,X)\to \hom(Y,X)
    \]
    is a bijection.
    \item It is \emph{internally $n$-step}, meaning that
	\[
        \ihom({\bbox}^k,X)\to\ihom({\bbcor}^k,X)
    \]
	is an isomorphism of cubesets for all $k>n$.
    \item For every $f\colon Y\to Z\in l(r(\Lambda_n))$
	(i.e. a trivial cofibration/retract of a relative cell complex w.r.t. corner inclusions of dimension $\geq n+1$),
    one has that
	\[
    	\ihom(Z, X)\to \ihom(Y,X)
	\]
	is an isomorphism.
    \item The map
	\[
        \ihom({\bbox}^{n+1},X)\to\ihom({\bbcor}^{n+1},X)
    \]
    is an isomorphism.
\end{enumerate}
\end{proposition}

\begin{proof}
	A morphism of cubesets is an isomorphism if and only if it is a bijection pointwise.
	This means that a cubeset $X$ is internally $n$-step precisely if
    \[
        \hom({\bbox}^k\times{\bbox}^j,X)\to\hom({\bbcor}^k\times{\bbox}^j,X)
    \]
	is a bijection for all $j\in \N_0$ and all $k>n$.
	Thus, the internal variants are stronger than the non-internal variants (being the internal ones at $j=0$).

	$(a)\implies(b)$ and $(c)\implies(d)$: It suffices to see that the preimage of the class of isomorphisms under a cocontinuous functor is saturated.
	For this note that the pushout along an isomorphism is again an isomorphism,
	as well as transfinite compositions of isomorphisms and retracts.

    $(a)\implies(c)$:
	We follow the route of Remark~\ref{rem:internal-nilfib}.
	In Corollary~\ref{cor:internal-nilfib} we have shown that the roof $\bbcor^{k}\times{\bbox}^j\to{\bbox}^{k+j}$
	is a relative cell complex built from corner inclusions ${\bbcor}^i\to{\bbox}^i$ for $k\le i$
	(one examines the proof carefully to see that there is no other $i$ involved, see Remark~\ref{prop:internal-nilfib}).
	This implies that the roofs ${\bbcor}^k\times{\bbox}^j\to{\bbox}^{k+j}$
	are in the saturated closure of all ${\bbcor}^i\to{\bbox}^i$ for $k\le i$.
	Thus we know that having unique corner completion implies unique internal corner completion.

    $(e)\implies(a)$:
    We need to show that the natural map
    \[
        \hom({\bbox}^{n+1+j},X)\to\hom({\bbcor}^{n+1+j},X)
    \]
    is bijective for all $j\in\N_0$.
    We have a factorization
    \begin{center}
    \begin{tikzcd}
	{\hom({\bbox}^{n+1+j},X)} && {\hom({\bbcor}^{n+1}\times{\bbox}^j,X)} \\
	& {\hom({\bbcor}^{n+1+j},X)}
	\arrow[from=1-1, to=1-3]
	\arrow[from=1-1, to=2-2]
	\arrow[from=2-2, to=1-3]
    \end{tikzcd}
    \end{center}
    where the horizontal map is bijective for all $j\in\N_0$ by assumption on the internal Hom.
    For $j=0$, the left map is bijective by assumption (since then the right map degenerates).
    In the induction step, suppose both the left and right map are bijections.
    Now the map ${\bbcor}^{n+1+j}\times{\bbox}^1\to{\bbcor}^{n+1+j+1}$ is a pushout
    along the ($n+1+j$)-dimensional corner inclusion ${\bbcor}^{n+1+j}\to{\bbox}^{n+1+j}$
    (see the proof of Proposition~\ref{prop:internal-nilfib}),
    so pullback along the former induces a bijection on $\hom$-sets.
    Moreover, by induction hypothesis, the map ${\bbcor}^{n+1}\times{\bbox}^j\to{\bbcor}^{n+1+j}$
    induces a bijection on the level of $\hom(-,X)$.
    Thus we can use the factorization
    \begin{center}
    \begin{tikzcd}
	{{\bbcor}^{n+1}\times{\bbox}^{j+1}} & {{\bbcor}^{n+1+j}\times{\bbox}^1} & {{\bbcor}^{n+1+j+1}}
	\arrow[from=1-1, to=1-2]
	\arrow[from=1-2, to=1-3]
    \end{tikzcd}
    \end{center}
    to see that the composite induces a bijection on $\hom(-,X)$.
    In turn this means that both $\hom({\bbcor}^{n+1+j},X)\to\hom({\bbcor}^{n+1}\times{\bbox}^j,X)$
    and hence $\hom({\bbox}^{n+1+j},X)\to\hom({\bbcor}^{n+1+j},X)$ are bijections, completing the induction.
\end{proof}

\begin{remark}
	Note that imposing the bijection for a single $k$ externally fails:
	That
	\[\hom({\bbox}^k,X)\to \hom({\bbcor}^k,X)\]
	is bijective is not equivalent to the map
    \[\ihom({\bbox}^k,X)\to \ihom({\bbcor}^k,X)\]
    being an isomorphism,
    since the latter considers all dimensions.
	As an explicit counterexample take some cubeset with $X(0)=X(1)=\ast$ but $X(2)\ne \ast$ and consider
	$k=1$.
\end{remark}

\begin{proposition}\label{prop:step-left-adjoint}
The inclusion of $n$-step cubesets into all cubesets admits a left adjoint which we call the canonical truncation $\tau_n\colon\cSet\to\cSet_n$.
\end{proposition}

\begin{proof}
This follows from \cite[Theorem 1.39]{Adamek1994}.
Alternatively, one can use the $\infty$-categorical version from \cite[Proposition 5.5.4.2]{Lurie2009}.
\end{proof}

First, let us remark that a certain fibrant replacement does not change the truncation.

\begin{lemma}
	Let $X\to X^{\mathrm{fib},n}$ be a fibrant replacement of $X$ starting in dimension $n+1$,
	i.e., a fibrant replacement w.r.t. $\Lambda_n$.
	Then
	\[
	    \tau_{n}X=\tau_n (X^{\mathrm{fib},n}).
	\]
\end{lemma}

\begin{proof}
The map $X\to X^{\mathrm{fib},n}$ is a trivial cofibration w.r.t. $\Lambda_n$,
i.e., $(X\to X^{\mathrm{fib},n})\in l(r(\Lambda_n))$.
By Proposition~\ref{prop:truncations-char}~(b), for every $n$-step cubeset $Y$, there is a bijection
\[
    \hom(X^{\mathrm{fib},n},Y) \to \hom(X,Y)
\]
natural in $X$ and $Y$.
Applying the truncation to both sides, the result follows from the Yoneda lemma.
\end{proof}

Thus, a good first step to truncate is to build a certain fibrant replacement.
Note that fibrant replacement a priori is not a left adjoint due to the impossibility of choosing a canonical corner completion.
Testing against objects with unique corner completion however makes this lift unique and thus we obtain the left adjointness.

\begin{remark}
Note that truncation does not preserve finite limits in general.
If it would do so, then $\cSet_n\to\cSet$ would be a geometric morphism
with the inclusion corresponding to a subtopos on $\cSet$.
Hence there exists a Grothendieck topology on $\bbox$ such that the $n$-step cubesets are the sheaves for this topology:
the sieves are precisely the injections $f\colon S\to{\bbox}^k$ for which pullback $f^\ast\colon\hom({\bbox}^k,X)\to\hom(S,X)$
is an isomorphism for each $n$-step cubeset $X$ (and $S$ corresponds to a subfunctor of the representable ${\bbox}^k$).
In particular, the corner inclusion ${\bbcor}^{n+1}\to{\bbox}^{n+1}$ is a cover for this topology.
But then Remark~\ref{rem:corner-no-topology} yields a contradiction to stability of covers under pullback.
\end{remark}

However, the truncation does preserve finite products.

\begin{proposition}\label{prop:n-trunc-ihom-prod}
If $X$ and $Y$ are cubesets and $Y$ is $n$-step, then $\ihom(X,Y)$ is $n$-step.
Moreover, the $n$-truncation preserves finite products.
\end{proposition}

\begin{proof}
First we show that $\ihom(X,Y)$ is $n$-step if $Y$ is so.
We need to see that pullback along the inclusion ${\bbcor}^k\to{\bbox}^k$ induces a bijection
\[
    \hom({\bbox}^k\times X,Y) = \hom({\bbox}^k,\ihom(X,Y)) \to \hom({\bbcor}^k,\ihom(X,Y)) = \hom({\bbcor}^k\times X,Y)
\]
for all $k>n$.
For $X = {\bbox}^m$ it follows from Corollary~\ref{cor:internal-nilfib}
and Remark~\ref{prop:internal-nilfib},
see also the proof of $(a)\implies(c)$ of Proposition~\ref{prop:truncations-char}.
But then it follows for general $X$ by writing $X$ as the colimit of representables,
using that colimits are stable under base change in $\cSet$.

For the second claim, one argues as in the last part of the proof of
Proposition~\ref{prop:concrete-ihom-prod}.
\end{proof}

The category of $n$-step cubesets additionally has the following closure properties.

\begin{proposition}\label{prop:n-trunc-closure}
The category of $n$-step cubesets $\cSet_n$ is closed under limits, filtered colimits and coproducts in $\cSet$.
\end{proposition}

\begin{proof}
	By continuity of the Hom-functor, $\cSet_n$ is closed under limits.
Closure under filtered colimits follows from the fact that ${\bbcor}^k$ is compact for all $k\in\N$.
Closure under coproducts follows from the fact that ${\bbcor}^k$ is connected for all $k\in\N$.
\end{proof}

\subsubsection{Computing the Truncation of Concrete Fibrant Cubesets}\label{ssec:truncation-compute}

The localisation $\tau_{n}X$ may always be computed naively by some iterated construction, see \cite[Construction 1.37]{Adamek1994}:
\begin{itemize}
	\item Glue to $X$ a $k$-cube for every $k$-corner which does not admit a filler in $X$ (for all $k> n$).
	\item Coequalize two $k$-cubes if they agree on their $k$-corner (for all $k>n$).
	\item Iterate this process until it stabilizes (at limit ordinals take the colimit).
\end{itemize}
In our case, the morphisms at which we localize have compact domain and codomain
and thus the process stabilizes after $\aleph_0$-many steps.
For general cubespaces, all these steps and in particular such an iterated construction seem unavoidable,
as the following examples show.
\begin{example}
\begin{enumerate}[(i)]
	\item For an example where one needs iteration of the identification step consider the double Kebab bread
	 
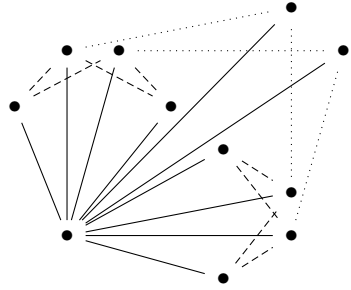
\begin{center}
\begin{tikzcd}[ampersand replacement=\&, row sep=tiny, column sep=tiny]
	\&\&\&\&\&\& \bullet \& \\
	\& \bullet \& \bullet \&\&\&\&\& \bullet \\
	\\
	\bullet \&\&\& \bullet \\
	\&\&\&\& \bullet \\
	\&\&\&\&\&\& \bullet \\
	\& \bullet \&\&\&\&\& \bullet \\
	\&\&\&\& \bullet
	\arrow[dotted, no head, from=2-2, to=1-7]
	\arrow[dashed, no head, from=2-2, to=4-4]
	\arrow[dotted, no head, from=2-3, to=2-8]
	\arrow[dashed, no head, from=2-3, to=4-4]
	\arrow[dashed, no head, from=4-1, to=2-2]
	\arrow[dashed, no head, from=4-1, to=2-3]
	\arrow[dashed, no head, from=5-5, to=6-7]
	\arrow[dashed, no head, from=5-5, to=7-7]
	\arrow[dotted, no head, from=6-7, to=1-7]
	\arrow[no head, from=7-2, to=1-7]
	\arrow[no head, from=7-2, to=2-2]
	\arrow[no head, from=7-2, to=2-3]
	\arrow[no head, from=7-2, to=2-8]
	\arrow[no head, from=7-2, to=4-1]
	\arrow[no head, from=7-2, to=4-4]
	\arrow[no head, from=7-2, to=5-5]
	\arrow[no head, from=7-2, to=6-7]
	\arrow[no head, from=7-2, to=7-7]
	\arrow[no head, from=7-2, to=8-5]
	\arrow[dotted, no head, from=7-7, to=2-8]
	\arrow[dashed, no head, from=8-5, to=6-7]
	\arrow[dashed, no head, from=8-5, to=7-7]
\end{tikzcd}
\end{center}
	\item For the necessity of the insertion step simply take a non-fibrant cubeset with many corners that do not yet exist, but assure that if they exist they are unique (e.g., take the $k$-skeleton of a nontrivial concrete fibrant cubeset).
	\item To see that in each identification step one needs to take the induced equivalence relation consider
	
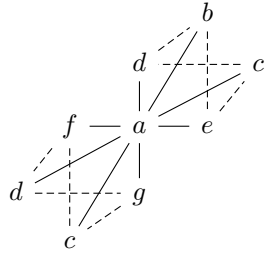
\begin{center}
\begin{tikzcd}[ampersand replacement=\&, row sep= tiny, column sep=tiny]
	\&\&\&\&\& b \& \\
	\&\&\& d \&\&\& c \\
	\\
	\& f \&\& a \&\& e \\
	\\
	d \&\&\& g \\
	\& c
	\arrow[dashed, no head, from=1-6, to=4-6]
	\arrow[dashed, no head, from=2-4, to=1-6]
	\arrow[dashed, no head, from=2-7, to=2-4]
	\arrow[no head, from=4-2, to=4-4]
	\arrow[dashed, no head, from=4-2, to=6-1]
	\arrow[dashed, no head, from=4-2, to=7-2]
	\arrow[no head, from=4-4, to=1-6]
	\arrow[no head, from=4-4, to=2-4]
	\arrow[no head, from=4-4, to=2-7]
	\arrow[no head, from=4-4, to=4-6]
	\arrow[no head, from=4-4, to=6-1]
	\arrow[no head, from=4-4, to=6-4]
	\arrow[no head, from=4-4, to=7-2]
	\arrow[dashed, no head, from=4-6, to=2-7]
	\arrow[dashed, no head, from=6-1, to=6-4]
	\arrow[dashed, no head, from=6-4, to=7-2]
\end{tikzcd},

		\end{center}
	where $b\sim c$ and $d\sim c$ but there is no cube witnessing $b\sim d$.
	\item To see that \emph{first identify all necessary cubes and afterward glue all that you need} does not suffice, consider
	 
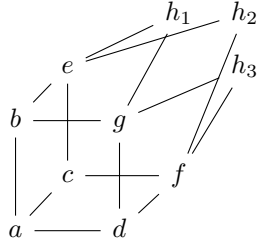
\begin{center}
\begin{tikzcd}[ampersand replacement=\&, row sep=tiny, column sep=tiny]
	\&\&\& {h_1} \& {h_2} \\
	\& e \&\&\& {h_3} \\
	b \&\& g \\
	\& c \&\& f \\
	a \&\& d
	\arrow[no head, from=2-2, to=1-4]
	\arrow[no head, from=2-2, to=1-5]
	\arrow[no head, from=3-1, to=2-2]
	\arrow[no head, from=3-1, to=3-3]
	\arrow[no head, from=3-3, to=1-4]
	\arrow[no head, from=3-3, to=2-5]
	\arrow[no head, from=4-2, to=2-2]
	\arrow[no head, from=4-2, to=4-4]
	\arrow[no head, from=4-4, to=1-5]
	\arrow[no head, from=4-4, to=2-5]
	\arrow[no head, from=5-1, to=3-1]
	\arrow[no head, from=5-1, to=4-2]
	\arrow[no head, from=5-1, to=5-3]
	\arrow[no head, from=5-3, to=3-3]
	\arrow[no head, from=5-3, to=4-4]
\end{tikzcd}.
\end{center}
Here, the three-dimensional corner has no filler.
If it had one, the new vertex would simultaneously be 2-equivalent to $h_1$, $h_2$ and $h_3$ and thus inserting the filler makes the $h_i$ equivalent.
 \end{enumerate}
\end{example}

Also note that concreteness is not preserved by the truncation.
However, together with fibrancy and concreteness, we may significantly simplify the computation of the truncation.

\begin{definition}\label{def:univ-replacement}
For $x,y\in X(0)$ define $x\sim_n y$ iff there exist $c,c'\in X(n+1)$ such that
\[
    c\rvert_{{\bbcor}^{n+1}} = c'\rvert_{{\bbcor}^{n+1}} \quad\text{and}\quad c(1) = x,\; c'(1) = y.
\]
\end{definition}

We can characterize this relation further as follows:

\begin{proposition}\label{prop:univ-replacement}
Let $X$ be a concrete fibrant cubeset, $n\in\N_0$ and $x,y\in X(0)$.
Then the following assertions are equivalent.
\begin{enumerate}[(a)]
    \item $x\sim_n y$,
    \item The configuration $c\rvert_{{\bbcor}^{n+1}} \equiv x$, $c(1) = y$ is in $X(n+1)$.
    \item For every $k\leq n+1$ and every cube $c\in X(k)$ with value $x$ at some vertex,
            the same cube with value $y$ at that vertex is in $X(k)$.
\end{enumerate}
In particular, the relation $\sim_n$ is an equivalence relation on $X(0)$.
\end{proposition}

This can e.g. be found in \cite{Gutman2020}
and follows mainly from the fact that every concrete fibrant cubeset has the \emph{gluing property}.
The property in (c) is also referred to as \emph{universal replacement}.
Using this we see that $x\sim_n y$ is an equivalence relation on $X(0)$, inducing a factor $X(0)/\sim_n$.
This yields a new cubeset $X/\sim_n$ whose value $X/\sim_n(m)$ is given by the image of $X(m)$ under the projection
\[
    X(0)^{2^m}\to (X(0)/\sim_n)^{2^m}.
\]
Explicitly this means that two $m$-cubes in $X$ are identified if all their points are equivalent in $X(0)$.
This cubeset thus is a factor of $X$ and allows for an easy description of $\tau_n X$, as the following proposition shows.

\begin{proposition}\label{prop:trunc_concfib}
If $X$ is a concrete fibrant cubeset, then the truncation $\tau_n X$ is given by $X/\sim_n$ and the unit
$\eta_X\colon X\to \tau_n X$ is a surjective fibration.
In particular, $\tau_n X$ is concrete and fibrant.
\end{proposition}

\begin{proof}
Denote $\eta\colon X\to\tau_n X$ the unit of the $n$-step adjunction and $\pi\colon X\to X/\sim_n$ the factor map.
Let us begin with establishing the existence of a morphism $f\colon X/\sim_n\to\tau_n X$ such that $f\circ\pi = \eta$, i.e., $\eta$ factors through $\pi$.
To do so we have to show that if $a$ and $b$ are two $k$-cubes in $X$ which are $n$-equivalent, i.e., $a(\omega)\sim_n b(\omega)$ for all $\omega\in\{0,1\}^k$,
then $\eta(a) = \eta(b)$.

If $k\leq n+1$ take a surjection $p\colon{\bbox}^{n+1}\to{\bbox}^k$.
Then $a' = p^\ast(a)\sim_n p^\ast(b) = b'$ in $X(n+1)$.
This allows us to first use universal replacement, and afterwards unique corner completion in $\tau_n X$.
Using universal replacement, the configuration
\[
    c_0\colon\{0,1\}^{n+1}\to X(0),\quad\omega\mapsto\begin{cases}
        b'(0),\quad &\omega = 0, \\
        a'(\omega),\quad &\omega\neq 0,
    \end{cases}
\]
is an element of $X(n+1)$ since $a'(0)\sim_n b'(0)$.
But after an appropriate rotation, $a'\rvert_{{\bbcor}^{n+1}} = c_0\rvert_{{\bbcor}^{n+1}}$
and so $\eta(a')\rvert_{{\bbcor}^{n+1}} = \eta(c_0)\rvert_{{\bbcor}^{n+1}}$.
Using unique corner completion in $\tau_n X$, we obtain that $\eta(a') = \eta(c_0)$.
Iterating this process, exchanging $a'(\omega)$ for $b'(\omega)$ one after another yields a sequence $a',c_0,c_1,\ldots,b'$ of ($n+1$)-cubes in $X$
with $\eta(a') = \eta(c_0) = \eta(c_1) = \ldots = \eta(b')$.
Taking the corresponding face of $a'$ and $b'$ we obtain that $\eta(a) = \eta(b)$.

For $k>n+1$, we induct on $k$: the induction start for $k = n+1$ has already been established.
If $k \geq n+2$ and $a\sim_n b$ for $a,b\in X(k)$,
then for every inert $\epsilon_i\colon{\bbox}^{k-1}\to{\bbox}^{k}$ we have that $\epsilon_i^\ast(a) \sim_n \epsilon_i^\ast(b)$ in $X(k-1)$.
By induction hypothesis, $\eta(\epsilon_i^\ast(a)) = \eta(\epsilon_i^\ast(b))$ which in turn yields $\eta(a)\rvert_{{\bbcor}^{k}} = \eta(b)\rvert_{{\bbcor}^{k}}$.
Using unique corner completion in $\tau_n X$ we obtain that $\eta(a) = \eta(b)$.
This gives the desired factorization $f\circ\pi = \eta$.

Next, let us prove that the cubeset $X/\sim_n$ is $n$-step.
For this we first show that each corner in $X/\sim_n$ of dimension $k > n$ has at most one completion.
Since $X/\sim_n$ is concrete, it suffices to consider the case $k = n+1$ (otherwise, take a diagonal ${\bbox}^{n+1}\to{\bbox}^k$ mapping the corner into the corner and hitting the point $1^k$).
So take an ($n+1$)-corner $\lambda$ in $X/\sim_n$ with completions $a$ and $b$.
Take $a',b'\in X(n+1)$ with $\pi(a') = a$ and $\pi(b') = b$.
Then we obtain that $\pi(a')\rvert_{{\bbcor}^{n+1}} = \lambda = \pi(b')\rvert_{{\bbcor}^{n+1}}$.
This means that $a'(\omega) \sim_n b'(\omega)$ for all $\omega\neq 1^{n+1}$.
Using universal replacement $2^{n+1}-1$ many times, the configuration
\[
    \{0,1\}^{n+1}\to X(0),\quad\omega\mapsto\begin{cases}
        a'(\omega),\quad &\omega\neq 1^{n+1}, \\
        b'(1^{n+1}),\quad &\omega = 1^{n+1},
    \end{cases}
\]
is in $X(n+1)$.
But this means that $a'(1^{n+1}) \sim_n b'(1^{n+1})$ and so $a(1^{n+1}) = b(1^{n+1})$.
By concreteness, $a = b$.

Next we prove that the map $\pi\colon X\to X/\sim_n$ is a fibration.
\begin{center}
\begin{tikzcd}
	{{\bbcor}^k} & X \\
	{{\bbox}^k} & {X/\sim_n}
	\arrow[from=1-1, to=1-2]
	\arrow[hook', from=1-1, to=2-1]
	\arrow["\pi", from=1-2, to=2-2]
	\arrow[dashed, from=2-1, to=1-2]
	\arrow[from=2-1, to=2-2]
\end{tikzcd}
\end{center}
For $k\leq n+1$ the lift exists by universal replacement: first choose a cube $c$ in $X$ that maps to ${\bbox}^k\to X/\sim_n$
(this is possible by surjectivity of $\pi$) and then replace all points of $c$ in the corner with the points of the given corner ${\bbcor}^k\to X$,
which are equal in $X/\sim_n$.
For $k\geq n+1$ again complete the corner arbitrarily in $X$.
Uniqueness of corner completion in $X/\sim_n$ asserts that this completion pushes down to the cube we started with in $X/\sim_n$.

In particular, $X/\sim_n$ is a fibrant cubeset:
\begin{center}
\begin{tikzcd}
	\ast & X \\
	{{\bbcor}^k} & {X/\sim_n} \\
	{{\bbox}^k} & \ast
	\arrow[from=1-1, to=1-2]
	\arrow[hook', from=1-1, to=2-1]
	\arrow["\pi", from=1-2, to=2-2]
	\arrow[dashed, from=2-1, to=1-2]
	\arrow[from=2-1, to=2-2]
	\arrow[hook', from=2-1, to=3-1]
	\arrow["p", from=2-2, to=3-2]
	\arrow[dashed, from=3-1, to=1-2]
	\arrow[from=3-1, to=3-2]
\end{tikzcd}
\end{center}
since $\pi$ is an epic fibration one can lift the $k$-corner in $\tau_nX$ to $X$ (inductively completing lower dimensional corners) and then use that $X$ is fibrant to complete to a $k$-cube in $X$ which pushes down to the desired completion.

Taking the previous steps together, we obtain that $X/\sim_n$ is a concrete fibrant $n$-step cubeset.
By adjunction, there exists a unique map $g$ making the diagram
\begin{center}
\begin{tikzcd}
	X & {\tau_nX} \\
	& {X/\sim_n}
	\arrow["\eta", from=1-1, to=1-2]
	\arrow["\pi"', from=1-1, to=2-2]
	\arrow["g", dashed, from=1-2, to=2-2]
\end{tikzcd}
\end{center}
commutative.
This means that $f\circ g = \id$ since $\eta$ is the unit of the adjunction, and $g\circ f = \id$ since $\pi$ is surjective.
So we have shown that $\tau_n X = X/\sim_n$.
\end{proof}

\begin{remark}
Note that in particular this implies that the equivalence relation of the $n$-truncation on $n+1$-cubes is simply generated by identifying two $n+1$-cubes if they have the same corner,
i.e., it identifies with the coequalizer of both arrows $\hom(\bbox^{n+1}\sqcup_{\bbcor^{n+1}}\bbox^{n+1},X)\to \hom(\bbox^{n+1},X)=X(n+1)$.
\end{remark}

\begin{remark}
The previous proposition shows that the filtered category $(\cSet_n)_n$ restricts to a filtration on concrete fibrant cubesets.
\begin{center}
\begin{tikzcd}
	\cSet & \cdots & {\cSet_2} & {\cSet_1} & {\cSet_0} \\
	\ccSetnil & \cdots & {\cSet^\mathrm{conc}_{\mathrm{fib},2}} & {\cSet^\mathrm{conc}_{\mathrm{fib},1}} & {\cSet^\mathrm{conc}_{\mathrm{fib},0}}
	\arrow[shift left, from=1-1, to=1-2]
	\arrow[shift left, hook, from=1-2, to=1-1]
	\arrow["{\tau_2}", shift left, from=1-2, to=1-3]
	\arrow[shift left, hook, from=1-3, to=1-2]
	\arrow["{\tau_1}", shift left, from=1-3, to=1-4]
	\arrow[shift left, hook, from=1-4, to=1-3]
	\arrow["{\tau_0}", shift left, from=1-4, to=1-5]
	\arrow[shift left, hook, from=1-5, to=1-4]
	\arrow[hook, from=2-1, to=1-1]
	\arrow[shift left, from=2-1, to=2-2]
	\arrow[shift left, hook, from=2-2, to=2-1]
	\arrow["{\tau_2}", shift left, from=2-2, to=2-3]
	\arrow[hook, from=2-3, to=1-3]
	\arrow[shift left, hook, from=2-3, to=2-2]
	\arrow["{\tau_1}", shift left, from=2-3, to=2-4]
	\arrow[hook, from=2-4, to=1-4]
	\arrow[shift left, hook, from=2-4, to=2-3]
	\arrow["{\tau_0}", shift left, from=2-4, to=2-5]
	\arrow[hook, from=2-5, to=1-5]
	\arrow[shift left, hook, from=2-5, to=2-4]
\end{tikzcd}
\end{center}\end{remark}

\begin{example}\label{ex:truncation_HK}
The $n$-truncation of a Host-Kra cubegroup associated to a filtered group $G_\bullet$ is given by passing to the filtration quotiented by $G_{n+1}$,
i.e. taking the filtered group defined by $G_0/G_{n+1}\ge G_1/G_{n+1}\ge\cdots$.
Thus, the length of the filtration is the step of the cubeset.

In particular, the $n$-truncation of $\cD_k(A)$ is $0$ for $n< k$,
and is $\cD_k(A)$ for all $n\ge k$.
Hence, $\cD_k(A)$ is $k$-truncated.
\end{example}

\begin{remark}
There are several different and more categorical constructions that yield the truncation in the concrete fibrant setting, but need iteration and the gluing step in the general case.
As an example of such a construction, consider the two canonical maps
\[
    \ihom({\bbox}^{n+1}\sqcup_{{\bbcor}^{n+1}}{\bbox}^{n+1},X)\times{\bbox}^{n+1}\to X,
\]
which are obtained from currying the two canonical maps
\[
    \ihom({\bbox}^{n+1}\sqcup_{{\bbcor}^{n+1}}{\bbox}^{n+1},X)\to \ihom({\bbox}^{n+1},X).
\]
Taking the coequalizer
\begin{center}
\begin{tikzcd}[ampersand replacement=\&]
	{\ihom(\bbox^{n+1}\sqcup_{\bbcor^{n+1}}\bbox^{n+1},X)\times {\bbox}^{n+1}} \&\& X \&\& {Y}
	\arrow["{\phi_2}"', shift right=1, from=1-1, to=1-3]
	\arrow["{\phi_1}", shift left=1, from=1-1, to=1-3]
	\arrow[two heads, from=1-3, to=1-5]
\end{tikzcd}
\end{center}
gives a quotient $Y$.
If $Y$ is $n$-step, then it agrees with the $n$-truncation of $X$.
\end{remark}

\subsubsection{Connectivity}\label{ssec:connectivity}

We further investigate $0$-step cubesets and talk about connectivity.
Recall the adjunction $i_!\dashv i^\ast$ from Lemma~\ref{lem:point-cset}.

\begin{proposition}\label{prop:0-trunc-computation}
The adjunction $i_!\dashv i^\ast$ induces an equivalence of categories $\set\simeq\cSet_0$.
In particular, a cubeset $X$ is $0$-step if and only if the counit $i_!i^\ast X\to X$ is an isomorphism.
Under this equivalence, the $0$-truncation $\tau_0\colon\cSet\to\set$ is given by
\[
    X\mapsto\pi_0(X) \coloneq (\tau_0X)(0) = \varinjlim_{n\in{\bbox}^\op}X(n).
\]
In particular, a cubeset $X$ is $0$-step if and only if the unit $X\to i_!\pi_0(X)$ is an isomorphism.
Moreover, the diagram
\begin{center}
\begin{tikzcd}
	{X(1)} & {X(0)} & {\pi_0(X)}
	\arrow["{s^\ast}", shift left, from=1-1, to=1-2]
	\arrow["{t^\ast}"', shift right, from=1-1, to=1-2]
	\arrow["\eta", from=1-2, to=1-3]
\end{tikzcd}
\end{center}
is a coequalizer diagram, where $s,t$ are the two maps ${\bbox}^0\to{\bbox}^1$.
\end{proposition}

\begin{proof}
Since the left adjoint is fully faithful, the adjunction induces an equivalence of categories
between $\set$ and the full subcategory of $\cSet$ spanned by objects $X$ such that the
counit $i_!i^\ast X\to X$ is an isomorphism.
Thus it suffices to show that the condition on the counit is equivalent to $X$ being $0$-step.
If $X\simeq i_!i^\ast X$, then $X$ is a constant presheaf which clearly has unique corner completion
in all dimensions and so is $0$-step.
In the other direction, assume that $X$ is $0$-step.
Recall that we can write
\[
    {\bbcor}^k = \varinjlim_{j\in\cJ}{\bbox}^j
\]
where $\cJ$ runs over all inert inclusions ${\bbox}^j\to{\bbox}^k$ for $j<k$.
Thus we have isomorphisms
\[
    X(k)\simeq\varprojlim_{j\in\cJ}X(j)
\]
for all $k\geq 1$.
By induction, we infer that $X(k)\simeq X(0)$ for all $k\geq 1$.
Thus, $X\simeq i_!i^\ast X$.

For the second claim note that if $X$ is a cubeset and $Y$ is $0$-step
we have by the previous considerations that
\[
    \hom(X,Y) = \hom(X,i_!i^\ast Y) = \hom\left(\varinjlim_{n\in{\bbox}^\op} X(n),Y(0)\right)
\]
where the second equality follows from Remark~\ref{rem:left-adjoint-i_!}.
Here, the cubeset $X$ is interpreted as a diagram $X\colon{\bbox}^\op\to\set$
whose colimit in $\set$ exists.
This means that $(\tau_0X)(0) = \varinjlim_{n\in{\bbox}^\op} X(n)$.

For the last assertion, given a map $f\colon X(0)\to S$ with $f\circ s^\ast = f\circ t^\ast$
we need to show that there is a map $g_0\colon\pi_0(X)\to S$ such that $f = g_0\circ\eta$.
By the adjunction $\pi_0(X)\dashv i_!$, this is equivalent to showing that there exists a map
$g\colon X\to i_!S$ such that $g_0 = f\colon X(0)\to i_!S(0)$.
Since $i_!S$ is concrete it suffices to assume that $X$ is concrete as well.
Given $c\in X(n)$ we have that $f(c(\omega)) = f(c(\omega'))$ for all $\omega,\omega'\in\{0,1\}^n$
since $f\circ s^\ast = f\circ t^\ast$.
Thus we can just define $g_n(c) = f(c(0^n))$ which determines a map $g\colon X\to i_!S$ of cubesets with $g_0 = f$.
Uniqueness of $g$ is clear by construction.
\end{proof}

\begin{remark}
The set $\pi_0(X)$ can be computed as follows:
it is the set $X(0)$ modded out by the equivalence relation generated by
\[
    x\sim_0 y\iff\;\text{there exists}\; h\in X(1)\;\text{such that}\; s^\ast(h) = h(0) = x,\; t^\ast(h) = h(1) = y.
\]
If $X$ has the gluing property, then this relation already is an equivalence relation.

Moreover note that the coequalizer in the proposition is reflexive.
This gives another proof of the fact that the $0$-truncation, or $\pi_0$, commutes with finite products.
\end{remark}

\begin{definition}\label{def:ergodic}
Let $X$ be a cubeset.
We say that $X$ is ergodic if $\pi_0(X) = \ast$.
More generally, we say that $X$ is \emph{$n$-ergodic} if $\tau_{n-1}X = \ast$.
\end{definition}

\begin{remark}
If $X$ is concrete and fibrant, then $X$ is $n$-ergodic if and only if $X(n) = X(0)^{2^n}$.
Moreover, since the truncation preserves finite products by Proposition~\ref{prop:n-trunc-ihom-prod},
$n$-ergodic cubesets are closed under taking finite products.
\end{remark}

\begin{example}\label{ex:EM_k_ergodic}
The Eilenberg-Mac Lane cubeset $\cD_k(A)$ from Example~\ref{ex:host_kra_cubegroups} is $k$-ergodic
by Example~\ref{ex:truncation_HK}.
\end{example}

There also exists a decomposition into ergodic components.

\begin{proposition}\label{prop:ergodic-decomp}
Let $X$ be a cubeset and $\eta\colon X\to\tau_0X$ the unit of the $0$-truncation.
For $p\in\pi_0(X)$ denote by
\[
    X_p(n) = \{ c\in X(n) \mid \eta(c(\omega)) = p\;\text{for all}\;\omega\in\{0,1\}^n \}.
\]
Then $X_p$ is an ergodic sub-cubeset of $X$ and $X = \coprod_{p\in\pi_0(X)} X_p$.
\end{proposition}

\begin{proof}
It is clear that $X_p$ is a sub-cubeset of $X$ since for $c\in X_p(m)$ and $f\colon{\bbox}^n\to{\bbox}^m$
we have that $\eta(X(f)(c)(\omega)) = \eta(c(f(\omega))) = p$ for all $\omega\in\{0,1\}^n$
and thus $X(f)(c)\in X_p(n)$.
This induces a natural injective map $\coprod_{p\in\pi_0(X)} X_p\to X$.
It also is surjective: given $c\in X(n)$, to any two $\omega,\omega'\in\{0,1\}^n$
there is a map $l\colon{\bbox}^1\to{\bbox}^n$ with $l(0) = \omega$, $l(1) = \omega'$.
Then $c\circ l\in X(1)$ and by Proposition~\ref{prop:0-trunc-computation} we have that
\[
    \eta(c(\omega)) = \eta(c(l(0))) = \eta(c(l(1))) = \eta(c(\omega')).
\]
Lastly, $X_p$ is ergodic: given $x,y\in X_p(0)$, since $\eta(x) = \eta(y) = p$
we have that $x\sim_0 y$ in $X$.
This means that there is a finite sequence $h_1,\ldots,h_k\in X(1)$
such that $h_1(0) = x$, $h_i(1) = h_{i+1}(0)$ for $1\leq i\leq k-1$ and $h_k(1) = y$.
But then we have that $\eta(h_i(0)) = p = \eta(h_i(1))$ for all $i$,
and so $h_i\in X_p(1)$.
Hence $x\sim_0 y$ also in $X_p$ which means that $X_p$ is ergodic.
\end{proof}

We call every $X_p$ an \emph{ergodic component} of $X$
and the decomposition of $X$ into the coproduct $\coprod_{p\in\pi_0(X)}X_p$ the
\emph{decomposition in ergodic components}.

Let us also comment on connectivity of $\bbGamma$-sets.
For this we need the following propositions, characterizing connected $\bbGamma$-sets.
First, one can decompose $\bbGamma$-sets as a coproduct of the fibers over its points.

\begin{proposition}\label{prop:connected-comp-gamma}
Let $X$ be a $\bbGamma$-set, let $p\in X(0)$ and for $\langle n\rangle\in\bbGamma$ define
\[
    X_p(n) = \{ c\in X(n) \mid c(0) = p \}.
\]
Then $X_p$ is a sub-$\bbGamma$-set of $X$ and $X = \coprod_{p\in X(0)} X_p$.
This is the decomposition of $X$ into its connected components.
\end{proposition}

\begin{proof}
Since there is a unique map $u\colon\ast\to\langle n\rangle$, $X_p$ is a $\bbGamma$-set, and it is clear that $X_p\sub X$.
It is clear that $X_p$ as a $\bbGamma$-set is connected since every $\bbGamma$-set with $X(0) = \ast$ is connected.
Lastly, $X(n) = \coprod_{p\in X(0)} X_p(n)$ is clear by definition, and so $X = \coprod_{p\in X(0)} X_p$.
\end{proof}

\begin{definition}
Let $\cC$ be a category.
An object $C\in\cC$ is called \emph{connected} if $\hom(C,-)\colon\cC\to\set$ preserves coproducts.
\end{definition}

\begin{proposition}\label{prop:connected-gamma-prop}
Let $X$ be a $\bbGamma$-set.
Then the following conditions are equivalent.
\begin{enumerate}[(a)]
    \item $X$ is connected.
    \item For every decomposition $X = X_1\sqcup X_2$ we have that $X_1$ is empty or equal to $X$.
    \item $X(0)$ is a singleton.
\end{enumerate}
\end{proposition}

\begin{proof}
$(a)\implies(b)$: Let $X = X_1\sqcup X_2$.
Then the identity on $X$ has to factor w.l.o.g. as $X\to X_1\to X$.
Therefore, $X_1\to X$ is an epimorphism, and hence an isomorphism.

$(b)\implies(c)$: Follows from Proposition~\ref{prop:connected-comp-gamma}.

$(c)\implies(a)$: If $X(0)$ is a point, then for any map $X\to\coprod_i Y_i$
its underlying point factors through one $Y_i$,
and so does every map $X(n)\to\coprod_i Y_i(n)$.
\end{proof}

Therefore, we have a decomposition of the $\bbGamma$-set $X$ into its connected components:
it is a decomposition $X = \coprod_{i\in I} X_i$ where $X_i$ are connected and $X_i\sub X$ are sub-$\bbGamma$-sets.

\subsubsection{Truncations over \texorpdfstring{$\bbGamma$}{Gamma}}\label{ssec:truncation-gamma}

Note that of course one can also define truncations over $\bbGamma$ instead of $\bbox$.
Here the correct set of morphisms to localize at is the refined $k$-corners for $k>n$ from Section~\ref{ssec:corner-gamma}.
Recall the set
\[
    \Lambda_{\bbGamma} \coloneq \{{\bbcor}^n_\omega\hookrightarrow{\bblox}^n \mid n\in\N,\,\omega\in\{0,1\}^n\setminus\{0^n\}\}
\]
from Definition~\ref{def:corner_in_gamma} and~\ref{def:omega_corner_gamma}.
This yields the variant
\[
    \Lambda_{\bbGamma,n} \coloneq \{{\bbcor}^k_\omega\hookrightarrow{\bblox}^k \mid k\geq n+1,\,\omega\in\{0,1\}^k\setminus\{0^k\}\}.
\]
of maps in $\Lambda_{\bbGamma}$ between objects of dimension $>n$.
Then we can define the truncations analogously to $\cSet$.

\begin{definition}
A $\bbGamma$-set $X$ is called $n$-step if it is local w.r.t. $\Lambda_{\bbGamma,n}$.
We denote by $\PSh(\bbGamma)_n$ the full subcategory of $n$-step $\bbGamma$-sets
and by $\tau_n^{\bbGamma}\colon\PSh(\bbGamma)\to\PSh(\bbGamma)_n$ the localization onto $\Lambda_{\bbGamma,n}$-local objects.
\end{definition}

\begin{remark}
Essentially all the same properties of the truncation that work with cubesets hold over $\bbGamma$ as well,
e.g., the subcategory of $n$-step $\bbGamma$-sets is closed under limits, filtered colimits and coproducts.
\end{remark}

Again, one can compute the truncation similarly to cubesets.
However, note that the truncation does not affect the value $X(0)$ at all:
the $\bbGamma$-set decomposes into a coproduct of its $X(0)$-values, see Proposition~\ref{prop:connected-comp-gamma},
and so it suffices to see that $\tau_n X(0) = \ast$ whenever $X(0) = \ast$.
But this follows from the explicit construction of the reflection $\tau_n X$ as a filtered colimit over objects which arise as connected colimit of objects $Y$ that all have $Y(0) = \ast$,
see \cite[Construction 1.37]{Adamek1994}.

But other than that, for a $\bbGamma$-set arising from a concrete fibrant cubeset, computing the truncation over $\bbGamma$ is analogous to computing it over $\bbox$.
This will need the following Definition.

\begin{definition}
Let $X$ be a concrete fibrant cubeset.
For $n=0$ we define $x\approx_n y$ for $x,y\in X(1)$ if $x(0) = y(0)$.
For $n>0$ we define $x\approx_n y$ for $x,y\in X(1)$ if there exist $\omega\in\{0,1\}^{n+1}\setminus\{0^{n+1}\}$ and $c,c'\in X(n+1)$ such that
\[
    c\rvert_{j_!{\bbcor}^{n+1}_\omega} = c'\rvert_{j_!{\bbcor}^{n+1}_\omega} \quad\text{and}\quad c(0,\omega) = x,\; c'(0,\omega) = y.
\]
\end{definition}

Here $c(0,\omega)$ denotes the composition $c\circ t$ where $t\colon{\bblox}^1\to{\bblox}^{n+1}$ is the unique linear map with $t(1) = \omega$.

\begin{proposition}
Let $X$ be a concrete fibrant cubeset, $n\in\N$ and $x,y\in X(1)$.
Then the following assertions are equivalent.
\begin{enumerate}[(a)]
    \item $x\approx_n y$.
    \item $x(0) = y(0)$ and $x(1)\sim_n y(1)$.
\end{enumerate}
In particular, the relation $\approx_n$ is an equivalence relation on $X(1)$ for all $n\in\N_0$.
\end{proposition}

\begin{proof}
$(a)\implies (b)$:
If $x\approx_n y$, there exist $\omega\in\{0,1\}^{n+1}\setminus\{0^{n+1}\}$ and $c,c'\in X(n+1)$ such that
\[
    c\rvert_{j_!{\bbcor}^{n+1}_\omega} = c'\rvert_{j_!{\bbcor}^{n+1}_\omega} \quad\text{and}\quad c(0,\omega) = x,\; c'(0,\omega) = y.
\]
In particular, $x(0) = y(0)$.
Moreover, we have seen in Lemma~\ref{lem:alternative-nilfibrant-concrete}
that $X$ is local w.r.t. $j_!{\bbcor}^k_\omega\to{\bbcor}^k_\omega$.
This means that $c\rvert_{{\bbcor}^{n+1}_\omega} = c'\rvert_{{\bbcor}^{n+1}_\omega}$.
Now take an automorphism of cubes exchanging $\omega$ and $1^{n+1}$ (using the appropriate flips).
This implies that $c\rvert_{{\bbcor}^{n+1}_{1^{n+1}}} = c'\rvert_{{\bbcor}^{n+1}_{1^{n+1}}}$ and in particular since ${\bbcor}^{n+1}\sub{\bbcor}^{n+1}_{1^{n+1}}$,
$c\rvert_{{\bbcor}^{n+1}} = c'\rvert_{{\bbcor}^{n+1}}$ and $c(1^{n+1}) = x(1)$, $c'(1^{n+1}) = y(1)$.
This means that $x(1)\sim_n y(1)$.

$(b)\implies (a)$:
If $x(1)\sim_n y(1)$ take the (linear!) degeneracy $p\colon{\bbox}^{n+1}\to{\bbox}^1$ along the $x_1$-coordinate axis.
Then $c = x\circ p$ is an ($n+1$)-cube in $X$ with $c(1^{n+1}) = x(1)$.
Using universal replacement, we see that the configuration
\[
    c'\colon\{0,1\}^{n+1}\to X(0),\quad\omega\mapsto\begin{cases}
        c(\omega),\quad &\omega\neq 1^{n+1}, \\
        y(1),\quad &\omega = 1^{n+1},
    \end{cases}
\]
is in $X(n+1)$.
Thus we have that
\[
    c\rvert_{j_!{\bbcor}^{n+1}_{1^{n+1}}} = c'\rvert_{j_!{\bbcor}^{n+1}_{1^{n+1}}} \quad\text{and}\quad c(0,\omega) = x,\; c'(0,\omega) = y.\qedhere
\]
\end{proof}

This equivalence relation yields a $\bbGamma$-set $X/\approx_n$ as follows:
for $m=0$ we set $X/\approx_n(0) = X(0)$ and for $m>0$ we define $X/\approx_n(m)$ to be the quotient of $X(m)$ that identifies two $m$-cubes $a,b$
if for all inclusions $l\colon{\bblox}^1\to{\bblox}^m$ we have that $a\circ l \approx_n b\circ l$.

The previous proposition gives us an easy description of $X/\approx_n$:
two $m$-cubes $a$ and $b$ are identified if and only if $a(0) = b(0)$ and $a(\omega) \sim_n b(\omega)$ for all $\omega\in\{0,1\}^m\setminus\{0^m\}$.
Our goal now is to show that this quotient $X/\approx_n$ actually is the truncation over $\bbGamma$.

\begin{proposition}\label{prop:trunc-concfib-gamma}
Let $X$ be a concrete fibrant cubeset and $n\in\N_0$.
Then the $\bbGamma$-truncation of degree $n$ of $j^\ast X$ is $X/\approx_n$ and the unit $j^\ast X\to\tau^{\bbGamma}_n j^\ast X$ is a $\bbGamma$-fibration.
\end{proposition}

\begin{proof}
The proof works essentially the same as the proof of Proposition~\ref{prop:trunc_concfib}.
In a first step, we show existence of a map making the diagram
\begin{center}
\begin{tikzcd}
	{j^\ast X} && {\tau^{\bbGamma}_n j^\ast X} \\
	& {X/\approx_n}
	\arrow["\eta", from=1-1, to=1-3]
	\arrow["\pi"', from=1-1, to=2-2]
	\arrow["f"', dashed, from=2-2, to=1-3]
\end{tikzcd}
\end{center}
commutative.
We compute $f$ sectionwise, and $k = 0$ is clear since $\tau^{\bbGamma}_n j^\ast X(0) = X(0) = X/\approx_n(0)$.
For $1\leq k\leq n+1$ we need to see that if $a,b$ are two $k$-cubes in $X$ with $a\approx_n b$, then $\eta(a) = \eta(b)$.
Take a linear surjection (a degeneracy) $p\colon{\bbox}^{n+1}\to{\bbox}^k$.
Then $a' = p^\ast(a) \approx_n p^\ast(b) = b'$.
Using universal replacement we can, one by one, exchange each point of $a'$, which is not at position $0$, by a point of $b'$ at the same position.
But instead of rotating $a'$, we now use unique corner completion of $\tau_n^{\bbGamma} j^\ast X$ against every $\omega$-corner ${\bbcor}^{n+1}_\omega$
to conclude that $\eta(a') = \eta(c_k)$ for each $c_k$ after a finite replacement, and finally, $\eta(a) = \eta(b)$.
For $k\geq n+2$ the same induction as in the proof of Proposition~\ref{prop:trunc_concfib} yields the claim.

The rest of the proof works verbatim as in the proof of Proposition~\ref{prop:trunc_concfib}
except that at $\omega = 0$, equality has to occur in the equivalence relation, and one has to consider all corners ${\bbcor}^k_\omega$.

For $n=0$ note that the Hom-isomorphism of the adjunction holds since both sides agree with $Y(0)^{X(0)}$.
\end{proof}

\subsection{Homotopy Groups}\label{ssec:structure-group}

In this section we define algebraic invariants of cubesets: their structure groups,
also known as homotopy groups in analogy to simplicial sets.

\begin{definition}\label{def:structure-group}
Let $X$ be an $n$-step cubeset, $x_0\in X(0)$ a point and $n\geq 1$.
Then define $\pi_n(X,x_0)$ to be the set
\[
    \{ c\colon{\bbox}^n\to X \mid c\lvert_{{\bbcor}^n} \equiv x_0 \}.
\]
We define a binary operation on $\pi_n(X,x_0)$ as follows.
Let $a,b\in\pi_n(X,x_0)$ and set $\lambda_i\colon{\bbox}^n\to X$ to be $\lambda_1 = a$, $\lambda_2 = b$ and $\lambda_i \equiv x_0$ for $2<i\le n+1$.
This defines a corner $\lambda_{a,b}\colon{\bbcor}^{n+1}\to X$ which admits a unique completion $\overline{a*b}\colon{\bbox}^{n+1}\to X$ to a cube in $X$.
Let $s\colon{\bbox}^n\to{\bbox}^{n+1}$ be the diagonal $(x_1,\ldots,x_n)\mapsto (x_1,x_1,x_2,\ldots,x_n)$.
Then define
\[
    a * b = \overline{a*b}\circ s \colon {\bbox}^n\to X.
\]
With this operation, we call $\pi_n(X,x_0)$ the \emph{$n$-th structure monoid of $X$.}
\end{definition}

Let us illustrate this definition in the cases $n=1$ and $n=2$
where the resulting composition $a*b$ is colored red.
\begin{center}
\begin{tikzcd}[column sep=small, row sep= small]
	&&&&& b && {a*b} \\
	b && {a*b} && {x_0} && a \\
	&&&&& {x_0} && {x_0} \\
	{x_0} && a && {x_0} && {x_0}
	\arrow[no head, from=1-6, to=1-8]
	\arrow[dashed, no head, from=1-6, to=3-6]
	\arrow[no head, from=2-1, to=2-3]
	\arrow[no head, from=2-5, to=1-6]
	\arrow[color={rgb,255:red,214;green,92;blue,92}, no head, from=2-5, to=1-8]
	\arrow[no head, from=2-5, to=2-7]
	\arrow[color={rgb,255:red,214;green,92;blue,92}, no head, from=2-5, to=4-5]
	\arrow[no head, from=2-7, to=1-8]
	\arrow[dashed, no head, from=3-6, to=3-8]
	\arrow[color={rgb,255:red,214;green,92;blue,92}, no head, from=3-8, to=1-8]
	\arrow[no head, from=4-1, to=2-1]
	\arrow[color={rgb,255:red,214;green,92;blue,92}, no head, from=4-1, to=2-3]
	\arrow[no head, from=4-1, to=4-3]
	\arrow[no head, from=4-3, to=2-3]
	\arrow[dashed, no head, from=4-5, to=3-6]
	\arrow[color={rgb,255:red,214;green,92;blue,92}, dashed, no head, from=4-5, to=3-8]
	\arrow[no head, from=4-5, to=4-7]
	\arrow[no head, from=4-7, to=2-7]
	\arrow[no head, from=4-7, to=3-8]
\end{tikzcd}
\end{center}

\begin{remark}
A priori, the choice of $s$ for extracting an $n$-cube out of the completed ($n+1$)-cube
seems arbitrary.
However, it essentially is the only choice:
given two inerts $\epsilon_i,\epsilon_j\colon{\bbox}^n\to{\bbox}^{n+1}$ with $i\neq j$,
there is exactly one subobject of ${\bbox}^{n+1}$ that is the complement of the symmetric difference of $\epsilon_i$ and $\epsilon_j$.
This is due to the fact that the complement of the symmetric difference of $\epsilon_i$ and $\epsilon_j$
is equal to the subobject $[\omega_i=\omega_j=0]\cup[\omega_i=\omega_j=1]$ of dimension $n$,
since it consists of $2\cdot 2^{n-1} = 2^n$ points.
We also see that it doesn't matter in the definition of the operation that we impose $a$ and $b$
on the sides $\epsilon_1$ and $\epsilon_2$, it could have also been any other two distinct lower faces.%
\footnote{In fact, since flips and coordinate permutations induce automorphisms on the set of $n+1$-cubes, making the same definition after some automorphism of $\bbox^{n+1}$ leads to the same composition on concrete cubesets.}
\end{remark}

The next proposition justifies the definition of structure \emph{monoid}.

\begin{proposition}\label{prop:structure-group-is-monoid}
The operation
\[
    *\colon\pi_n(X,x_0)\times\pi_n(X,x_0)\to\pi_n(X,x_0),\quad (a,b)\mapsto a*b
\]
makes $\pi_n(X,x_0)$ an abelian monoid.
This construction upgrades to a functor $\cSet_{\ast,n}\to\cMon$.
\end{proposition}

\begin{proof}
\textbf{Commutativity:}
Let $a,b\in\pi_n(X,x_0)$.
In order to show that $a*b = b*a$ we want to see that
one can act on the ($n+1$)-cube witnessing the composition by an automorphism that interchanges $a$ with $b$ but leaves the rest of the cube, and in particular the side inducing the composition, invariant.
This can be described for $n=1$ by the following diagram
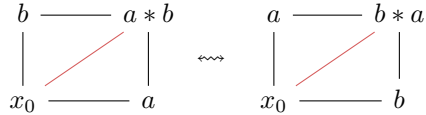
\begin{center}
\begin{tikzcd}
	b & {a*b} & a & {b*a} \\
	{x_0} & a & {x_0} & b
	\arrow[no head, from=1-1, to=1-2]
	\arrow[no head, from=1-3, to=1-4]
	\arrow[no head, from=2-1, to=1-1]
	\arrow[color={rgb,255:red,214;green,92;blue,92}, no head, from=2-1, to=1-2]
	\arrow[no head, from=2-1, to=2-2]
	\arrow[""{name=0, anchor=center, inner sep=0}, no head, from=2-2, to=1-2]
	\arrow[""{name=1, anchor=center, inner sep=0}, no head, from=2-3, to=1-3]
	\arrow[color={rgb,255:red,214;green,92;blue,92}, no head, from=2-3, to=1-4]
	\arrow[no head, from=2-3, to=2-4]
	\arrow[no head, from=2-4, to=1-4]
	\arrow["\leftrightsquigarrow"{description}, draw=none, from=0, to=1]
\end{tikzcd}
\end{center}
where the $2$-cubes relate by switching the coordinates $x_1$ and $x_2$.
This argument works in every dimension and is formalized as follows.

For $k\geq 2$ denote by $\sigma_{12}^k\in\Aut({\bbox}^{k})$ the automorphism exchanging coordinates $x_1$ and $x_2$.
Note that $\sigma_{12}^k$ maps the corner onto the corner,
i.e., restricts to an automorphism $\sigma_{12}^k\rvert\colon{\bbcor}^k\to{\bbcor}^k$.
We first show that $\overline{b*a}\circ\sigma_{12}^{n+1} = \overline{a*b}$.
Since $X$ is $n$-step it suffices to see that $\lambda_{b,a}\circ\sigma_{12}^{n+1}\rvert = \lambda_{a,b}$.
But this can be tested against inerts $\epsilon_i\colon{\bbox}^n\to{\bbcor}^{n+1}$.
There we have that $\sigma_{12}^{n+1}\rvert\circ\epsilon_i = \epsilon_i\circ\sigma_{12}^n$ for $i\geq 3$
and $\sigma_{12}^{n+1}\rvert\circ\epsilon_1 = \epsilon_2$, $\sigma_{12}^{n+1}\rvert\circ\epsilon_2 = \epsilon_1$
since $\sigma_{12}^{n+1}$ can be decomposed as
\[
    \sigma_{12}^2\times\id\colon{\bbox}^2\times{\bbox}^{n-1}\to{\bbox}^2\times{\bbox}^{n-1}.
\]
This shows the desired equality for $i = 1,2$.
For $i\geq 3$ we have that
\[
      \lambda_{b,a}\circ\sigma_{12}^{n+1}\rvert\circ\epsilon_i
    = \lambda_{b,a}\circ\epsilon_i\circ\sigma_{12}^n
    = x_0\circ\sigma_{12}^n
    = x_0
    = \lambda_{a,b}\circ\epsilon_i,
\]
verifying the claim there as well.
From $\sigma_{12}^{n+1}\circ s = s$ we conclude that
\[
    a*b = \overline{a*b}\circ s = \overline{b*a}\circ\sigma_{12}^{n+1}\circ s = b*a.
\]

\textbf{Identity element:}
Clearly, the constant $n$-cube $x_0$ is the obvious candidate to be the neutral element.
To see that it actually is, consider the case $n=1$:
the $2$-cube
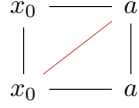
\begin{center}
\begin{tikzcd}
	{x_0} & a \\
	{x_0} & a
	\arrow[no head, from=1-1, to=1-2]
	\arrow[no head, from=2-1, to=1-1]
	\arrow[color={rgb,255:red,214;green,92;blue,92}, no head, from=2-1, to=1-2]
	\arrow[no head, from=2-1, to=2-2]
	\arrow[no head, from=2-2, to=1-2]
\end{tikzcd}
\end{center}
is a cube in $X$: it is the degeneration of $a$ into an additional dimension.
By definition of the operation, $a*x_0 = a$.

In general, let $d\colon{\bbox}^{n+1}\to{\bbox}^n$ denote the degeneracy omitting the first coordinate,
which is left inverse to both the inert $\epsilon_1$ and to the diagonal $s$.
Take $a\in\pi_n(X,x_0)$ and consider $a\circ d\colon{\bbox}^{n+1}\to X$.
Then $a\circ d\circ\epsilon_1 = a$.
For $i\geq 2$ we have that $d\circ\epsilon_i\colon{\bbox}^n\to{\bbox}^n$ factors through an inert face map, and hence $a\circ d\circ\epsilon_i \equiv x_0$.
Thus, $a\circ d = \overline{a*x_0}$ and
\[
    a*x_0 = \overline{a*x_0}\circ s = a\circ d\circ s = a.
\]
By commutativity, we also have $x_0*a = a$.
Thus, $x_0$ is the identity element of the monoid.

\textbf{Associativity:} Take elements $a,b,c\in\pi_n(X,x_0)$ and the ($n+1$)-cubes $\overline{a*b}$, etc.,
that witness their composition.
Together, they yield a corner of an ($n+2$)-cube.
We use unique corner completion in dimension $n+2$ to obtain a point $a*b*c$
witnessing $(a*b)*c$ and $a*(b*c)$ as $n$-dimensional subcubes,
so that they have to be the same.
The situation is summarized for $n=1$ by the following image,
where $\overline{(a*b)*c}$ is colored red, $\overline{a*(b*c)}$ blue and $(a*b)*c = a*(b*c)$ in yellow.
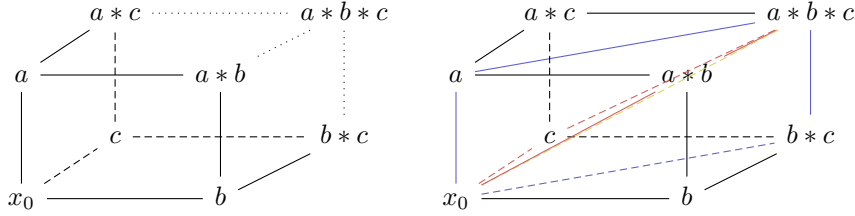
\begin{center}
\begin{tikzcd}[column sep=small, row sep=small]
	& {a*c} && a*b*c && {a*c} &&& a*b*c \\
	a && {a*b} && a &&& {a*b} \\
	& c && {b*c} && c &&& {b*c} \\
	{x_0} && b && {x_0} &&& b
	\arrow[dotted, no head, from=1-2, to=1-4]
	\arrow[no head, from=1-6, to=1-9]
	\arrow[no head, from=2-1, to=1-2]
	\arrow[no head, from=2-1, to=2-3]
	\arrow[no head, from=2-1, to=4-1]
	\arrow[dotted, no head, from=2-3, to=1-4]
	\arrow[no head, from=2-5, to=1-6]
	\arrow[color={rgb,255:red,92;green,92;blue,214}, no head, from=2-5, to=1-9]
	\arrow[no head, from=2-5, to=2-8]
	\arrow[color={rgb,255:red,214;green,92;blue,92}, no head, from=2-8, to=1-9]
	\arrow[dashed, no head, from=3-2, to=1-2]
	\arrow[dashed, no head, from=3-2, to=3-4]
	\arrow[dotted, no head, from=3-4, to=1-4]
	\arrow[dashed, no head, from=3-6, to=1-6]
	\arrow[color={rgb,255:red,214;green,92;blue,92}, dashed, no head, from=3-6, to=1-9]
	\arrow[dashed, no head, from=3-6, to=3-9]
	\arrow[color={rgb,255:red,92;green,92;blue,214}, no head, from=3-9, to=1-9]
	\arrow[dashed, no head, from=4-1, to=3-2]
	\arrow[no head, from=4-1, to=4-3]
	\arrow[no head, from=4-3, to=2-3]
	\arrow[no head, from=4-3, to=3-4]
	\arrow[color={rgb,255:red,214;green,214;blue,92}, dashed, no head, from=4-5, to=1-9]
	\arrow[color={rgb,255:red,92;green,92;blue,214}, no head, from=4-5, to=2-5]
	\arrow[color={rgb,255:red,214;green,92;blue,92}, no head, from=4-5, to=2-8]
	\arrow[color={rgb,255:red,214;green,92;blue,92}, dashed, no head, from=4-5, to=3-6]
	\arrow[color={rgb,255:red,92;green,92;blue,214}, dashed, no head, from=4-5, to=3-9]
	\arrow[no head, from=4-5, to=4-8]
	\arrow[no head, from=4-8, to=2-8]
	\arrow[no head, from=4-8, to=3-9]
\end{tikzcd}
\end{center}
Now to the formal argument:
denote $\lambda_1 = \overline{a*b}$, $\lambda_2 = \overline{a*c}$ and $\lambda_3 = \overline{b*c}$ the corresponding ($n+1$)-cubes
to the compositions $a*b$, $a*c$ and $b*c$, respectively (so that $\overline{a*b}\circ s = a*b$, etc.).
Let $\lambda_i \equiv x_0$ for $4\leq i\leq n+2$.
That these ($n+1$)-cubes $\lambda_i$ yield a corner $\lambda\colon{\bbcor}^{n+2}\to X$
can be verified by a direct computation using Remark~\ref{rem:corner-combinatorial}.
Now take the completion $d\colon{\bbox}^{n+2}\to X$ of $\lambda$.
We define the maps of cubes
\begin{align*}
    r\colon{\bbox}^{n+1}\to{\bbox}^{n+2},\quad &(x_1,\ldots,x_{n+1})\mapsto(x_1,x_2,x_2,x_3,\ldots,x_{n+1}) \\
    b\colon{\bbox}^{n+1}\to{\bbox}^{n+2},\quad &(x_1,\ldots,x_{n+1})\mapsto(x_1,x_1,x_2,\ldots,x_{n+1})
\end{align*}
which ought to extract $\overline{(a*b)*c}$ and $\overline{a*(b*c)}$, respectively, out of $d$.%
\footnote{As depicted in red ($r$) and blue ($b$) for $n=1$. If you can't see the colors, as, e.g., the second author, it doesn't matter, really.}
To see that $d\circ r = \overline{(a*b)*c}$ we need to verify that both cubes agree on their corner.
For $i=1$ we have that $r\circ\epsilon_1^n = \epsilon_1^{n+1}\circ s$,
so
\[
    d\circ r\circ\epsilon_1^n = d\circ\epsilon_1^{n+1}\circ s = \overline{a*b}\circ s = a*b.
\]
For $i=2$ we calculate $r\circ\epsilon_2^n = \epsilon_2^{n+1}\circ\epsilon_2^n$,
hence
\[
    d\circ r\circ\epsilon_2^n = d\circ\epsilon_2^{n+1}\circ\epsilon_2^n = \overline{a*c}\circ\epsilon_2^n = c.
\]
For $j\geq 3$ we obtain $r\circ\epsilon_j^n = \epsilon_{j+1}^{n+1}\circ r'$ for some (linear) $r'\colon{\bbox}^n\to{\bbox}^{n+1}$ and therefore,
\[
    d\circ r\circ\epsilon_j^n = d\circ\epsilon_{j+1}^{n+1}\circ r' = x_0\circ r' = x_0.
\]
This verifies the claim.
The analogous statement for $b$ is verified using similar arguments.
From the fact that $r\circ s = b\circ s$ we finally obtain
\[
    (a*b)*c = \overline{(a*b)*c}\circ s = d\circ r\circ s = d\circ b\circ s = a*(b*c)
\]
and so the operation is associative.

\textbf{Functoriality:}
If $f\colon(X,x_0)\to(Y,y_0)$ is a map of pointed $n$-step cubesets,
it preserves corners and the basepoint and so induces a map $f_\ast\colon\pi_n(X,x_0)\to\pi_n(Y,y_0)$.
Moreover, it is a monoid homomorphism since $f$ is a natural transformation and preserves corners.
\end{proof}

\begin{remark}\label{rem:structure-mon-not-grouplike}
In general, even for $n$-step $X$, the monoid $\pi_n(X,x_0)$ does in general not have inverses.
One would like to define, for $a\in\pi_n(X,x_0)$, an inverse $a^{-1}$ as follows (in case $n=1$):
we search for an element $a^{-1}\in\pi_1(X,x_0)$ such that the 2-cube on the left is in $X$,
and the red diagonal is the constant $x_0$.
\begin{center}
\begin{tikzcd}
	{a^{-1}} & {x_0} & {a^{-1}} & {x_0} \\
	{x_0} & a & {x_0} & a
	\arrow[no head, from=1-1, to=1-2]
	\arrow[dotted, no head, from=1-4, to=1-3]
	\arrow[no head, from=2-1, to=1-1]
	\arrow[color={rgb,255:red,214;green,92;blue,92}, no head, from=2-1, to=1-2]
	\arrow[no head, from=2-1, to=2-2]
	\arrow[no head, from=2-2, to=1-2]
	\arrow[dotted, no head, from=2-3, to=1-3]
	\arrow[no head, from=2-3, to=2-4]
	\arrow[no head, from=2-4, to=1-4]
\end{tikzcd}
\end{center}
So instead of completion at $x_0$, one flips the cube so that one can now use corner completion at $a^{-1}$.
The problem now is that the (anti-)diagonal connecting $x_0$ with $x_0$ need not be the constant $x_0$
(arising from the degeneracy).
This can only be guaranteed if $X$ has unique gluing property.
Then $X$ is concrete and we will later see that $\pi_n(X,x_0)$ becomes a group in that case.
\end{remark}

\begin{definition}\label{def:structure-monoid-general}
For a general pointed cubeset $(X,x_0)$, the $n$-th structure (homotopy) monoid is given by $\pi_n(\tau_n X,x_0)$.
For $n=0$ let $\pi_0(X) = \tau_0 X(0)$ be the set of \emph{ergodic components},
see Proposition~\ref{prop:0-trunc-computation}.
\end{definition}

This definition determines functors $\pi_0\colon\cSet\to\set$
and $\pi_n\colon\cSet_\ast\to\cMon$ for $n\geq 1$.
Since $\pi_0$ is left adjoint, it preserves colimits.

\begin{example}\label{ex:HK_homotopy_groups}
Discrete cubesets $i_\ast S$ have trivial structure monoids, since they have trivial truncations.

The structure monoids of the Host-Kra cubegroup $\HK(G_\bullet)$ are given by
$\pi_n(G,1)=G_n/G_{n+1}$ which in this case are actually groups.
In particular, the structure groups of $\cD_k(A)$ are
\[
    \pi_n(A,0) = \begin{cases}
        A,\quad & n = k, \\
        0,\quad & n \neq k.
    \end{cases}
\]
\end{example}

\begin{proposition}\label{prop:structure-monoid-prod}
Let $(X,x_0)$ and $(Y,y_0)$ be pointed cubesets and $n\in\N$.
Then the canonical monoid homomorphism
\[
    \pi_n(X\times Y,(x_0,y_0))\to\pi_n(X,x_0)\times\pi_n(Y,y_0)
\]
is an isomorphism.
\end{proposition}

\begin{proof}
Since truncation commutes with finite products by Proposition~\ref{prop:n-trunc-ihom-prod},
we can assume that $X$ and $Y$ are $n$-step.
But in this case, $\pi_n(X,x_0)$ sits in the pullback square
\begin{center}
\begin{tikzcd}
	{\pi_n(X,x_0)} & \ast \\
	{\hom({\bbox}^n,X)} & {\hom({\bbcor}^n,X)}
	\arrow[from=1-1, to=1-2]
	\arrow[from=1-1, to=2-1]
	\arrow["{x_0}", from=1-2, to=2-2]
	\arrow[from=2-1, to=2-2]
\end{tikzcd}.
\end{center}
Thus, the claim follows from the continuity of the $\hom$-functor and the commutation of limits with limits.
\end{proof}

In the case of concrete fibrant cubesets where one has an explicit description of the truncation $\tau_n$ it is possible to give another definition of structure monoid.
For this, let $X$ be a concrete fibrant cubeset with basepoint $x_0$.
As before, we would like to define $\pi_n(X,x_0)$ to be the set of all $n$-cubes in $X$ whose corner is $x_0$.
Then multiplication should be defined in the same way:
given $a,b\in\pi_n(X,x_0)$, their product $a*b$ is given by the diagonal $s$ of the completion
\begin{center}
\begin{tikzcd}
	b && b & {a*b} & b & {a*b} \\
	{x_0} & a & {x_0} & a & {x_0} & a
	\arrow[dotted, no head, from=1-3, to=1-4]
	\arrow[no head, from=1-5, to=1-6]
	\arrow[""{name=0, anchor=center, inner sep=0}, no head, from=2-1, to=1-1]
	\arrow[no head, from=2-1, to=2-2]
	\arrow[""{name=1, anchor=center, inner sep=0}, no head, from=2-3, to=1-3]
	\arrow[no head, from=2-3, to=2-4]
	\arrow[""{name=2, anchor=center, inner sep=0}, dotted, no head, from=2-4, to=1-4]
	\arrow[""{name=3, anchor=center, inner sep=0}, no head, from=2-5, to=1-5]
	\arrow[color={rgb,255:red,214;green,92;blue,92}, no head, from=2-5, to=1-6]
	\arrow[no head, from=2-5, to=2-6]
	\arrow[no head, from=2-6, to=1-6]
	\arrow["\rightsquigarrow"{marking, allow upside down, pos=0.75}, draw=none, from=0, to=1]
	\arrow["\rightsquigarrow"{marking, allow upside down}, draw=none, from=2, to=3]
\end{tikzcd}
\end{center}
But there's a catch: if $X$ is not $1$-step, the completion is not unique!
So we need to build a quotient of $\pi_n(X,x_0)$ making multiplication well-defined.
In the case sketched for $n=1$ this would mean that given two completions $c$ and $d$ of the same corner
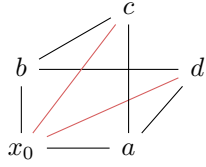
\begin{center}
\begin{tikzcd}[column sep=small, row sep=small]
	&& c & \\
	b &&& d \\
	\\
	{x_0} && a
	\arrow[no head, from=2-1, to=1-3]
	\arrow[no head, from=2-1, to=2-4]
	\arrow[color={rgb,255:red,214;green,92;blue,92}, no head, from=4-1, to=1-3]
	\arrow[no head, from=4-1, to=2-1]
	\arrow[color={rgb,255:red,214;green,92;blue,92}, no head, from=4-1, to=2-4]
	\arrow[no head, from=4-1, to=4-3]
	\arrow[no head, from=4-3, to=1-3]
	\arrow[no head, from=4-3, to=2-4]
\end{tikzcd}
\end{center}
we would need to identify $c$ and $d$ based at $x_0$ (the two red lines).
This is precisely what we will do.

\begin{construction}\label{cons:structure-monoid-concrete}
Let $X$ be a concrete fibrant cubeset with point $x_0\in X(0)$.
Then define $\bbcor_n(X,x_0)$ to be the pullback
\begin{center}
\begin{tikzcd}
	{\bbcor_n(X,x_0)} & \ast \\
	{\hom({\bbox}^n,X)} & {\hom({\bbcor}^n,X)}
	\arrow[from=1-1, to=1-2]
	\arrow[from=1-1, to=2-1]
	\arrow["{x_0}", from=1-2, to=2-2]
	\arrow[from=2-1, to=2-2]
\end{tikzcd}
\end{center}
which is the set of all $n$-cubes in $X$ that have value $x_0$ on their corner.
Now we want to identify two $n$-cubes in $X$ if there is an $(n+1)$-homotopy connecting the two cubes.
Let $s\colon{\bbox}^n\to{\bbox}^{n+1}$ be the map $(x_1,\ldots,x_n)\mapsto (x_1,x_1,x_2,\ldots,x_n)$.
Since $s$ maps the $n$-corner into the ($n+1$)-corner,
it induces an inclusion
\[
    i\colon{\bbox}^n\sqcup_{{\bbcor}^n}{\bbox}^n\to{\bbox}^{n+1}\sqcup_{{\bbcor}^{n+1}}{\bbox}^{n+1}.
\]
We define an equivalence relation on $\bbcor_n(X,x_0)$ by taking the pullback diagram
\begin{center}
\begin{tikzcd}
	{R} & {\hom({\bbox}^{n+1}\sqcup_{{\bbcor}^{n+1}}{\bbox}^{n+1},X)} \\
	{\bbcor_n(X,x_0)\times\bbcor_n(X,x_0)} & {\hom({\bbox}^n\sqcup_{{\bbcor}^n}{\bbox}^n,X)}
	\arrow[from=1-1, to=1-2]
	\arrow[from=1-1, to=2-1]
	\arrow["{i^\ast}", from=1-2, to=2-2]
	\arrow[hook, from=2-1, to=2-2]
\end{tikzcd}.
\end{center}
Now $\tilde{\pi}_n(X,x_0)$ is defined to be the coequalizer of the diagram
\begin{center}
\begin{tikzcd}
	R & {\bbcor_n(X,x_0)} & {\tilde{\pi}_n(X,x_0)}
	\arrow[shift left, from=1-1, to=1-2]
	\arrow[shift right, from=1-1, to=1-2]
	\arrow[from=1-2, to=1-3]
\end{tikzcd}.
\end{center}
\end{construction}

To get a better understanding of this coequalizer, let us describe which elements it identifies.

\begin{lemma}\label{lem:concrete-equiv-rel-structure-monoid}
Let $X$ be a concrete fibrant cubeset, $x_0\in X$ and $c,d\in\bbcor_n(X,x_0)$.
Then the following assertions are equivalent.
\begin{enumerate}[(a)]
    \item We have that $[c] = [d]$ in $\tilde{\pi}_n(X,x_0)$.
    \item We have that $\tau_n(c) = \tau_n(d)$ for $\tau_n\colon X\to\tau_n X$ the $n$-truncation.
    \item The cube $[c,d]$ given by $[c,d]\circ\epsilon_1 = c$, $[c,d]\circ\overline{\epsilon}_1 = d$ is a cube in $X(n+1)$.
    \item There exist $e,f\in X(n+1)$ such that $e\lvert_{{\bbcor}^{n+1}} = f\lvert_{{\bbcor}^{n+1}}$
        and $e\circ s = c$, $f\circ s = d$.
\end{enumerate}
\end{lemma}

\begin{proof}
$(b)\implies(c)$: This follows from universal replacement~\ref{prop:univ-replacement} and Proposition~\ref{prop:trunc_concfib}.

$(c)\implies(d)$: Take $e = [c,c]$ and $f = [c,d]$, and use concreteness.

$(d)\implies(b)$: This follows from Proposition~\ref{prop:trunc_concfib}.

Lastly we show that the set $\bbcor_n(X,x_0)$ modded out by the equivalence relation described in $(d)$
is the coequalizer of the two maps $R\to\bbcor_n(X,x_0)$.
A careful unraveling of the definitions shows that $R$ can be identified with the set
\[
    \left\{ (c,d)\in X(n+1)^2 \mid c\lvert_{{\bbcor}^{n+1}} = d\lvert_{{\bbcor}^{n+1}},\; c\circ s\lvert_{{\bbcor}^n} \equiv x_0 \equiv d\circ s\lvert_{{\bbcor}^n} \right\}
\]
and the two maps $R\to\bbcor_n(X,x_0)$ are just the projections $(c,d)\mapsto c\circ s$ and $(c,d)\mapsto d\circ s$.
But this precisely is the equivalence relation of $(d)$.
\end{proof}

\begin{remark}
Since ${\bbox}^{n+1} = {\bbox}^n\times{\bbox}^1$, and hence
\[
    X(n+1) = \hom({\bbox}^{n+1},X) = \hom({\bbox}^1,\ihom({\bbox}^n,X)),
\]
the characterization $(c)$ in the preceding Lemma can be interpreted as the existence of a homotopy,
i.e., a map ${\bbox}^n\times{\bbox}^1\to X$ connecting the $n$-cubes $c$ and $d$ with one another.
\end{remark}

Analogously to the general case, we can also equip the new definition of $\pi_n(X,x_0)$
with the structure of a monoid.

\begin{construction}\label{cons:concrete-structure-monoid-mult}
For $[a],[b]\in\tilde{\pi}_n(X,x_0)$ define $a*b$ as in Definition~\ref{def:structure-group}.
This yields a binary operation
\[
    *\colon\tilde{\pi}_n(X,x_0)\times\tilde{\pi}_n(X,x_0)\to\tilde{\pi}_n(X,x_0),\quad (a,b)\mapsto [a*b].
\]
This operation is well-defined by universal replacement:
in an ($n+1$)-cube $\overline{a*b}$ witnessing the composition,
we can replace $a$ with any equivalent $a'$ and likewise for $b$
by Lemma~\ref{lem:concrete-equiv-rel-structure-monoid}
without altering the appearances of $x_0$ and the value at $1^{n+1}$.
This shows that, if $a\sim a'$ and $b\sim b'$, then also $a*b\sim a'*b'$.
\end{construction}

\begin{proposition}\label{prop:concrete-structure-group-is-monoid}
The operation $*\colon\tilde{\pi}_n(X,x_0)\times\tilde{\pi}_n(X,x_0)\to\tilde{\pi}_n(X,x_0)$ makes $\tilde{\pi}_n(X,x_0)$
an abelian monoid.
It upgrades to a functor $\ccSet_{\ast,\mathrm{fib}}\to\cMon$.
\end{proposition}

\begin{proof}
The proof is verbatim the same as the proof of Proposition~\ref{prop:structure-group-is-monoid} apart from the fact that all completions taken are not unique,
but this deficit is countered by the identification on $\tilde{\pi}_n(X,x_0)$.
\end{proof}

For $X$ a concrete fibrant cubeset, the two different definitions of structure monoid agree.
For this, we need the following Lemma.

\begin{proposition}\label{prop:trunc-structure-monoid}
Let $X$ be a concrete fibrant cubeset with truncation $\tau_n\colon X\to\tau_n X$ and $x_0\in X$.
Then $\tau_{n,\ast}\colon\pi_k(X,x_0)\to\pi_k(\tau_n X,x_0)$ is an isomorphism for $k\leq n$,
and $\pi_k(\tau_n X,x_0) = 0$ for $k > n$.

In particular, the two definitions of structure monoid are naturally isomorphic for concrete fibrant cubesets.
\end{proposition}

\begin{proof}
That $\pi_k(\tau_n X,x_0) = 0$ for $k>n$ is clear since the only $k$-dimensional cube having the constant corner $x_0$ is the constant cube $x_0$ itself by unique corner completion of $\tau_n X$ for $k>n$.

Let $k\leq n$.
By functoriality, $\tau_n$ induces a monoid homomorphism $\tau_{n,\ast}\colon\pi_k(X,x_0)\to\pi_k(\tau_n X,x_0)$.
This map is a bijection on underlying sets by Lemma~\ref{lem:concrete-equiv-rel-structure-monoid}.

For the last claim, note that it suffices to assume that $X$ is $n$-step.
Then $\tau_n$ is the identity and so by Lemma~\ref{lem:concrete-equiv-rel-structure-monoid},
we have that $\bbcor_n(X,x_0) = \tilde{\pi}_n(X,x_0)$.
\end{proof}

\begin{remark}\label{rem:nil-cofibrations}
Recall the inclusion
$i\colon{\bbox}^k\sqcup_{{\bbcor}^k}{\bbox}^k\to{\bbox}^{k+1}\sqcup_{{\bbcor}^{k+1}}{\bbox}^{k+1}$
from Construction~\ref{cons:structure-monoid-concrete} that is
obtained from gluing $s\colon{\bbox}^k\to{\bbox}^{k+1}$ along itself.
Using universal replacement~\ref{prop:univ-replacement} one can show for $X$ concrete fibrant
the truncation $\tau_n\colon X\to\tau_n X$ has right lifting against $i$, i.e.,
in the commutative diagram
\begin{center}
\begin{tikzcd}
	{{\bbox}^k\sqcup_{{\bbcor}^k} {\bbox}^k} & X \\
	{{\bbox}^{k+1}\sqcup_{{\bbcor}^{k+1}}{\bbox}^{k+1}} & {\tau_n X}
	\arrow["{a\sqcup b}", from=1-1, to=1-2]
	\arrow["i"', from=1-1, to=2-1]
	\arrow["{\tau_n}", from=1-2, to=2-2]
	\arrow[dashed, from=2-1, to=1-2]
	\arrow["{c\sqcup d}"', from=2-1, to=2-2]
\end{tikzcd}
\end{center}
there exists a diagonal filler for $k\leq n$.
So the right lifting property of $\tau_n$ against $i$ is what yields injectivity of $\tau_{n,\ast}$.
This can be done more generally:
Denote
\[
    \Gamma = \{{\bbox}^n\sqcup_{{\bbcor}^n}{\bbox}^n\to{\bbox}^{n+1}\sqcup_{{\bbcor}^{n+1}}{\bbox}^{n+1} : n\in\N_0\cup\{-1\}\}
\]
where the maps are given by diagonal embeddings $i$ arising from $s$.
Now if a map $f\colon X\to Y$ between concrete fibrant cubesets%
\footnote{or as long as the definition of structure monoid makes sense}
has right lifting property against $\Gamma$, then $f_\ast$ induces an injection of structure monoids
for all $n\geq 1$ (and even of sets for $n=0$).
\end{remark}

Another interesting fact is that, if $X$ is concrete,
the proposed construction for inverses in Remark~\ref{rem:structure-mon-not-grouplike}
does actually work.

\begin{proposition}\label{prop:concrete-structure-group-is-group}
If $X$ is a concrete fibrant cubeset and $x_0\in X$, then $\pi_n(X,x_0)$ is grouplike.
\end{proposition}

\begin{proof}
Given an element $a\in\pi_n(X,x_0)$ we want to find $b\in\pi_n(X,x_0)$
such that (in case $n=1$) the left square is a $2$-cube in $X$.
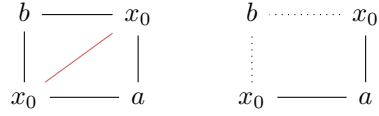
\begin{center}
\begin{tikzcd}
	b & {x_0} & b & {x_0} \\
	{x_0} & a & {x_0} & a
	\arrow[no head, from=1-1, to=1-2]
	\arrow[dotted, no head, from=1-4, to=1-3]
	\arrow[no head, from=2-1, to=1-1]
	\arrow[color={rgb,255:red,214;green,92;blue,92}, no head, from=2-1, to=1-2]
	\arrow[no head, from=2-1, to=2-2]
	\arrow[no head, from=2-2, to=1-2]
	\arrow[dotted, no head, from=2-3, to=1-3]
	\arrow[no head, from=2-3, to=2-4]
	\arrow[no head, from=2-4, to=1-4]
\end{tikzcd}
\end{center}
So instead of completing $a$ and $b$ to $x_0$ we use a flip in order to center $a$ in the corner
and obtain the completion, which then after flipping once again yields an element $b$ such that $a*b = x_0$.

Denote by $\rho_j^k\colon{\bbox}^{k}\to{\bbox}^{k}$ the flip $x_j\mapsto 1-x_j$.
For $k=n+1$ and $j=2$ this flip restricts to $\epsilon_1^n$ so that
\[
    \rho_2^{n+1}\circ\epsilon_1^n = \epsilon_1^n\circ\rho_1^n.
\]
Let $\lambda\colon{\bbcor}^{n+1}\to X$ denote the corner given by $\lambda_1 = \lambda_2 = a\circ\rho_1^n$
and $\lambda_i \equiv x_0$ for $i\geq 3$.
Using Remark~\ref{rem:corner-combinatorial} one can verify that this actually is a corner
(using that $\rho_1^n\circ\epsilon_j^{n-1} = \epsilon_j^n\circ\rho_1^{n-1}$ for $j\geq 2$).
Let us denote its completion by $c$.
We claim that $\overline{a*a^{-1}} = c\circ\rho_2^{n+1}$ has the desired inverse element of $a$ as its second inert face.
For $i\geq 3$ we have that
\[
    \overline{a*a^{-1}}\circ\epsilon_i^n = c\circ\rho_2^{n+1}\circ\epsilon_i^n
    = c\circ\epsilon_i^{n+1}\circ\rho_2^n = x_0.
\]
For $i=1$ we obtain
\[
    \overline{a*a^{-1}}\circ\epsilon_1^n = c\circ\rho_2^{n+1}\circ\epsilon_1^n = c\circ\epsilon_1^{n+1}\circ\rho_1^n = \lambda_1\circ\rho_1^n = a.
\]
Define $a^{-1} = \overline{a*a^{-1}}\circ\epsilon_2^n$.
By construction, $a^{-1}\lvert_{{\bbcor}^n} \equiv x_0$
which can be checked via Remark~\ref{rem:corner-combinatorial},
hence $a^{-1}\in\pi_n(X,x_0)$.

Now to the part that uses concreteness of $X$, namely $\overline{a*a^{-1}}\circ s \equiv x_0$.
Here we can check on points:
if $x_1 = 0$, then
\[
    \overline{a*a^{-1}}\circ s(0,x_2,\ldots,x_n)
    = c(\rho_2^{n+1}(0,0,x_2,\ldots,x_n))
    = c(0,1,x_2,\ldots,x_n)
    = a(0,x_2,\ldots,x_n)
    \equiv x_0
\]
and if $x_1 = 1$, then
\[
    \overline{a*a^{-1}}\circ s(1,x_2,\ldots,x_n)
    = c(\rho_2^{n+1}(1,1,x_2,\ldots,x_n))
    = c(1,0,x_2,\ldots,x_n)
    = a(0,x_2,\ldots,x_n)
    \equiv x_0.
\]
Therefore, we obtain $a*a^{-1} = x_0$ and $a^{-1}$ is the inverse of $a$.
\end{proof}

Another important aspect of the structure monoid is that if two points $x_0,y_0\in X(0)$
are in the same ergodic component, then there is a canonical isomorphism $\pi_n(X,x_0)\to\pi_n(X,y_0)$
of structure monoids.

\begin{proposition}\label{prop:structure-monoid-independence-base-point}
Let $X$ be a cubeset and $z\in X(1)$ with vertices $x_0,y_0\in X(0)$.
Then there is a canonical monoid isomorphism $h_z\colon\pi_n(X,x_0)\to\pi_n(X,y_0)$.

More generally, if $[x_0]=[y_0]$ in $\pi_0(X)$, then $\pi_n(X,x_0)\simeq\pi_n(X,y_0)$.
If $X$ has the unique gluing property, then this isomorphism is canonical as well.
\end{proposition}

\begin{proof}
We can assume w.l.o.g. that $X$ is $n$-step.
We construct a map $h_z\colon\pi_n(X,x_0)\to\pi_n(X,y_0)$, depending on $z$.
In the case $n=2$, this looks as follows: since $z = (x_0,y_0)\in X(1)$
one can degenerate this 1-cube to a 2-cube in $X$ along which we can put an element $a\in\pi_2(X,x_0)$
(depicted in blue).
Then using (unique!) corner completion in dimension 3 in $X$ yields on the back face
the corresponding element $h_z(a)\in\pi_2(X,y_0)$.
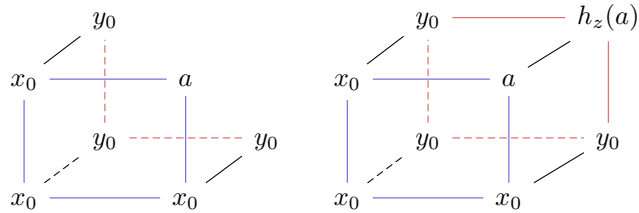
\begin{center}
\begin{tikzcd}[column sep=small, row sep=small]
	& {y_0} &&&& {y_0} && {h_z(a)} \\
	{x_0} && a && {x_0} && a \\
	& {y_0} && {y_0} && {y_0} && {y_0} \\
	{x_0} && {x_0} && {x_0} && {x_0}
	\arrow[color={rgb,255:red,214;green,92;blue,92}, no head, from=1-6, to=1-8]
	\arrow[no head, from=2-1, to=1-2]
	\arrow[color={rgb,255:red,92;green,92;blue,214}, no head, from=2-1, to=2-3]
	\arrow[no head, from=2-5, to=1-6]
	\arrow[color={rgb,255:red,92;green,92;blue,214}, no head, from=2-5, to=2-7]
	\arrow[no head, from=2-7, to=1-8]
	\arrow[color={rgb,255:red,214;green,92;blue,92}, dashed, no head, from=3-2, to=1-2]
	\arrow[color={rgb,255:red,214;green,92;blue,92}, dashed, no head, from=3-2, to=3-4]
	\arrow[color={rgb,255:red,214;green,92;blue,92}, dashed, no head, from=3-6, to=1-6]
	\arrow[color={rgb,255:red,214;green,92;blue,92}, dashed, no head, from=3-6, to=3-8]
	\arrow[color={rgb,255:red,214;green,92;blue,92}, no head, from=3-8, to=1-8]
	\arrow[color={rgb,255:red,92;green,92;blue,214}, no head, from=4-1, to=2-1]
	\arrow[dashed, no head, from=4-1, to=3-2]
	\arrow[color={rgb,255:red,92;green,92;blue,214}, no head, from=4-1, to=4-3]
	\arrow[color={rgb,255:red,92;green,92;blue,214}, no head, from=4-3, to=2-3]
	\arrow[no head, from=4-3, to=3-4]
	\arrow[color={rgb,255:red,92;green,92;blue,214}, no head, from=4-5, to=2-5]
	\arrow[dashed, no head, from=4-5, to=3-6]
	\arrow[color={rgb,255:red,92;green,92;blue,214}, no head, from=4-5, to=4-7]
	\arrow[color={rgb,255:red,92;green,92;blue,214}, no head, from=4-7, to=2-7]
	\arrow[no head, from=4-7, to=3-8]
\end{tikzcd}
\end{center}
Let $p_1^k\colon{\bbox}^k\to{\bbox}^1$ be the degeneracy onto the $x_1$-coordinate.
For $a\in\pi_n(X,x_0)$ we define an ($n+1$)-corner in $X$ given by $\lambda_1 = a$ and $\lambda_i = z\circ p_1^n$ for $i\geq 2$.
Using Remark~\ref{rem:corner-combinatorial} we see that this actually constitutes a corner in $X$,
where one uses that $p_1^n\circ\epsilon_1 \equiv 0$ and $p_1^n\circ\epsilon_i = p_1^{n-1}$ independent of $i\geq 2$.
This corner can be uniquely completed in $X$ to obtain an ($n+1$)-cube $\overline{h_z(a)}$.
Let $h_z(a) = \overline{h_z(a)}\circ\overline{\epsilon}_1^n$.
This $h_z(a)$ satisfies that $h_z(a)\circ\epsilon_i^{n-1} \equiv y_0$:
since $\overline{\epsilon}_1^n\circ\epsilon_i^{n-1} = \epsilon_{i+1}^n\circ\overline{\epsilon}_1^{n-1}$
we have that
\[
    h_z(a)\circ\epsilon_i^{n-1}
    = \overline{h_z(a)}\circ\overline{\epsilon}_1^n\circ\epsilon_i^{n-1}
    = \overline{h_z(a)}\circ\epsilon_{i+1}^n\circ\overline{\epsilon}_1^{n-1}
    = z\circ p_1^n\circ\overline{\epsilon}_1^{n-1}
    = z\circ 1
    \equiv y_0
\]
for $1\leq i\leq n$.
Therefore, the assignment $a\mapsto h_z(a)$ gives a well-defined map
\[
    h_z\colon\pi_n(X,x_0)\to\pi_n(X,y_0).
\]
It is easy to see that $h$ has an inverse:
recall that $\rho_1^k\colon{\bbox}^k\to{\bbox}^k$ is the cube automorphism $x_1\mapsto 1-x_1$.
Given $a\in\pi_n(X,x_0)$ take the ($n+1$)-cube $d$ constructed above having $a$ and $b = h_z(a)$
on its front and back $x_1$-faces.
Let $c = \overline{h_z(a)}\circ\rho_1^{n+1}$ and $w = z\circ\rho_1^1$.
Then $c\circ\epsilon_1 = b$, $c\circ\epsilon_i = w\circ p_1^n$ and $c\circ\overline{\epsilon}_1 = a$.
Therefore, $c$ is the ($n+1$)-cube $\overline{h_w(b)}$ witnessing $b$ and $h_w(b)$ and furthermore, $h_w(b) = a$.
This implies that $h_w(h_z(a)) = a$.
Interchanging $w$ with $z$, using that $\rho_1^1\circ\rho_1^1 = \id$, yields that $h_w$ is the inverse of $h_z$.

It remains to show that $h_z(a*b) = h_z(a)*h_z(b)$ for all $a,b\in\pi_n(X,x_0)$.
In case of $n=1$ this can be summarized as follows in an image:
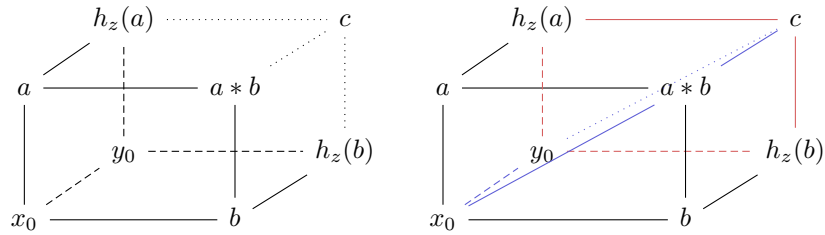
\begin{center}
\begin{tikzcd}[column sep=small, row sep=small]
	& {h_z(a)} && c && {h_z(a)} &&& c \\
	a && {a*b} && a &&& {a*b} \\
	& {y_0} && {h_z(b)} && {y_0} &&& {h_z(b)} \\
	{x_0} && b && {x_0} &&& b
	\arrow[dotted, no head, from=1-2, to=1-4]
	\arrow[color={rgb,255:red,214;green,92;blue,92}, no head, from=1-6, to=1-9]
	\arrow[no head, from=2-1, to=1-2]
	\arrow[no head, from=2-1, to=2-3]
	\arrow[no head, from=2-1, to=4-1]
	\arrow[dotted, no head, from=2-3, to=1-4]
	\arrow[no head, from=2-5, to=1-6]
	\arrow[no head, from=2-5, to=2-8]
	\arrow[color={rgb,255:red,92;green,92;blue,214}, no head, from=2-8, to=1-9]
	\arrow[dashed, no head, from=3-2, to=1-2]
	\arrow[dashed, no head, from=3-2, to=3-4]
	\arrow[dotted, no head, from=3-4, to=1-4]
	\arrow[color={rgb,255:red,214;green,92;blue,92}, dashed, no head, from=3-6, to=1-6]
	\arrow[color={rgb,255:red,92;green,92;blue,214}, dotted, no head, from=3-6, to=1-9]
	\arrow[color={rgb,255:red,214;green,92;blue,92}, dashed, no head, from=3-6, to=3-9]
	\arrow[color={rgb,255:red,214;green,92;blue,92}, no head, from=3-9, to=1-9]
	\arrow[dashed, no head, from=4-1, to=3-2]
	\arrow[no head, from=4-1, to=4-3]
	\arrow[no head, from=4-3, to=2-3]
	\arrow[no head, from=4-3, to=3-4]
	\arrow[no head, from=4-5, to=2-5]
	\arrow[color={rgb,255:red,92;green,92;blue,214}, no head, from=4-5, to=2-8]
	\arrow[color={rgb,255:red,92;green,92;blue,214}, dashed, no head, from=4-5, to=3-6]
	\arrow[no head, from=4-5, to=4-8]
	\arrow[no head, from=4-8, to=2-8]
	\arrow[no head, from=4-8, to=3-9]
\end{tikzcd}
\end{center}
on one inert $\epsilon_1$ of the 3-corner we put the 2-cube $\overline{a*b}$ witnessing the multiplication $a*b$,
on the other two inerts $\epsilon_2$ and $\epsilon_3$ we put the witnesses $\overline{h_z(a)}$ and $\overline{h_z(b)}$ for $a\mapsto h_z(a)$ and $b\mapsto h_z(b)$.
Then using unique 3-corner completion gives us a 3-cube in $X$
which now has as 2-cube in blue the witness $\overline{h_z(a*b)}$
and $\overline{h_z(a)*h_z(b)}$ in red.
Taking the dotted blue diagonal we must have that $h_z(a)*h_z(b) = h_z(a*b)$ since completion is unique.

Now to the formal argument:
let $a,b\in\pi_n(X,x_0)$ and set $\lambda_1 = \overline{a*b}$, $\lambda_2 = \overline{h_z(a)}$, $\lambda_3 = \overline{h_z(b)}$ and $\lambda_i = z\circ p_1^{n+1}$ for $i\geq 4$.
This determines an ($n+2$)-corner $\lambda$ in $X$:
one can directly verify the conditions of Remark~\ref{rem:corner-combinatorial} once again making use of the fact that
$p_1^{n+1}\circ\epsilon_1 \equiv 0$ and $p_1^{n+1}\circ\epsilon_i = p_1^n$ independent of $i\geq 2$.
Completing this corner $\lambda$ yields an ($n+2$)-cube $c$ in $X$.
Similarly to the proof of associativity in Proposition~\ref{prop:structure-group-is-monoid}
we define maps of cubes $r = \overline{\epsilon}_1^{n+1}$ and
\[
    b\colon{\bbox}^{n+1}\to{\bbox}^{n+2},\quad (x_1,\ldots,x_{n+1})\mapsto(x_1,x_2,x_2,\ldots,x_{n+1})
\]
which ought to extract $\overline{h_z(a)*h_z(b)}$ (depicted in red $r$) and $\overline{h_z(a*b)}$ (in blue $b$), respectively, out of $c$.
To see that $\overline{h_z(a)*h_z(b)} = c\circ r$, note that $r\circ\epsilon_i^n = \overline{\epsilon}_1^{n+1}\circ\epsilon_i^n = \epsilon_{i+1}^{n+1}\circ\overline{\epsilon}_1^n$ and therefore
\[
    c\circ r\circ\epsilon_i^n = c\circ\epsilon_{i+1}^{n+1}\circ\overline{\epsilon}_1^n
    = \begin{cases}
        \overline{h_z(a)}\circ\overline{\epsilon}_1^n,\quad & i = 1, \\
        \overline{h_z(b)}\circ\overline{\epsilon}_1^n,\quad & i = 2, \\
        z\circ p_1^{n+1}\circ\overline{\epsilon}_1^n,\quad & i\geq 3,
    \end{cases}
    = \begin{cases}
        h_z(a),\quad & i = 1, \\
        h_z(b),\quad & i = 2, \\
        y_0,\quad & i\geq 3,
    \end{cases}
    = \overline{h_z(a)*h_z(b)}\circ\epsilon_i^n.
\]
For $\overline{h_z(a*b)} = c\circ b$ note that $b\circ\epsilon_1^n =\epsilon_1^{n+1}\circ s$,
$b\circ\epsilon_2^n = \epsilon_2^{n+1}\circ\epsilon_2^n$ and $b\circ\epsilon_i^n = \epsilon_{i+1}^{n+1}\circ b'$ for $i\geq 3$ and some (linear) $b'\colon{\bbox}^n\to{\bbox}^{n+1}$.
From this we obtain
\[
    c\circ b\circ\epsilon_i =
    \begin{cases}
        c\circ\epsilon_1\circ s,\quad & i = 1, \\
        c\circ\epsilon_2\circ\epsilon_2,\quad & i = 2, \\
        c\circ\epsilon_{i+1}\circ b',\quad & i\geq 3
    \end{cases} =
    \begin{cases}
        \overline{a*b}\circ s,\quad & i = 1, \\
        \overline{h_z(a)}\circ\epsilon_2,\quad & i = 2, \\
        z\circ p_1^{n+1}\circ b',\quad & i \geq 3
    \end{cases} =
    \begin{cases}
        a*b,\quad & i = 1, \\
        z\circ p_1^n,\quad & i = 2, \\
        z\circ p_1^n,\quad & i\geq 3
    \end{cases} =
    \overline{h_z(a*b)}\circ\epsilon_i
\]
where for the case $i\geq 3$ we have used that $c\circ\epsilon_j = z\circ p$ for $j\geq 4$.
Now lastly take note that $r\circ s = \overline{\epsilon}_1\circ s = b\circ\overline{\epsilon}_1$ and so
\[
    h_z(a)*h_z(b) = \overline{h_z(a)*h_z(b)}\circ s = c\circ r\circ s = c\circ b\circ\overline{\epsilon}_1 = \overline{h_z(a*b)}\circ\overline{\epsilon}_1 = h_z(a*b)
\]
proving the claim.

In the general setting, if $[x_0]=[x_1]$ in $\pi_0(X)$,
then there exists a finite zig-zag of 1-cubes $z_1,\ldots,z_k$ connecting $x_0$ and $x_1$,
see Proposition~\ref{prop:0-trunc-computation}.
Composing the corresponding isomorphisms $h_{z_i}$ gives a monoid isomorphism $\pi_n(X,x_0)\simeq\pi_n(X,y_0)$.

Lastly, if $X$ has the unique gluing property, then $[x_0]=[x_1]$ in $\pi_0(X)$
means that there exists a $z\in X(1)$ with vertices $x_0$ and $x_1$.
Since $X$ is concrete by Proposition~\ref{prop:unique-gluing-concrete},
this $z$ is unique.
\end{proof}

We have seen that if $z\in X(1)$ and $w = z\circ\rho_1^1$ then we have that $h_z\circ h_w = \id = h_{z(1)}$ where $z(1)$ denotes the degenerate 1-cube.
We would like to extend this property of the composition of different $h$.
However, in general, there is not much that can be hoped for:
given $z,w\in X(1)$ there need not exist an element $z.w$ which acts as the composition of $z$ and $w$.
And if $X$ has the gluing property, which guarantees existence of $z.w$,
then missing uniqueness doesn't let us make a statement about $h_{z.w}$ in terms of $h_z$ and $h_w$.
This can be fixed by assuming that $X$ has the unique gluing property,
in which case $X$ automatically is concrete.
But then everything works as imagined as the following Lemma shows.

\begin{lemma}\label{lem:homotopy-action-well-defined-global}
Let $X$ be a concrete $n$-step cubeset with the gluing property,
$x_0,y_0,z_0\in X(0)$ such that $(x_0,y_0),(y_0,z_0)\in X(1)$.
Then
\[
    h_{(y_0,z_0)}\circ h_{(x_0,y_0)} = h_{(x_0,z_0)}.
\]
\end{lemma}

\begin{proof}
Let $a\in\pi_n(X,x_0)$, $b = h_{(x_0,y_0)}(a)\in\pi_n(X,y_0)$
as well as $u = (x_0,y_0)$, $v = (y_0,z_0)$ and $w = (x_0,z_0)$.
Take the ($n+1$)-cubes $\overline{h_u(a)}$ and $\overline{h_v(b)}$
which are depicted for $n=2$ as the left and middle cube, respectively.
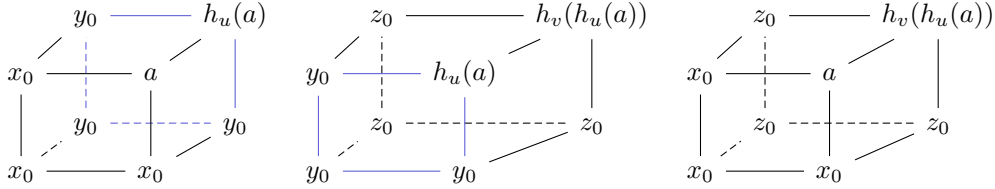
\begin{center}
\begin{tikzcd}[row sep=tiny, column sep = tiny]
	& {y_0} && {h_u(a)} && {z_0} && {h_v(h_u(a))} && {z_0} && {h_v(h_u(a))} \\
	{x_0} && a && {y_0} && {h_u(a)} && {x_0} && a \\
	& {y_0} && {y_0} && {z_0} && {z_0} && {z_0} && {z_0} \\
	{x_0} && {x_0} && {y_0} && {y_0} && {x_0} && {x_0}
	\arrow[color={rgb,255:red,92;green,92;blue,214}, no head, from=1-2, to=1-4]
	\arrow[no head, from=1-6, to=1-8]
	\arrow[no head, from=1-10, to=1-12]
	\arrow[no head, from=2-1, to=1-2]
	\arrow[no head, from=2-1, to=2-3]
	\arrow[no head, from=2-3, to=1-4]
	\arrow[no head, from=2-5, to=1-6]
	\arrow[color={rgb,255:red,92;green,92;blue,214}, no head, from=2-5, to=2-7]
	\arrow[no head, from=2-7, to=1-8]
	\arrow[no head, from=2-9, to=1-10]
	\arrow[no head, from=2-9, to=2-11]
	\arrow[no head, from=2-11, to=1-12]
	\arrow[color={rgb,255:red,92;green,92;blue,214}, dashed, no head, from=3-2, to=1-2]
	\arrow[color={rgb,255:red,92;green,92;blue,214}, dashed, no head, from=3-2, to=3-4]
	\arrow[color={rgb,255:red,92;green,92;blue,214}, no head, from=3-4, to=1-4]
	\arrow[dashed, no head, from=3-6, to=1-6]
	\arrow[dashed, no head, from=3-6, to=3-8]
	\arrow[no head, from=3-8, to=1-8]
	\arrow[dashed, no head, from=3-10, to=1-10]
	\arrow[dashed, no head, from=3-10, to=3-12]
	\arrow[no head, from=3-12, to=1-12]
	\arrow[no head, from=4-1, to=2-1]
	\arrow[dashed, no head, from=4-1, to=3-2]
	\arrow[no head, from=4-1, to=4-3]
	\arrow[no head, from=4-3, to=2-3]
	\arrow[no head, from=4-3, to=3-4]
	\arrow[color={rgb,255:red,92;green,92;blue,214}, no head, from=4-5, to=2-5]
	\arrow[dashed, no head, from=4-5, to=3-6]
	\arrow[color={rgb,255:red,92;green,92;blue,214}, no head, from=4-5, to=4-7]
	\arrow[color={rgb,255:red,92;green,92;blue,214}, no head, from=4-7, to=2-7]
	\arrow[no head, from=4-7, to=3-8]
	\arrow[no head, from=4-9, to=2-9]
	\arrow[dashed, no head, from=4-9, to=3-10]
	\arrow[no head, from=4-9, to=4-11]
	\arrow[no head, from=4-11, to=2-11]
	\arrow[no head, from=4-11, to=3-12]
\end{tikzcd}
\end{center}
Then $\overline{h_u(a)}\circ\overline{\epsilon}_1 = h_u(a) = b = \overline{h_v(b)}\circ\epsilon_1$
(depicted in blue).
Hence, those two cubes can be glued along the $x_1$-axis to a cube $c\in X(n+1)$ satisfying
$c\circ\epsilon_1 = \overline{h_u(a)}\circ\epsilon_1 = a$
and
$c\circ\overline{\epsilon}_1 = \overline{h_v(b)}\circ\overline{\epsilon}_1 = h_v(b)$,
which can be viewed on the right.
By comparing points we infer that $c = \overline{h_w(a)}$ and hence
$h_w(a) = h_v(b) = h_v(h_u(a))$.
\end{proof}

We close this section with an alternative description of the multiplication in $\pi_n(X,x_0)$
in terms of the isomorphism $h_{(x_0,y_0)}$ for concrete cubesets.

\begin{lemma}\label{lem:homotopy-action-well-defined}
Let $X$ be a concrete $n$-step cubeset with the (unique) gluing property and $x_0\in X(0)$.
Then for all $a,b\in\pi_n(X,x_0)$ we have that
\[
    (h_{(x_0,y_0)}(a))(1^n) = (a*b)(1^n)
\]
where $y_0 = b(1^n)$.
\end{lemma}

\begin{proof}
Let $u = (x_0,y_0)\in X(1)$ and denote by $[h_u(a),h_u(a)] = h_u(a)\circ r^{n+1}$
the cube $h_u(a)$ degenerated along the $\omega_1$-axis to an ($n+1$)-cube.
Here, $r^k\colon{\bbox}^k\to{\bbox}^{k-1},\omega\mapsto(\omega_2,\ldots,\omega_k)$
denotes the projection omitting the first coordinate.
We show that, gluing to $\overline{a*b}$ iteratively along lower faces $\epsilon_k$,
one obtains an ($n+1$)-cube $e$ that satisfies
\[
    e\rvert_{{\bbcor}^{n+1}} = [h_u(a),h_u(a)]\rvert_{{\bbcor}^{n+1}}.
\]
Using unique corner completion, we obtain that
\[
    (a*b)(1^n) = \overline{a*b}(1^{n+1}) = e(1^{n+1}) = [h_u(a),h_u(a)](1^{n+1}) = (h_u(a))(1^n).
\]
In the case of $n=2$, the gluing procedure looks as follows.
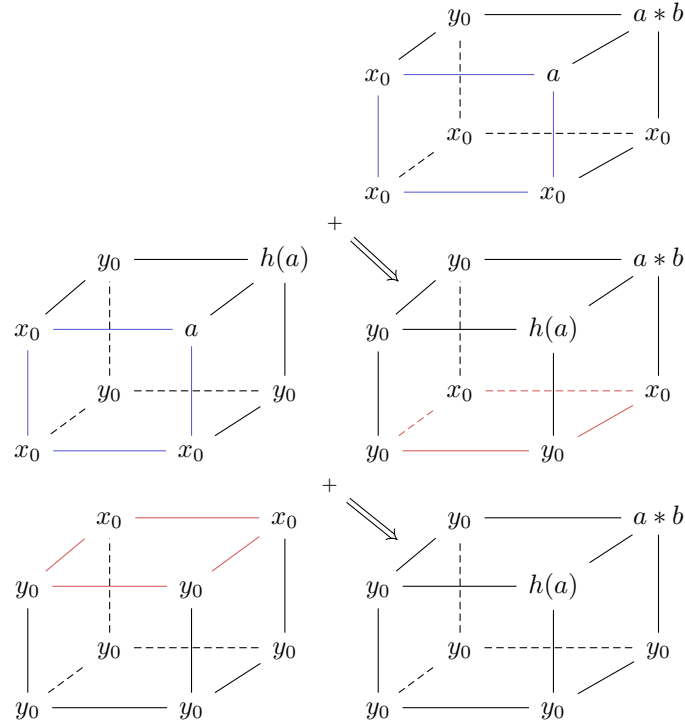
\begin{center}
\begin{tikzcd}[column sep=small, row sep=small]
	&&&&& {y_0} && {a*b} \\
	&&&& {x_0} && a \\
	&&&&& {x_0} && {x_0} \\
	&&&& {x_0} && {x_0} \\
	& {y_0} && {h(a)} && {y_0} && {a*b} \\
	{x_0} && a && {y_0} && {h(a)} \\
	& {y_0} && {y_0} && {x_0} && {x_0} \\
	{x_0} && {x_0} && {y_0} && {y_0} \\
	& {x_0} && {x_0} && {y_0} && {a*b} \\
	{y_0} && {y_0} && {y_0} && {h(a)} \\
	& {y_0} && {y_0} && {y_0} && {y_0} \\
	{y_0} && {y_0} && {y_0} && {y_0}
	\arrow[no head, from=1-6, to=1-8]
	\arrow[no head, from=2-5, to=1-6]
	\arrow[color={rgb,255:red,92;green,92;blue,214}, no head, from=2-5, to=2-7]
	\arrow[no head, from=2-7, to=1-8]
	\arrow[dashed, no head, from=3-6, to=1-6]
	\arrow[dashed, no head, from=3-6, to=3-8]
	\arrow[no head, from=3-8, to=1-8]
	\arrow[color={rgb,255:red,92;green,92;blue,214}, no head, from=4-5, to=2-5]
	\arrow[dashed, no head, from=4-5, to=3-6]
	\arrow[color={rgb,255:red,92;green,92;blue,214}, no head, from=4-5, to=4-7]
	\arrow[""{name=0, anchor=center, inner sep=0}, "{+}"{description}, draw=none, from=4-5, to=5-4]
	\arrow[color={rgb,255:red,92;green,92;blue,214}, no head, from=4-7, to=2-7]
	\arrow[no head, from=4-7, to=3-8]
	\arrow[no head, from=5-2, to=5-4]
	\arrow[no head, from=5-2, to=6-1]
	\arrow[no head, from=5-4, to=6-3]
	\arrow[no head, from=5-6, to=5-8]
	\arrow[color={rgb,255:red,92;green,92;blue,214}, no head, from=6-1, to=6-3]
	\arrow[""{name=1, anchor=center, inner sep=0}, no head, from=6-5, to=5-6]
	\arrow[no head, from=6-5, to=6-7]
	\arrow[no head, from=6-7, to=5-8]
	\arrow[dashed, no head, from=7-2, to=5-2]
	\arrow[dashed, no head, from=7-2, to=7-4]
	\arrow[dashed, no head, from=7-2, to=8-1]
	\arrow[no head, from=7-4, to=5-4]
	\arrow[no head, from=7-4, to=8-3]
	\arrow[dashed, no head, from=7-6, to=5-6]
	\arrow[color={rgb,255:red,214;green,92;blue,92}, dashed, no head, from=7-6, to=7-8]
	\arrow[no head, from=7-8, to=5-8]
	\arrow[color={rgb,255:red,92;green,92;blue,214}, no head, from=8-1, to=6-1]
	\arrow[color={rgb,255:red,92;green,92;blue,214}, no head, from=8-1, to=8-3]
	\arrow[color={rgb,255:red,92;green,92;blue,214}, no head, from=8-3, to=6-3]
	\arrow[no head, from=8-5, to=6-5]
	\arrow[color={rgb,255:red,214;green,92;blue,92}, dashed, no head, from=8-5, to=7-6]
	\arrow[color={rgb,255:red,214;green,92;blue,92}, no head, from=8-5, to=8-7]
	\arrow[""{name=2, anchor=center, inner sep=0}, "{+}"{description}, draw=none, from=8-5, to=9-4]
	\arrow[no head, from=8-7, to=6-7]
	\arrow[color={rgb,255:red,214;green,92;blue,92}, no head, from=8-7, to=7-8]
	\arrow[color={rgb,255:red,214;green,92;blue,92}, no head, from=9-2, to=9-4]
	\arrow[no head, from=9-6, to=9-8]
	\arrow[color={rgb,255:red,214;green,92;blue,92}, no head, from=10-1, to=9-2]
	\arrow[color={rgb,255:red,214;green,92;blue,92}, no head, from=10-1, to=10-3]
	\arrow[color={rgb,255:red,214;green,92;blue,92}, no head, from=10-3, to=9-4]
	\arrow[""{name=3, anchor=center, inner sep=0}, no head, from=10-5, to=9-6]
	\arrow[no head, from=10-5, to=10-7]
	\arrow[no head, from=10-7, to=9-8]
	\arrow[dashed, no head, from=11-2, to=9-2]
	\arrow[dashed, no head, from=11-2, to=11-4]
	\arrow[no head, from=11-4, to=9-4]
	\arrow[dashed, no head, from=11-6, to=9-6]
	\arrow[dashed, no head, from=11-6, to=11-8]
	\arrow[no head, from=11-8, to=9-8]
	\arrow[no head, from=12-1, to=10-1]
	\arrow[dashed, no head, from=12-1, to=11-2]
	\arrow[no head, from=12-1, to=12-3]
	\arrow[no head, from=12-3, to=10-3]
	\arrow[no head, from=12-3, to=11-4]
	\arrow[no head, from=12-5, to=10-5]
	\arrow[dashed, no head, from=12-5, to=11-6]
	\arrow[no head, from=12-5, to=12-7]
	\arrow[no head, from=12-7, to=10-7]
	\arrow[no head, from=12-7, to=11-8]
	\arrow[between={0.2}{0.8}, Rightarrow, from=0, to=1]
	\arrow[between={0.2}{0.8}, Rightarrow, from=2, to=3]
\end{tikzcd}
\end{center}
In general, start with $\overline{a*b}$ (the upper right most cube)
whose first inert side is $a = \overline{a*b}\circ\epsilon_1$ (depicted in blue).
This agrees with $a = \overline{h_u(a)}\circ\epsilon_1$ (the blue side of the upper left most cube),
so we can glue those cubes together, obtaining a new cube $c_2\in X(n+1)$ with
$c_2\circ\epsilon_1 = \overline{h_u(a)}\circ\overline{\epsilon}_1 = h_u(a)$
and $c_2\circ\overline{\epsilon}_1 = \overline{a*b}\circ\overline{\epsilon}_1$.
Now, we wish to turn all the remaining appearances of $x_0$ into $y_0$,
using only gluing at lower faces $\epsilon_k$ for $k\geq 3$ to arrive at the cube $e\in X(n+1)$
(the lower right most cube) satisfying
\[
    e\rvert_{{\bbcor}^{n+1}} = [h_u(a),h_u(a)]\rvert_{{\bbcor}^{n+1}}.
\]
If $n=1$, set $e=c_2$.
Otherwise, on to the remaining steps of gluing:
For $2\leq k\leq n$ define the configuration
\[
    d'_k\colon\{0,1\}^k\to X(0),\quad\omega\mapsto\begin{cases}
        x_0,\quad & \omega_1=\ldots=\omega_{k-1}=1\;\text{and}\;\omega_k=0, \\
        y_0,\quad & \text{else}
    \end{cases}
\]
which is in $X(k)$ since $x_0\sim_{n-1}y_0$ and $k\leq n$.
Let
\[
    q_{k}\colon{\bbox}^{n+1}\to{\bbox}^{k},\quad \omega\mapsto (\omega_1,\omega_3,\ldots,\omega_{k+1})
\]
and define $d_k = d'_k\circ q_k\in X(n+1)$.

Assume that we have inductively defined $c_k$ for $2\leq k\leq n$
which satisfies, for all $\omega\in\{0,1\}^{n+1}$ with $\omega_{k+1}=0$, that
\[
    c_k(\omega) = \begin{cases}
        x_0,\quad &\omega_1=\omega_3=\ldots=\omega_k = 1, \\
        y_0,\quad & \text{else}
    \end{cases}
\]
and $c_k(1^{n+1}) = \overline{a*b}(1^{n+1})$.
(This is, in particular, the case for $c_2$: it is the gluing
of $\overline{h_{u}(a)}$ (for $\omega_1 = 0$, having the value $y_0$)
and $\overline{a*b}$ (for $\omega_1 = 1$, having the value $x_0$).)
Then $c_k\circ\epsilon_{k+1} = d_k\circ\epsilon_{k+1}$ by induction hypothesis (for $c_k$)
and definition (for $d_k$).
This yields the gluing $c_{k+1}$ of $c_k$ and $d_k$ along the $\omega_{k+1}$-axis
so that $c_{k+1}\circ\epsilon_{k+1} = d_k\circ\overline{\epsilon}_{k+1}$
and $c_{k+1}\circ\overline{\epsilon}_{k+1} = c_k\circ\overline{\epsilon}_{k+1}$.
If $k\leq n-1$, we can compute by case distinction on $\omega_{k+1}$
for all $\omega\in\{0,1\}^{n+1}$ with $\omega_{k+2}=0$ that
\[
    c_{k+1}(\omega) = \begin{cases}
        x_0,\quad &\omega_1=\omega_3=\ldots=\omega_{k+1} = 1, \\
        y_0,\quad & \text{else}
    \end{cases}
\]
since if $\omega_{k+1}=0$ we are in the case of $d_k(\omega)$ with $\omega_{k+1}=1$
which has value $y_0$,
and if $\omega_{k+1}=1$ we are in the case of $c_k(\omega)$ where it follows from the induction hypothesis.
Moreover, we have that $c_{k+1}(1^{n+1}) = c_k(1^{n+1}) = \overline{a*b}(1^{n+1})$
by induction hypothesis.

Let $e = c_{n+1}$.
For all $\omega\in\bigcup_{i\geq 3}\epsilon_i$ (i.e., $\omega_i=0$ for some $i\geq 3$) we have that $e(\omega) = y_0$
which can be seen by reverse induction from $n+1$ to $3$ on the biggest $i$ such that $\omega_i=0$.
Moreover, we have that
\begin{align*}
    e(0,0,1,\ldots,1) &= c_2(0,0,1,\ldots,1) = h_{u}(a)(0,1,\ldots,1) = y_0, \\
    e(0,1,1,\ldots,1) &= c_2(0,1,1,\ldots,1) = h_{u}(a)(1,1,\ldots,1), \\
    e(1,0,1,\ldots,1) &= c_2(1,0,1,\ldots,1) = \overline{a*b}(1,0,1,\ldots,1) = b(1,\ldots,1) = y_0.
\end{align*}
This shows that $e\rvert_{{\bbcor}^{n+1}} = [h_u(a),h_u(a)]\rvert_{{\bbcor}^{n+1}}$
which is all that was left to prove.
\end{proof}

\subsubsection{Homotopy Groups over \texorpdfstring{$\bbGamma$}{Gamma}}\label{ssec:structure-group-gamma}

Recall the wide subcategory inclusion $j\colon\bbGamma\to\bbox$.
This induces a restriction $j^\ast\colon\cSet\to\PSh(\bbGamma)$ which has a left and a right adjoint $j_!$ and $j_\ast$, respectively,
both pairs of adjoint functors make up a geometric morphism (in opposite directions), see Lemma~\ref{lem:j-geometric-morphism}.
In $\bbGamma$, there is a more subtle notion of corner than in $\bbox$ which we have introduced in Section~\ref{ssec:corner-gamma}.
The main takeaway was that for a concrete cubeset $X$, the two different notions of fibrancy ($X$ is fibrant and $j^\ast X$ is fibrant) agree, see Proposition~\ref{prop:nilfib-restricted-nilfib-concrete}.
Moreover, in Section~\ref{ssec:truncation-gamma} we have seen that there also exists a truncation over $\bbGamma$ which, for a concrete fibrant cubeset $X$, can be computed analogously,
yielding the quotient $X/\approx_n$ (which is only a $\bbGamma$-set, not a cubeset).
Now the reason for why one should be interested in a definition of structure monoid over $\bbGamma$ is the different notion of corner.
In Remark~\ref{rem:structure-mon-not-grouplike} we have seen that for a general cubeset $X$ there is no hope in obtaining a grouplike monoid.
However, this defect is not present over $\bbGamma$ for $n=1$ anymore: to define inverses, just use the corner
\begin{center}
\begin{tikzcd}
	& {x_0} & {a^{-1}} & {x_0} \\
	{x_0} & a & {x_0} & a
	\arrow[dashed, no head, from=1-3, to=1-4]
	\arrow[no head, from=2-1, to=1-2]
	\arrow[no head, from=2-1, to=2-2]
	\arrow[""{name=0, anchor=center, inner sep=0}, draw=none, from=2-2, to=1-2]
	\arrow[""{name=1, anchor=center, inner sep=0}, dashed, no head, from=2-3, to=1-3]
	\arrow[no head, from=2-3, to=1-4]
	\arrow[no head, from=2-3, to=2-4]
	\arrow[dashed, no head, from=2-4, to=1-4]
	\arrow["\leadsto"{description}, draw=none, from=0, to=1]
\end{tikzcd}
\end{center}
whose completion gives us the inverse of $a$.
This gives us yet another way to define the structure monoid for a cubeset $X$
for which it will then turn out to be a group if $n=1$.
The good thing: everything stays the same for $X$ concrete fibrant.

\begin{definition}
Let $X$ be a $\bbGamma$-set, $x_0\in X$ and $Y = \tau_n^{\bbGamma} X$ for $n\geq 1$.
Then define $\pi_n^{\bbGamma}(X,x_0)$ to be the set
\[
    \{ c\in Y(n) \mid c\rvert_{{\bbcor}^n_1}\equiv x_0 \}.
\]
We define a binary operation on $\pi_n^{\bbGamma}(X,x_0)$:
given $a,b\in\pi_n(X,x_0)$, set $\lambda_1 = a$, $\lambda_2 = b$ and $\lambda_i \equiv x_0$ for $2<i\leq n+1$.
This defines a corner $\lambda_{a,b}\colon{\bbcor}^{n+1}_1\to Y$ whose unique completion we denote by $\overline{a*b}\in Y(n+1)$.
Let $s\colon{\bbox}^n\to{\bbox}^{n+1}$ be the map $(x_1,\ldots,x_n)\mapsto (x_1,x_1,x_2, \ldots,x_n)$.
Then define
\[
    a * b = \overline{a*b}\circ s \in \pi_n(X,x_0).
\]
With this operation, we call $\pi_n^{\bbGamma}(X,x_0)$ the \emph{$n$-th structure monoid of $X$ over $\bbGamma$}.
\end{definition}

This is verbatim the same Definition as~\ref{def:structure-group}
where the structure monoid is defined over $\bbox$.
The only difference here is that we are using the truncation w.r.t. $\bbGamma$
and so there is no reason to expect that the two definitions agree in general.

\begin{remark}
Over $\bbGamma$, every object decomposes into a coproduct of objects
which all have only one point by Proposition~\ref{prop:connected-comp-gamma}.
This means that the only notion of $\pi_0(X)$ for $X$ a $\bbGamma$-set is to take $\hom(\ast,X) = X(0)$
(which agrees with $\hom(\ast,\tau_0^{\bbGamma} X)$).
\end{remark}

\begin{proposition}\label{prop:structure-monoid-gamma}
The operation
\[
    *\colon\pi_n^{\bbGamma}(X,x_0)\times\pi_n^{\bbGamma}(X,x_0)\to\pi_n^{\bbGamma}(X,x_0),\quad (a,b)\mapsto a*b
\]
makes $\pi_n^{\bbGamma}(X,x_0)$ an abelian monoid.
Moreover, the monoid $\pi_1^{\bbGamma}(X,x_0)$ is grouplike.
This construction upgrades to a functor $\pi_n^{\bbGamma}\colon\PSh(\bbGamma)_{\ast}\to\cMon$.
\end{proposition}

\begin{proof}
The proof for commutativity, associativity and the existence of an identity element
works the same as in the proof of Proposition~\ref{prop:structure-group-is-monoid}
since there all morphisms in $\bbox$ we have used are actually in $\bbGamma$.
Also functoriality is clear.

Next we show the existence of inverse elements in the case of $n=1$.
Given an element $a\in\pi_n^{\bbGamma}(X,x_0)$ we want to find $b\in\pi_n(X,x_0)$
such that the left square, with red diagonal being the identity (constant cube $x_0$) is in $X(2)$ .
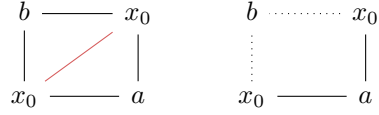
\begin{center}
\begin{tikzcd}
	b & {x_0} & b & {x_0} \\
	{x_0} & a & {x_0} & a
	\arrow[no head, from=1-1, to=1-2]
	\arrow[dotted, no head, from=1-4, to=1-3]
	\arrow[no head, from=2-1, to=1-1]
	\arrow[color={rgb,255:red,214;green,92;blue,92}, no head, from=2-1, to=1-2]
	\arrow[no head, from=2-1, to=2-2]
	\arrow[no head, from=2-2, to=1-2]
	\arrow[dotted, no head, from=2-3, to=1-3]
	\arrow[no head, from=2-3, to=2-4]
	\arrow[no head, from=2-4, to=1-4]
\end{tikzcd}
\end{center}
So instead of completing $a$ and $b$ to $x_0$, we use the corner ${\bbcor}^2_\omega$ at $\omega=(0,1)$ to obtain a completion
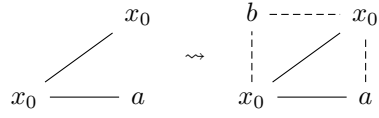
\begin{center}
\begin{tikzcd}
	& {x_0} & {b} & {x_0} \\
	{x_0} & a & {x_0} & a
	\arrow[dashed, no head, from=1-3, to=1-4]
	\arrow[no head, from=2-1, to=1-2]
	\arrow[no head, from=2-1, to=2-2]
	\arrow[""{name=0, anchor=center, inner sep=0}, draw=none, from=2-2, to=1-2]
	\arrow[""{name=1, anchor=center, inner sep=0}, dashed, no head, from=2-3, to=1-3]
	\arrow[no head, from=2-3, to=1-4]
	\arrow[no head, from=2-3, to=2-4]
	\arrow[dashed, no head, from=2-4, to=1-4]
	\arrow["\leadsto"{description}, draw=none, from=0, to=1]
\end{tikzcd}
\end{center}
which then produces the desired inverse element $b$.
\end{proof}

\begin{remark}\label{rem:structure-mon-not-grouplike-gamma}
For $n\geq 2$ there is no possibility of defining an inverse in $\pi_n^{\bbGamma}(X,x_0)$
analogous to the case $n=1$.
If, for example in the 3-cube we omit the vertex $\omega = (0,1,1)$,
the corner ${\bbcor}^3_\omega$ looks as follows.
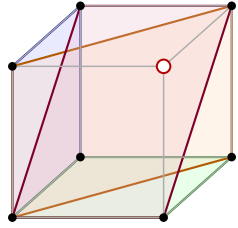
\begin{center}
\begin{tikzpicture}[
    line cap=round,
    line join=round,
    edge/.style={draw=gray!65, line width=0.55pt},
    frontedge/.style={draw=black, line width=0.65pt},
    vertex/.style={circle, fill=black, inner sep=1.15pt},
    missing/.style={circle, draw=red!70!black, fill=white, line width=0.75pt, inner sep=1.8pt},
    faceA/.style={fill=blue!35,   draw=blue!60!black,   fill opacity=0.23, line width=0.8pt},
    faceB/.style={fill=green!35,  draw=green!50!black,  fill opacity=0.23, line width=0.8pt},
    faceC/.style={fill=orange!35, draw=orange!70!black, fill opacity=0.23, line width=0.8pt},
    faceD/.style={fill=purple!30, draw=purple!65!black, fill opacity=0.23, line width=0.8pt}
]

\begin{scope}[shift={(15.0,4.8)}]
  \coordinate (A) at (0,0);
  \coordinate (B) at (2,0);
  \coordinate (C) at (0,2);
  \coordinate (D) at (2,2);
  \coordinate (E) at (0.9,0.8);
  \coordinate (F) at (2.9,0.8);
  \coordinate (G) at (0.9,2.8);
  \coordinate (H) at (2.9,2.8);

  \filldraw[faceA] (A)--(C)--(G)--(E)--cycle;
  \filldraw[faceB] (A)--(B)--(F)--(E)--cycle;
  \filldraw[faceC] (A)--(C)--(H)--(F)--cycle;
  \filldraw[faceD] (A)--(B)--(H)--(G)--cycle;

  \draw[edge] (A)--(B)--(D)--(C)--cycle;
  \draw[edge] (E)--(F)--(H)--(G)--cycle;
  \draw[edge] (A)--(E);
  \draw[edge] (B)--(F);
  \draw[edge] (C)--(G);
  \draw[edge] (D)--(H);

  \node[vertex]  at (A) {};
  \node[vertex]  at (B) {};
  \node[vertex]  at (C) {};
  \node[missing] at (D) {};
  \node[vertex]  at (E) {};
  \node[vertex]  at (F) {};
  \node[vertex]  at (G) {};
  \node[vertex]  at (H) {};
\end{scope}
\end{tikzpicture}
\end{center}
Then one would like to set the lower face to be constant $x_0$,
the left face to be the element $a\in\pi_2(X,x_0)$ whose inverse we wish to construct,
and the diagonal $[x_1=x_2]$ to be also constant $x_0$.
But then there is the diagonal $[x_1=x_3]$ left,
whose second inert (= lower) face has to be equal to the 1-dimensional diagonal of $a$,
while having constant $x_0$ on the opposite (= upper) face and on the first inert (= lower) face.
But there is no guarantee for this 2-cube to exist in $X$.

The failure can be fixed if $X$ arises from a concrete cubeset
since then the points suffice to determine the face $[x_1=x_3]$ uniquely.
In particular, in the proof of Proposition~\ref{prop:concrete-structure-group-is-group} no flip is required (if one performs the same Definitions as before, beginning from Construction~\ref{cons:structure-monoid-concrete}).
\end{remark}

Let us close this section with the result that the structure groups of a concrete fibrant cubeset
can be equally computed over $\bbox$ or over $\bbGamma$.

\begin{proposition}
Let $X$ be a concrete fibrant cubeset and $x_0\in X$.
Then there is an isomorphism $\pi_n(X,x_0)\simeq\pi_n^{\bbGamma}(j^\ast X,x_0)$.
\end{proposition}

\begin{proof}
The $\bbGamma$-set $j^\ast(X/\sim_n)$ is $n$-step over $\bbGamma$:
since $X/\sim_n = \tau_n X$ is $n$-step over $\bbox$ by Proposition~\ref{prop:trunc_concfib},
it is $\Lambda_n$-local, and thus $j^\ast(X/\sim_n)$ is $\Lambda_{\bbGamma,n}$-local by concreteness
and Proposition~\ref{prop:nilfib-restricted-nilfib-concrete}.
Hence $\tau_n^{\bbGamma}(j^\ast(X/\sim_n),x_0)) = \pi_n^{\bbGamma}(j^\ast(X-/\sim_n),x_0)$
and from the explicit definition of structure monoid over $\bbGamma$
it is immediate that $\pi_n^{\bbGamma}(j^\ast(X/\sim_n),x_0) = \pi_n(X/\sim_n,x_0)$.

Thus, it suffices to show that
$\pi_n^{\bbGamma}(j^\ast(X/\sim_n),x_0)\simeq\pi_n^{\bbGamma}(j^\ast X,x_0)$.
By the universal property of $\tau_n^{\bbGamma}j^\ast X = X/\approx_n$,
see Proposition~\ref{prop:trunc-concfib-gamma}, there is a commutative diagram
\begin{center}
\begin{tikzcd}
	{j^\ast X} & {X/\approx_n} \\
	& {j^\ast(X/\sim_n)}
	\arrow["{\eta^{\bbGamma}}", from=1-1, to=1-2]
	\arrow["{j^\ast(\eta)}"', from=1-1, to=2-2]
	\arrow["f", dashed, from=1-2, to=2-2]
\end{tikzcd}.
\end{center}
We show that the monoid homomorphism
\begin{align*}
    f_\ast\colon\{c\in (X/\approx_n)(n) \mid c\rvert_{{\bbcor}^n_1}\equiv x_0 \}
    &\to\{c\in j^\ast(X/\sim_n)(n) \mid c\rvert_{{\bbcor}^n_1}\equiv x_0 \}, \\
    c&\mapsto f(c)
\end{align*}
is bijective.
For surjectivity take $c\in(X/\sim_n)(n)$ with $c\rvert_{{\bbcor}^n}\equiv x_0$.
Take $c'\in X(n)$ with $\eta(c') = c$.
Using universal replacement we can assume that $c'(0) = x_0$.
This then yields $d = \eta^{\bbGamma}(c') \in (X/\approx_n)(n)$ with $d\rvert_{{\bbcor}^n_1}\equiv x_0$
and $f(d) = c$.
For injectivity let $c,d\in\pi_n^{\bbGamma}(X/\approx_n,x_0)$ with $f(c) = f(d)$.
Take $c',d'\in (j^\ast X)(n) = X(n)$ with $\eta^{\bbGamma}(c') = c$, $\eta^{\bbGamma}(d') = d$,
in particular $c'(0) = x_0 = d'(0)$.
Moreover,
\[
    \eta(c') = f(\eta^{\bbGamma}(c')) = f(c) = f(d) = \eta(d').
\]
This means that for all $\omega\neq 0$ we have that $c'(\omega) \sim_n d'(\omega)$.
But this implies that $c = \eta^{\bbGamma}(c') = \eta^{\bbGamma}(d') = d$.
So $f_\ast$ is a monoid isomorphism.
\end{proof}

\begin{corollary}
If $X$ is a concrete fibrant cubeset, then $\pi_n^{\bbGamma}(j^\ast X,x_0)$ is grouplike for all $n\in\N$.
\end{corollary}

\subsection{The Postnikov Tower of a Cubeset}\label{ssec:postnikov}

Now we have all the techniques at hand which allow one to prove the weak structure theorem.
This theorem gives insight in the structure of concrete fibrant cubesets.
It appears as \cite[Theorem 3.2.19]{Candela2017} and \cite[Corollary 7.20]{Gutman2020}
and in both cases its proof covers large parts of the papers.

\begin{theorem}[Weak Structure Theorem]\label{thm:weak-structure}
Let $X$ be a concrete ergodic fibrant cubeset.
Then there exists a tower of surjective fibrations
\begin{center}
\begin{tikzcd}
	X & \cdots & {\tau_2X} & {\tau_1X} & {\tau_0X}
	\arrow[from=1-1, to=1-2]
	\arrow[from=1-2, to=1-3]
	\arrow[from=1-3, to=1-4]
	\arrow[from=1-4, to=1-5]
\end{tikzcd}
\end{center}
such that $\tau_n X$ is a concrete fibrant $n$-step cubeset
and $\tau_n X$ is a principal $\cD_n(\pi_n(X,x_0))$-bundle over $\tau_{n-1} X$.
This tower is natural in $X$.
\end{theorem}

By principal $\cD_n(\pi_n(X,x_0))$-bundle we mean principal bundle in the wide sense,
i.e., without the assumption of local triviality.
In general, given a group object $G$ in a category $\mcC$,
a \emph{$G$-torsor} is an object $X$ together with a $G$-action $\alpha\colon G\times X\to X$
that is \emph{free} and \emph{transitive}, meaning that the \emph{shear map}
\[
    \alpha\times\pi_2\colon G\times X\to X\times X,
\]
which arises from the action $\alpha$ and the projection $\pi_2$ onto $X$, is an isomorphism.
Informally, surjectivity corresponds to transitivity and injectivity is freeness.
Now a \emph{principal $G$-bundle over $X$} is a relative $G$-torsor over a base space $X$,
i.e., a $G$-torsor in the slice category $\mcC_{/X}$.
Spelled out, this means that it is a map $p\colon P\to X$ together with a $G$-action
$\alpha\colon G\times P\to P$ over $X$ such that the shear map
\[
    \alpha\times\pi_2\colon G\times P\to P\times_X P
\]
is an isomorphism and the diagram
\begin{center}
\begin{tikzcd}
	{G\times P} & P & X
	\arrow["\alpha", shift left, from=1-1, to=1-2]
	\arrow["{\pi_2}"', shift right, from=1-1, to=1-2]
	\arrow["p", from=1-2, to=1-3]
\end{tikzcd}
\end{center}
is a coequalizer diagram.
Note that if $\mcC$ is a (Grothendieck) topos,
these conditions are equivalent to asserting that $p\colon P\to X$ is an epimorphism
and that the shear map is an isomorphism, by Giraud's axioms, see \cite[Appendix]{MacLane1994}.
And we say that $p$ is \emph{locally trivial} if there exists a cover $U\to X$ such that
there is a pullback diagram of the form
\begin{center}
\begin{tikzcd}
	{G\times U} & P \\
	U & X
	\arrow[from=1-1, to=1-2]
	\arrow["{\pi_2}"', from=1-1, to=2-1]
	\arrow["p", from=1-2, to=2-2]
	\arrow[from=2-1, to=2-2]
\end{tikzcd}.
\end{center}
If $\mcC$ is an $\infty$-topos, the situation becomes substantially easier
since colimits are better behaved.
Namely, it suffices to require that $X = P/G$
and then the shear map automatically is an equivalence,
and furthermore the principal bundle is locally trivial, see \cite{Nikolaus2014}.

\begin{remark}
If $X$ is not necessarily ergodic, then we can decompose $X$ into its ergodic components,
see Proposition~\ref{prop:ergodic-decomp}.
Applying the weak structure theorem to each of these components and noting the truncation
$\tau_n$ preserves the coproduct, see Proposition~\ref{prop:n-trunc-closure},
we obtain the weak structure theorem for non-ergodic cubesets.
However, then the fibration $\tau_nX\to\tau_{n-1}X$ is not necessarily a principal bundle for a group anymore,
but rather a principal groupoid bundle for the groupoid $\pi_n(X)$.

Since this variant is notationally more elaborate we won't formulate it here.
But we note that morally, $\pi_n(X)$ shouldn't just be a groupoid, but an $n$-groupoid,
and the whole $X$ an $\infty$-groupoid (= anima).
This view on $X$ is taken in \cite{Bihlmaier2026a}.
\end{remark}

To establish the weak structure theorem take a concrete fibrant ergodic cubeset $X$.
Let us look at the unit $\eta\colon X\to\tau_{n-1}X$
where $X$ is $n$-step.
Since $\eta$ is a surjective fibration we obtain a fiber sequence for every $x_0\in X$:
\begin{center}
\begin{tikzcd}
	{(E,x_0)} & {(X,x_0)} & {(\tau_{n-1}X,x_0)}
	\arrow["e", from=1-1, to=1-2]
	\arrow["\eta", from=1-2, to=1-3]
\end{tikzcd}.
\end{center}

\begin{proposition}\label{prop:fiber-ergodic-trunc}
The cubeset $E$ is a concrete fibrant cubeset that is both $n$-ergodic and $n$-step.
Moreover, $e_\ast\colon\pi_n(E,x_0)\to\pi_n(X,x_0)$ is an isomorphism.
\end{proposition}

\begin{proof}
Subobjects of concrete cubesets are concrete, so $E$ is concrete.

Now for $k>n$, given a $k$-corner $\lambda$ in $E$, one can complete it in $\tau_nX$ to a $k$-cube $c$ (since $\tau_nX$ is fibrant).
Under $\eta$, the $k$-corner $e\circ\lambda$ of $c$ is mapped to $x_0$.
By unique corner completion in $\tau_{n-1}X$ we must have that $\eta\circ c \equiv x_0$ and so $c$ factors through $e$, i.e., $c\in E$.
\begin{center}
\begin{tikzcd}
	{{\bbcor}^k} & E & {X} & {\tau_{n-1}X} \\
	{{\bbox}^k} & \ast & \ast & \ast
	\arrow["\lambda", from=1-1, to=1-2]
	\arrow[hook', from=1-1, to=2-1]
	\arrow["e", from=1-2, to=1-3]
	\arrow[from=1-2, to=2-2]
	\arrow["\eta", from=1-3, to=1-4]
	\arrow[from=1-3, to=2-3]
	\arrow[from=1-4, to=2-4]
	\arrow[dashed, from=2-1, to=1-2]
	\arrow["c"{description, pos=0.3}, dashed, from=2-1, to=1-3]
	\arrow[from=2-1, to=2-2]
	\arrow["{=}"{description}, draw=none, from=2-2, to=2-3]
	\arrow["{=}"{description}, draw=none, from=2-3, to=2-4]
\end{tikzcd}
\end{center}
Uniqueness of the completion in $E$ follows from uniqueness in $X$.

Given $x,y\in E(0)$ we have that $\eta(x) = \eta(y) = x_0$ in $\tau_{n-1}X$
and so $x$ and $y$ are ($n-1$)-equivalent in $X$, i.e., $x\sim_{n-1} y$.
Using universal replacement~\ref{prop:univ-replacement} we see that
we can arbitrarily interchange $x$ and $y$ in every $k$-cube of $X$ for $k\leq n$.
Therefore, every map $\{0,1\}^k\to E(0)$ is in $E(k)$ and so $E(k) \simeq \prod_{0\to k} E(0)$ for $k\leq n$,
hence $E$ is $n$-ergodic.
From the previous two items it also follows that $E$ is fibrant.

The last assertion follows from the fact that $X$ is $n$-step,
hence $\pi_n(X,x_0) = \{c\in X(n)\mid c\rvert_{{\bbcor}^n}\equiv x_0\}$.
Given such a $c\in\pi_n(X,x_0)$ we have that $\tau_{n-1}(c)\equiv x_0$ and thus $c\in E(n)$.
\end{proof}

Now it remains to show that $E\simeq\cD_n(A)$ for some abelian group $A$.
Which abelian group is available? Of course the structure group $\pi_n(E,x_0) = \pi_n(X,x_0)$.

\begin{proposition}\label{prop:fiber-torsor}
Let $E$ be an $n$-ergodic $n$-step concrete cubeset.
Let $x_0\in E$ and $A = \pi_n(E,x_0)$.
Then there is an action $\alpha\colon\cD_n(A)\times E\to E$ such that $E$ is an $\cD_n(A)$-torsor.
\end{proposition}

\begin{proof}
We have that $A = \pi_n(E,x_0) = \{ c\in E(n) \mid c\rvert_{{\bbcor}^n} \equiv x_0 \}$ because $E$ is $n$-step.
Since $E$ is $n$-ergodic and concrete, evaluation at the uppermost vertex induces a bijection $A\simeq E(0)$.
This allows us to define an action
\[
    \alpha_0\colon A\times E(0)\to E(0)
\]
which is just given by applying the bijection $E(0)\simeq A$ and then acting by left multiplication on $A$,
and afterwards forgetting again to $E(0)$.
For $k\leq n$ this defines by pushforward of $\alpha_0$ an action
\[
    \alpha_k\colon\cD_n(A)(k)\times E(k) = \hom_\set(\{0,1\}^k,A\times E(0)) \to \hom_\set(\{0,1\}^k,E(0)) = E(k)
\]
natural in $k$ (which is just the pointwise action).
For $k\geq n+1$ recall that we can write
\[
    {\bbcor}^k = \varinjlim_{j\in\cJ}{\bbox}^j
\]
where $\cJ$ runs over all inert inclusions ${\bbox}^j\to{\bbox}^k$ for $j<k$.
Suppose we have defined the action $\alpha_m\colon\cD_n(A)(m)\times E(m)\to E(m)$ for all $m<k$ natural in $m$.
This yields a map
\[
    \varprojlim_{j\in\cJ}\cD_n(A)(j)\times E(j) \to \varprojlim_{j\in\cJ} E(j)
\]
which then after using the isomorphism
\begin{align*}
    \cD_n(A)(k)\times E(k)
    &= \hom({\bbox}^k,\cD_n(A)\times E) \\
    &\simeq \hom({\bbcor}^k,\cD_n(A)\times E) \\
    &= \varprojlim_{j\in\cJ}\hom({\bbox}^j,\cD_n(A)\times E) \\
    &= \varprojlim_{j\in\cJ}\cD_n(A)(j)\times E(j)
\end{align*}
determines an action $\alpha_k\colon\cD_n(A)(k)\times E(k)\to E(k)$
(it is an action since limits commute with limits and, in particular, products) natural in $m\leq k$.
Putting every dimension together, we obtain an action
\[
    \alpha\colon\cD_n(A)\times E\to E.
\]
We go on to show that the shear map is an isomorphism.
This is clear for $k = 0$ using the bijection $A\simeq E(0)$.
But then inductively it follows for all $k\in\N$
because pushforward and taking limits is functorial, and so preserves isomorphisms.
Since $E$ is non-empty, we infer that also $E\simeq\cD_n(A)$ as cubesets by general nonsense:
every non-empty $G$-torsor is (non-canonically) isomorphic to $G$.
\end{proof}

The action $\alpha$ on $E$ does, a priori, depend on the choice of basepoint $x_0$.
Since $\pi_n(E,x_0)\simeq\pi_n(E,y_0)$ for any choice of basepoints $x_0,y_0\in E$,
we obtain an action of $\pi_n(E,y_0)$ on $E$.
Moreover, the action is compatible with this isomorphism in the following sense.
Recall from Proposition~\ref{prop:structure-monoid-independence-base-point}
that if $x_0,y_0\in E$ are two points then there is a canonical group isomorphism
$h_{(x_0,y_0)}\colon\pi_n(E,x_0)\to\pi_n(E,y_0)$.

\begin{lemma}\label{lem:action-independence-basepoint}
Let $E$ be an $n$-ergodic $n$-step concrete cubeset.
Let $x_0,y_0\in E$ and $A = \pi_n(E,x_0)$, $B = \pi_n(E,y_0)$.
Then the diagram
\begin{center}
\begin{tikzcd}
	{\cD_n(A)\times E} && {\cD_n(B)\times E} \\
	& E
	\arrow["{h\times\id}", from=1-1, to=1-3]
	\arrow["\alpha"', from=1-1, to=2-2]
	\arrow["\beta", from=1-3, to=2-2]
\end{tikzcd}
\end{center}
is commutative where $h = \cD_n(h_{(x_0,y_0)})$
and $\alpha$, $\beta$ are defined as in (the proof of) Proposition~\ref{prop:fiber-torsor}.
\end{lemma}

\begin{proof}
Since $E$ and $\cD_n(A)$ are concrete it suffices to see commutativity on level $k=0$
which means that
\[
    \beta_0\circ(h_{(x_0,y_0)}\times\id_0) = \alpha_0.
\]
Let $x\in E(0)$ and $a\in A$.
Write $x_{y_0}\in E(n)$ for the corresponding element in $\pi_n(E,y_0)$,
and analogously $x_{x_0}\in\pi_n(E,x_0)$ for $x_0$.
By definition,
\[
    \beta_0\circ(h_{(x_0,y_0)}\times\id_0)(a,x) = (h_{(x_0,y_0)}(a)*x_{y_0})(1^n)
    \quad\text{and}\quad
    \alpha_0(a,x) = (a*x_{x_0})(1^n).
\]
We compute that
\begin{align*}
       (a*x_{x_0})(1^n)
    &= (h_{(x_0,x)}(a))(1^n) \\
    &= (h_{(y_0,x)}(h_{(x,y_0)}(h_{(x_0,x)}(a))))(1^n) \\
    &= (h_{(x,y_0)}(h_{(x_0,x)}(a))*x_{y_0})(1^n) \\
    &= (h_{(x_0,y_0)}(a)*x_{y_0})(1^n),
\end{align*}
where we have used Lemma~\ref{lem:homotopy-action-well-defined}
in the first and third equality and Lemma~\ref{lem:homotopy-action-well-defined-global}
in the second and fourth equality.
\end{proof}

In the last step we show that the map $X\to\tau_{n-1}X$ is a principal $\cD_n(A)$-bundle
for $X$ $n$-step and $A = \pi_n(X,x_0)$.

\begin{proposition}\label{prop:unit-principal-bundle}
Let $X$ be an ergodic concrete fibrant $n$-step cubeset, $x_0\in X(0)$ and $A = \pi_n(X,x_0)$.
Then the unit $\eta\colon X\to\tau_{n-1}X$ is a principal $\cD_n(A)$-bundle.
\end{proposition}

\begin{proof}
We begin by defining an action $\alpha_0\colon A\times X(0)\to X(0)$.
Given $x\in X(0)$, by ergodicity we have that $(x_0,x)\in X(1)$.
Recall the monoid isomorphism $h_{(x_0,x)}\colon\pi_n(X,x_0)\to\pi_n(X,x)$ from Proposition~\ref{prop:structure-monoid-independence-base-point}.
Then define
\[
    \alpha_0((a,x)) = h_{(x_0,x)}(a)(1^n)\in X(0).
\]
This defines an action of $A$ on $X(0)$:
since $h_{(x_0,x)}$ is a monoid homomorphism we have that
\[
    \alpha_0((e_{x_0},x)) = h_{(x_0,x)}(e_{x_0})(1^n) = e_x(1^n) = x.
\]
Associativity follows from Lemma~\ref{lem:homotopy-action-well-defined}:
given $a,b\in A$ and $x\in X(0)$ we obtain for $z_0 = \alpha_0(b,x) = h_{(x_0,x)}(b)(1^n)$ that
$z_0\sim_{n-1} x$ and thus
\begin{align*}
       \alpha_0(a,\alpha_0(b,x))
    &= h_{(x_0,z_0)}(a)(1^n) \\
    &= h_{(x,z_0)}(h_{(x_0,x)}(a))(1^n) \\
    &= (h_{(x_0,x)}(a)*h_{(x_0,x)}(b))(1^n) \\
    &= h_{(x_0,x)}(a*b)(1^n)
    = \alpha_0(a*b,x)
\end{align*}
where we have used Lemma~\ref{lem:homotopy-action-well-defined-global} in the second equality.
Moreover, the action leaves the fibers of $\eta$ over $Y = \tau_{n-1}X$ invariant:
By definition, the $n$-cube $h_{(x_0,x)}(a)$ satisfies $h_{(x_0,x)}(a)(\omega) = x$ for all $\omega\neq 1$
and hence $\alpha_0(a,x)\sim_{n-1} x$, which in turn means that $\eta_0(\alpha_0(a,x)) = \eta_0(x)$.

Since $i_\ast\colon\set\to\cSet$ preserves finite products,
we obtain an action $i_\ast(\alpha_0)\colon i_\ast(A\times X(0))\to i_\ast X(0)$ of (concrete) cubesets.
Since $X$ and $\cD_n(A)$ are concrete it suffices to show that we obtain a factorization
\begin{center}
\begin{tikzcd}
	{\cD_n(A)\times X} & {i_\ast(A\times X(0))} \\
	X & {i_\ast X(0)}
	\arrow[hook, from=1-1, to=1-2]
	\arrow["\alpha"', dashed, from=1-1, to=2-1]
	\arrow["{i_\ast(\alpha_0)}", from=1-2, to=2-2]
	\arrow[hook, from=2-1, to=2-2]
\end{tikzcd}.
\end{center}
Then clearly the action has to arise pointwise,
so we are left to show that this is well-defined.
For $k\leq n$ we use the fact that $\alpha_0$ leaves the fibers of $\eta_0$ invariant
and so we can use universal replacement for $\sim_{n-1}$ to obtain the desired result.
For $k\geq n+1$ we use unique corner completion in $X$ to define the action
pointwise on the corner and then complete the cube,
analogously to the argument in the proof of Proposition~\ref{prop:fiber-torsor}.
This also directly shows that the action is invariant over $\eta$ by checking it pointwise
which means that $\eta\circ\alpha = p_Y\circ(\id\times\eta)$.

Moreover, the action has the following property, analogous to Lemma~\ref{lem:action-independence-basepoint}:
if $x_0,y_0\in X$ and $A = \pi_n(X,x_0)$, $B = \pi_n(X,y_0)$ then the diagram
\begin{center}
\begin{tikzcd}
	{\cD_n(A)\times X} && {\cD_n(B)\times X} \\
	& X
	\arrow["{h\times\id}", from=1-1, to=1-3]
	\arrow["\alpha"', from=1-1, to=2-2]
	\arrow["\beta", from=1-3, to=2-2]
\end{tikzcd}
\end{center}
is commutative where $h = \cD_n(h_{(x_0,y_0)})$.
This follows by checking it on points (since everything is concrete) from Lemma~\ref{lem:homotopy-action-well-defined-global}.

It remains to show that each fiber $E$ is an $\cD_n(A)$-torsor:
By invariance over $\eta$ there is an induced action $\cD_n(A)\times E\to E$.
First assume that $x_0\in E$.
Then from Lemma~\ref{lem:action-independence-basepoint} we see that
the restricted action agrees with the action defined in Proposition~\ref{prop:fiber-torsor}.
In particular, $E$ is an $\cD_n(A)$-torsor.
Otherwise, shift the action by $h$ to a basepoint in $E$ using the previous observation
and then argue analogously.
\end{proof}

\begin{proof}[Proof of Theorem~\ref{thm:weak-structure}]
The existence of the tower of fibrations is clear:
since adjunctions and hence units compose, each unit $X\to\tau_n X$ and $\tau_n X\to\tau_{n-1} X = \tau_{n-1}\tau_n X$
are surjective fibrations by Proposition~\ref{prop:trunc_concfib}.
The same proposition also asserts that $\tau_n X$ is a concrete fibrant $n$-step cubeset.
Naturality in $X$ follows from functoriality of $\tau_n$.

It remains to show that $\tau_n X\to\tau_{n-1} X$ is a principal $\cD_n(\pi_n(X,x_0))$-bundle.
First note that since $X$, and hence $\tau_n X$, is ergodic, $A = \pi_n(X,x_0)$ doesn't depend on $x_0$
by Proposition~\ref{prop:structure-monoid-independence-base-point}.
Moreover, $\pi_n(X,x_0) = \pi_n(\tau_n X,x_0)$ by Proposition~\ref{prop:trunc-structure-monoid}
and so Proposition~\ref{prop:unit-principal-bundle} gives the claim.
\end{proof}

Note that, different from the situation of simplicial sets and their Postnikov towers,
this construction of the tower does not converge as it identifies too much information.

\begin{example}
Let $S$ be a nontrivial set and let $X = i_\ast S$, i.e., $X(n) = S^{2^n}$.
Then $\tau_n X = \ast$ for every $n\in\N_0$ and so the canonical map $X\to\varprojlim_n\tau_nX = \ast$ is not an isomorphism.
\end{example}

However, it is a weak equivalence.

\begin{definition}
A map $f\colon X\to Y$ between cubesets is called a \emph{weak equivalence}
if it induces a bijection $f_\ast\colon\pi_0(X)\to\pi_0(Y)$ and
monoid isomorphisms $f_\ast\colon\pi_n(X,x_0)\to\pi_n(Y,f(x_0))$ for all $x_0\in X$ and $n\in\N$.
\end{definition}

\begin{remark}[Model Category Structures]
It is a very canonical question to ask whether there exists the structure of a model category on the category of cubesets that has this notion of weak equivalences, and/or the notion of fibration from Definition~\ref{def:nil-fibration} as fibrations.
Whilst we are not able to fully answer this question, let us remark on some obstructions,
indicating that for the present choice of fibrations from Definition~\ref{def:nil-fibration}
the answer might be negative.
\begin{enumerate}[(i)]
    \item There is no model structure on $\cSet$ with the present choice of fibrations with cofibrations being the class of monomorphisms.
		This follows from Cisinski's \emph{model structures ex nihilo} machinery in \cite{Cisinski2025}: The subobject classifier $\Omega$
		has the problem that the trivial cofibrations $\ell(r(\Lambda))$ do not consist of $\Omega$-anodyne extensions,
since in general $\ihom(\Omega, X)\to X\times_Y \ihom(\Omega, Y)$ is not surjective.
    \item We conjecture that the structure groups of the fibrant replacement of the point, $\ast_{\mathrm{fib}}$, are non-trivial.
		Thus, there should exist a trivial cofibration, e.g., $\ast\to \ast_{\mathrm{fib}}$,
		that does \emph{not} induce isomorphisms on structure groups.
		This shows that the present choice of fibrations together with the canonical choice of weak equivalences as those maps that induce isomorphisms on all structure groups
		does not induce a model structure.
	\end{enumerate}

	We were not able to disprove if using either normal monomorphisms as cofibrations, or using
	\[
    \Gamma = \{{\bbox}^n\sqcup_{{\bbcor}^n}{\bbox}^n\to{\bbox}^{n+1}\sqcup_{{\bbcor}^{n+1}}{\bbox}^{n+1} : n\in\N_0\cup\{-1\}\}
\]
	as generating cofibrations (see Remark~\ref{rem:nil-cofibrations}), yields a (useful) model structure on $\cSet$.
\end{remark}

\section{Condensed Cubesets}\label{sec:condensed-cubeset}

As of now we have only considered the category $\cSet$ of \enquote{discrete} cubesets without adhering to a topology
on the collections of $n$-cubes.
However, a great deal of the theory of nilspaces deals with \emph{topological} cubesets,
called \emph{cubespaces}, especially compact cubespaces, in order to appropriately treat Host-Kra theory.
After all, dynamical systems should always preserve a topological or measure theoretical structure!
Therefore, in \cite{Gutman2020}, the theory of nilspaces is treated topologically from the beginning,  making the arguments more involved,
and in \cite{Candela2017a}, the arguments from \cite{Candela2017} are revisited again to show that they are consistent with the topological aspects.
In both cases, every construction has to be checked for continuity of maps,
or the Hausdorff property of quotients, etc.
Moreover, the arguments usually are only considered for compact second-countable Hausdorff spaces (= metrizable compact spaces).

In this section we go a different route to install a topology on a cubeset
that is more global and doesn't adhere to the concrete constructions made in the previous sections.
One of the main differences to classical approaches is that we don't work with topological spaces,
but rather a replacement thereof: condensed sets.
The category of condensed sets $\CondSet$ is categorically much better behaved
than the category of topological spaces as it behaves just like a (Grothendieck) topos
(see also Remark~\ref{rem:size-issues}) on the one hand,
and like a category of algebras of an algebraic theory on the other hand.
Nevertheless, the category of condensed sets contains the category of compactly generated weak Hausdorff spaces as a full subcategory
and thus there is no loss of information when one wants to recover (most) topological spaces.
The good categorical properties, especially its definition as a category of sheaves on a coherent site,
allow one to apply the theory of presentable ($\infty$-)categories.
We achieve the topological enrichment by tensoring the category $\cSet$ of cubesets
with the category of condensed sets $\CondSet$ in $\PrL$, the $\infty$-category of presentable $\infty$-categories.

\subsection{Tensor Products of Presentable Categories}

We begin by recalling the tensor product of presentable $(\infty,1)$-categories from \cite[Section 4.8.1]{Lurie2017},
see also \cite[Section 0713]{Kerodon2026}.
In the following, $\infty$-category will always mean $(\infty,1)$-category.
Denote by $\PrL$ the $\infty$-category consisting of presentable categories
whose morphisms are functors $\cC\to\cD$ preserving small colimits, see \cite[Section 5.5]{Lurie2009}.
By the adjoint functor theorem, any colimit preserving functor $\cC\to\cD$ between presentable categories is left adjoint.
We denote the full sub-$\infty$-category of $\Fun(\cC,\cD)$ consisting of colimit preserving functors by $\Fun^\uL(\cC,\cD)$.

The category $\PrL$ has a symmetric monoidal tensor product $\otimes\colon\PrL\times\PrL\to\PrL$.
It represents separately colimit preserving functors, i.e., there is an equivalence of $\infty$-categories
\[
    \Fun^\uL(\cC\otimes\cD,\cE) \simeq \Fun_{\mathrm{colim}\times\mathrm{colim}}(\cC\times\cD,\cE)
\]
where the right hand side denotes the full sub-$\infty$-category of functors $\cC\times\cD\to\cE$
that preserve colimits in both variables separately.
Moreover, the tensor structure is closed with internal Hom object $\Fun^\uL(\cC,\cE)$,
i.e., there is a natural equivalence
\[
    \Fun^\uL(\cC\otimes\cD,\cE) \simeq \Fun^\uL(\cD,\Fun^\uL(\cC,\cE)).
\]
We denote the tensor unit of $\PrL$ by $\Ani$ (which is the free cocompletion of a point).
To give an explicit description of the tensor product we need the following definition.

\begin{definition}\label{def:cont-fun}
Let $\cK$ be a family of small $\infty$-categories and $S,\cD$ $\infty$-categories with all limits of forms in $\cK$.
Then we define $\Fun^\cK(S,\cD)$ to be the full sub-$\infty$-category of $\Fun(S,\cD)$ spanned by functors preserving $\cK$-indexed limits.
\end{definition}

\begin{example}\label{ex:cont-fun-cats}
\begin{enumerate}[(i)]
    \item If $\cK$ consists of all small $\infty$-categories and $S=\cC^\op$ with $\cC$ a presentable $\infty$-category,
            then we have that $\Fun^\cK(\cC^\op,\cD)$ consists of all limit preserving functors $\cC^\op\to\cD$
            which equivalently are right adjoint functors, and we denote this $\infty$-category by $\Fun^\uR(\cC^\op,\cD)$.
    \item If $\cK$ is empty, then $\Fun^\cK(S,\cD) = \Fun(S,\cD)$.
    \item If $\cK$ consists of all finite $\infty$-categories, $S$ is small and $\cD$ a presentable $\infty$-category,
            then we denote $\Fun^\cK(S,\cD)$ by $\Fun^\mathrm{lex}(S,\cD)$ (= left exact functors)
            which again is a presentable $\infty$-category.
    \item If $\cK$ consists of all finite sets, $S$ is small and $\cD$ a presentable $\infty$-category,
            we write $\Fun^\times(S,\cD)$ for $\Fun^\cK(S,\cD)$ which is the presentable $\infty$-category of finite product preserving functors.            
\end{enumerate}
\end{example}

\begin{proposition}\label{prop:pres-tens-prod}
Let $\cC$ and $\cD$ be presentable $\infty$-categories.
Then there is a canonical equivalence
\[
    \cC\otimes\cD \simeq \Fun^\uR(\cC^\op,\cD) \simeq \Fun^\uL(\cC,\cD^\op)^\op
\]
\end{proposition}

\begin{proof}
This is \cite[Proposition 4.8.1.17]{Lurie2017}.
\end{proof}

\begin{corollary}\label{cor:pres-tens-prod}
Let $\cK$ be a collection of small $\infty$-categories and $S$ a small $\infty$-category with limits of form in $\cK$.
For $\cC = \Fun^\cK(S,\Ani)$ we have that
\[
    \cC\otimes\cD \simeq \Fun^\cK(S,\cD).
\]
\end{corollary}

\begin{proof}
We have that
\[
    \cC\otimes\cD\simeq\Fun^\uR(\cC^\op,\cD)\simeq\Fun^\uL(\Fun^\cK((S^\op)^\op,\Ani),\cD^\op)^\op\simeq\Fun_\cK(S^\op,\cD^\op)^\op\simeq\Fun^\cK(S,\cD)
\]
where the second to last equivalence follows from \cite[Theorem 5.3.6.2]{Lurie2009}.
\end{proof}

We will also need the following lemma on the functoriality of the tensor product.
In some special cases, this is \cite[Section 2.2]{Haine2025}
but we weren't able to find a complete reference, so we record it here for completeness.

\begin{lemma}\label{lem:pres-tens-prod-fun}
Let $\cK$ be a family of small $\infty$-categories, $S$ a small $\infty$-category and $\cC = \Fun^\cK(S,\Ani)$.
If $f^\ast\colon\cD\to\cD'$ is a map in $\PrL$ that preserves limits of forms in $\cK$,
then the diagram
\begin{center}
\begin{tikzcd}
	{\cC\otimes\cD} & {\cC\otimes\cD'} \\
	{\Fun^\cK(S,\cD)} & {\Fun^\cK(S,\cD')}
	\arrow["{\id\otimes f^\ast}", from=1-1, to=1-2]
	\arrow["\simeq"', from=1-1, to=2-1]
	\arrow["\simeq", from=1-2, to=2-2]
	\arrow["{(f^\ast)\circ -}"', from=2-1, to=2-2]
\end{tikzcd}
\end{center}
is commutative.
\end{lemma}

\begin{proof}
By \cite[Remark 071F]{Kerodon2026}, under the equivalence $\cC\otimes\cD\simeq\Fun^\uR(\cC^\op,\cD)$,
the right adjoint part of $\id\otimes f^\ast$ is given by postcomposition
\[
    f_\ast\circ -\colon\Fun^\uR(\cC^\op,\cD')\to\Fun^\uR(\cC^\op,\cD)
\]
with the right adjoint $f_\ast$ of $f^\ast$.
Thus, under the equivalence $\cC\otimes\cD\simeq\Fun^\cK(S,\cD)$,
the right adjoint of $\id\otimes f^\ast$ is also given by postcomposition with $f_\ast$.
But the left adjoint of postcomposition with $f_\ast$ on the whole functor category is given by
\[
    f^\ast\circ -\colon\Fun(S,\cD)\to\Fun(S,\cD').
\]
Since $f^\ast$ preserves $\cK$-indexed limits, it restricts to a functor
\[
    f^\ast\circ -\colon\Fun^\cK(S,\cD)\to\Fun^\cK(S,\cD')
\]
which then also is the left adjoint to $f_\ast\circ -$.
\end{proof}

Furthermore, the tensor product of presentable categories behaves nicely under truncation.

\begin{example}\label{ex:pres-tens-prod}
\begin{enumerate}[(i)]
    \item If $n\geq -2$, then $\cC\otimes\tau_n\Ani = \tau_n\cC$
        where $\tau_n$ denotes the left adjoint to the inclusion of $n$-truncated objects in $\cC$ and $\Ani$, respectively.
        In particular, $\tau_n\cC\otimes\cD = \cC\otimes\tau_n\cD = \tau_n\cC\otimes\tau_n\cD$.
    \item For $n = 0$ we have that $\tau_0\Ani = \set$ and so $\cC\otimes\set = \tau_0\cC$.
    \item If $\cC$ is a 1-category and $\cD$ any $\infty$-category, then
        $\cC\otimes\cD = \tau_0\cC\otimes\cD = \cC\otimes\tau_0\cD$ which is a 1-category.
    \item The 1-category $\set$ is an idempotent algebra in $\PrL$.
        By the previous we have that
        \[
            \set\otimes\set = \set\otimes\tau_0\Ani = \tau_0\set = \set.
        \]
    \item In particular, the tensor product restricts to 1-categories:
        We have that, if $\cC$ and $\cD$ are presentable 1-categories, then $\cC\otimes_\set\cD = \cC\otimes\cD$,
        and $\set$ is the tensor unit there.
        For this special case see also \cite{Bird1984}.
\end{enumerate}
\end{example}

\subsection{Condensed Cubesets}

To pass from \enquote{discrete} cubesets to \enquote{topological} cubesets we want to apply this tensor product
to the case where $\cC$ is the category of cubesets and $\cD$ is the category of condensed sets.
Let us first recall the category of condensed sets.

We implicitly fix an uncountable strong limit cardinal $\kappa$.

\begin{definition}\label{def:cond}
\begin{enumerate}[(i)]
    \item Denote by $\extr$ the category of extremally disconnected compact Hausdorff spaces (and continuous maps)
            and by $\extr_\kappa$ its full subcategory of spaces with topological weight less than $\kappa$.
    \item We denote the sifted Ind-completion of $\extr_\kappa$ by $\CondAni$, the $\infty$-category of \emph{($\kappa$-)condensed anima}.
            Its category of (homotopy!) discrete objects (0-truncation) is denoted by $\CondSet$,
            the category of \emph{($\kappa$-)condensed sets}.
\end{enumerate}
\end{definition}

Condensed sets have been introduced in \cite{Scholze2026,Scholze2026a,Clausen2026}
with a particular focus on applications in analytic and complex geometry.
For a thorough introduction to the basics of condensed sets see \cite{Bihlmaier2025}.

\begin{proposition}\label{prop:cond-sheaves}
There are equivalences of ($\infty$-)categories
\[
    \CondAni\simeq\Fun^\times(\extr_\kappa^\op,\Ani)\quad\text{and}\quad\CondSet\simeq\Fun^\times(\extr_\kappa^\op,\set).
\]
\end{proposition}

\begin{proof}
This follows from \cite[Proposition 5.5.8.15, Remark 5.5.8.16]{Lurie2009} in the $\Ani$-valued case
and \cite[Corollary 2.8]{Adamek2001} in the $\set$-valued case.
\end{proof}

\begin{remark}\label{rem:size-issues}
At first it might seem arbitrary to fix a strong limit cardinal $\kappa$.
We mainly do so in order to ensure that $\CondAni$ and $\CondSet$ are presentable
since they are ($\infty$-, resp. 1-)sheaf categories on a small site
and thus we can simply apply all aspects of the theory of presentable $\infty$-categories without further ado.
However, one can certainly take the colimit over all $\kappa$ and obtain the \enquote{actual} category of condensed anima/sets,
see also \cite[Remark 1.4]{Scholze2026} and \cite[Theorem 2.4.2]{Bihlmaier2025}.
Then, one has to work with accessible sheaves on $\extr$ and extend the theory of tensor product of presentable categories
to the slightly more general setting to which condensed ($\infty$-)categories belong.
\end{remark}

\begin{remark}\label{rem:cond-sheaves-comparison}
By the comparison lemma, see \cite[Theorem 2.3.10]{Bihlmaier2025}, \cite[Appendix, Corollary 4.3]{MacLane1994} and \cite[Remark 4.4.8]{Kuijper2026},
the $\infty$-category of condensed anima is equivalent to the $\infty$-category of hypersheaves on the site $\prolak$.
Here, $\prolak$ is the category of profinite sets (= Stone spaces) of weight less than $\kappa$,
and covers are given by finite families of jointly surjective maps.
\end{remark}

As usual with a (1-)topos, there are certain finiteness and separation conditions, see \cite[Chapter 2.5]{Bihlmaier2025}.

\begin{definition}\label{def:cond-qcqs}
\begin{enumerate}[(i)]
    \item A cover of a condensed set $X$ is a family $(U_i\to X)_{i\in I}$ such that the map $\coprod_{i\in I}U_i\to X$ is surjective.
    \item A condensed set $X$ is \emph{quasicompact} if every cover has a finite subcover.
    \item A condensed set $X$ is \emph{quasiseparated} if for any two maps $f,g\colon K\to X$ with $K$ quasicompact,
            the pullback $K\times_X K$ is quasicompact.
\end{enumerate}
\end{definition}

A condensed set can be thought of as a replacement for a topological space.
Namely, the assignment
\[
    X \mapsto (E\mapsto C(E,X))
\]
maps a topological space $X$ to a condensed set: its $E$-valued points are simply the continuous functions $E\to X$,
see \cite[Section 2.4]{Bihlmaier2025}.
Thus, this functor is faithful and, if $X$ is ($\kappa$-)compactly generated, then all those functions recover $X$ completely.
This includes locally ($\kappa$-)compact and metrizable spaces.

\begin{proposition}\label{prop:cond-top-compactly-gen}
The functor $\Top\to\CondSet$ is fully faithful on $\kappa$-compactly generated topological spaces
and induces an equivalence of the category $\CHaus_\kappa$ of compact Hausdorff spaces of weight less than $\kappa$
and the full subcategory $\qcqs$ of quasicompact quasiseparated condensed sets.
\end{proposition}

\begin{proof}
See \cite[Corollary 2.3.29]{Bihlmaier2025} for the first, and \cite[Theorem 2.5.46]{Bihlmaier2025} for the second statement.
\end{proof}

\begin{definition}\label{def:cond-cube}
The category of \emph{condensed cubesets} is $\cSet\otimes\CondAni = \cSet\otimes\CondSet$.
We denote this category by $\cCond$.
\end{definition}

\begin{proposition}\label{prop:cond-cube-descr}
Explicitly, the category $\cCond$ of condensed cubesets is canonically equivalent to any of the following categories.
\begin{enumerate}[(a)]
    \item The category of right adjoint functors $\cSet^\op\to\CondSet$,
    \item the category of right adjoint functors $\CondSet^\op\to\cSet$,
    \item the category of functors ${\bbox}^\op\to\CondSet$ \emph{(cubical condensed sets)},
    \item the category of finite product preserving functors $\extr_\kappa^\op\to\cSet$ \emph{(condensed cubesets)},
    \item the category of functors $\extr_\kappa^\op\times{\bbox}^\op\to\set$ that preserve finite products in the first variable.
\end{enumerate}
\end{proposition}

\begin{proof}
The first two follow from Proposition~\ref{prop:pres-tens-prod}.
The equivalence of $(a)$ and $(c)$ follows from Corollary~\ref{cor:pres-tens-prod}, as does the equivalence of $(b)$ and $(d)$.
The equivalence with $(e)$ then follows from currying.
\end{proof}

\begin{remark}
We have already considered a similar construction when talking about cubegroups:
those form the category of functors ${\bbox}^\op\to\Grp$, and so can be described as $\cSet\otimes\Grp$.
Thus, considering topological cubesets (= cubespaces) is essentially the same as working with cubegroups (or any other algebraic object).
The reason for working in condensed sets is that those behave just like any other algebraic theory,
unlike topological spaces, and so admit a description as a sheaf category.
\end{remark}

Recall the inclusion $i\colon\ast\to\bbox$ and the corresponding adjoint functors
$i_!\dashv i^\ast\dashv i_\ast$ from Lemma~\ref{lem:point-cset}.
Also, recall that for a condensed set $X$, the functor given by evaluation at the point $X\mapsto X(\ast)$
admits a left adjoint, given by sending a set $S$ to the sheafification of the constant presheaf
$E\mapsto S$ for $E\in\extr$.
This left adjoint is denoted $S_\ud$ since it agrees with the set $S$
as discrete topological space, viewed as condensed set, see \cite[Lemma 2.4.38]{Bihlmaier2025}.

\begin{lemma}
\begin{enumerate}[(i)]
\item The adjunction $i_!\dashv i^\ast$ induces an adjunction
        $i_!\otimes\id\colon\CondSet\rightleftarrows\cCond\lon i^\ast\otimes\id$.
        Moreover, the diagram
        \begin{center}
        \begin{tikzcd}
            	{\set\otimes\CondSet} & {\cSet\otimes\CondSet} \\
            	\\
            	{\Fun^\times(\extr_\kappa^\op,\set)} & {\Fun^\times(\extr_\kappa^\op,\cSet)}
            	\arrow[""{name=0, anchor=center, inner sep=0}, "{i_!\otimes\id}"', from=1-1, to=1-2]
            	\arrow["\simeq"', from=1-1, to=3-1]
            	\arrow[""{name=1, anchor=center, inner sep=0}, "{i^\ast\otimes\id}"', curve={height=12pt}, from=1-2, to=1-1]
            	\arrow["\simeq", from=1-2, to=3-2]
            	\arrow[""{name=2, anchor=center, inner sep=0}, "{i_!\circ -}"', from=3-1, to=3-2]
            	\arrow[""{name=3, anchor=center, inner sep=0}, "{i^\ast\circ -}"', curve={height=12pt}, from=3-2, to=3-1]
            	\arrow["\dashv"{anchor=center, rotate=84}, draw=none, from=0, to=1]
            	\arrow["\dashv"{anchor=center, rotate=84}, draw=none, from=2, to=3]
        \end{tikzcd}
        \end{center}
        is commutative.
\item The adjunction $(-)_\ud\dashv(-)(\ast)$ induces an adjunction
        $\id\otimes(-)_\ud\colon\cSet\rightleftarrows\cCond\lon \id\otimes(-)(\ast)$.
        Moreover, the diagram
        \begin{center}
        \begin{tikzcd}
	        {\cSet\otimes\set} & {\cSet\otimes\CondSet} \\
            	\\
            	{\Fun({\bbox}^\op,\set)} & {\Fun({\bbox}^\op,\CondSet)}
            	\arrow[""{name=0, anchor=center, inner sep=0}, "{\id\otimes(-)_\ud}"', from=1-1, to=1-2]
            	\arrow["\simeq"', from=1-1, to=3-1]
            	\arrow[""{name=1, anchor=center, inner sep=0}, "{\id\otimes(-)(*)}"', curve={height=12pt}, from=1-2, to=1-1]
            	\arrow["\simeq", from=1-2, to=3-2]
            	\arrow[""{name=2, anchor=center, inner sep=0}, "{(-)_\ud\circ -}"', from=3-1, to=3-2]
            	\arrow[""{name=3, anchor=center, inner sep=0}, "{(-)(*)\circ -}"', curve={height=12pt}, from=3-2, to=3-1]
            	\arrow["\dashv"{anchor=center, rotate=69}, draw=none, from=0, to=1]
            	\arrow["\dashv"{anchor=center, rotate=68}, draw=none, from=2, to=3]
        \end{tikzcd}
        \end{center}
        is commutative.
\end{enumerate}
\end{lemma}

\begin{proof}
That the adjunctions $i_!\dashv i^\ast$ and $(-)_\ud\dashv(-)(\ast)$ upgrade to adjunctions
between the respective tensor products follows from \cite[Section I.1.6.1]{Gaitsgory2017}
where it is shown that the tensor product of presentable $(\infty,1)$-categories
extends to a symmetric monoidal structure on the $(\infty,2)$-category $\PrL$
(where the proof in the unstable case works as in the stable case),
see also \cite[Section 1.2]{Haine2025}.
In particular, the tensor product preserves adjunctions in $\PrL$.

The second claim in (i) follows from Lemma~\ref{lem:pres-tens-prod-fun}
since both $i_!$ and $i^\ast$ preserve limits.
The second claim in (ii) also follows from Lemma~\ref{lem:pres-tens-prod-fun} since $\cK =\emptyset$ in this case.
\end{proof}

\begin{definition}
A condensed cubeset in the essential image of the functor $(-)_\ud\colon\cSet\to\cCond$ is called a \emph{discrete condensed cubeset},
or simply \emph{cubeset}.
\end{definition}

This abuse of notation is justified by the fact that $(-)_\ud$ is fully faithful,
and so its essential image is equivalent to $\cSet$.

Furthermore, there is also a functor $i_!\colon\CondSet\to\cCond$.
This yields, in particular, a Yoneda embedding
\[
    \extr_\kappa\times\bbox\to\cCond,\quad (E,n)\mapsto i_!E\times{\bbox}^n_\ud
\]
since it is given by the composition
\begin{center}
\begin{tikzcd}
	{\extr_\kappa\times\bbox} & {\Fun((\extr_\kappa\times\bbox)^\op,\set)=\Fun(\extr_\kappa^\op,\cSet)} & \cCond.
	\arrow["y", from=1-1, to=1-2]
	\arrow["\Sh", from=1-2, to=1-3]
\end{tikzcd}
\end{center}
To see those two functors to be equal,
note that the Yoneda lemma on the level of presheaves takes the form
\[
    (E,n)\mapsto\left((S,m)\mapsto C(S,E)\times\hom({\bbox}^m,{\bbox}^n)\right)
\]
and since sheafification commutes with products and is computed pointwise on $\bbox$,
we have that
\begin{align*}
       \Sh\left(C(-,E)\times\hom({\bbox}^m,{\bbox}^n)\right)
    &= C(-,E)\times\Sh(\hom({\bbox}^m,{\bbox}^n)) \\
    &= C(-,E)\times C(-,\hom({\bbox}^m,{\bbox}^n)) \\
    &= C(-,E\times\hom({\bbox}^m,{\bbox}^n)).
\end{align*}
The classical Yoneda lemma also gives rise to an enriched Yoneda lemma.
First, note that the Hom-set $\hom(X,Y)$ of two condensed cubesets $X$ and $Y$
naturally has the structure of a condensed set:
its $E$-valued points are given by $\hom(X\times i_!E,Y)$.
This defines a condensed set since $i_!$ preserves coproducts as left adjoint,
so we have that
\[
    X\times i_!(E\sqcup E') = X\times (i_!E\sqcup i_!E') = (X\times i_!E)\sqcup(X\times i_!E')
\]
for $E,E'\in\extr_\kappa$,
where we have used that coproducts are stable under base change in $\CondSet$
and hence by pointwise computation also in $\cCond$.

\begin{lemma}[Enriched Yoneda Lemma]\label{lem:yoneda-enriched}
Let $X$ be a condensed cubeset and $n\in\N_0$.
Then there is an isomorphism of condensed sets
\[
    X(n) \simeq \hom({\bbox}^n_\ud,X).
\]
\end{lemma}

\begin{proof}
We show the claim pointwise on $E\in\extr_\kappa$.
There we have that
\[
    \hom({\bbox}^n_\ud,X)(E) = \hom({\bbox}^n_\ud\times i_!E,X) = X(E,n),
\]
where the last equality follows from the classical Yoneda lemma.
\end{proof}

In the following subsections we will revisit the main notions of cubesets introduced throughout the first chapters
and explain what happens to them after tensoring with $\CondSet$.

\subsection{Concrete Cubesets}

The following proposition allows us to translate properties of discrete cubesets to condensed cubesets.

\begin{proposition}
The concrete cubeset adjunction $(-)^\mathrm{conc}\colon\cSet\rightleftarrows\ccSet$
induces by tensoring with $\CondSet$ an adjunction
$(-)^\mathrm{conc}\colon\cCond\rightleftarrows\cCond^\mathrm{conc}$
with fully faithful right adjoint.
The category $\cCond^\mathrm{conc}$ is canonically identified with the full subcategory of $\cCond$
spanned by objects $X$ for which the map
\[
    X(n) \to \prod_{0\to n} X(0)
\]
is an injection of condensed sets for every $n\in\N$.
Moreover, the left adjoint $(-)^\mathrm{conc}$ is given by
\[
    X^\mathrm{conc}(n) = \im\left(X(n)\to\prod_{0\to n}X(0)\right).
\]
\end{proposition}

\begin{proof}
The existence of the adjunction is clear by functoriality in $\PrL$.
Since both the inclusion $\cSet^\mathrm{conc}\to\cSet$ and its left adjoint $(-)^\mathrm{conc}$ preserve finite products
by Proposition~\ref{prop:concrete-ihom-prod},
under the equivalence $\cSet\otimes\CondSet\simeq\Fun^\times(\extr_\kappa^\op,\cSet)$,
the tensored adjunction is given by pushforward of the original one by Lemma~\ref{lem:pres-tens-prod-fun}.
This implies the rest of the claims.
\end{proof}

\subsection{Fibrant Cubesets}

Since the functor $(-)_\ud\colon\cSet\to\cCond$ preserves colimits,
the corner ${\bbcor}^n_\ud$ is the colimit in $\cCond$ over all inert inclusions ${\bbox}^k_\ud\to{\bbox}^n_\ud$ for $k<n$.

\begin{definition}
A condensed cubeset $X$ is \emph{fibrant} if the map
\[
    X(n) = \hom({\bbox}^n_\ud,X)\to\hom({\bbcor}^n_\ud,X) = \varprojlim_{\substack{{\bbox}^k\to{\bbox}^n\\k<n}}X(k)
\]
is a surjection of condensed sets.
\end{definition}

Equivalently, $X$ is fibrant if for each $E\in\extr_\kappa$, the corresponding cubeset $X(E)$ is fibrant.
Similarly, one defines fibrations of condensed cubesets.

Recall the set of maps
\[
    \bprism\coloneq \{ {{\bbox}^{n}\sqcup_{{\bbox}^{n-1}}{\bbox}^{n}}\to{({\bbox}^{n}\sqcup_{{\bbox}^{n-1}}{\bbox}^{n})\sqcup_{{\bbox}^{n-1}\sqcup{\bbox}^{n-1}}{\bbox}^n} \mid n\in\N \}
\]
from Definition~\ref{def:gluing-property}.

\begin{definition}
A condensed cubeset $X$ has the \emph{gluing property} (\emph{unique gluing property}, \emph{at most unique gluing property}) if the pullback
\[
    f^\ast\colon\hom(B,X)\to\hom(A,X)
\]
is an epimorphism (isomorphism, monomorphism) of condensed sets for every $f\in\bprism$.
\end{definition}

Equivalently, $X$ has the gluing property if for each $E\in\extr_\kappa$, the cubeset $X(E)$ has the gluing property (and likewise for the other properties).

\begin{proposition}
Let $X$ be a condensed cubeset.
Then the following conditions are equivalent.
\begin{enumerate}[(a)]
    \item It has at most unique gluing.
    \item It is concrete.
\end{enumerate}
Moreover, if $X$ is fibrant, then it has the gluing property.
\end{proposition}

\begin{proof}
This follows from testing pointwise against $E\in\extr_\kappa$.
\end{proof}

\subsection{Truncated Cubesets}

Next, we investigate $n$-step condensed cubesets.

\begin{proposition}
Let $n\in\N_0$.
The $n$-truncation $\tau_n\colon\cSet\rightleftarrows\cSet_n$ induces by tensoring with $\CondSet$ an adjunction
$\tau_n\colon\cCond_n\rightleftarrows\cCond$ with fully faithful right adjoint.
The category $\cCond_n$ is canonically identified with the full subcategory of $\cCond$ of $n$-step condensed cubesets $X$:
those for which the map
\[
    \hom({\bbox}^k_\ud,X)\to\hom({\bbcor}^k_\ud,X)
\]
is an isomorphism for every $k>n$.
Moreover, the left adjoint can be computed pointwise on $E\in\extr_\kappa$.
In particular, if $X$ is concrete and fibrant, then the unit $\eta\colon X\to\tau_n X$ is a surjective fibration
and $\tau_n X$ is concrete and fibrant.
\end{proposition}

\begin{proof}
The existence of the $\CondSet$-enriched adjunction follows from functoriality of the tensor product.
Both the inclusion $\cSet_n\to\cSet$ and its left adjoint $\tau_n$ preserve finite products, see Proposition~\ref{prop:n-trunc-ihom-prod}.
Therefore, the adjunction is simply given by pushforward under the equivalence $\cCond\simeq\Fun^\times(\extr_\kappa^\op,\cSet)$
by Lemma~\ref{lem:pres-tens-prod-fun}.
This implies fully faithfulness of the right adjoint and the computation of the left adjoint pointwise on $E\in\extr_\kappa$.
The subcategory $\cCond_n$ is canonically identified with those $X\in\cCond$ such that $X(E)\in\cSet_n$ for every $E\in\extr_\kappa$,
i.e. the map
\[
    X(E)(k) = \hom({\bbox}^k,X(E))\to\hom({\bbcor}^k,X(E)) = (\varprojlim_{j\in\cJ} X(E)(j)) = (\varprojlim_{j\in\cJ} X(j))(E)
\]
is a bijection for every $k>n$.
But this is precisely the condition that the map
\[
    \hom({\bbox}^k_\ud,X)\to\hom({\bbcor}^k_\ud,X)
\]
is an isomorphism of condensed sets for every $k>n$ by the enriched Yoneda lemma~\ref{lem:yoneda-enriched}.
The last claim follows by the pointwise statement on $\extr_\kappa$ from Proposition~\ref{prop:trunc_concfib}.
\end{proof}

\begin{proposition}\label{prop:cond-cube-trunc-qcqs}
Let $X$ be a concrete fibrant condensed cubeset.
If $X$ is (pointwise) qcqs, then so is $\tau_n X$.
\end{proposition}

\begin{proof}
Since the unit $\eta\colon X\to\tau_n X$ is surjective, $\tau_n X$ is quasicompact.
For quasiseparatedness it suffices to see that $\tau_n X(0)$ is quasiseparated:
since $X$ is concrete, the map $X(n)\to\prod_{0\to n}X(0)$ is injective
and subobjects and limits of quasiseparated condensed sets are quasiseparated by \cite[Lemma 2.5.47]{Bihlmaier2025}.
Let $Y = X(0)$ and $Z = \tau_nX(0)$.
Thus it suffices to see that $Z$ is quasiseparated,
which is equivalent to the kernel pair $Y\times_Z Y$ being a closed subobject of $Y\times Y$,
i.e. the inclusion $Y\times_Z Y\to Y\times Y$ is quasicompact.
Since $Y\times Y$ is compact Hausdorff, we need to see that $Y\times_Z Y$ is given by a closed subset of $Y\times Y$.
Consider the pullback square
\begin{center}
\begin{tikzcd}
	W & {\hom({\bbox}^{n+1}_\ud,X)} \\
	{\hom({\bbox}^{n+1}_\ud,X)} & {\hom({\bbcor}^{n+1}_\ud,X)}
	\arrow[from=1-1, to=1-2]
	\arrow[from=1-1, to=2-1]
	\arrow[from=1-2, to=2-2]
	\arrow[from=2-1, to=2-2]
\end{tikzcd}.
\end{center}
The condensed set $W$ is qcqs since all other vertices of the diagram are,
and we have that $Y\times_Z Y$ is the image of the map
\[
    W\to X(n+1)\times X(n+1)\to X(0)\times X(0) = Y\times Y
\]
by Proposition~\ref{prop:trunc_concfib}.
But images of qcqs objects in quasiseparated objects are qcqs again, so that $Y\times_Z Y$ is closed in $Y\times Y$.
\end{proof}

\begin{definition}
A condensed cubeset $X$ is \emph{$n$-ergodic} if $\tau_{n-1} X = \ast$.
\end{definition}

In particular, if $X$ is concrete and fibrant, then $X$ is $n$-ergodic if and only if $X(n) = X^{2^n}$.

\subsection{Structure Groups}

There also is the enrichment on structure groups.
For the following statement, we have to leave the setting of presentable categories and work externally in $\Cat$.

\begin{proposition}
Let $n\in\N$.
The functor $\pi_n\colon\cSet_\ast\to\cMon$ upgrades to a functor
\[
    \pi_n\colon\cCond_\ast\to\cMon(\CondSet) = \Cond(\cMon).
\]
We have that $\pi_n(X,x_0)(E) = \pi_n(X(E),x_0)$.
If $X$ is concrete fibrant, then $\pi_n(X,x_0)$ is a condensed abelian group.
If $X$ additionally is qcqs, so is $\pi_n(X,x_0)$.
\end{proposition}

\begin{proof}
The functor $\pi_n$ induces by pushforward a functor $\Fun(\extr_\kappa^\op,\cSet_\ast)\to\Fun(\extr_\kappa^\op,\cMon)$.
Since $\pi_n$ preserves finite products by Proposition~\ref{prop:structure-monoid-prod}
it restricts to a functor
\[
    \cSet_\ast\otimes\CondSet\simeq\Fun^\times(\extr_\kappa^\op,\cSet_\ast)\to\Fun^\times(\extr_\kappa^\op,\cMon)\simeq\cond(\cMon).
\]
If $X$ is qcqs, then $\tau_n X$ is as well, see Proposition~\ref{prop:cond-cube-trunc-qcqs},
and so the last claim follows from the fact that if $X$ is $n$-step, then $\pi_n(X,x_0)$ fits in the pullback square
\begin{center}
\begin{tikzcd}
	{\pi_n(X,x_0)} & \ast \\
	{\hom({\bbox}^n,X)} & {\hom({\bbcor}^n,X)}
	\arrow[from=1-1, to=1-2]
	\arrow[from=1-1, to=2-1]
	\arrow["{x_0}", from=1-2, to=2-2]
	\arrow[from=2-1, to=2-2]
\end{tikzcd}.
\end{center}
\end{proof}

\begin{remark}
The analogous statement holds for $\pi_0$:
since $\pi_0$ is left adjoint, tensoring with $\CondSet$ gives a functor $\pi_0\colon\cCond\to\CondSet$.
The functor is given by pointwise application of $\pi_0\colon\cSet\to\set$ since $\pi_0$ preserves finite products.
Moreover, it is computed by $\pi_0(X) = \tau_0 X(0)$.
\end{remark}

\begin{proposition}
Let $n\in\N$.
Then the functor $\cD_n\colon\Ab\to\cSet$ upgrades to a functor $\cD_n\colon\CondAb\to\cCond$
which factors through the subcategory of condensed abelian cubegroups.
We have that $\cD_n(A)(E) = \cD_n(A(E))$.
\end{proposition}

\begin{proof}
The pointwise pushforward $\cD_n\circ -\colon\Fun(\extr_\kappa^\op,\Ab)\to\Fun(\extr_\kappa^\op,\cSet)$
restricts to a functor $\cD_n\colon\CondAb\to\cCond$ since it preserves finite products by Remark~\ref{rem:EM_preserves_products}.
\end{proof}

\subsection{Weak Structure Theorem}

Now we are ready to state and prove a weak structure theorem for condensed cubesets.

\begin{theorem}\label{thm:cond-weak-structure}
Let $X$ be a condensed concrete ergodic fibrant cubeset.
Then there exists a tower of surjective fibrations
\begin{center}
\begin{tikzcd}
	X & \cdots & {\tau_2X} & {\tau_1X} & {\tau_0X}
	\arrow[from=1-1, to=1-2]
	\arrow[from=1-2, to=1-3]
	\arrow[from=1-3, to=1-4]
	\arrow[from=1-4, to=1-5]
\end{tikzcd}
\end{center}
such that $\tau_n X$ is a condensed concrete fibrant $n$-step cubeset
and $\tau_n X$ is a principal $\cD_n(\pi_n(X,x_0))$-bundle over $\tau_{n-1} X$.
This tower is natural in $X$.
\end{theorem}

Note that by projectivity of the point in $\cCond$, there is always a choice of some basepoint for nonempty $X$.

\begin{proof}
The theorem follows from the pointwise statement for the cubeset $X(E)$ from Theorem~\ref{thm:weak-structure}.
Note that the fact that the fibrations $\tau_n X\to\tau_{n-1} X$ are quotient maps
is automatic since in $\CondSet$, every epimorphism is regular.
Let us also remark that for $\tau_n X\to\tau_{n-1} X$ to be a principal $\cD_n(\pi_n(X,x_0))$-bundle,
it suffices to show that the shear map $\cD_n(\pi_n(X,x_0))\times\tau_n X\to\tau_n X\times_{\tau_{n-1}X}\tau_n X$ is an isomorphism,
since then the diagram
\begin{center}
\begin{tikzcd}
	{\cD_n(\pi_n(X,x_0))\times\tau_nX} & {\tau_n X} & {\tau_{n-1}X}
	\arrow[shift left, from=1-1, to=1-2]
	\arrow[shift right, from=1-1, to=1-2]
	\arrow[from=1-2, to=1-3]
\end{tikzcd}
\end{center}
is a coequalizer by surjectivity (= effective epimorphism) of $\tau_nX\to\tau_{n-1}X$.
\end{proof}

\begin{remark}
Recall that in this context, a principal $G$-bundle is a principal bundle in the wide sense,
i.e. without asserting local triviality of the action.
This is in line with \cite[Definition 2.1.6]{Candela2017a} and \cite[Theorem 5.4]{Gutman2020}
where no assumption on local triviality is made.
\end{remark}

By Proposition~\ref{prop:cond-cube-trunc-qcqs} the condensed weak structure theorem has the following corollary.

\begin{corollary}\label{cor:compact-weak-structure}
Let $X$ be a compact ergodic nilspace.
Then there is a tower of surjective fibrations
\begin{center}
\begin{tikzcd}
	X & \cdots & {\tau_2X} & {\tau_1X} & {\tau_0X}
	\arrow[from=1-1, to=1-2]
	\arrow[from=1-2, to=1-3]
	\arrow[from=1-3, to=1-4]
	\arrow[from=1-4, to=1-5]
\end{tikzcd}
\end{center}
such that $\tau_n X$ is a compact $n$-step nilspace
and the compact group $\pi_n(X,x_0)$ acts freely on $\tau_n X$
such that the orbits are the fibers of the quotient $\tau_n X\to\tau_{n-1} X$.
\end{corollary}

We have broken our convention and used classical terminology here,
the reason being an easier comparability to the classical theorem,
see \cite[Theorem 5.4]{Gutman2020} and \cite[Proposition 2.1.9]{Candela2017a}.
Note that a nilspace is an $n$-step concrete fibrant cubeset for some $n\in\N$,
see \cite[Definition 3.6]{Gutman2020} and \cite[Definition 1.2.1]{Candela2017}, \cite[Definition 1.0.2]{Candela2017a},
and as usual, $X$ is notationally identified with $X(0)$.
Moreover, compact is used to denote a compact second-countable Hausdorff space, see \cite[Definition 1.0.1]{Candela2017a}.

\begin{remark}
In general, there is no reason to believe that the principal bundle $\tau_n X\to\tau_{n-1}X$ is locally trivial,
i.e., the existence of a cover $U\to \tau_{n-1}X$ such that after base change to $U$,
the principal bundle is the trivial one $U\times G\to U$.
However, if $\pi_n(X,x_0)$ is a compact Lie group, then local triviality of the principal bundle is automatic
by \cite[Theorem 3.3]{Gleason1950}, see also \cite[Proposition 2.5.2]{Candela2017a}.
\end{remark}

\printbibliography[heading=bibintoc]
\end{document}